\documentclass[10pt]{amsart} % I find amsat easier to read on a computer, feel free to switch it back - B
\usepackage[top=10pt, bottom=10pt, margin=60pt]{geometry} 

\usepackage[demo]{graphicx}
\usepackage{tikz}
\usepackage[nice]{nicefrac}
\usepackage{amssymb}
\usepackage{epstopdf}
\usepackage{amsmath,amsthm,amscd,amssymb} %, mathrsfs} % I don't have the fonts for mathrsfs on my laptop - B
\usepackage{verbatim}
\allowdisplaybreaks
\usepackage{latexsym}
\usepackage[colorlinks,citecolor=red,pagebackref,hypertexnames=false]{hyperref}
\usepackage{mathtools}
\usepackage{esint}
\usepackage{enumerate}
\usepackage{bbm}
\usepackage{esint, enumitem, mathrsfs}
\usepackage{forest}

 \usepackage{todonotes}
\usepackage{MnSymbol,wasysym}

\usepackage[framemethod=tikz]{mdframed}
\mdfsetup{linecolor=blue,backgroundcolor=gray!10,roundcorner=10pt,leftmargin=2pt,innerleftmargin=2pt,innerrightmargin=2pt}
\usepackage{pgf,tikz, caption, float, tcolorbox, pgfplots}

\usepackage{subcaption}

\usepackage{cleveref}

\renewcommand\thesubfigure{(\alph{subfigure})}

\usepackage{tikz}
\usetikzlibrary{decorations.pathreplacing}

\newcommand{\RNum}[1]{\uppercase\expandafter{\romannumeral #1\relax}}

\newcommand\restr[2]{{% we make the whole thing an ordinary symbol
  \left.\kern-\nulldelimiterspace % automatically resize the bar with \right
  #1 % the function
  \vphantom{|} % pretend it's a little taller at normal size
  \right|_{#2} % this is the delimiter
  }}

\numberwithin{equation}{section}

\theoremstyle{plain}
\newtheorem{theorem}{Theorem}[section]
\newtheorem{lemma}[theorem]{Lemma}

\newtheorem{proposition}[theorem]{Proposition}

\theoremstyle{definition}

\theoremstyle{remark}
\newtheorem{remark}[theorem]{Remark}

\def\BBN {{\mathbb N}}
\def\BBZ {{\mathbb Z}}

\def\BBR {{\mathbb R}}
\def\BBC {{\mathbb C}}

\newcommand*{\rom}[1]{\expandafter\@slowromancap\romannumeral #1@}
\newcommand{\f}{\widehat}
\newcommand{\m}{\mathfrak{m}}

\newcommand{\ds}{\displaystyle}
\newcommand{\p}{\tau}
\newcommand{\h}{\mathscr{H}}
\newcommand{\el}{\mathscr{L}}

\newcommand{\ii}{\mathscr}

\title[\parbox{14cm}{\centering{}} \quad]{Multi-parameter Flag Leibniz Rules of Arbitrary Complexity in mixed-norm spaces}

\author[C. Benea]{Cristina Benea}
\address[C. Benea]{Universit\'e de Nantes, Laboratoire de Math\'ematiques Jean Leray, Nantes 44322, France}
\email{cristina.benea@univ-nantes.fr}

\author[Y. Zhai]{Yujia Zhai}
\address[Y. Zhai]{CNRS - Universit\'e de Nantes, Laboratoire de Math\'ematiques Jean Leray, Nantes 44322, France}
\email{yujia.zhai@univ-nantes.fr}

\begin{document}

\renewcommand{\thesubfigure}{\roman{subfigure}}

\forestset{
    my tree/.style={
      for tree={
        shape=circle,
        fill=black,
        minimum width=4pt,
        inner sep=1pt,
        anchor=center,
        line width=1pt,
        s sep+=20pt,
      },
    },
  }

\forestset{
    my treeSep/.style={
      for tree={
        shape=circle,
        fill=black,
        minimum width=4pt,
        inner sep=1pt,
        anchor=center,
        line width=1pt,
        s sep+=35pt,
      },
    },
  }

%\date{\today}

\begin{abstract}
We prove multi-parameter Leibniz rules corresponding to flag paraproducts of arbitrary complexity in mixed-norm spaces, including endpoint estimates. The proof relies on multi-linear harmonic analysis techniques and a quantitative treatment of the commutators introduced by Bourgain and Li. The argument is robust and applicable to a generic class of multipliers, including (symmetric) Mikhlin multipliers of positive order and asymmetric variants of partial differential operators and Mikhlin multipliers of positive order. %symbols either generated by partial differential operators or of Mikhlin type of positive orders.
\end{abstract}

%\footnote{\today}

\maketitle

\section{Introduction} \label{introduction}

\subsection{Motivation and Main results}
In this work, we study multi-parameter flag Leibniz rules, which include the particular bi-parameter example
{\fontsize{9}{9}
\begin{align} \label{bi_Muscalu}
&\|D_{(1)}^{\beta_1}D^{\beta_2}_{(2)}\big(D_{(1)}^{\alpha_1}D_{(2)}^{\alpha_2}(f_1 f_2) f_3 D_{(1)}^{\gamma_1}D_{(2)}^{\gamma_2}(f_4 f_5)\big) \|_{L^{\vec r}}  \\
 \lesssim 
& \|D_{(1)}^{\alpha_1+\beta_1}D_{(2)}^{\alpha_2+\beta_2} f_1\|_{L^{\vec p_1}} \|f_2\|_{L^{\vec p_2 }} \|f_3\|_{L^{\vec p_3}}\|D_{(1)}^{\gamma_1}D_{(2)}^{\gamma_2}f_4\|_{L^{\vec p_4}}\|f_5\|_{L^{\vec p_5}} + \| f_1\|_{L^{\vec p_1}} \|D_{(1)}^{\alpha_1+\beta_1}D_{(2)}^{\alpha_2+\beta_2}f_2\|_{L^{\vec p_2 }} \|f_3\|_{L^{\vec p_3 }}\|D_{(1)}^{\gamma_1}D_{(2)}^{\gamma_2}f_4\|_{L^{\vec p_4 }}\|f_5\|_{L^{\vec p_5 }} +\nonumber\\
& \|D_{(1)}^{\alpha_1+\beta_1}f_1\|_{L^{\vec p_1 }} \|D_{(2)}^{\alpha_2} f_2\|_{L^{\vec p_2}} \|D_{(2)}^{\beta_2}f_3\|_{L^{\vec p_3}}\|D_{(1)}^{\gamma_1}f_4\|_{L^{\vec p_4 }}\|D_{(2)}^{\gamma_2}f_5\|_{L^{\vec p_5}} + \|D_{(1)}^{\alpha_1}f_1\|_{L^{\vec p_1 }} \|D_{(2)}^{\alpha_2+\beta_2} f_2\|_{L^{\vec p_2 }} \|f_3\|_{L^{\vec p_3 }}\|D_{(1)}^{\gamma_1+\beta_1}f_4\|_{L^{\vec p_4}}\|D_{(2)}^{\gamma_2}f_5\|_{L^{\vec p_5 }} \nonumber \\
 &+ \ldots  \text{  140 other similar terms}. \nonumber
\end{align}}Above, $\beta_1, \beta_2, \alpha_1, \alpha_2, \gamma_1, \gamma_2 \geq 0$
and the spaces $L^{\vec{p}_i}(\BBR^{d_1} \times \BBR^{d_2})$ represent mixed-norm Lebesgue spaces $L^{p_i^1}_x(L_y^{p_i^2})$ (see Section \ref{sec:Section2}, definition \ref{eq:def:mixed:norm}) with $1 \leq p^1_i, p^2_i \leq \infty$ for all $1 \leq i \leq 5$. 
Additionally, we require that $$ \sum_{i=1}^5 \frac{1}{p^1_i} = \frac{1}{r^1}, \qquad \sum_{i=1}^5\frac{1}{p^2_i} = \frac{1}{r^2}, $$ and
\[
\begin{array}{ c c c }
 \frac{1}{r^2}< \frac{d_2+\beta_2}{d_2}, & \frac{1}{p_1^2}+\frac{1}{p_2^2}  <  \frac{d_2+\alpha_2}{d_2}, & \frac{1}{p_4^2}+\frac{1}{p_5^2}  <   \frac{d_2+\gamma_2}{d_2} ,\\ 
  \frac{1}{r^1}< \min( \frac{d_1+\beta_1}{d_1} , \frac{d_2+\beta_2}{d_2}), \qquad  & \frac{1}{p_1^1}+\frac{1}{p_2^1}  <  \min(\frac{d_1+\alpha_1}{d_1},   \frac{d_2+\alpha_2}{d_2}), \qquad & \frac{1}{p_4^1}+\frac{1}{p_5^1}  < \min( \frac{d_1+\gamma_1}{d_1}, \frac{d_2+\gamma_2}{d_2}).
\end{array}
\]

The Leibniz rule \eqref{bi_Muscalu} confirms the fact that fractional partial derivatives\footnote{For $\alpha_1, \alpha_2 \geq 0$, we consider the partial differential operators $D_{(1)}^{\alpha_1}, D_{(2)}^{\alpha_2}$, initially defined on the space $\mathcal{S}(\mathbb{R}^{d_1+d_2})$ of Schwartz functions -- via the Fourier transform -- by formulas \begin{align} \label{partial_diff}
D_{(1)}^{\alpha_1}f:= \ii F^{-1}\left(|\xi|^{\alpha_1}\f{f}(\xi,\eta)\right), \qquad D_{(2)}^{\alpha_2}f:= \ii F^{-1}\left(|\eta|^{\alpha_2}\f{f}(\xi,\eta)\right).
\end{align}
In contrast, the homogeneous differential operator $D^{\beta}$ is defined on the Schwartz space $\mathcal{S}(\mathbb{R}^d)$ by 
\begin{equation} \label{homo_diff}
D^{\beta} f := \mathscr{F}^{-1}\big(|\xi|^{\beta} \f{f}(\xi) \big).
\end{equation}
} acting in various ways on products of functions are properly distributed among the functions, provided these are elements of some mixed-norm $L^{\vec p}$ spaces, with $\vec p=(p^1, \ldots, p^N), 1 \leq p^1, \ldots, p^N \leq \infty$.

The difficulty -- and thus the interest -- of flag Leibniz rules such as \eqref{bi_Muscalu} resides in the fact that straightaway composition arguments are insufficient when the input functions are too close to $L^1$, in spite of them having the form of and distributing the derivatives as compositions of simpler Leibniz rules. This is a feature shared with the \emph{flag paraproducts} introduced by Muscalu in \cite{flag_paraproducts}, which ressemble compositions of Coifman-Meyer multipliers. In fact, Coifman-Meyer multipliers are usually invoked in the study of fractional Leibniz rules, and in particular the boundedness of the one-parameter flag paraproduct from \cite{flag_paraproducts} implies the one-parameter flag Leibniz rule.\footnote{Modulo some endpoints.}

When it comes to fractional partial derivatives acting independently on various variables -- as in \eqref{bi_Muscalu} -- a new layer of difficulty is added since the boundedness of multi-parameter flag paraproducts remains, to our knowledge, a difficult open problem. However, combining multilinear harmonic analysis techniques with the method introduced by Bourgain and Li \cite{BourgainLi-kato} for proving Leibniz rules for input data in $L^\infty$ (another situation in which Coifman-Meyer multipliers cannot be invoked), we are able to prove multi-parameter flag Leibniz rules of arbitrary complexity. 

Differences and similarities between the reduction of Leibniz rules to the boundedness of Coifman-Meyer multipliers and the Bourgain-Li method will be discussed in Section \ref{sec:differences:Coifman-Meyer}.

Leibniz rules of various types have been extensively investigated and widely used in nonlinear PDEs. The simplest Leibniz rule, acting on functions defined on $\BBR^d$, takes the form
\begin{align}\label{Leib_one_paraproduct}
\| D^{\beta} (f_1 f_2) \|_{L^r} \lesssim \|D^{\beta}f_1\|_{L^{p_1}}\|f_2\|_{L^{p_2}} + \|f_1\|_{L^{p_1}}\|D^{\beta}f_2\|_{L^{p_2}}
\end{align}
where 
\begin{align} \label{exp_Leib_one_paraproduct}
& 1 \leq p_1, p_2 \leq \infty, \quad  \frac{d}{d+\alpha} < r \leq \infty, \quad  \frac{1}{p_1} + \frac{1}{p_2} = \frac{1}{r}.
\end{align}
As mentioned above, the Coifman-Meyer theorem \cite{CoifMeyer-ondelettes} implies the Leibniz rule \eqref{Leib_one_paraproduct}, but only in the range 
\[
 1 < p_1, p_2 \leq \infty, \quad  \frac{d}{d+\alpha} < r < \infty, \quad  \frac{1}{p_1} + \frac{1}{p_2} = \frac{1}{r}.
\] 

Coifman-Meyer operators are associated in frequency to Mikhlin symbols: they are $n$-linear operators $T_m$ described\footnote{Given a frequency symbol $m$, we denote by $T_m$ the associated $n$-linear operator.} by
\[
T_m(f_1, \ldots, f_n):= \int_{\BBR^{nd}} m(\xi_1, \ldots, \xi_n) \hat f_1(\xi_1) \cdot \ldots \cdot \hat f_n(\xi_n) e^{2 \pi i x(\xi_1+\ldots+\xi_n)} d \xi_1 \ldots d \xi_n,
\]
where $m$ is a Mikhlin symbol (of order $0$) satisfying
\begin{equation}
\label{eq:Mikhlin:symbol}
\big| \partial_{\xi_1}^{\gamma_1} \ldots  \partial_{\xi_n}^{\gamma_n}  m(\xi_1, \ldots, \xi_n) \big| \lesssim  \big(|\xi_1|+ \ldots+ |\xi_n| \big)^{- |\gamma_1|- \ldots - |\gamma_n|}
\end{equation}
for sufficiently many\footnote{In certain situations, finding minimal regularity conditions for Mikhlin symbols becomes important. This will not be the case for our applications concerning Leibniz rules for homogeneous (or inhomogeneous, as we will see later) partial fractional differential operators, as the symbols involved will either be smooth or will be a suitable superposition of smooth symbols.} multi-indices $\gamma_1, \ldots, \gamma_n$. The class of Mikhlin symbols on $\BBR^{nd}$ will be denoted $\mathcal{M}(\BBR^{nd})$.

Since multilinear Coifman-Meyer operators (which are particular cases of multilinear Calderon-Zygmund operators) do not satisfy $L^\infty \times L^\infty \mapsto L^\infty$, or (strong type) $L^1 \times L^p \mapsto L^{p \over {p+1}}$ bounds for $1 \leq p \leq \infty$, a different approach is required for dealing with these endpoint estimates. This was introduced in \cite{BourgainLi-kato}, where Bourgain and Li proved %- among other things - 
that 
\begin{align*}
\| D^{\beta} (f_1 f_2) \|_{L^\infty} \lesssim \|D^{\beta}f_1\|_{L^{\infty}}\|f_2\|_{L^{\infty}} + \|f_1\|_{L^{\infty}}\|D^{\beta}f_2\|_{L^{\infty}},
\end{align*}
a result conjectured in \cite{GrafakosMaldonadoNaibo-Kato-endpoint}. In fact, the authors proved in \cite{BourgainLi-kato} a Kato-Ponce commutator estimate involving Besov norms\footnote{The Besov norms $\| \cdot\|_{\dot{B}^{\beta}_{p,\infty}}$ associated to the real parameter $\beta$ and to the Lebesgue exponent $1 \leq p \leq \infty$, will be explicitly defined in Section \ref{sec:LP&Besov} -- see \eqref{eq:def:Beson:1param}.}
\begin{align}
\label{Leib:comm:infty}
\| D^{\beta} (f_1 f_2) - (D^{\beta}f_1) f_2 - f_1 (D^{\beta}f_2) \|_{L^\infty} \lesssim \|f_1\|_{L^{\infty}}\|f_2\|_{\dot{B}^{\beta}_{\infty,\infty}} + \|f_1\|_{\dot{B}^{\beta}_{\infty,\infty}}\|f_2\|_{L^{\infty}}.
\end{align}
The appearance of Besov norms should already suggest that a scale-by-scale analysis will be performed, and that the estimates obtained in this way will be summed according to their magnitude; this is in sharp contrast with the approach for Coifman-Meyer multipliers, in which the orthogonality between different scales plays a crucial role.

The case when (at least) one of $p_1$ or $p_2$ is equal to $1$ was proved by Oh and Wu in \cite{OhWu}, by applying the methods introduced in \cite{BourgainLi-kato}. It represented a first instance of a (strong-type) Leibniz rule for input functions in $L^1$.

Bi-parameter Leibniz rules, such as
\begin{align}\label{Leib_bi_paraproduct}
\|D_{(1)}^{\beta_1} D_{(2)}^{\beta_2} (f_1f_2) \|_{L^{r}(\BBR^{d_1 +d_2})} \lesssim & \|D_{(1)}^{\beta_1}D_{(2)}^{\beta_2}f_1\|_{L^{p_1}(\BBR^{d_1 +d_2})}\|f_2\|_{L^{p_2}(\BBR^{d_1 +d_2})} +  \|f_1\|_{L^{p_1}(\BBR^{d_1 +d_2})}\|D_{(1)}^{\beta_1}D_{(2)}^{\beta_2}f_2\|_{L^{p_2}(\BBR^{d_1 +d_2})} + \nonumber \\
&  \|D_{(1)}^{\beta_1}f_1\|_{L^{p_1}(\BBR^{d_1 +d_2})}\|D_{(2)}^{\beta_2}f_2\|_{L^{p_2}(\BBR^{d_1 +d_2})} +  \|D_{(2)}^{\beta_2}f_1\|_{L^{p_1}(\BBR^{d_1 +d_2})}\|D_{(1)}^{\beta_1}f_2\|_{L^{p_2}(\BBR^{d_1 +d_2})},
\end{align}
at least for
\begin{align}
\label{eq:range:bi-param:CM}
&1 < p_1, p_2 \leq \infty,\quad \max\big(  \frac{d_1}{d_1+\beta_1},  \frac{d_2}{d_2+\beta_2} \big) < r < \infty, \quad \frac{1}{p_1} + \frac{1}{p_2} = \frac{1}{r},
\end{align}
are a consequence of bi-parameter paraproducts' boundedness within the same range \eqref{eq:range:bi-param:CM}, as proved in \cite{bi-parameter_paraproducts}. The endpoints $p_1=p_2=r=\infty$ and strong estimates in the case $p_1=1$ or $p_2=1$ for \eqref{Leib_bi_paraproduct} are contained in the work of Oh and Wu \cite{OhWu}.

Mixed-norm estimates for bi-parameter Leibniz rules were obtained more recently, as a consequence of mixed-norm estimates for bi-parameter paraproducts \cite{vv_BHT}, \cite{quasiBanachHelicoid}\footnote{Although stated for functions defined on $\BBR \times \BBR$, the Leibniz rules in \cite{vv_BHT}, \cite{quasiBanachHelicoid} remain valid in higher dimensions.}; more exactly, it was proved that the inequality
\begin{align}\label{Leib_bi_paraproduct:mixed}
\|D_{(1)}^{\beta_1} D_{(2)}^{\beta_2} (f_1f_2) \|_{L^{r^1}_x(L^{r^2}_y)} \lesssim & \|D_{(1)}^{\beta_1}D_{(2)}^{\beta_2}f_1\|_{L^{p_1^1}_x(L^{p_1^2}_y)}\|f_2\|_{L^{p_2^1}_x(L^{p_2^2}_y)} +  \|f_1\|_{L^{p_1^1}_x(L^{p_1^2}_y)}\|D_{(1)}^{\beta_1}D_{(2)}^{\beta_2}f_2\|_{L^{p_2^1}_x(L^{p_2^2}_y)}  \nonumber \\
&+  \|D_{(1)}^{\beta_1}f_1\|_{L^{p_1^1}_x(L^{p_1^2}_y)}\|D_{(2)}^{\beta_2}f_2\|_{L^{p_2^1}_x(L^{p_2^2}_y)} +  \|D_{(2)}^{\beta_2}f_1\|_{L^{p_1^1}_x(L^{p_1^2}_y)}\|D_{(1)}^{\beta_1}f_2\|_{L^{p_2^1}_x(L^{p_2^2}_y)},
\end{align}
holds for Lebesgue exponents satisfying
\begin{align}
&1 < p^i_1, p^i_2\leq \infty,\quad  \frac{1}{p_1^i} + \frac{1}{p_2^i} = \frac{1}{r^i}, \quad \text{for all $1 \leq i \leq 2$}, 
& \frac{d_2}{d_2+\beta_2} < r^2 <\infty, \quad \max \big(   \frac{d_1}{d_1+\beta_1},  \frac{d_2}{d_2+\beta_2} \big) < r^1 <\infty . \nonumber
\end{align}
The same result, including the endpoints $r^i=\infty$ (which forces $p_1^i=p_2^i= \infty$) for some $1 \leq i \leq 2$, or $p^i_j=1$ for some $1\leq i, j \leq 2$, were proved by Oh and Wu \cite{OhWu} using the Bourgain-Li method and thus avoiding mixed-norm estimates for Coifman-Meyer multipliers. Other partial results were obtained in \cite{UMDparaproducts}.

Multi-parameter multilinear operators are especially interesting; unlike their linear analogues, they take as input several functions, and yield as output only one function, so that linear techniques (freezing a variable, using vector-valued estimates) are not easily applicable. Mixed-norm estimates for multi-parameter multilinear operators present an additional difficulty, and in general they require sharper estimates for the concerned operator (localization, weighted estimates, etc; see \cite{vv_BHT}, \cite{quasiBanachHelicoid}).

The results above in \eqref{Leib_bi_paraproduct:mixed} can be extended to the $N$-parameter case; except for a few endpoints,\footnote{Certain complications appear when one tries to prove mixed-norm estimates for $N$-parameter paraproducts, when some of the input functions are in mixed-norm Lebesgue spaces involving $L^\infty$.} this result is implicit in \cite{vv_BHT}, \cite{quasiBanachHelicoid}. The full result, including $L^1$ and $L^\infty$ endpoints is implicit in \cite{OhWu}.

We would like to comment that although \eqref{Leib_one_paraproduct}, \eqref{Leib_bi_paraproduct} and \eqref{Leib_bi_paraproduct:mixed} above describe the bi-linear case, the $n$-linear case for $n > 2$ remains valid and it can be proved by the same methods.

The one-parameter flag Leibniz rule can be perceived as the Leibniz rule for compositions of fractional differential operators. The simplest example is
\begin{align} \label{Leib_one_flag}
 \|D^{\beta}\left(D^{\alpha}(f_1f_2) f_3\right)\|_{L^r}& \lesssim  \|D^{\alpha+\beta}f_1\|_{L^{p_1}}\|f_2\|_{L^{p_2}} \|f_3\|_{L^{p_3}} + \|f_1\|_{L^{p_1}}\|D^{\alpha+\beta}f_2\|_{L^{p_2}} \|f_3\|_{L^{p_3}}  \\
& +\|D^{\alpha}f_1\|_{L^{p_1}}\|f_2\|_{L^{p_2}} \|D^\beta f_3\|_{L^{p_3}} +\|f_1\|_{L^{p_1}}\|D^{\alpha}f_2\|_{L^{p_2}} \|D^\beta f_3\|_{L^{p_3}},  \nonumber
\end{align}
where 
\begin{align}\label{exp_Leib_one_flag}
& 1 \leq p_1,p_2, p_3 \leq \infty, \quad \frac{1}{p_1} + \frac{1}{p_2} + \frac{1}{p_3} = \frac{1}{r},  \quad  0 \leq  \frac{1}{r} < \frac{d+\beta}{d}, \quad 0 \leq  \frac{1}{p_1} + \frac{1}{p_2} < \frac{d+\alpha}{d}.
\end{align}
The Leibniz rule \eqref{Leib_one_flag} (except for endpoints $(p_1, p_2) = (\infty, \infty)$ or strong-type estimates when $p_i=1$ for some $1 \leq i \leq 3$) is a consequence of the boundedness of the one-parameter flag paraproduct \cite{flag_paraproducts}. The endpoint case $(p_1, p_2) = (\infty, \infty)$ can be derived by iteratively applying the endpoint estimate \eqref{Leib_one_paraproduct} due to Bourgain and Li \cite{BourgainLi-kato}.

It is worth pointing out that interpreting the Leibniz rule \eqref{Leib_one_flag} as a composition of two classical Leibniz rules (as described in \eqref{Leib_one_paraproduct}) and iteratively invoking \eqref{Leib_one_paraproduct} only yields a limited range of exponents, namely the case $1 \leq p_1, p_2, p_3 \leq \infty , 0 \leq \frac{1}{p_1} + \frac{1}{p_2} \leq 1$ and $\max(\frac{1}{2}, \frac{d}{d+\beta}) < r \leq \infty$. In order to achieve boundedness in the nontrivial range $1 < \frac{1}{p_1} + \frac{1}{p_2} <\min(2, \frac{d+\alpha}{d})$ and $\max(\frac{1}{3}, \frac{d}{d+\beta}) < r \leq \infty $, one can decompose $D^{\beta}\big(D^{\alpha}(f_1f_2) f_3\big)$ into a sum of flag paraproducts whose boundedness was proved by Muscalu \cite{flag_paraproducts} and later extended by Miyachi and Tomita \cite{MiyachiTomita-flag} to Hardy spaces input data.

The flag paraproduct should be thought of as compositions of Coifman-Meyer multipliers: in frequency, the associated symbol is a product of singular Mikhlin symbols
\begin{equation}
\label{eq:symbol:flag:paraprod}
m(\xi):= \prod_{S \subseteq \{1, \ldots, n  \}}m_S(\xi_S)
\end{equation}
where $\xi:= (\xi_i)_{i=1}^n \in \mathbb{R}^{dn}$, $\xi_S := (\xi_i)_{i \in S} \in \mathbb{R}^{d \cdot \text{card}(S)}$ and $m_S \in \mathcal{M}(\mathbb{R}^{d \cdot \text{card}(S)})$. Hence the singularity set associated to $m$ consists of unions of subspaces of various dimensions, which can be further organized into a union of ordered subspaces -- or flags. Flag paraproducts do not satisfy $L^\infty \times \ldots \times L^\infty  \to L^\infty$, or strong type estimates $L^1 \times L^{p_2} \times \ldots \times  L^{p_n} \to L^r$; nevertheless, we will see that the flag Leibniz rule remains true even in these particular situations. 

Although the boundedness of the generic multi-parameter flag paraproduct is still an open problem, as mentioned previously, a particular case of bi-parameter flag paraproducts was proved independently in \cite{MuscaluZhai} and \cite{LuPipherZhang_flag}.
 That leads to some specific example of bi-parameter Leibniz rules in the full range of boundedness, modulo certain endpoints. However, the generic multi-parameter versions of \eqref{Leib_one_flag} known before were those that could be obtained as a result of compositions of \eqref{Leib_bi_paraproduct} or \eqref{Leib_bi_paraproduct:mixed} -- which leaves out a significant range of Lebesgue exponents. We will show that the Bourgain-Li method can be adapted to proving multi-parameter flag Leibniz rules of arbitrary complexity.

Especially in the context of flag Leibniz rules (and of flag paraproducts, as one would expect), it becomes convenient to use \emph{rooted tree} representations. The $n$-linear Leibniz rule for
\begin{align} \label{Leib_1_generic}
D^{\beta}(f_1 \ldots f_n),
\end{align}
corresponds to the simplest tree
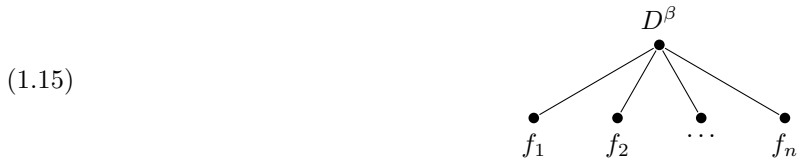
\begin{equation}
\label{tree:Leib_1_generic}
\vcenter{\hbox{\begin{forest}
my tree
[, label={above: $D^{\beta}$}
  	[, label = {below: $f_1$}
	]
	[, label = {below: $f_2$}
	]
	[, label= {below:$\ldots$}
	]
	[, label = {below: $f_n$}
	]
]
\end{forest}}}
\end{equation}
where the root of the tree is the vertex associated to the differential operator $D^\beta$ and the leaves of the tree are the vertices associated to the $n$ functions that the differential operator acts on.  

The multilinear expression $D^\beta(D^\alpha(f_1 \cdot f_2) f_3)$ on the left hand side of \eqref{Leib_one_flag} can be represented by the following rooted tree:
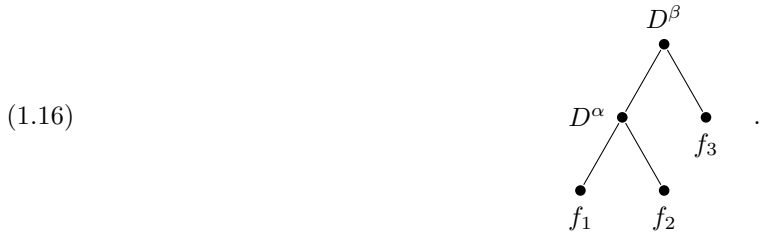
\begin{equation}
\label{tree:3-lin:paraprod}
\begin{array}{cc}
\vcenter{\hbox{\begin{forest}
my tree
[, label = {above: $D^\beta$}
	[, label = {left: $D^\alpha$}
		[, label = {below: $f_1$}]
		[, label = {below: $f_2$}]
	]
	[, label = {below: $f_3$}]
]
\end{forest}}} &.
\end{array}
\end{equation}

In general, a rooted tree $\mathcal{G}$ consists of a collection of vertices which are organized according to their depth; the root -- denoted $\mathfrak{r}_{\mathcal{G}}$ -- has depth zero, the direct descendants of the root have depth $1$, and so on. The vertices which don't have any descendants are called \emph{leaves} -- and $\mathcal L_{\mathcal G}$ denotes the collection of leaves in the rooted tree $\mathcal{G}$; all the other vertices make up $\mathcal V$, the collection of vertices that have at least one descendant.\footnote{It is more natural to request that every vertex which is not a leaf has at least two descendants.}

To each $l \in \mathcal L_{\mathcal G}$ we associate a function acting on $\BBR^d$, and to each vertex $v \in \mathcal V$ we associate a fractional differential operator $D^{\beta^v}$ for some $\beta^v \geq 0$. With an abuse of notation, we identify the collection $\mathcal L_{\mathcal G}$ of leaves with the collection of functions $\{ f_1, \ldots, f_n \}$, and similarly, the collection of vertices $\mathcal V$ is identified with the collection of fractional differential operators $\{ D^{\beta^v}: v \in \mathcal V  \}$.

For each $v \in \mathcal{V}$, we define the set 
\begin{equation} \label{leaf_subtree}
\mathcal{L}(v) := \{i \in \{1,\ldots, n\}: f_i \  \text{is a descendant of} \  v \}.
\end{equation}
As a consequence, $\mathcal{L}(\text{root}) = \{1, \ldots, n\}$.

It is also easy to verify that 
for any $v_1, v_2 \in \mathcal{V}$, only one of the following situations can happen:
\begin{equation}\label{obs:ancester_descendant}
(1) \ \ \mathcal{L}(v_1) \cap \mathcal{L}(v_2) = \emptyset; \ \ 
%\item
(2)\ \  \mathcal{L}(v_1) \subseteq \mathcal{L}(v_2); \ \ 
%\item
(3)\ \  \mathcal{L}(v_2) \subseteq \mathcal{L}(v_1).
\end{equation}

Last but not least, the \emph{complexity} of the tree $\mathcal G$ is defined as the maximal depth among the leaves vertices; or equivalently, as the maximum length of upward paths from a leaf to the root. For example, \eqref{tree:Leib_1_generic} is a tree of complexity 1 while \eqref{tree:3-lin:paraprod} is of complexity 2. 

We can adapt the tree representation to the multi-parameter setting by substituting the homogenous differential operators \eqref{homo_diff} with partial differential operators \eqref{partial_diff}. The 5-linear expression on the left hand side of \eqref{bi_Muscalu} indeed corresponds to a tree of complexity 2 represented by Figure \ref{fig:5}(\subref{fig:5:2:param}) bellow; next to it, Figure \ref{fig:5}(\subref{fig:5:1:param}) depicts its one-parameter equivalent.

\begin{figure}[!htbp]
\begin{subfigure}[t]{.45\textwidth}
  \centering
 \begin{forest}
my treeSep
[, label={above: $D^{\beta}$}
  	[, label = {left: $D^{\alpha}$}
		[, label = {below: $f_1$}
		]
		[, label = {below: $f_2$}
		]
	]
	[, label = {below: $f_3$}
	]
	[, label = {right: $D^{\gamma}$}
		[, label = {below: $f_4$}
		]
		[, label = {below: $f_5$}
		]
	]
]
\end{forest}
  \caption{Tree corresponding to the one-parameter Leibniz rule $D^\beta(D^\alpha(f_1 f_2) f_3 D^\gamma(f_4 f_5))$} 
  \label{fig:5:1:param}
    \end{subfigure}
 \hfill    
\begin{subfigure}[t]{.45\textwidth}
  \centering  
  \begin{forest}
my treeSep
[, label={above: $D_{(1)}^{\beta_1}D_{(2)}^{\beta_2}$}
  	[, label = {left: $D_{(1)}^{\alpha_1}D_{(2)}^{\alpha_2}$}
		[, label = {below: $f_1$}
		]
		[, label = {below: $f_2$}
		]
	]
	[, label = {below: $f_3$}
	]
	[, label = {right: $D_{(1)}^{\gamma_1}D_{(2)}^{\gamma_2}$}
		[, label = {below: $f_4$}
		]
		[, label = {below: $f_5$}
		]
	]
]
\end{forest}
   \caption{Tree corresponding to the bi-parameter Leibniz rule $D_{(1)}^{\beta_1}D_{(2)}^{\beta_2}(D_{(1)}^{\alpha_1}D_{(2)}^{\alpha_2}(f_1 f_2) f_3 D_{(1)}^{\gamma_1}D_{(2)}^{\gamma_2}(f_4 f_5))$} 
   \label{fig:5:2:param}
  \end{subfigure}
    \caption{One and bi-parameter flag Leibniz rules}
\label{fig:5}  
\end{figure}

Building upon the above structures, one can obtain a rooted tree representation for $N$-parameter flag Leibniz rules  as well; in such a situation, the leaves $\mathcal L_{\mathcal G}$ correspond to functions $\{f_1, \ldots, f_n \}$ defined on $\BBR^{d_1} \times \ldots \times \BBR^{d_N}$, and the vertices in $\mathcal V$ to generic fractional partial differential operators $D_{(1)}^{\beta^v_1} \ldots D_{(N)}^{\beta^v_N}$, where each $\beta_j^v \geq 0$ indicates the partial derivatives associated to $v$ in the $j$-th parameter. As before, we identify a vertex $v$ with its corresponding differential operator, and the set $\mathcal L_{\mathcal G}$ with the set of functions $\{f_1, \ldots, f_n \}$.

In this paper, we establish multi-parameter flag Leibniz rules of arbitrary complexity for input data in $L^{\vec p}$ spaces, with $ 1\leq \vec p \leq \infty$.\footnote{This inequality is to be understood componentwise.} More precisely, we prove an $N$-parameter, $n$-linear Leibniz rule associated to a rooted tree $\mathcal{G}$ with $n$ leaves. The possible distribution of derivatives $D_{(1)}^{\beta^v_1} \ldots D_{(N)}^{\beta^v_N}$ among the $n$ functions $f_1, \ldots, f_n$ is described by the tensor map $\delta_1\otimes \ldots \otimes \delta_N:  \mathcal{V}^N \to \{1,\ldots, n \}^N$; for $1 \leq j \leq N$, $\delta_j$ represents a map
\begin{equation} \label{delta:derivative_distribution}
\delta_j: \mathcal{V} \rightarrow \{1,\ldots, n \}
\end{equation}
satisfying the following two conditions:
\begin{enumerate}[label=(\roman*), leftmargin=*]
\item\label{deriv:distrib:i} for any  $v \in \mathcal{V}$, $\delta_j(v) \in \mathcal{L}(v)$ (equivalently, $D^{\beta_j^v}_{(j)}$ derivatives are attributed to $f_{\delta_j(v)}$, one of the leaves descending from the vertex $v$).
\item \label{deriv:distrib:ii} if $\mathfrak{S}(v)$ denotes the set of non-leaf, direct descendants of the vertex $v$ (for some $v \in \mathcal{V}$), then
$$\delta_j(v) = \delta_j(w), \ \ \text{for some} \ \  w \in \mathfrak{S}(v).$$
\end{enumerate}

This latter condition ensures that the distribution of derivatives agrees with the composition law. As a consequence of \ref{deriv:distrib:ii}, if $\delta(v) = l$ for some $v \in \mathcal{V}$, then for any non-leaf vertex $w$ along the path from $v$ to the leaf $f_l$, we have $\delta(w) = l$. Due to observation \eqref{obs:ancester_descendant}, conditions \ref{deriv:distrib:i} and \ref{deriv:distrib:ii} can be simultaneously satisfied.
\medskip 

We denote by $\mathcal D(\mathcal{V})$ (abbreviated as $\mathcal{D}$) the collection of maps $\delta_j: \mathcal{V} \to \{1,\ldots, n \}$ satisfying conditions \ref{deriv:distrib:i} and \ref{deriv:distrib:ii} above for $1 \leq j \leq N$. We will abbreviate $\underbrace{\mathcal{D} \times \ldots \mathcal{D}}_{N \ \ \text{copies}}$ as $\mathcal{D}^N$.

We notice that such maps are well-defined since the leading partial derivatives only hit one function they act on at a time, but $\delta_j$  is not necessarily injective nor surjective -- in the example (\ref{bi_Muscalu}), $D^{\beta_1}_{(1)}$ and $D^{\alpha_1}_{(1)}$ can hit $f_1$ simultaneously (thus not injective) and it is also possible that no derivative hits $f_1$ at all (thus not surjective). With abuse of notation, we will denote for any $1 \leq l \leq n$,
\begin{equation}
\label{eq:def:distrib:deriv}
\delta_j^{-1}(l) := \sum_{v \in \mathcal{V}: \delta_j(v) = l} \beta_j^v
\end{equation}
and if $l \nin \text{range}(\delta_j)$, then $\delta_j^{-1}(l)  = 0$; hence $\delta_j^{-1}(l)$ keeps track of the number of partial derivatives attributed to $f_l$ by the map $\delta_j$.

For any $v \in \mathcal{V}$, $p_{v}$ denotes the Lebesgue exponent defined by 
\begin{equation}
\label{eq:def:Lebesg:Exp:vertex}
\frac{1}{p_v}= \sum_{i \in \mathcal{L}(v)} \frac{1}{p_i}.
\end{equation}
In the case of $N$-parameters rooted trees, we consider $N$-tuples $\vec p_{v}= (p_v^1, \ldots, p_v^N)$ and each $p_v^j$, for $1 \leq j \leq N$, is defined by 
\begin{equation}
\label{eq:def:Lebesg:Exp:vertex:muti:param}
\frac{1}{p^j_v}= \sum_{i \in \mathcal{L}(v)} \frac{1}{p_i^j}.
\end{equation}

\begin{figure}[!htbp]
\begin{forest}
my tree
[, label={above: $D_{(1)}^{\beta_1^{\mathfrak{r}_{\mathcal{G}}}} \ldots D_{(N)}^{\beta_N^{\mathfrak{r}_{\mathcal{G}}}}$}
  	[, label = {left: $D_{(1)}^{\beta_1^{1}}\ldots D_{(N)}^{\beta_N^{1}}$}
  		[, label = {left: {\fontsize{5}{5}}}
			[, label = {below: $f_1$}]
			[, label = {below: $\ldots$}]
		]
		[, label = {left: {\fontsize{5}{5}}}
			[, label = {below: $\ldots$}]
			[, label = {below: $\ldots$}]
		]
	]
	[, label = {left: $D_{(1)}^{\beta_1^{2}}\ldots D_{(N)}^{\beta_N^{2}}$}
		[, label = {left: {\fontsize{5}{5}}}
			[, label = {below: $\ldots$}]
			[, label = {below: $\ldots$}]
		]
		[, label = {left: {\fontsize{5}{5}}}
			[, label = {below: $\ldots$}]
			[, label = {below: $\ldots$}]
		]
	]
  	[, label = {right: $D_{(1)}^{\beta_1^{n_1}}\ldots D_{(N)}^{\beta_N^{n_1}}$}
		[, label = {left: {\fontsize{5}{5}}}
			[, label = {below: $ \ldots$}]
			[, label = {below: $\ldots$}]
		]
		[, label = {right:{\fontsize{5}{5}}}
			[, label = {below: $\ldots$}]
			[, label = {below: $f_n$}]
		]
	]
]
\end{forest}
\caption{A rooted tree of arbitrary complexity associated to an $N$-parameter flag Leibniz rule.}
\end{figure}
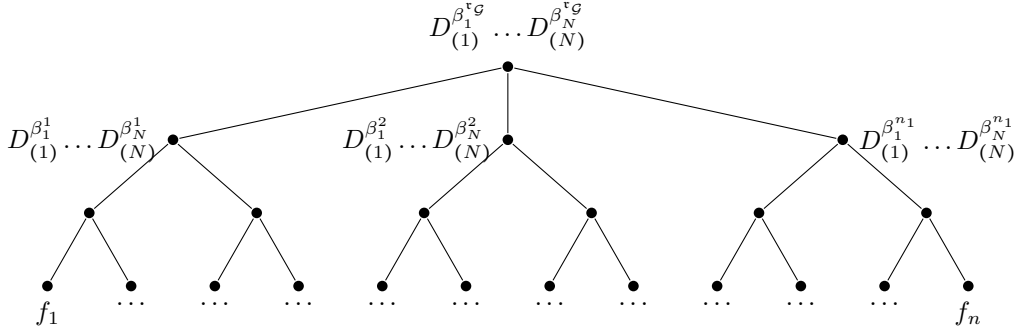

Now we are ready to state our main result:
\begin{theorem}
\label{thm:main}
Let $\mathcal{G}$ be a rooted tree of root $\mathfrak{r}_{\mathcal G}$, and to every $v \in \mathcal V$ we associate the $N$-parameter fractional differential operator $D_{(1)}^{\beta^v_1} \ldots D_{(N)}^{\beta^v_N}$, with $\beta_1^v, \ldots, \beta_N^v \geq 0$. Let $\mathcal D^N$ be the collection of maps $\delta_1\otimes \ldots \otimes \delta_N: \mathcal{V}^N \to \{1,\ldots, n \}^N$ satisfying conditions \ref{deriv:distrib:i} and \ref{deriv:distrib:ii}, which describe the admissible distributions of derivatives. If $T_{\mathcal{G}}$ denotes the $n$-linear operator indicated by the rooted tree $\mathcal{G}$, then for any functions $f_1, \ldots, f_n \in \mathcal{S}(\BBR^{d_1} \times \ldots \times \BBR^{d_N})$,
 \begin{align} \label{induction_result}
\|T_{\mathcal{G}}(f_1, \ldots, f_n) \|_{\vec{r}} \lesssim \sum_{\delta_1\otimes \ldots \otimes \delta_N \in \mathcal D^N}\prod_{l=1}^n\|D_{(1)}^{\delta_1^{-1}(l)} \ldots D_{(N)}^{\delta_N^{-1}(l)}f_{l}\|_{\vec{p_l}},
\end{align}
provided that $1\leq p^j_l \leq \infty$ for all $1 \leq j \leq N$, $1 \leq l \leq n$,
\begin{equation}
\label{condition:Holder}
\frac{1}{r^j}= \sum_{l =1}^n \frac{1}{p^j_l}, \quad \text{for all      } 1 \leq j \leq N,
\end{equation}
\begin{equation}
\label{cond:thm:1}
\frac{d_N}{d_N+\beta_N^{\mathfrak{r}_{\mathcal G}}}< r^N, \quad \max(\frac{d_N}{d_N+\beta_N^{\mathfrak{r}_{\mathcal G}}}, \frac{d_{N-1}}{d_{N-1}+\beta_{N-1}^{\mathfrak{r}_{\mathcal G}}} )< r^{N-1}, \ldots, \max \big( \frac{d_N}{d_N+\beta_N^{\mathfrak{r}_{\mathcal G}}}, \ldots, \frac{d_1}{d_1+\beta_1^{\mathfrak{r}_{\mathcal G}}}  \big)< r^1 
\end{equation}
and in general, for any $v \in \mathcal V$, 
\begin{equation}
\label{cond:thm:2}
\frac{d_N}{d_N+\beta_N^{v}}< p_v^N, \quad  \max(\frac{d_N}{d_N+\beta_N^{v}}, \frac{d_{N-1}}{d_{N-1}+\beta_{N-1}^{v}} )< p^{N-1}_v, \ldots, \max \big( \frac{d_N}{d_N+\beta_N^{v}}, \ldots, \frac{d_1}{d_1+\beta_1^{v}}  \big)< p_v^1. 
\end{equation}
Whenever $\beta_j^v \in 2 \BBN$, the corresponding conditions on $p^{j}_v$ in \eqref{cond:thm:2} above can be removed.
\end{theorem}

The only previously known case of the above result corresponds to $N=1$: the one-parameter flag Leibniz rules associated to trees of arbitrary complexity are a consequence\footnote{Modulo endpoints.} of the boundedness of flag paraproducts of arbitrary complexity from \cite{flag_paraproducts}.

We remark that the mixed norms 
\[
\|\cdot\|_{L^{\vec p_l}(\BBR^{d_1}\times \ldots \BBR^{d_N})}=\| \ldots \| \cdot\|_{L^{ p_l^N (\BBR^{d_N})}} \ldots \|_{{L^{ p_l^1}(\BBR^{d_1})}}
\]
in Theorem \ref{thm:main} can be further replaced by 
\[
\| \ldots \| \cdot\|_{L^{ \vec p_l^N (\BBR^{d_N})}} \ldots \|_{L^{\vec p_l^1}(\BBR^{d_1})};
\]
that is, each $\|\cdot\|_{L^{p_l^j}(\BBR^{d_j})}$ norm can be replaced by the mixed norm $\|\cdot\|_{L^{\vec p_l^j}(\BBR^{d_j})}$, as long as $\vec{p}^j_l = (p^j_{l,1}, \ldots, p^j_{l, d_j}) \in [1, \infty]^{d_j}$. In this situation we require (component-wise) conditions analogous to \eqref{condition:Holder}, \eqref{cond:thm:1} and \eqref{cond:thm:2}:
\begin{equation*}
\frac{1}{\vec{r}^j}= \sum_{l =1}^n \frac{1}{\vec{p}^j_l},\quad \text{for all    } 1 \leq j \leq N,
\end{equation*}
\begin{equation*}
\frac{d_N}{d_N+\beta_N^{\mathfrak{r}_{\mathcal G}}}< \min(r_1^N,\ldots, r_{d_N}^N), %\quad \max(\frac{d_N}{d_N+\beta_N^{\mathfrak{r}_{\mathcal G}}}, \frac{d_{N-1}}{d_{N-1}+\beta_{N-1}^{\mathfrak{r}_{\mathcal G}}} )< r^{N-1}, 
\ldots, \max \big( \frac{d_N}{d_N+\beta_N^{\mathfrak{r}_{\mathcal G}}}, \ldots, \frac{d_1}{d_1+\beta_1^{\mathfrak{r}_{\mathcal G}}}  \big)< \min(r^1_1, \ldots, r^1_{d_1})
\end{equation*}
and for any $v \in \mathcal V$, 
\begin{equation*}
\frac{d_N}{d_N+\beta_N^{v}}< \min(p_{v,1}^N,\ldots, p_{v,d_N}^N), %\quad  \max(\frac{d_N}{d_N+\beta_N^{v}}, \frac{d_{N-1}}{d_{N-1}+\beta_{N-1}^{v}} )< \min(p^{N-1}_{v,1}, \ldots p^{N-1}_{v,d_{N-1}}), 
\ldots, \max \big(\frac{d_N}{d_N+\beta_N^{v}}, \ldots, \frac{d_1}{d_1+\beta_1^{v}}  \big)< \min(p_{v,1}^1, \ldots, p_{v,d_1}^1) . 
\end{equation*}
This remains true for all our results, namely Theorem \ref{thm:main}-\ref{thm:asymm:Mikhlin}.

Moreover, for a rooted tree $\mathcal{G}$, we can associate to each vertex $v \in \mathcal{V}$ the inhomogeneous differential operator $J_{(1)}^{\beta^v_1} \ldots J_{(N)}^{\beta^v_N}$\footnote{For $1 \leq j \leq N$, we denote by $J_{(j)}^{\beta_j}$, with $\beta_j \geq 0$, the inhomogeneous partial differential operator defined on the space of Schwartz functions $\mathcal{S}(\mathbb{R}^{d_1}\times \ldots \BBR^{d_N})$ by 
\begin{align*} 
J_{(j)}^{\beta_j}f:= \ii F^{-1}\big((1+|\xi^j|^2)^{\frac{\beta_j}{2}}\f{f}(\xi^1,\ldots, \xi^N)\big).
\end{align*}
%where $\xi^j \in \BBR^{d_j}$ and $ 1 \leq j \leq N$.
} instead of the homogeneous differential operator $D_{(1)}^{\beta^v_1} \ldots D_{(N)}^{\beta^v_N}$. The multilinear operator -- denoted by $T^{J}_{\mathcal{G}}$ -- satisfies the same mixed-norm estimates as $T_{\mathcal{G}}$ described in Theorem \ref{thm:main} and can be treated in the same fashion; the only notable difference appears at the level of cone/paraproduct decompositions, since the inhomogeneous partial differential operators do not pick out small frequency scales. Details on how to adjust the decompositions can be found in Grafakos-Oh \cite{graf-Leibniz_rules} or Oh-Wu \cite{OhWu}. 

Now we return to our initial examples -- the explicit Leibniz rule \eqref{bi_Muscalu} -- in order to clarify the notation in our main theorem. In the example \eqref{bi_Muscalu}, $n = 5$ and $N=2$, so that the set of leaves $\mathcal{L}_{\mathcal{G}}$ consists of 
\[
\mathcal{L}_{\mathcal{G}}=\{f_1, f_2, f_3, f_4, f_5\},
\]
and the collection $\mathcal V$ of non-leaf vertices of
\[
\mathcal V= \{\vec \beta= (\beta_1, \beta_2), \vec \alpha=(\alpha_1, \alpha_2), \vec \gamma=(\gamma_1, \gamma_2) \}.
\]
Then 
$$\mathcal L(\vec \beta)=\{1, 2, 3, 4, 5\}, \quad \mathcal L(\vec \alpha)=\{1, 2\}, \quad \mathcal L(\vec \gamma)=\{ 4, 5\}  $$
and for any $j= 1,2$,  $\delta_j: \{ \beta_j, \alpha_j, \gamma_j \} \to \{1, 2, 3, 4, 5\}$ must satisfy the condition that 
\begin{align*}
 \delta_j(\alpha_j) \in \{1,2 \}, \quad \delta_j(\gamma_j) \in \{4,5 \},
\end{align*}
and
\begin{align*}
& \delta_j(\beta_j) \in \{1, 2, 3, 4, 5\} \ \ \text{with} \ \ \delta_j(\beta_j) = \delta_j(\alpha_j) \ \ \text{or} \ \ \delta_j(\beta_j) = \delta_j(\gamma_j) \ \ \text{or} \ \  \delta_j(\beta_j) = 3.
\end{align*}
As a result, there are in total 144 choices of $\delta_1 \otimes \delta_2$! The first term in the right-hand side of \eqref{bi_Muscalu} indeed corresponds to the particular choice of the maps 
\begin{align*}
& \delta_1(\beta_1) = \delta_1(\alpha_1) = 1, \delta_1(\gamma_1) = 4; \qquad \delta_2(\beta_2) = \delta_2(\alpha_2) = 1, \delta_2(\gamma_2) = 4,
\end{align*}
whereas the fourth term corresponds to 
\begin{align*}
& \delta_1(\beta_1) = \delta_1(\gamma_1) = 4, \delta_1(\alpha_1) = 1; \qquad \delta_2(\beta_2) = \delta_2(\alpha_2) = 2, \delta_2(\gamma_2) = 5.
\end{align*}

We can equivalently represent the flag Leibniz rules in frequency -- since the fractional (partial) differential operators themselves are defined in frequency (see \eqref{partial_diff}, \eqref{homo_diff}). Starting from the observation that
\begin{equation}
\label{eq:D:beta:n}
D^\beta(f_1 \, f_2\, \ldots f_n)= \mathcal F^{-1} (|\xi_1+ \ldots+\xi_n|^\beta \hat{f_1}(\xi_1) \cdot \ldots \cdot \hat f_n(\xi_n)),
\end{equation}
we realize that a correspondence can be established between the trees
\begin{equation}
\label{tree:corresp:D:beta}
\begin{array}{cccc}
\vcenter{\hbox{
\begin{forest}
my tree
[, label={above: $D^{\beta}$}
  	[, label = {below: $f_1$}
	]
	[, label = {below: $f_2$}
	]
	[, label= {below:$\ldots$}
	]
	[, label = {below: $f_n$}
	]
]
\end{forest}}} & \leftrightsquigarrow  &
\vcenter{\hbox{\begin{forest}
my tree
[, label={above: $|\xi_1+\ldots+\xi_n|^{\beta}$}
  	[, label = {below: $\hat f_1(\xi_1)$}
	]
	[, label = {below: $\hat f_2(\xi_2)$}
	]
	[, label= {below:$\ldots$}
	]
	[, label = {below: $\hat f_n(\xi_n)$}
	]
]
\end{forest}}}& . 
\end{array}
\end{equation}

One should notice that in the frequency representation of the flag Leibniz rules, the set of leaves $\hat{\mathcal{L}}_{\mathcal G}$ consists of $\{ \hat f_1(\xi_1), \ldots \hat f_n(\xi_n) \}$. The vertices $v \in \mathcal V$ should be identified, in the one-parameter case, with the symbols
\begin{equation}
\label{cond:tree:multiplier:1}
 |\sum_{i \in \mathcal{L} (v)} \xi_i|^{\beta^v};
\end{equation}
in the $N$-parameters case, the non-leaf vertices appearing in the flag Leibniz rules should be identified with 
\begin{equation}
\label{cond:tree:multiplier:N}
\prod_{j=1}^N |\sum_{i \in \mathcal{L} (v)} \xi^j_i|^{\beta_j^v}.
\end{equation}

The representation in frequency of the 5-linear flag appearing in \eqref{bi_Muscalu} and of its one-parameter analogue are represented in Figures \ref{fig:freq:5} (\subref{fig:freq:5:2}) and \ref{fig:freq:5}(\subref{fig:freq:5:1}) below.
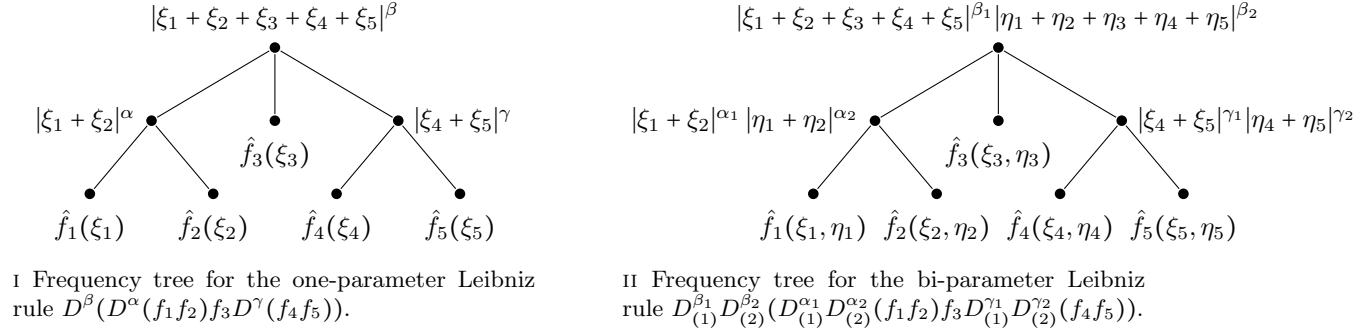
\begin{figure}[!htbp]
\hspace*{\fill}
\begin{subfigure}[t]{.40\textwidth}
  \centering
 \begin{forest}
my treeSep
[, label={above: $|\xi_1+\xi_2+\xi_3+\xi_4+\xi_5|^{\beta}$}
  	[, label = {left: $|\xi_1+\xi_2|^{\alpha}$}
		[, label = {below: $\hat f_1(\xi_1)$}
		]
		[, label = {below: $\hat f_2(\xi_2)$}
		]
	]
	[, label = {below: $\hat f_3(\xi_3)$}
	]
	[, label = {right: $|\xi_4+\xi_5|^{\gamma}$}
		[, label = {below: $\hat f_4(\xi_4)$}
		]
		[, label = {below: $\hat f_5(\xi_5)$}
		]
	]
]
\end{forest}

  \caption{Frequency tree for the one-parameter Leibniz rule $D^\beta(D^\alpha(f_1 f_2) f_3 D^\gamma(f_4 f_5))$.} 
  \label{fig:freq:5:1}
    \end{subfigure}
 \hfill    
\begin{subfigure}[t]{.40\textwidth}
  \centering  
  \begin{forest}
my treeSep
[, label={above: $|\xi_1+\xi_2+\xi_3+\xi_4+\xi_5|^{\beta_1} |\eta_1+\eta_2+\eta_3+\eta_4+\eta_5|^{\beta_2}$}
  	[, label = {left: $|\xi_1+\xi_2|^{\alpha_1} \, |\eta_1+\eta_2|^{\alpha_2}$}
		[, label = {below: $\hat f_1 (\xi_1, \eta_1)$}
		]
		[, label = {below: $\hat f_2(\xi_2, \eta_2)$}
		]
	]
	[, label = {below: $\hat f_3(\xi_3, \eta_3)$}
	]
	[, label = {right: $|\xi_4+\xi_5|^{\gamma_1}|\eta_4+\eta_5|^{\gamma_2}$}
		[, label = {below: $\hat f_4(\xi_4, \eta_4)$}
		]
		[, label = {below: $\hat f_5(\xi_5, \eta_5)$}
		]
	]
]
\end{forest}
   \caption{Frequency tree for the bi-parameter Leibniz rule $D_{(1)}^{\beta_1}D_{(2)}^{\beta_2}(D_{(1)}^{\alpha_1}D_{(2)}^{\alpha_2}(f_1 f_2) f_3 D_{(1)}^{\gamma_1}D_{(2)}^{\gamma_2}(f_4 f_5))$.} 
   \label{fig:freq:5:2}
  \end{subfigure}
 \hspace*{\fill}
  \caption{Frequency representation of one and bi-parameter flag Leibniz rules}
\label{fig:freq:5}
\end{figure}

Interestingly, the same methods imply the boundedness of multi-parameter flag multipliers that do not correspond directly to multi-parameter flag Leibniz rules, since the structures of the trees are different for different parameters. One such example is the following $5$-linear, bi-parameter expression
\begin{align}\label{example:asymmetric_diff}
&      \qquad \qquad T_{\mathcal{G}_1\otimes \mathcal{G}_2}(f_1, f_2, f_3, f_4, f_5)(x,y)  \\
:= &\int_{\BBR^{10}} |\xi_1+\xi_2+\xi_3+\xi_4+\xi_5|^{\beta_1} |\xi_1+\xi_2|^{\alpha_1}|\xi_3+\xi_4+\xi_5|^{\gamma_1}|\xi_3+\xi_4|^{\zeta_1}  |\eta_1+\eta_2+\eta_3+\eta_4+\eta_5|^{\beta_2}|\eta_1+\eta_3|^{\alpha_2}|\eta_2+\eta_4|^{\gamma_2}\nonumber \\
& \quad \prod_{l=1}^5 \hat{f_l}(\xi_l, \eta_l)e^{2 \pi i (x, y) \cdot (\xi_1+\xi_2+\xi_3+\xi_4+\xi_5, \eta_1+\eta_2+\eta_3+\eta_4+\eta_5)} d \xi d \eta. \nonumber
\end{align}

More generally, we can consider $n$-linear, $N$-parameter operators that can be represented as $T_{\mathcal{G}_1 \otimes \ldots \otimes \mathcal{G}_N}$, where each $\mathcal{G}_j$, for $1 \leq j \leq N$, indicates the frequency tree in the $j$-th parameter.

For any $1 \leq j \leq N$, we denote by $\mathcal{V}_j$ the set of vertices with at least one descendant and by $\mathcal L_j$ the set of vertices with no descendants. Whereas $\mathcal L_j$ is always going to be identified with the collection of functions $\{ f_1, \ldots, f_n \}$, the collections $\mathcal{V}_j$ on the other hand can be quite different due to the distinct tree structures associated to each parameter. Let $(\xi_l^j)_{1 \leq l \leq n}$ denote the frequency variables for the $j$-th parameter. To every vertex $v_j \in \mathcal{V}_j$, we associate a symbol $\displaystyle \big|\sum_{l \in \mathcal{L}(v_j)}\xi^j_l\big|^{\beta^{v_j}_j}$ with $\beta^{v_j}_j \geq 0$. The distribution of derivatives is still described by maps $\delta_1 \otimes \ldots \otimes \delta_N$, where for every $1 \leq j \leq N$,
$$
\delta_j : \mathcal{V}_j \rightarrow \{1, \ldots, n\}
$$
satisfies the conditions \ref{deriv:distrib:i} and \ref{deriv:distrib:ii}, with $\mathcal{V}$ replaced by $\mathcal{V}_j$. We denote by $\mathcal{D}(\mathcal{V}_j)$ the collection of such $\delta_j$s, so that $\mathcal{D}(\mathcal{V}_1) \times \ldots \times \mathcal{D}(\mathcal{V}_N)$ represents the collection of admissible distributions of derivatives among the $n$ functions $f_1, \ldots, f_n$.

We now state the Leibniz-type estimates for multi-parameter flag multipliers with asymmetric symbols generated by partial differential operators:
%\footnote{The result in Theorem \ref{thm:an:example} can be further extended by replacing each rooted tree $\mathcal{G}_j$ with a union of rooted trees.}:
%\todo{see fotnote}

\begin{theorem}
\label{thm:an:example}
Let $T_{\mathcal{G}_1\otimes \ldots \otimes \mathcal{G}_N}$ denote the $n$-linear operator associated to $\mathcal{G}_j$ for $1 \leq j \leq N$ such that each $\mathcal{G}_j$ denotes the frequency tree for the $j$-th parameter. Suppose that every vertex $v_j \in \mathcal{V}_j$ is associated to a symbol $\displaystyle \big|\sum_{l \in \mathcal{L}(v_j)}\xi^j_l\big|^{\beta^{v_j}_j}$ with $\beta^{v_j}_j \geq 0$. Then for any functions $f_1, \ldots, f_n \in \mathcal{S}(\BBR^{d_1} \times \ldots \times \BBR^{d_N})$, we have
\begin{align*}
\|T_{\mathcal{G}_1\otimes \ldots \otimes \mathcal{G}_N} (f_1, \ldots, f_n) \|_{\vec r} \lesssim \sum_{\delta_1\otimes  \ldots \otimes \delta_N \in \mathcal{D}(\mathcal{V}_1) \times \ldots \times \mathcal{D}(\mathcal{V}_N)}\prod_{l=1}^n\|D_{(1)}^{\delta_1^{-1}(l)}\ldots D_{(N)}^{\delta_N^{-1}(l)}f_{l}\|_{\vec{p_l}},
\end{align*} 
for any $1\leq p_l^j \leq \infty$, $1 \leq j \leq N$, $1 \leq l \leq n$, 
\[
\frac{1}{r^j}= \sum_{l =1}^n \frac{1}{p^j_l}, \quad \text{for all      } 1 \leq j \leq N,
\]
and
\begin{equation}\label{exponents:asymmetric_differential}
\max\Big\{\frac{d_j}{d_j+\beta^{v_j}_j}: v_j \in \mathcal{V}_j, 1 \leq j \leq N \Big\} < \min\{r^j: 1 \leq j \leq N  \}.
\end{equation}
\end{theorem}

We remark that the condition \eqref{exponents:asymmetric_differential} on the Lebesgue exponents  is only sufficient, and that in the case when $\beta_j^{v_j} \in 2 \BBN$ it can be disregarded.

Moreover, the same method allows us to prove mixed-norm Leibniz-type estimates for multilinear Mikhlin multipliers of order $\beta>0$, by systematically reducing them to estimates for linear Mikhlin multipliers. This extends to flags associated to Mikhlin multipliers of strictly positive order, both in one-parameter and in multi-parameter settings.

Let $\beta \in \BBR$. We say $m(\xi_1, \ldots, \xi_n)$ is a Mikhlin symbol \emph{of order $\beta$} provided that $m : \BBR^{dn} \to \BBC$ is smooth away from the origin $\{(\xi_1,\ldots, \xi_n) = 0 \}$ and satisfies the condition
\begin{equation}
\label{eq:positive:order:class}
\big| \partial_{\xi_1}^{\gamma_1} \ldots  \partial_{\xi_n}^{\gamma_n}  m(\xi_1, \ldots, \xi_n) \big| \lesssim  \big(|\xi_1|+ \ldots+ |\xi_n| \big)^{\beta- |\gamma_1|- \ldots - |\gamma_n|}
\end{equation}
for sufficiently many multi-indices $\gamma_1, \ldots, \gamma_n$. We denote\footnote{It should be clear from the context what is the space of variables a symbol $m \in \mathcal M_\beta$ acts on. Especially for flags, it will be more convenient to leave this implicit, since the number of variables depends on each vertex of the rooted tree.} by $\mathcal M_\beta$ the class of symbols satisfying the above conditions.

\begin{theorem}
\label{thm:multipliers:positive:order}
Let $T_{\mathcal{G}}$ denote the $n$-linear operator indicated by a rooted tree $\mathcal{G}$ of root $\mathfrak{r}_{\mathcal G}$, where every vertex $v \in \mathcal V$ is associated to a symbol $m_v \in \mathcal M_{\beta^v}$ of order $\beta^v >0$.  Let $\mathcal D$ be the collection of maps defined in \eqref{delta:derivative_distribution}, satisfying conditions \ref{deriv:distrib:i} and \ref{deriv:distrib:ii}. Then for any functions $f_1, \ldots, f_n \in \mathcal{S}(\BBR^{d})$,
we have
 \begin{align} \label{eq:positive:order:1:param}
\|T_{\mathcal{G}}(f_1, \ldots, f_n) \|_{r} \lesssim \sum_{\delta \in \mathcal D}\prod_{l=1}^n\|D^{\delta^{-1}(l)} f_{l}\|_{p_l},
\end{align}
provided that $1 < p_l < \infty$, $1 \leq l \leq n$, and $1/n<r <\infty$ satisfy the H\"older condition 
\[
\frac{1}{r}= \frac{1}{p_1}+\ldots+\frac{1}{p_n}.
\]
\end{theorem}

Notice that in this situation we fail to recover precisely the $L^\infty$ and $L^1$ endpoints.

In the $N$-parameter case, we consider $\beta_1, \ldots, \beta_N \in \BBR$ and we say that $m$ is an $N$-parameter Mikhlin symbol of order $(\beta_1, \ldots, \beta_N)$, simply written as $m \in \mathcal M_{\beta_1, \ldots, \beta_N}$, if $m : \BBR^{(d_1+ \ldots d_N) n} \to \BBC$ is smooth away from the region $\displaystyle \bigcup_{j=1}^N\{(\xi_1, \ldots \xi_n) \in \BBR^{(d_1+\ldots + d_N)n}: (\xi_1^j, \ldots, \xi_n^j)=0 \text{  in   } \BBR^{d_j n}\}$ and satisfies the condition
\begin{equation}
\label{eq:positive:order:class:multi:param}
\big| \partial_{\xi^1_1}^{\gamma^1_1} \ldots \partial_{\xi^N_1}^{\gamma^N_1} \ldots  \partial_{\xi_n}^{\gamma_n^1} \ldots \partial_{\xi_n^N}^{\gamma_n^N}  m(\xi_1, \ldots, \xi_n) \big| \lesssim \prod_{j=1}^N  \big(|\xi_1^j|+ \ldots+ |\xi_n^j| \big)^{\beta_j- |\gamma^j_1|- \ldots - |\gamma^j_n|}
\end{equation}
for sufficiently many multi-indices $\gamma^1_1, \ldots, \gamma^N_1, \ldots, \gamma^1_n, \ldots, \gamma^N_n$.

Then we have the following $N$-parameter result for symbols which tensorize in each parameter:

\begin{theorem}
\label{thm:Nparam:tensor:positive:multipl}
Let $T_{\mathcal{G}}$ denote the $n$-linear operator indicated by a rooted tree $\mathcal{G}$ of root $\mathfrak{r}_{\mathcal G}$, where every vertex is associated to a symbol $m_v$ satisfying 
\begin{equation} \label{symbol_N_parameter}
m_v(\xi_{i_1}^1, \ldots, \xi_{i_1}^N, \ldots, \xi_{i_k}^1, \ldots, \xi_{i_k}^N)= \prod_{j=1}^N m_v^j(\xi_{i_1}^j, \ldots, \xi_{i_k}^j),
\end{equation}
with $m_v^j \in  \mathcal{M}_{\beta^v_j}$ for $\beta^v_j>0$.  Let $\mathcal D$ be the collection of maps defined in \eqref{delta:derivative_distribution}, satisfying conditions \ref{deriv:distrib:i} and \ref{deriv:distrib:ii}. Then for any functions $f_1, \ldots, f_n \in \mathcal{S}(\BBR^{d_1} \times \ldots \times \BBR^{d_N})$, we have
\begin{align} \label{eq:positive:order:1:param}
\|T_{\mathcal{G}}(f_1, \ldots, f_n) \|_{\vec r} \lesssim  \sum_{\delta_1\otimes \ldots \otimes \delta_N \in \mathcal D^N}\prod_{l=1}^n\|D_{(1)}^{\delta_1^{-1}(l)} \ldots D_{(N)}^{\delta_N^{-1}(l)}f_{l}\|_{\vec{p_l}},
\end{align}
provided that $1< p_1^1, \ldots, p_1^N, p_2^1, \ldots, p_2^N, \ldots, p_n^1, \ldots, p_n^N < \infty$, $\frac{1}{n}<r^1, \ldots, r^N < \infty$ satisfy component-wise the H\"older condition 
\[
\frac{1}{\vec r}= \frac{1}{ \vec p_1}+\ldots+\frac{1}{\vec p_n}.
\] 
\end{theorem}

Although we expect the non-tensorized equivalent result to remain true in all its generality, that will be analyzed in an upcoming paper. For now we only examine the depth-1 result:
\begin{theorem}\label{thm:multiparameter_nontensor}
If $\beta_1, \ldots, \beta_N >0$ and $m \in \mathcal M_{\beta_1, \ldots, \beta_N}$ is an $N$-parameter Mikhlin symbol satisfying \eqref{eq:positive:order:class:multi:param}, then the associated multiplier $T_{m}$ satisfies
\[
\|T_m(f_1, \ldots, f_n)\|_{L^{\vec r}} \lesssim  \sum_{\substack{\sigma_1^1, \ldots, \sigma^N_n \in \{ 0, 1 \} \\ \sigma_1^j+ \ldots+ \sigma_n^j=1}}\prod_{l=1}^n\|D^{\sigma^1_l \beta_1}_{(1)} \ldots D^{\sigma^N_l \beta_N}_{(N)} f_{l}\|_{\vec p_l},
\]
for any $1< p_1^1, \ldots, p_1^N, p_2^1, \ldots, p_2^N, \ldots, p_n^1, \ldots, p_n^N < \infty$, $\frac{1}{n}<r^1, \ldots, r^N < \infty$ such that component-wise the following H\"older condition holds
\[
\frac{1}{\vec r}= \frac{1}{ \vec p_1}+\ldots+\frac{1}{\vec p_n}.
\] 
\end{theorem}

We notice that the symbol does not have to obey the symmetry in \eqref{symbol_N_parameter}; in fact, we also obtain Leibniz-type estimates for multipliers associated to asymmetric Mikhlin symbols, as long as they tensorize. One such example is 
\begin{align}\label{example:asymmetric_diff:M}
&T_{\mathcal{G}_1\otimes \mathcal{G}_2}(f_1, f_2, f_3, f_4, f_5)(x,y) \nonumber \\
:= &\int_{\BBR^6} m_1(\xi_1,\xi_2,\xi_3,\xi_4,\xi_5)m_2(\xi_1,\xi_2)m_3(\xi_3, \xi_4,\xi_5) m_4(\xi_3,\xi_4) \tilde{m}_1(\eta_1,\eta_2,\eta_3,\eta_4,\eta_5) \tilde{m}_2(\eta_1,\eta_3) \tilde{m}_3(\eta_2,\eta_4)\nonumber \\
& \quad \prod_{l=1}^5 \hat{f_l}(\xi_l, \eta_l)e^{2 \pi i (x, y) \cdot (\xi_1+\xi_2+\xi_3+\xi_4+\xi_5, \eta_1+\eta_2+\eta_3+\eta_4+\eta_5)} d \xi d \eta,
\end{align}
which can be perceived as an extension of \eqref{example:asymmetric_diff} in the off-diagonal regions.

The following theorem gives a general formulation on the boundedness of multipliers associated to tensorized asymmetric Mikhlin symbols:
\begin{theorem}
\label{thm:asymm:Mikhlin}
Let $T_{\mathcal{G}_1\otimes \ldots \otimes \mathcal{G}_N}$ denote the $n$-linear operator associated to $\mathcal{G}_j$ for $1 \leq j \leq N$ such that each $\mathcal{G}_j$ denotes the frequency tree for the $j$-th parameter. Suppose that every vertex $v_j \in \mathcal{V}_j$ is associated to a symbol $m^j_{v_j} \in \mathcal{M}_{\beta^{v_j}_j}$ with $\beta_v^{v_j} >0$. Let $\mathcal D(\mathcal{V}_j)$ be the collection of maps $\delta_j: \mathcal{V}_j \rightarrow \{1,\ldots, n\}$ defined in \eqref{delta:derivative_distribution}, satisfying conditions \ref{deriv:distrib:i} and \ref{deriv:distrib:ii}. Then for any $f_1, \ldots, f_n \in \mathcal{S}(\BBR^{d_1} \times \ldots \times \BBR^{d_N})$,
\begin{align*}
\|T_{\mathcal{G}_1\otimes \ldots \otimes \mathcal{G}_N} (f_1, \ldots, f_n) \|_{\vec r} \lesssim  \sum_{\delta_1\otimes \ldots \otimes \delta_N \in \mathcal{D}(\mathcal{V}_1) \times \ldots \times \mathcal{D}(\mathcal{V}_N)} \prod_{l=1}^n\|D_{(1)}^{\delta_1^{-1}(l)} \ldots D_{(N)}^{\delta_N^{-1}(l)}f_{l}\|_{\vec{p_l}},
\end{align*} 
for any $1< p_1^1, \ldots, p_1^N, p_2^1, \ldots, p_2^N, \ldots, p_n^1, \ldots, p_n^N < \infty$, $\frac{1}{n}<r^1, \ldots, r^N < \infty$ satisfying component-wise the H\"older condition 
\[
\frac{1}{\vec r}= \frac{1}{ \vec p_1}+\ldots+\frac{1}{\vec p_n}.
\] 
\end{theorem}

Finally, the same type of reasoning allows us also to reprove \emph{smoothing properties} of $n$-linear Mikhlin multipliers associated to a symbol of negative order, in the mixed-norm multi-parameter setting. Our prototypical example consists of the $n$-linear $N$-parameter fractional integral operator, whose frequency symbol is given by
\[
(|\xi^1_1|^2 + \ldots |\xi^1_n|^2)^{-\frac{\nu_1}{2}} \ldots (|\xi_1^N|^2 + \ldots + |\xi_n^N|^2)^{-\frac{\nu_N}{2}},
\]
for $0 \leq \nu_j \leq n d_j$, for $1 \leq j \leq N$. This is motivated by the works of Hart-Torres-Wu \cite{HartTorresWu-smoothing-bil} and Yang-Liu-Wu \cite{YangLiuWu-smoothing}, where smoothing properties for less regular multipliers are studied in the mixed-norm and respectively in the bi-parameter setting. More concrete statements and sketches of proofs will be detailed in Section \ref{remark:smoothing}. 

Although the smoothing properties and the results in Theorem \ref{thm:multiparameter_nontensor} are not new, our intention here is to illustrate that a careful, quantitative scale-by-scale analysis (which includes improved estimates thanks to the introduction of certain commutators) offers and alternative route to proving them. We will elaborate on this method in the next section, as well as in Section \ref{Bourgain-Li_hilow}.

\subsection{Strategy}\label{strategy}
We provide an overview of our methodology and draw a comparison with the approach based on Coifman-Meyer multipliers. For the sake of simplicity, we will focus on dimension one but the discussion can be easily extended to higher dimensions. 

To gain some intuition of the Leibniz rules described in Theorem \ref{thm:main}, we observe that the derivatives capture a function's oscillation rate, and that can be understood through the Fourier transform. Because of \eqref{eq:D:beta:n} and an observation that goes back to Bony \cite{bony-initial-paraprod}, it is natural to decompose the frequency space into regions
\[
\tilde R_l:= \{ (\xi_1, \ldots, \xi_n): |\xi_1+ \ldots + \xi_n| \sim |\xi_l| \},
\]
since in that case we expect to have $D^\beta( (f_1 \cdot \ldots \cdot f_n)\mid_{\tilde R_j}) \sim f_1 \cdot \ldots f_{l-1} \cdot (D^\beta f_l) \cdot f_{l+1} \cdot \ldots \cdot f_n$.

Indeed, if we restrict our attention to the region\footnote{Throughout the paper, we say that two positive expressions $E_1$ and $E_2$ are equivalent and we write $E_1 \sim E_2$ if there exists $C>0$ so that $C^{-1} E_1 \leq E_2 \leq C E_1$. Correspondingly, we say that $E_1$ is much smaller than $E_2$ and write $E_1 \ll E_2$ if there exists $C>0$ (which in general will be implicitly depending on the dimension, number of functions involved) so that $E_1 \leq C^{-1} E_2$.} $R_1:=\{(\xi_1, \ldots, \xi_n) : |\xi_1|\gg |\xi_2|, \ldots ,|\xi_n| \}$, we have that 
\[
\begin{array}{cccc}
\vcenter{\hbox{\begin{forest}
my treeSep
[, label={above: $|\xi_1+\ldots+\xi_n|^{\beta}$}
  	[, label = {below: $\hat f_1(\xi_1)$}
	]
	[, label = {below: $\hat f_2(\xi_2)$}
	]
	[, label= {below:$\ldots$}
	]
	[, label = {below: $\hat f_n(\xi_n)$}
	]
]
\end{forest}}}  & \sim &
\vcenter{\hbox{\begin{forest}
my treeSep
[, label={above: $\,$}
  	[, label = {below: $|\xi_1|^\beta \hat  f_1(\xi_1)$}
	]
	[, label = {below: $\hat f_2(\xi_2)$}
	]
	[, label= {below:$\ldots$}
	]
	[, label = {below: $\hat f_n(\xi_n)$}
	]
]
\end{forest}}} & \,
\end{array}
\]
and if $\tilde{\chi}_{R_1}$ is a smooth function adapted to the region $R_1$, then
\[
|\xi_1+\ldots+\xi_n|^{\beta} \cdot \tilde{\chi}_{R_1}(\xi_1, \ldots, \xi_n)= m(\xi_1, \ldots, \xi_n) |\xi_1|^\beta \cdot \tilde{\chi}_{R_1}(\xi_1, \ldots, \xi_n),
\]
where 
\[
m(\xi_1, \ldots, \xi_n):= \frac{|\xi_1+\ldots+\xi_n|^{\beta}}{|\xi_1|^\beta} \cdot \tilde{\chi}_{R_1}(\xi_1, \ldots, \xi_n).
\]

Now the key point is to notice that $m(\xi_1, \ldots, \xi_n)$ is a classical multilinear Mikhlin symbol: it is only singular at $\xi_1=\xi_2= \ldots=\xi_n=0$ (in the region $R_1$, $|\xi_1|=0$ is equivalent to $ (\xi_1, \ldots, \xi_n) =0$), and it decays fast away from the origin. So once the functions $f_1, \ldots, f_n$ are jointly restricted in frequency to the region $R_1$,
\[
D^\beta(f_1, f_2, \ldots, f_n)=T_m(D^\beta f_1, f_2, \ldots, f_n).
\]
This is, in short,\footnote{For more details, and a comparison with the Bourgain-Li approach, see Section \ref{sec:differences:Coifman-Meyer}.} how the boundedness of Coifman-Meyer multipliers (associated to Mikhlin symbols) imply Leibniz-type estimates such as \eqref{Leib_one_paraproduct}, \eqref{Leib_bi_paraproduct}, \eqref{Leib_bi_paraproduct:mixed}, \eqref{Leib_one_flag}, etc. Of course, this excludes certain endpoints.

On the other hand, the methodology that we rely on are \emph{commutators} -- originally introduced by Bourgain and Li \cite{BourgainLi-kato} -- which manage to capture a certain cancellation between different scales. To start with, we write $R_1$ as 
\begin{equation}\label{conical_1}
R_1:= \bigcup_{k_1 \gg k_2, \ldots, k_n \in \BBZ} R_{k_1, k_2, \ldots, k_n }:=\bigcup_{k_1 \gg k_2, \ldots, k_n} \{ (\xi_1, \ldots, \xi_n) : |\xi_l| \sim 2^{k_l} \text{     for $1 \leq l \leq n$} \}
\end{equation}
and notice that on $R_1$,
\begin{equation}
\label{eq:intr:comm}
|\xi_1+\ldots+\xi_n|^{\beta} = |\xi_1|^\beta + (|\xi_1+\ldots+\xi_n|^{\beta}-|\xi_1|^\beta).
\end{equation}

The first term on the right hand side of the above identity seems to be exactly what we wanted: we can indeed replace $|\xi_1+\ldots+\xi_n|^{\beta}$ by $|\xi_1|^{\beta}$; for the second term, we notice that it becomes
\begin{equation}
\label{eq:commutator}
\begin{aligned}
&\big( \int_0^1 \beta \, |\xi_1+ t(\xi_2+ \ldots \xi_n)|^{\beta-2} (\xi_1+ t(\xi_2+ \ldots \xi_n)) dt \big) (\xi_2+\ldots+\xi_n) \\
&= \sum_{l=2}^n  \big( \int_0^1 \beta \, |\xi_1+ t(\xi_2+ \ldots \xi_n)|^{\beta-2} (\xi_1+ t(\xi_2+ \ldots \xi_n)) dt \big) \cdot \xi_l \\
:&= \sum_{l=2}^n m_{C_\beta,1,l}(\xi_1, \xi_2 + \ldots \xi_n).
\end{aligned}
\end{equation}
We restrict $m_{C_\beta,1,l}$ to the regions $R_{k_1,k_l,\pm}$ defined by  
\begin{align*}
R_{k_1,k_l,+} := & \{(\xi_1, \ldots \xi_n): \xi_1 \sim 2^{k_1}, |\xi_l| \sim 2^{k_l} \ll 2^{k_1},   |\xi_{\tilde{l}}| \ll 2^{k_1} \text{ for } \tilde{l} \neq 1, l\}, \\
R_{k_1,k_l,-} := & \{(\xi_1, \ldots \xi_n): \xi_1\sim -2^{k_1}, |\xi_l| \sim 2^{k_l} \ll 2^{k_1},   |\xi_{\tilde{l}}| \ll 2^{k_1} \text{ for } \tilde{l} \neq 1, l\}, 
\end{align*}
which are contained in $R_{k_1, \ldots, k_n} \text{\ \ for\ \ } k_1 \gg k_2, \ldots, k_n$, and denote by $m_{C_\beta,1,l,\pm}^{k_1, k_l}$ the localized symbol. We notice that $T_{m_{C_\beta, 1,l}^{k_1, k_l,\pm}}$ -- the multilinear operator whose symbols is given precisely by $m^{k_1, k_l,\pm}_{C_\beta,1,l}$ -- satisfies
\begin{equation}
\label{eq:improved:comm}
\| T_{m^{k_1, k_l,\pm}_{C_\beta, 1,l}} \|_{L^{p_1} \times \ldots \times L^{p_n} \to L^p} \lesssim 2^{k_1(\beta-1)} 2^{k_l}, \qquad \text{for   } 1 \leq p_1, \ldots, p_n \leq \infty.
\end{equation}
This quantified interaction between different scales will allow us to sum over the scales $k_1$ and $k_l$, as we will see in more detail in Section \ref{Bourgain-Li_hilow}.

For flag Leibniz rules, such as $D^\beta(D^\alpha(f_1 f_2) f_3 D^\gamma(f_4 f_5))$ whose frequency tree representation appears in Figure \ref{fig:5}(\subref{fig:5:1:param}), it is natural to restrict our attention to similar frequency regions: if 
\[
 R'_1=\{ (\xi_1, \ldots, \xi_5):  |\xi_1| \gg |\xi_2|, |\xi_3| , |\xi_4| \gg |\xi_5|   \},
 \]
then 
\begin{align*}
&|\xi_1+\xi_2+\xi_3+\xi_4+\xi_5|^\beta \, |\xi_1+\xi_2|^\alpha \, |\xi_4+\xi_5|^\gamma \tilde{\chi}_{R'_1}(\xi_1, \xi_2, \xi_3, \xi_4, \xi_5) \, \hat f_1 (\xi_1) \hat f_2 (\xi_2) \hat f_3  (\xi_3) \hat f_4 (\xi_4) \hat f_5 (\xi_5) \\
& = \frac{ |\xi_1+\xi_2+\xi_3+\xi_4+\xi_5|^\beta }{|\xi_1|^\beta} \, \frac{|\xi_1+\xi_2|^\alpha}{|\xi_1|^\alpha} \frac{|\xi_4+\xi_5|^\gamma}{|\xi_4|^\gamma}  \tilde{\chi}_{R'_1}(\xi_1, \xi_2, \xi_3, \xi_4, \xi_5) \ \,  \widehat{ D^{\alpha+\beta} f_1} (\xi_1) \hat f_2 (\xi_2) \hat f_3  (\xi_3) \widehat{ D^\gamma f_4} (\xi_4) \hat f_5 (\xi_5). 
\end{align*}

This implies that the Leibniz rule is a consequence of the boundedness of the $5$-linear operator associated to the symbol 
\[
\frac{ |\xi_1+\xi_2+\xi_3+\xi_4+\xi_5|^\beta }{|\xi_1|^\beta} \, \frac{|\xi_1+\xi_2|^\alpha}{|\xi_1|^\alpha} \frac{|\xi_4+\xi_5|^\gamma}{|\xi_4|^\gamma}  \tilde{\chi}_{R'_1}(\xi_1, \xi_2, \xi_3, \xi_4, \xi_5),
\]
which can also be represented as
\begin{equation}
\label{eq:product:Mikhlin}
m_\beta(\xi_1, \xi_2, \xi_3, \xi_4, \xi_5) \cdot m_\alpha(\xi_1, \xi_2) \cdot m_\gamma (\xi_4, \xi_5),
\end{equation}
where $m_\beta(\xi_1, \xi_2, \xi_3, \xi_4, \xi_5), m_\alpha(\xi_1, \xi_2), m_\gamma (\xi_4, \xi_5)$ are all Mikhlin multipliers when smoothly restricted to the region $R'_1$. Expressions such as \eqref{eq:product:Mikhlin} are prototypes of symbols associated to flag paraproducts -- compare to the definition in \eqref{eq:symbol:flag:paraprod}.

We will however choose a different path for estimating $D^\beta(D^\alpha(f_1 f_2) f_3 D^\gamma(f_4 f_5))$, which is closer to the Bourgain-Li approach, since it will allow us to also treat the multi-parameter flag Leibniz rules. The main steps of the strategy are:
\begin{enumerate}
\item \label{dec:root:symbol}
splitting of the root symbol and appearance of commutators in the off-diagonal frequency regions;
\item  \label{dec:Fourier:series}
Fourier series decomposition for the symbols;
\item \label{dec:tensorization:subtrees}
tensorization into frequency-localized subtrees. %\todo{see if this is okay}
\end{enumerate}

In the treatment of Leibniz rules of complexity 1, Step \ref{dec:root:symbol} distributes the derivative to a leaf \eqref{eq:intr:comm} while for a generic flag Leibniz rule, it will pass the derivative from the root to one of its direct descendants. This, together with the application of Step \ref{dec:Fourier:series}, allows us to reduce the estimation of the original tree to subtrees of lower complexities, which leads to Step \ref{dec:tensorization:subtrees}. 

The same methodology can be employed to develop Leibniz-type estimates for operators associated to Mikhlin symbols of order $\beta$ for $\beta>0$. We observe that in the region $R_1$ given by \eqref{conical_1}, the condition \eqref{eq:positive:order:class} satisfied by the Mikhlin symbol $m_{\beta}(\xi_1, \ldots, \xi_n)$ can be reformulated as
\begin{equation*}
\big| \partial_{\xi_1}^{\gamma_1} \ldots  \partial_{\xi_n}^{\gamma_n}  m_\beta(\xi_1, \ldots, \xi_n) \big| \lesssim  \big(|\xi_1|+ \ldots+ |\xi_n| \big)^{\beta- |\gamma_1|- \ldots - |\gamma_n|} \sim |\xi_1|^{\beta- |\gamma_1|- \ldots - |\gamma_n|}.
\end{equation*}
Such a heuristic computation suggests that $m_{\beta}(\xi_1, \xi_2, \ldots, \xi_n)$ behaves like $m_{\beta}(\xi_1, 0, \ldots, 0)$ on $R_1$. We can make it rigorous by invoking Step \ref{dec:root:symbol} and introducing the commutator
\begin{equation*}
m_{\beta}(\xi_1, 0, \ldots, 0) + \left(m_{\beta}(\xi_1, \xi_2, \ldots, \xi_n) - m_{\beta}(\xi_1, 0, \ldots, 0)\right)
\end{equation*}
and the scale-by-scale analysis is applicable in this setting as well. 

We remark that although this strategy works for obtaining Leibniz-type estimates for Mikhlin multipliers with symbols of strictly positive order, it fails in the case of classical Mikhlin symbols, which correspond to order zero.

To extend the methodology to the multi-parameter setting, we notice that for multipliers which are tensor product of symbols in each parameter -- such as the ones involved in Theorem \ref{thm:main} and Theorem \ref{thm:Nparam:tensor:positive:multipl} -- Steps \ref{dec:root:symbol} and \ref{dec:Fourier:series} can be performed independently in each parameter. Step \ref{dec:tensorization:subtrees} -- which allows a decoupling of the rooted tree into subtrees of lower complexities -- can be attained thanks to Step \ref{dec:Fourier:series}; further computations concern the distribution of derivatives -- this is carried out independently for each parameter so that the mixed Besov and Lebesgue norms (see Section \ref{sec:LP&Besov}) naturally appear. 
\medskip

\subsection{Structure of the paper}
The paper is organized as follows: in Section \ref{sec:Section2} we introduce the necessary terminology and review the Bourgain-Li approach from \cite{BourgainLi-kato}, putting the accent on certain novel aspects that will be involved in treating the mixed-norm multi-parameter generic flag Leibniz rule. In Section \ref{sec:5lin:1param} we present a 5-linear flag Leibniz rule, and in Section \ref{sec:2param:5linflag} its bi-parameter version; these particular examples are interesting enough to capture the complexity of the general case, without being too technical. In Section \ref{generic_induction} we present in detail the inductive proof of our main Theorem \ref{thm:main} and illustrate briefly a modified induction requested by Theorem \ref{thm:an:example}. Finally, in Section \ref{sec:generic:multipliers} we discuss Leibniz-type estimates for generic Mikhlin multipliers of positive order (Theorem \ref{thm:multipliers:positive:order}, \ref{thm:Nparam:tensor:positive:multipl}, \ref{thm:multiparameter_nontensor}) and recover the smoothing properties described in Theorem \ref{thm:smoothing:N:param}.
\\

\subsection{Acknowledgements}
C. Benea acknowledges partial support from ANR project RAGE ANR-18-CE40-0012 and research grant PN-III-P1-1.1-TE-2019-2275 from UEFISCDI, Romania. Y. Zhai's research is supported by ERC project FAnFArE no. 637510 and the region Pays de la Loire.

\section{Notation and useful results}
\label{sec:Section2}
In this section we set the notation that will be used throughout the paper, and review the Bourgain-Li method from \cite{BourgainLi-kato}.

\subsection{Littlewood-Paley projections and Besov spaces}~\\
\label{sec:LP&Besov}
Let $N \geq 1$, and let $\vec p=(p^1, \ldots, p^N)$ be an $N$-tuple of positive Lebesgue exponents: that is, we assume that $0< p^1, \ldots, p^N \leq \infty$. For functions on $\BBR^{d_1} \times \ldots \times \BBR^{d_N}$, we define the mixed (quasi)norm 
\begin{equation}
\label{eq:def:mixed:norm}
\|f\|_{L^{\vec p}}=\|f\|_{L^{p^1}_{x_1} L^{p^2}_{x_2} \ldots L^{p^N}_{x_N}}:= \Big( \int_{\BBR^{d_1}} \ldots \int_{\BBR^{d_{N-1}}}    \big(  \int_{\BBR^{d_N}} |f(x_1,  \ldots, x_{N-1}, x_N )|^{p^N}  d \, x_N\big)^{\frac{p^{N-1}}{p^N}} d \, x_{N-1} \big)^{\frac{p^{N-2}}{p^{N-1}}} \ldots d x_1 \Big)^\frac{1}{p^1}.
\end{equation}
Whenever one of the $p^j$ is equal to $\infty$, the integration in the $x_j$ variable is replaced by taking the essential supremum with respect to $x_j$.

Then the space $L^{\vec p}(\BBR^{d_1+\ldots+d_N})$ consists of all the functions defined on $\BBR^{d_1+\ldots+d_N}$, with finite $\| \cdot \|_{L^{\vec p}}$ norm:
\begin{equation}
\label{eq:def:mixedLp}
L^{\vec p}(\BBR^{d_1+\ldots+d_N}):= \{ f: \BBR^{d_1+\ldots+d_N} \to \BBC : \| f \|_{L^{\vec p}} < \infty  \}.
\end{equation}

\medskip
\begin{remark}We record a few useful properties of mixed-norm spaces: \mbox{} 
\begin{enumerate}
\item If $\vec p =(p^1, \ldots, p^N)$ with $1 \leq p^j \leq \infty$ for all $1 \leq j \leq N$, then $\| \cdot \|_{L^{\vec p}}$ is a norm and $L^{\vec p}(\BBR^{d_1+\ldots+d_N})$ is a Banach space.
\item Generally, $\| \cdot \|_{L^{\vec p}}$ is a quasi-norm and $L^{\vec p}(\BBR^{d_1+\ldots+d_N})$ is a quasi-Banach space.
\item If $\tau>0$ is so that
\[
\tau \leq \min \big( 1, \min_{1 \leq j \leq N} p^j \big),
\] 
then $\| \cdot \|_{L^{\vec p}}^\tau$ is subadditive.
\end{enumerate}
\end{remark}
\medskip

Throughout the paper, we make use of the classical Littlewood-Paley decompositions.\footnote{Small perturbations of the base functions $\varphi$ and $\psi$ will not change the inherent properties of the Littlewood-Paley families $\{ \psi(2^k \cdot)  \}_{k \in \BBZ}$. These perturbations will be denoted generically $\tilde \varphi$ and $\tilde \psi$.} On $\BBR^d$, we start with $\varphi: \BBR^d \to [0, \infty)$ a radial function so that $0 \leq \varphi \leq 1$, $\varphi \equiv 1$ on $\{ |\xi| \leq 1 \}$, $\varphi \equiv 0$ on $\{ |\xi| \geq 2 \}$; then we define $\psi(\xi):= \varphi(\xi) - \varphi(2 \xi)$, which is supported on $\{ 1/2 \leq |\xi| \leq 2 \}$, and we obtain
\[
\sum_{k \in \BBZ} \psi(2^{-k} \xi) =1, \qquad \text{for all     } \xi \neq 0.
\]
Hence for any Schwartz function $f \in \mathcal{S}(\BBR^d)$, 
\[
\hat f(\xi)= \sum_{k \in \BBZ} \psi(2^{-k} \xi) \, \hat f(\xi)  \qquad \text{for all     } \xi \neq 0.
\]
If we denote $\psi_k(\xi):= \psi(2^{-k} \xi)$, then the identity above reads in space as
\[
f(x)= \sum_{k \in \BBZ} f \ast \check \psi_{k}(x).
\]

For any $k \in \BBZ$, $\Delta_k$ denotes the Littlewood-Paley projection associated to the frequency region $\{ |\xi| \sim 2^k  \}$:
\[
f \mapsto \Delta_k f:= f \ast \check \psi_{k}= \ii F^{-1} (\hat f \cdot \psi_k).
\]

Then $\ds f = \sum_{k \in \BBZ} \Delta_k f$, and since the functions $ \check \psi_{k}$ are $L^1$-normalized,\footnote{That is, $\| \check \psi_{k}\|_1=\| \check \psi\|_1$ uniformly in $k$.} we have uniformly in $k \in \BBZ$,
\begin{equation}
\label{eq:Young}
\| \Delta_k f\|_p \leq \| \check \psi\|_1 \|f\|_p, \qquad \text{for any     } 1 \leq p \leq \infty. 
\end{equation}

For functions on $\BBR^{d_1} \times \ldots \times \BBR^{d_N}$, we consider the multi-parameter Littlewood-Paley decomposition
\[
f= \sum_{k_1, \ldots, k_N \in \BBZ}   \Delta_{k_1}^{(1)} \Delta_{k_2}^{(2)} \ldots \Delta_{k_N}^{(N)} f,
\]
where 
\begin{equation}
\Delta_{k_j}^{(j)} f (x_1, \ldots, x_{j-1}, x_j, x_{j+1}, \ldots, x_N):= f \ast_j \check \psi_{k_j} (x_1, \ldots, x_N) = \int_{\BBR^{d_j}} f(x_1, \ldots, x_{j-1}, x_j- t, x_{j+1}, \ldots, x_N) \check \psi_{k_j}(t) dt.
\end{equation}

As before in \eqref{eq:Young}, we have $\ds \| \Delta_{k_j}^{(j)} f\|_{L^p(\BBR^{d_1+\ldots+d_N})}  \lesssim \|f\|_{L^p(\BBR^{d_1+\ldots+d_N})}$ for any $1 \leq p \leq \infty$, and moreover, the mixed-norm estimate
\begin{equation}
\label{eq:Young:mixed}
\| \Delta^{(j)}_{k_j} f\|_{\vec p} \lesssim \|f\|_{\vec p}
\end{equation}
holds for any $\vec p= (p^1, \ldots, p^N)$ with $1 \leq p^1, \ldots, p^N \leq \infty$. This is due to the observation that, for almost every $(x_1, \ldots, x_j)$, 
\[
\|\Delta_{k_j}^{(j)} f (x_1, \ldots, x_{j-1}, x_j,\cdot) \|_{L^{p^{j+1}}_{x_{j+1}}  \ldots L^{p^N}_{x_N}} \lesssim \int_{\BBR^{d_j}} \|f(x_1, \ldots, x_{j-1}, x_j-t, \cdot)\|_{L^{p^{j+1}}_{x_{j+1}}  \ldots L^{p^N}_{x_N}} \,|\check \psi_{k_j}^{(j)}(t)| dt,
\]
which is nothing but a direct application of Minkowski's integral inequality (which is appropriate since $1 \leq p^1, \ldots, p^N \leq \infty$). Then we use Young's convolution inequality in the $x_j$ variable and integrate in the remaining variables to obtain \eqref{eq:Young:mixed}.

More generally,\footnote{This can be further extended by replacing each $\| \cdot\|_{L^{p^j}(\BBR^{d_j})}$ norm with the mixed-norm $\|\cdot\|_{L^{p_1^j}L^{p_2^j} \ldots L^{p_{d_j}^j}}    $, as long as all the Lebesgue exponents are between $1$ and $\infty$.} we obtain
\[
\|\Delta_{k_1}^{(1)} \Delta_{k_2}^{(2)} \ldots \Delta_{k_N}^{(N)} f\|_{\vec p} \lesssim \|f\|_{\vec p}, \quad \text{for any       } \vec p= (p^1, \ldots, p^N) \text{     with      } 1 \leq p^1, \ldots, p^N \leq \infty.
\]

Next, for any $s \in \BBR$ and any $1 \leq p \leq \infty$, we introduce the homogeneous Besov norms $\|\cdot\|_{\dot B^s_{p, \infty}}$ on $\mathcal{S}(\BBR^d)$ as
\begin{equation}
\label{eq:def:Beson:1param}
\|f\|_{\dot B^s_{p, \infty}}:= \sup_{k \in \BBZ} 2^{k s} \| \Delta_k f \|_p.
\end{equation}
A straightforward, but important observation is the inequality
\begin{equation}
\label{eq:obs:besov}
\|f\|_{\dot B^s_{p, \infty}} \lesssim \|D^s f\|_p,
\end{equation}
which holds true whenever the right hand side is finite.

Similarly, the $N$-parameter Besov norms of functions in $\mathcal{S}(\BBR^{d_1} \times \ldots \times \BBR^{d_N})$ are defined by
\begin{equation}
\label{eq:def:Beson:Nparam}
\|f\|_{\dot B^{s_1}_{p^1, \infty} \ldots  \dot B^{s_N}_{p^N, \infty} }:= \sup_{k_1, \ldots, k_N} 2^{k_1 s_1 } \ldots 2^{k_N s_N}  \|\Delta_{k_1}^{(1)} \Delta_{k_2}^{(2)} \ldots \Delta_{k_N}^{(N)} f\|_{\vec p}.
\end{equation}

One can also consider mixed Besov and Lebesgue norms: let $1 \leq m \leq N$ and $i_1, \ldots, i_m \in \{ 1, \ldots, N  \}$, and $s_{i_1}, \ldots, s_{i_m} \in \BBR$; then for any function $f \in \mathcal{S}(\BBR^{d_1}\times \ldots \times \BBR^{d_N}) \to \BBC$, we define
\begin{equation}
\label{eq:def:Beson:Nparam:mix}
\|f\|_{L^{p^1} \ldots  L^{p^{i_1-1}}  \dot B^{s_{i_1}}_{p^{i_1}, \infty} L^{p^{i_1+1}} \ldots  L^{p^{i_2-1}}  \dot B^{s_{i_2}}_{p^{i_2}, \infty} L^{p^{i_2+1}}  \ldots L^{p^N}}:= \sup_{k_{i_1}, \ldots, k_{i_m}} 2^{k_{i_1}  s_{i_1} } \ldots 2^{k_{i_m}  s_{i_m}}  \|\Delta_{k_{i_1}}^{(i_1)} \Delta_{k_{i_2}}^{(i_2)} \ldots \Delta_{k_{i_m}}^{(i_m)} f\|_{\vec p}.
\end{equation}

Since
\[
2^{k_j  s_j} \| \Delta^{(j)}_{k_j} f\|_{\vec p} = \|\tilde{\Delta}^{(j)}_{k_j} D^{s_j}_{(j)} f\|_{\vec p} \lesssim \|D^{s_j}_{(j)} f\|_{\vec p},
\]
for a slightly different Littlewood-Paley projection $\tilde \Delta^{(j)}_{k_j}$ having similar support properties, we deduce that in general
\[
\|f\|_{L^{p^1} \ldots  L^{p^{i_1-1}}  \dot B^{s_{i_1}}_{p^{i_1}, \infty} L^{p^{i_1+1}} \ldots  L^{p^{i_2-1}}  \dot B^{s_{i_2}}_{p^{i_2}, \infty} L^{p^{i_2+1}}  \ldots L^{p^N}} \lesssim 
 \|D^{s_{i_1}}_{(i_1)} \ldots D^{s_{i_m}}_{(i_m)} f\|_{\vec p},
\]
provided all the Lebesgue indices are in the Banach regime: $\vec p= (p^1, \ldots, p^N)$ with $1 \leq p^1, \ldots, p^N \leq \infty$. We will also need a mixed-norm interpolation result: for any $-\infty<s_0  < s_1 <\infty$ and any $0 \leq \theta \leq 1$ so that $s:= \theta s_0+ (1-\theta) s_1$,
\begin{equation}
\label{eq:interpolation:Besov}
\|f\|_{X_1 \ldots X_{j-1} \dot B^{s}_{p^j, \infty} X_{j+1} \ldots X_N} \lesssim \|f\|_{X_1 \ldots X_{j-1} \dot B^{s_0}_{p^j, \infty} X_{j+1} \ldots X_N}^{\theta} \cdot  \|f\|_{X_1 \ldots X_{j-1} \dot B^{s_1}_{p^j, \infty} X_{j+1} \ldots X_N}^{1-\theta},
\end{equation}
where the norms $\|\cdot\|_{X_i}$, for $i \in \{ 1, \ldots,  N \} \setminus \{ j \}$, denote either a Lebesgue $\|\cdot\|_{L^{p^i}}$ or a Besov $\| \cdot \|_{\dot B^{\tilde s_i}_{p^i, \infty}}$ norm. This is a straightforward consequence of the identity $\ds 2^{k s }= \big(2^{k s_0} \big)^{\theta}  \big(2^{k s_1} \big)^{1-\theta}$.

Oftentimes, this interpolation inequality will be used in the form 
\begin{equation}
\label{eq:interpolation:Besov:epsilon}
\|f\|_{X_1 \ldots X_{j-1} \dot B^{\epsilon}_{p^j, \infty} X_{j+1} \ldots X_N} \lesssim \|f\|_{X_1 \ldots X_{j-1} \dot B^{0}_{p^j, \infty} X_{j+1} \ldots X_N}^{{\beta-\epsilon} \over \beta} \cdot  \|f\|_{X_1 \ldots X_{j-1} \dot B^{\beta}_{p^j, \infty} X_{j+1} \ldots X_N}^{\epsilon \over \beta},
\end{equation}
where\footnote{In Section \ref{remark:smoothing}, the same inequality with $\beta < \epsilon < 0$ will be needed.} $0 < \epsilon< \beta$. 
\medskip

Before we proceed, we need to introduce a few extra operators and their properties. Recalling the definition of $\Delta_k$ and $\psi_k$, we define\footnote{In what follows, we want to make sure that the scale $2^\ell$ is much smaller than $2^k$: $2^\ell \ll 2^k$; this translates into assuming the existence of $\mathfrak{c}>0$ large enough (depending implicitly on the dimension, the number of functions involved, etc) so that $\ell \leq k - \mathfrak{c}$. Moreover, when $2^\ell \nll 2^k$ (so that $\ell > k - \mathfrak{c}$), we write $2^{\ell} \succ 2^{k}$ or equivalently $\ell \succ k$.}
\begin{equation}
\label{eq:def:Sk}
S_k f(x):= \sum_{\ell \ll k}\Delta_{\ell} f(x),
\end{equation}
which is of the form $S_k f(x)= f \ast \tilde \varphi_k(x)$, where $\tilde \varphi_k(x)= 2^{dk} \tilde \varphi (2^{k} x)$. As a consequence, 
\[
\|S_k f\|_p \lesssim \|f\|_p \qquad \text{for all    } 1 \leq p \leq \infty,
\]
uniformly in $k \in \BBZ$. Similarly, $\Delta_{\leq k}f := \sum_{\ell \leq k}f$ satisfies 
\[
\|\Delta_{\leq k}f \|_p \lesssim \|f\|_p \qquad \text{for all    } 1 \leq p \leq \infty \text{    and for all  } k \in \BBZ. 
\]

On the other hand, 
\begin{equation}
\label{eq:def:Delta>}
\Delta_{\succ k} f(x) := \sum_{\ell \succ k}\Delta_{\ell} f(x)=\sum_{\ell \geq k - \mathfrak{c}}\Delta_{\ell} f(x),
\end{equation}
and since $\Delta_{\succ k} f(x)+S_k f(x) = f(x)$, we again have
\begin{equation}\label{young_Delta>}
%\|\Delta_{\geq k} f\|_p \lesssim \|f\|_p \qquad \text{for all    } 1 \leq p \leq \infty,
\|\Delta_{\succ k} f\|_p \lesssim \|f\|_p \qquad \text{for all    } 1 \leq p \leq \infty,
\end{equation}
uniformly in $k \in \BBZ$.

For functions on $\BBR$, we also define the directional projection operators
\begin{equation}
\label{eq:direction:LP:def}
\Delta_{k,+} f(x):= \ii F^{-1}(\hat{f} \cdot \psi_{k,+})(x), \quad \Delta_{k,-} f(x):= \ii F^{-1}(\hat{f} \cdot \psi_{k,-})(x)
\end{equation}
where 
$\psi_{k,+}(\xi) := \psi_k(\xi) \chi_{\{\xi \geq 0\}}$ and $\psi_{k,-}(\xi) := \psi_k( \xi) \chi_{\{\xi <0\}}$. In dimension one, the region $\{|\xi| \sim 2^k\}$ naturally splits into two intervals, namely $\{\xi \sim 2^k \}$ and $\{\xi \sim -2^k \}$, which justifies the notion of ``directional" projection. We observe that $\Delta_{k,+}f =\tilde \Delta_{k,+}\Delta_{k} f$ for $\tilde \Delta_{k,+}$ associated to $\tilde \psi_{k, +}$, a function which is $\equiv 1$ on $\{ 2^{k-1} \leq \xi \leq 2^{k+1} \}$ (similarly for $\Delta_{k,-}f = \tilde \Delta_{k,-}\Delta_{k} f$), so that 
\begin{equation}
\label{eq:direction:LP:same}
\|\Delta_{k,+} f\|_p, \|\Delta_{k,-} f\|_p \lesssim \|\Delta_{k} f\|_p \lesssim \|f\|_p \qquad \text{for all    } 1 \leq p \leq \infty.
\end{equation}

All these definitions can be reformulated as Fourier projections onto the $j$th coordinate:
\begin{align}\label{projection_<>}
& S^{(j)}_{k_j} f(x):= \sum_{\ell \ll k_j}\Delta^{(j)}_{\ell} f(x),\qquad \text{and} \qquad \Delta^{(j)}_{\succ k_j} f(x):= \sum_{\ell > k_j - \mathfrak{c}}\Delta^{(j)}_{\ell} f(x). 
%\Delta^{(j)}_{\geq k_j} f(x):= \sum_{\ell \geq k}\Delta^{(j)}_{\ell} f(x).
\end{align}

For any function $f \in \mathcal{S}(\BBR^N)$, we define the directional projection operator on the $j$-th parameter by
\begin{align*}
 \Delta_{k_j,\pm}^{(j)} f (x_1, \ldots, x_{j-1}, x_j, x_{j+1}, \ldots, x_N)=& \int_{\BBR^{d_j}} f(x_1, \ldots, x_{j-1}, x_j- t, x_{j+1}, \ldots, x_N) \check \psi_{k_j,\pm}(t) dt.
\end{align*}
As before, for any $\vec p=(p^1, \ldots, p^N)$ with $1 \leq p^1, \ldots, p^N \leq \infty$, 
\[
\|S^{(j)}_{k_j} f\|_{\vec p}, \: \|\Delta^{(j)}_{\succ k_j} f\|_{\vec p}\lesssim \|f\|_{\vec p} \qquad \text{and} \qquad \|\Delta_{k_j,+}^{(j)}f\|_{\vec{p}}, \|\Delta_{k_j,-}^{(j)}f\|_{\vec{p}} \lesssim \|\Delta_{k_j}^{(j)}f\|_{\vec{p}} \lesssim \|f\|_{\vec{p}}.
\]

The Fourier series decomposition, which plays an important role in tensorizing the operator associated to a flag Leibniz rule into subtrees of lower complexity, will introduce certain modulations, which are however inconsequential: if $P_k^{(j)}$ is any of the projections $\Delta^{(j)}_k$, $S^{(j)}_k$ or directional projections $\Delta^{(j)}_{k,+}$, $\Delta^{(j)}_{k,-}$ (so that $P_k^{(j)}$ can also written as a convolution with the function $2^{kd_j}\check\phi(2^k x_j)$), then for any $a \in \BBR^{d_j}$, $P_{k, a}^{(j)}$ denotes
\begin{equation}
\label{eq:def:mod:proj}
P_{k, a}^{(j)} f(x_1, \ldots, x_{j-1}, x_j, x_{j+1}, \ldots, x_N):= \int_{\BBR^{d_j}} \phi(2^{-k} \xi_j) e^{2 \pi i a \cdot \frac{\xi_j}{2^k}} (\ii F^{(j)} f)(x_1, \ldots, x_{j-1}, \xi_j, x_{j+1}, \ldots, x_N) e^{2 \pi i x_j \cdot \xi_j} d \xi_j.
\end{equation} 
But $P_{k, a}^{(j)} f$ is simply the convolution in the $j$th coordinate between $f$ and $\check \phi_k$, evaluated at $x_j+ \frac{a}{2^k}$:
\[
P_{k, a}^{(j)} f(x_1, \ldots, x_j, \ldots, x_N)= f \ast_j \check \phi_k(x_1, \ldots,  x_j+  \frac{a}{2^k}, \ldots, x_N).
\]

Due to the trivial identity $\ds \| P_{k, a}^{(j)} f\|_{\vec p} = \|P_{k}^{(j)} f \|_{\vec p}$, we deduce 
\begin{equation}
\label{eq:trivial:modulation}
\|S^{(j)}_{k_j, a} f\|_{\vec p} \lesssim \|f\|_{\vec p} \qquad \text{and} \qquad \|\Delta^{(j)}_{ k_j, a} f\|_{\vec p},  \, \|\Delta^{(j)}_{k_j,+, a} f\|_{\vec p}, \, \|\Delta^{(j)}_{ k_j, -, a} f\|_{\vec p} \lesssim \|\Delta^{(j)}_{k_j} f\|_{\vec p} \lesssim \|f\|_{\vec p},
\end{equation}
for any $\vec p=(p^1, \ldots, p^N)$ with $1 \leq p^1, \ldots, p^N \leq \infty$.

\subsection{A review of the Bourgain-Li approach} \label{Bourgain-Li_hilow} ~\\
Here we present in dimension one\footnote{This assumption allows for a simplification of the notations, without restricting the method's generality -- see Remark \ref{remark:higher:dim}.} 
some elements of the Bourgain-Li proof of the bilinear Leibniz rule
\begin{align}\label{Leib_one_paraproduct:strategy}
\| D^{\alpha} (f g) \|_{L^p} \lesssim \|D^{\alpha}f\|_{L^{p_1}}\|g\|_{L^{p_2}} + \|f\|_{L^{p_1}}\|D^{\alpha} g \|_{L^{p_2}}
\end{align}
where $1 \leq p_1, p_2 \leq \infty, \quad  \frac{1}{1+\alpha} < p \leq \infty, \quad  \frac{1}{p_1} + \frac{1}{p_2} = \frac{1}{p}$, and $\alpha \geq 0$.
Due to the introduction of ``commutators'', the use of Coifman-Meyer multipliers can be completely avoided, thus extending the range of Leibniz rules beyond that of Coifman-Meyer multipliers. In distinction to \cite{BourgainLi-kato}, here we quantify the improvement produced by the commutator terms as an interaction between different scales, which in turn requires a suitable double Fourier series decomposition.

We start with Littlewood-Paley decompositions for the functions $f$ and $g$
\[
f = \sum_{k \in \BBZ} \Delta_k f, \qquad g= \sum_{\ell \in \BBZ} \Delta_\ell g,
\]
so that $D^{\alpha} (f g) $ becomes
\begin{equation}
\label{eq:initial:LP}
\sum_{k, \ell \in \BBZ} \int_{\BBR^2} |\xi_1+\xi_2|^\alpha  \widehat{\Delta_k f}(\xi_1)  \widehat{\Delta_\ell g}(\xi_2) e^{ 2 \pi i x (\xi_1+\xi_2)} d \xi_1 d \xi_2.
\end{equation}

We have several possibilities, depending whether $ |\xi_1+\xi_2| \sim |\xi_2| \gg |\xi_1|, \quad |\xi_1+\xi_2| \sim |\xi_1| \gg |\xi_2|,$ or $|\xi_1| \sim|\xi_2| \geq |\xi_1+\xi_2| $:
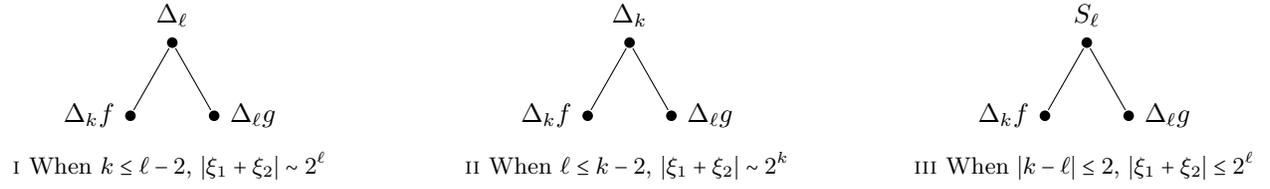
\begin{figure}[!htbp]
\begin{subfigure}[t]{.3\textwidth}
  \centering
  \begin{forest}
    my tree
    [, label={above:$\Delta_\ell $}
           [, label={left:$\Delta_k f$}]
          [, label={right:$\Delta_\ell g$}]
    ] 
  \end{forest} 
  \caption{When $k \leq \ell -2$, $|\xi_1+\xi_2| \sim 2^\ell$}
    \end{subfigure}
 \hfill    
\begin{subfigure}[t]{.3\textwidth}
  \centering  
   \begin{forest}
    my tree
    [, label={above:$\Delta_k $}
           [, label={left:$\Delta_k f$}]
          [, label={right:$\Delta_\ell g$}]
    ]
  \end{forest}
   \caption{When $\ell \leq k -2$, $|\xi_1+\xi_2| \sim 2^k$}
  \end{subfigure}
 \hfill 
  \begin{subfigure}[t]{.3\textwidth}
  \centering
  \begin{forest}
    my tree
    [, label={above:$S_\ell$}
           [, label={left:$\Delta_k f$}]
          [, label={right:$\Delta_\ell g$}]
    ] 
  \end{forest} 
\caption{When $|k -\ell| \leq 2$,  $|\xi_1+\xi_2| \leq 2^\ell$}
    \end{subfigure} 
\end{figure}

This allows us to decompose $D^\alpha(f \cdot g)$ as 
\begin{align*}
&\sum_{k \ll \ell} \int_{\BBR^2} |\xi_1+\xi_2|^\alpha  \widehat{\Delta_k f}(\xi_1)  \widehat{\Delta_\ell g}(\xi_2) e^{ 2 \pi i x (\xi_1+\xi_2)} d \xi_1 d \xi_2 \\
+&\sum_{k \gg \ell} \int_{\BBR^2} |\xi_1+\xi_2|^\alpha  \widehat{\Delta_k f}(\xi_1)  \widehat{\Delta_\ell g}(\xi_2) e^{ 2 \pi i x (\xi_1+\xi_2)} d \xi_1 d \xi_2 \\
+&\sum_{ |k - \ell| \leq 2} \int_{\BBR^2} |\xi_1+\xi_2|^\alpha  \widehat{\Delta_k f}(\xi_1)  \widehat{\Delta_\ell g}(\xi_2) e^{ 2 \pi i x (\xi_1+\xi_2)} d \xi_1 d \xi_2:= I+II+III,
\end{align*}
as suggested by the figure below:

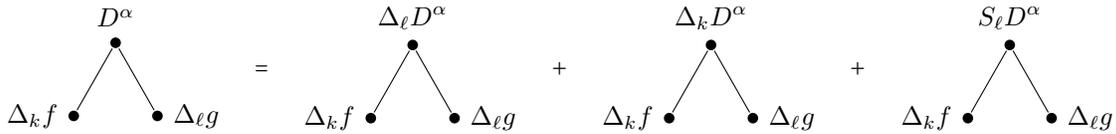
\[
\begin{array}{cccccccc}
\vcenter{\hbox{\begin{forest}
  my tree
    [, label={above:$D^\alpha$}
           [, label={left:$\Delta_k f$}]
          [, label={right:$\Delta_\ell g$}]
    ]
    \end{forest}}} & = &
  \vcenter{\hbox{\begin{forest}
    my tree
    [, label={above:$\Delta_\ell D^\alpha$}
           [, label={left:$\Delta_k f$}]
          [, label={right:$\Delta_\ell g$}]
    ]
  \end{forest}}} &+& 
    \vcenter{\hbox{\begin{forest}
    my tree
    [, label={above:$\Delta_k D^\alpha$}
           [, label={left:$\Delta_k f$}]
          [, label={right:$\Delta_\ell g$}]
    ]
  \end{forest}}}
  &+& 
    \vcenter{\hbox{\begin{forest}
    my tree
    [, label={above:$S_\ell D^\alpha$}
           [, label={left:$\Delta_k f$}]
          [, label={right:$\Delta_\ell g$}]
    ]
  \end{forest}}}&.
\end{array}    
\]

We study each of the cases $I$ and $III$ by taking a closer look at the associated multiplier -- this will be sufficient since $I$ and $II$ are symmetric. We highlight the main steps:
\begin{enumerate}[label=(\arabic*), leftmargin=*]
\item In treating $I$, we approximate $$|\xi_1+\xi_2|^\alpha =|\xi_2|^\alpha+`` \text{  error  ''},$$
so that it becomes 
\begin{align} \label{lem2.2:identity}
&\sum_{k \ll \ell} \int_{\BBR^2} |\xi_2|^\alpha  \widehat{\Delta_k f}(\xi_1)  \widehat{\Delta_\ell g}(\xi_2) e^{ 2 \pi i x (\xi_1+\xi_2)} d \xi_1 d \xi_2 \\
+& \sum_{k \ll \ell} \int_{\BBR^2} \big( |\xi_1+\xi_2|^\alpha -|\xi_2|^\alpha)  \widehat{\Delta_k f}(\xi_1)  \widehat{\Delta_\ell g}(\xi_2) e^{ 2 \pi i x (\xi_1+\xi_2)} d \xi_1 d \xi_2 \\
:=& \sum_{k \ll \ell}   \Delta_k f(x) \cdot (\Delta_\ell D^\alpha g)(x)   + \sum_{k \ll \ell}  [D^\alpha, \Delta_k f] \Delta_\ell g(x).
\end{align}

The term $[D^\alpha, \Delta_k f] \Delta_\ell g(x)$ represents a \emph{commutator}; we will see that it behaves better than the initial $D^\alpha( \Delta_k f \cdot \Delta_\ell g)$, in a way that can be expressed quantitatively. 

\item In estimating the first part, it is convenient to switch the order of summation, which produces
\begin{align*}
& \sum_{k,  \ell}   \Delta_k f(x) \cdot (\Delta_\ell D^\alpha g)(x) -  \sum_{k \succ \ell}   \Delta_k f(x) \cdot (\Delta_\ell D^\alpha g)(x) \\
=& f(x) \cdot D^\alpha g(x) -  \sum_{k \succ \ell}   \Delta_k f(x) \cdot (\Delta_\ell D^\alpha g)(x). 
\end{align*} 

\item With this, $I$ is converted into 
\begin{align*}
I &=  f(x) \cdot D^\alpha g(x) -   \sum_{k \succ \ell}   \Delta_k f(x) \cdot (\Delta_\ell D^\alpha g)(x) +  \sum_{k \ll \ell}  [D^\alpha, \Delta_k f] \Delta_\ell g(x) \\
&:= f(x) \cdot D^\alpha g(x) - I_A+I_B.
\end{align*}
\item The term $III$ reduces essentially to 
\begin{align*}
\sum_{\ell \in \BBZ} \int_{\BBR^2} |\xi_1+\xi_2|^\alpha  \widehat{\Delta_\ell f}(\xi_1)  \widehat{\Delta_\ell g}(\xi_2) e^{ 2 \pi i x (\xi_1+\xi_2)} d \xi_1 d \xi_2.
\end{align*}
\end{enumerate}

Next, we claim that it is sufficient to have precise estimates for the corresponding bilinear operators, with scales $k$ and $\ell$ fixed. 

\begin{lemma}
\label{lemma:multipliers:simpleParaproduct}
Let $1 \leq p_1, p_2 \leq \infty$, ${1 \over 2} \leq p \leq \infty$ satisfy $\frac{1}{p_1}+\frac{1}{p_2}=\frac{1}{p}$, and $k, \ell \in \BBZ$. Then we have
\begin{equation}
\label{result:local:product}
\| \Delta_k f \cdot (\Delta_\ell D^\alpha g)\|_p \lesssim  2^{\ell \alpha}  \| \Delta_k f\|_{p_1} \:  \|\Delta_\ell g\|_{p_2}.
\end{equation}
If $k \ll \ell$, 
\begin{equation}
\label{result:commutator:1}
\|  [D^\alpha, \Delta_k f] \Delta_\ell g\|_p \lesssim 2^{(\alpha-1) \ell} 2^k   \| \Delta_k f\|_{p_1} \:  \|\Delta_\ell g\|_{p_2}
\end{equation}
and 
\begin{equation}\label{result:Leibniz:Whitney}
\| D^\alpha \left(S_{\ell} f \Delta_\ell g\right)\|_p \lesssim  2^{\ell \alpha}  \|S_\ell f\|_{p_1} \:  \|\Delta_\ell g\|_{p_2}.
\end{equation}

Under the additional assumption that $\frac{1}{\alpha+1}<p \leq \infty$, 
\begin{equation}
\label{result:local:diag}
\| D^\alpha( \Delta_\ell f \cdot \Delta_\ell  g)\|_p \lesssim 2^{\ell \alpha}  \| \Delta_\ell f\|_{p_1} \:  \|\Delta_\ell g\|_{p_2}.
\end{equation}

\end{lemma}

Now we show how the fixed-scale estimates listed above allow us to control the terms $I$ and $III$, and thus prove \eqref{Leib_one_paraproduct:strategy}. Once that concluded, we will return to the proof of Lemma \ref{lemma:multipliers:simpleParaproduct}, since it allows to illustrate some of the main ideas needed for dealing with the more general Theorem \ref{thm:main}.

\begin{proof}[Proof of the Leibniz rule \eqref{Leib_one_paraproduct:strategy} assuming Lemma \ref{lemma:multipliers:simpleParaproduct}]
We let $\tau\leq \min(1, p)$ so that $\| \cdot \|_{L^{\vec p}}^\tau$ is subadditive.

\begin{enumerate} [leftmargin=*]
\item[$\bullet$] \underline{estimating $III$:}

Using \eqref{result:local:diag} of Lemma \ref{lemma:multipliers:simpleParaproduct} and the Besov norms definitions \eqref{eq:def:Beson:1param}, we have for $\frac{1}{1+\alpha}< p \leq \infty$ and $\tau \leq \min(1, p)$:
\begin{align*}
\|III\|_p^\tau & \lesssim \sum_{\ell \in \BBZ} 2^{\ell \alpha \tau} \|\Delta_\ell f \|_{p_1}^{\tau} \|\Delta_\ell g \|_{p_2}^{\tau} \\
& \lesssim  \sum_{\ell \in \BBZ} \min \big(  2^{\ell \alpha \tau} \|f\|_{\dot{B}^0_{p_1, \infty}}^{\tau} \|g\|_{\dot{B}^0_{p_2, \infty}}^{\tau} , 2^{- \ell \alpha \tau}  \|f\|_{\dot{B}^{\alpha}_{p_1, \infty}}^{\tau} \|g\|_{\dot{B}^{\alpha}_{p_2, \infty}}^{\tau}       \big).
\end{align*}
Optimizing in $\ell$, it is not difficult to see that 
\begin{align*}
\|III\|_p & \lesssim \|f\|_{\dot{B}^\alpha_{p_1, \infty}}^{\frac{1}{2} }  \|g\|_{\dot{B}^\alpha_{p_2, \infty}}^{\frac{1}{2}} \|f\|_{\dot{B}^0_{p_1, \infty}}^{\frac{1}{2}} \|g\|_{\dot{B}^0_{p_2, \infty}}^{\frac{1}{2} } \\
&\lesssim \|D^{\alpha}f\|_{L^{p_1}}\|g\|_{L^{p_2}} + \|f\|_{L^{p_1}}\|D^{\alpha} g \|_{L^{p_2}}.
\end{align*}
\item[$\bullet$] \underline{estimating $I_A$:} 

For any $k, \ell \in \BBZ$ and $0< \epsilon <\alpha$ we have, as a consequence of \eqref{result:local:product}, the following two inequalities:
\begin{align}
\| \Delta_k f \cdot (\Delta_\ell D^\alpha g)\|_p &\lesssim  2^{\ell \alpha}  \| \Delta_k f\|_{p_1} \:  \|\Delta_\ell g\|_{p_2} \lesssim 2^{\ell \alpha}  \|f\|_{\dot{B}^0_{p_1, \infty}} \|g\|_{\dot{B}^0_{p_2, \infty}} \\
\| \Delta_k f \cdot (\Delta_\ell D^\alpha g)\|_p &\lesssim 2^{-k \alpha }2^{\ell (\alpha -\epsilon)}  \big(2^{k \alpha} \| \Delta_k f\|_{p_1}\big) \:  \big(2^{\ell \epsilon} \|\Delta_\ell g\|_{p_2}\big) \lesssim  2^{-k \alpha }2^{\ell (\alpha -\epsilon)} \|f\|_{\dot{B}^\alpha_{p_1, \infty}} \|g\|_{\dot{B}^\epsilon_{p_2, \infty}}.
\end{align}

If we raise them to the power $\tau \leq \min(1, p)$ and sum over $\ell$ with $k \succ \ell$, we obtain 
\begin{align*}
\|I_A\|_p^\tau \lesssim \sum_{k \in \BBZ} \min \big( 2^{k \alpha \tau} \|f\|_{\dot{B}^0_{p_1, \infty}}^{\tau} \|g\|_{\dot{B}^0_{p_2, \infty}}^{\tau}, 2^{-k \epsilon \tau }  \|f\|_{\dot{B}^\alpha_{p_1, \infty}}^{\tau} \|g\|_{\dot{B}^\epsilon_{p_2, \infty}}^{\tau}\big).
\end{align*}

Optimizing over $k \in \BBZ$, we have as before
\begin{align*}
\|I_A\|_p^\tau &\lesssim \big(  \|f\|_{\dot{B}^0_{p_1, \infty}}^{\tau} \|g\|_{\dot{B}^0_{p_2, \infty}}^{\tau}  \big)^{\frac{\epsilon}{\alpha+\epsilon}}  \big(   \|f\|_{\dot{B}^\alpha_{p_1, \infty}}^{\tau} \|g\|_{\dot{B}^\epsilon_{p_2, \infty}}^{\tau} \big)^{\frac{\alpha}{\alpha+\epsilon}}.
\end{align*}

Using the interpolation of Besov norms mentioned earlier in \eqref{eq:interpolation:Besov:epsilon}, we deduce

\begin{align*}
\|I_A\|_p^\tau&\lesssim \|f\|_{\dot{B}^\alpha_{p_1, \infty}}^{\frac{\alpha}{\alpha+\epsilon} \tau}  \|g\|_{\dot{B}^\alpha_{p_2, \infty}}^{\frac{\epsilon}{\alpha+\epsilon} \tau} \|f\|_{\dot{B}^0_{p_1, \infty}}^{\frac{\epsilon}{\alpha+\epsilon} \tau} \|g\|_{\dot{B}^0_{p_2, \infty}}^{\frac{\alpha}{\alpha+\epsilon} \tau} \lesssim  \big( \|D^{\alpha}f\|_{L^{p_1}}\|g\|_{L^{p_2}} + \|f\|_{L^{p_1}}\|D^{\alpha} g \|_{L^{p_2}} \big)^\tau.
\end{align*}

\item[$\bullet$] \underline{estimating $I_B$:} 

Due to \eqref{result:commutator:1}, $I_B$ is similar to $I_A$. We start by noticing that, for $k \ll \ell$ fixed, 
\begin{equation}
\label{eq:min:1}
\big \| [D^\alpha, \Delta_k f] \Delta_\ell g  \big\|^\tau_p \lesssim  2^{(\alpha-1) \ell \tau} 2^{k \tau}  \|f\|_{\dot B^0_{p_1, \infty}}^\tau  \|g\|_{\dot B^0_{p_2, \infty}}^\tau
\end{equation} 
and on the other hand, if $0<\epsilon<\min(\alpha, 1)$,  
\begin{align}
\label{eq:min:2}
\big \| [D^\alpha, \Delta_k f] \Delta_\ell g  \big\|^\tau_p &\lesssim 2^{-\ell \tau} 2^{k(1-\epsilon) \tau} 2^{k \epsilon \tau}  \| \Delta_k f\|_{p_1}^{\tau} 2^{ \ell \alpha \tau}\|\Delta_\ell g\|_{p_2}^{\tau} \nonumber \\
 &\lesssim  2^{-\ell \tau} 2^{k(1-\epsilon) \tau} \|f\|_{\dot B^{\epsilon}_{p_1, \infty}}^\tau  \|g\|_{\dot B^{\alpha}_{p_2, \infty}}^\tau.
\end{align}

Combining \eqref{eq:min:1} and \eqref{eq:min:2} for $\ell$ fixed, and summing over $k \ll \ell$, we obtain 
\begin{equation}
\label{eq:min:comm}
\sum_{k: k \ll \ell } \big \| [D^\alpha, \Delta_k f] \Delta_\ell g  \big\|^\tau_p  \lesssim \min \big(  2^{\alpha \ell \tau}  \|f\|_{\dot B^0_{p_1, \infty}}^\tau  \|g\|_{\dot B^0_{p_2, \infty}}^\tau ,    2^{- \epsilon \ell \tau} \|f\|_{\dot B^{\epsilon}_{p_1, \infty}}^\tau  \|g\|_{\dot B^{\alpha}_{p_2, \infty}}^\tau  \big).
\end{equation}

Now the summation in $\ell$ is straightforward:
\begin{align*}
\|I_B\|_p^\tau &\lesssim \sum_{\ell \in \BBZ}  \min \big(  2^{\alpha \ell \tau}  \|f\|_{\dot B^0_{p_1, \infty}}^\tau  \|g\|_{\dot B^0_{p_2, \infty}}^\tau ,    2^{- \epsilon \ell \tau} \|f\|_{\dot B^{\epsilon}_{p_1, \infty}}^\tau  \|g\|_{\dot B^{\alpha}_{p_2, \infty}}^\tau  \big) \\
& \lesssim \big(   \|f\|_{\dot{B}^0_{p_1, \infty}}  \|g\|_{\dot{B}^0_{p_2, \infty}}  \big)^\frac{\tau \epsilon}{\alpha+\epsilon} \big(   \|f\|_{\dot{B}^\epsilon_{p_1, \infty}}  \|g\|_{\dot{B}^\alpha_{p_2, \infty}} \big)^\frac{\tau \alpha}{\alpha+\epsilon}. 
\end{align*}

Using again the interpolation of Besov norms from \eqref{eq:interpolation:Besov:epsilon}, we deduce
\begin{equation}
\label{ea:conc:comm}
\|I_B\|_p^\tau \lesssim \|f\|_{\dot{B}^\alpha_{p_1, \infty}}^{\frac{\epsilon}{\alpha+\epsilon} \tau}  \|g\|_{\dot{B}^\alpha_{p_2, \infty}}^{\frac{\alpha}{\alpha+\epsilon} \tau} \|f\|_{\dot{B}^0_{p_1, \infty}}^{\frac{\alpha}{\alpha+\epsilon} \tau} \|g\|_{\dot{B}^0_{p_2, \infty}}^{\frac{\epsilon}{\alpha+\epsilon} \tau} \lesssim  \big( \|D^{\alpha}f\|_{L^{p_1}}\|g\|_{L^{p_2}} + \|f\|_{L^{p_1}}\|D^{\alpha} g \|_{L^{p_2}} \big)^\tau.
\end{equation}
\end{enumerate}
Finally, due to a trivial application of H\"older's inequality ($\| f \cdot D^\alpha g\|_p \lesssim \| f\|_{p_1} \cdot \|D^\alpha g\|_{p_2}$), the Leibniz rule \eqref{Leib_one_paraproduct:strategy} follows.
\end{proof}

The proof above illustrates the main principle of our paper: in order to estimate multilinear operators associated to a symbol of positive order, it is sufficient to obtain quantitative estimates for the associated Littlewood-Paley pieces. We return now to the proof of \eqref{result:local:product}-\eqref{result:local:diag}.

\begin{proof}[Proof of Lemma \ref{lemma:multipliers:simpleParaproduct}]
The inequality \eqref{result:local:product} is a direct consequence of H\"older's inequality, and \eqref{result:Leibniz:Whitney} can be derived from \eqref{lem2.2:identity} and \eqref{result:commutator:1} in a straightforward manner. So our main focus will be proving \eqref{result:local:diag} and \eqref{result:commutator:1}. For this, we appropriately use Fourier series decompositions for the localized symbols -- our approach (see \ref{strategy}) for proving Theorem \ref{thm:main} will also rely on Fourier series decompositions, but in that setting they will allow us to easily tensorize the flag paraproduct into simpler object -- see also Remark \ref{remark:Bourgain-Li:generalizations}.

\vskip .2cm
\begin{enumerate} [leftmargin=15pt]
\item \label{FS:diagonal} \underline{\emph{Fourier series decomposition for the ``diagonal'' term}}
\vskip .2cm
Written in frequency, $D^\alpha( \Delta_\ell f \cdot \Delta_\ell  g)(x)$ becomes
\begin{align*}
 & \int_{\BBR^2} |\xi_1+\xi_2|^\alpha \widehat{\Delta_\ell f} (\xi_1) \widehat{\Delta_\ell g} (\xi_2) e^{2 \pi i x(\xi_1+\xi_2)} d \xi_1 d \xi_2 \\
=&   \int_{\BBR^2} |\xi_1+\xi_2|^\alpha \phi(2^{-\ell} (\xi_1+\xi_2)) \widehat{\Delta_\ell f} (\xi_1) \widehat{\Delta_\ell g} (\xi_2) e^{2 \pi i x(\xi_1+\xi_2)} d \xi_1 d \xi_2,
\end{align*}
where $\phi$ is a smooth, radial function that is equal to $1$ on a neighborhood of $\{ |\zeta| \leq 4 \}$ and is supported on $\{ |\zeta| \leq 8 \}$. We proceed with the Fourier series decomposition of $|\xi_1+\xi_2|^\alpha \phi(2^{-\ell} (\xi_1+\xi_2))$ on $[-2^{\ell+3}, 2^{\ell+3}]$. If we denote $\zeta=\xi_1+\xi_2$, we have
\begin{equation}
\label{eq:Fourier:dec:diagonal}
|\zeta|^\alpha \phi(2^{-\ell} \zeta)= \sum_{\tilde L \in \BBZ} C^{\ell}_{\tilde L} e^{2\pi i \frac{\tilde L}{2^{\ell+4}} \zeta},
\end{equation}
where 
\begin{align*}
C^{\ell}_{\tilde L} & = \tilde C \frac{1}{2^{\ell+4}} \int_{[-2^{\ell+3}, 2^{\ell+3}] } |\zeta'|^\alpha \phi(2^{-\ell} \zeta') \, e^{-2\pi i \frac{\tilde L}{2^{\ell+4}} \zeta'} d \zeta' = 2^{\ell \alpha}  \tilde C  \int_{[-8,8] } |\zeta'|^\alpha \phi(\zeta') \, e^{\frac{-2\pi i \tilde L \zeta'}{2^4}} d \zeta' \\
& := 2^{\ell \alpha} C_{\tilde L}.
\end{align*}

Because we are integrating $|\zeta|^\alpha$ close to the origin, the coefficients $C_{\tilde L}$ only have limited decay (see \cite[Lemma~1]{graf-Leibniz_rules}):
\begin{equation}
\label{eq:limited:decay}
|C_{\tilde L}| \lesssim \frac{1}{(1+|\tilde L|)^{1+\alpha}}.
\end{equation}

This means that we can express $D^\alpha( \Delta_\ell f \cdot \Delta_\ell  g)$ as
\begin{align}
\label{FD:delta:l:delta:l}
D^\alpha( \Delta_\ell f \cdot \Delta_\ell  g)(x) = \sum_{\tilde L \in \BBZ} 2^{\ell \alpha} C_{\tilde L} \: \big(  \Delta_{\ell, {\tilde L \over 2^\ell}} f \big)(x)    \big(  \Delta_{\ell, {\tilde L \over 2^\ell}} g \big)(x).
\end{align}

And this implies, for $\frac{1}{1+\alpha} < \tau \leq  \min(1,p)$,
\begin{align*}
\|D^\alpha( \Delta_\ell f \cdot \Delta_\ell  g)\|_p^\tau & \lesssim \sum_{\tilde L} 2^{\ell \alpha \tau} C_{\tilde L}^\tau \| \Delta_{\ell, {\tilde L \over 2^\ell}} f \|_{p_1}^\tau \| \Delta_{\ell, {\tilde L \over 2^\ell}} g \|_{p_2}^\tau \lesssim 2^{\ell \alpha \tau} \,  \|  \Delta_{\ell} f \|_{p_1}^\tau \|  \Delta_{\ell} g \|_{p_2}^\tau.
\end{align*}

\medskip 

\item \label{FS:comm} \underline{\emph{Fourier series decomposition for the commutator symbol}}
\vskip .2cm
Now $k \ll \ell$ are fixed and the commutator $[D^\alpha, \Delta_k f] \Delta_\ell g$ is given by 
\begin{align}\label{eq:comm:1}
&\int_{\BBR^2} \big(|\xi_1+\xi_2|^\alpha -|\xi_2|^\alpha \big) \widehat{\Delta_k f}(\xi_1)  \widehat{\Delta_\ell g}(\xi_2) e^{ 2 \pi i x (\xi_1+\xi_2)} d \xi_1 d \xi_2\\
=& 2^k \int_{\BBR^2} \frac{|\xi_1+\xi_2|^\alpha -|\xi_2|^\alpha}{\xi_1}  \tilde{\tilde{\psi}}_k(\xi_1) \widehat{ \Delta_k f}(\xi_1)  \widehat{\Delta_\ell g}(\xi_2) e^{ 2 \pi i x (\xi_1+\xi_2)} d \xi_1 d \xi_2,\nonumber
\end{align}
with $ \tilde{\tilde{\psi}}_k(\xi_1)=\frac{ \xi_1}{2^k} \tilde \psi_k(\xi_1)$, and $\tilde \psi_k(\xi_1) \equiv 1$ on the support of $ \psi_k(\xi_1)$.

We let 
\begin{equation}
\label{def:commutator}
m_{C_{\alpha}}(\xi_2, \xi_1)  := \frac{| \xi_1 +\xi_2|^{\alpha} -  |\xi_2|^{\alpha}}{\xi_1}= \alpha \int_0^1 |\xi_2 + t\xi_1|^{\alpha-2} (\xi_2 + t\xi_1) dt
\end{equation}
denote\footnote{The symbol $m_{C_{\alpha}}(\xi_2, \xi_1)$, associated to the region $\{|\xi_1| \ll |\xi_2|\}$, is central in the commutator's analysis. On the contrary, if we restrict our attention to the cone $\{|\xi_2| \ll |\xi_1|\}$, we need to study the behavior of $m_{C_{\alpha}}(\xi_1, \xi_2)$ defined by  $\frac{| \xi_1 +\xi_2|^{\alpha} -  |\xi_1|^{\alpha}}{\xi_2}$.} the symbol measuring the average contribution of the commutator, localized to the cone $R_2 := \{(\xi_1, \xi_2):|\xi_1| \ll |\xi_2|\}$. Since $k \ll \ell$, we can further (smoothly) localize it to the region $\{ |\xi_1|\ll 2^{\ell}, |\xi_2| \sim 2^\ell \}$:
\begin{equation} \label{symbol_undirectional_WC}
 m^{\ell}_{C_{\alpha}}(\xi_2, \xi_1)= m_{C_{\alpha}}(\xi_2, \xi_1)  \tilde \varphi_{\ell}(\xi_1) \tilde \psi_{\ell}(\xi_2).
\end{equation}

In view of the fact that we want to use a Fourier series decomposition in both the variables $\xi_1$ and $\xi_2$, $ m^{\ell}_{C_{\alpha}}$ will be split as
\begin{equation} \label{eq:doubleFS_+/-}
m^{\ell}_{C_{\alpha}}(\xi_2, \xi_1) = m^{\ell,+}_{C_{\alpha}}(\xi_2, \xi_1) + m^{\ell,-}_{C_{\alpha}}(\xi_2, \xi_1),
\end{equation}
where 
\begin{equation}
\label{eq:doubleFS:3}
m^{\ell,+}_{C_{\alpha}}(\xi_2, \xi_1) := m_{C_{\alpha}}(\xi_2, \xi_1)  \tilde \varphi_{\ell}(\xi_1) \tilde \psi_{\ell}^+(\xi_2), \qquad
m^{\ell,-}_{C_{\alpha}}(\xi_2, \xi_1) :=m_{C_{\alpha}}(\xi_2, \xi_1)  \tilde \varphi_{\ell}(\xi_1) \tilde \psi_{\ell}^-(\xi_2).
\end{equation}
Here $\tilde \psi_{\ell}^+$ is a bump function compactly supported on $[2^{\ell-2}, 2^{\ell+2}]$, $\tilde \psi_{\ell}^-$ a bump function compactly supported on $[-2^{\ell+2}, -2^{\ell-2}]$, and 
$\tilde \varphi_{\ell}$ on $[-2^{\ell-3}, 2^{\ell-3}]$. Overall we have
\begin{align}
[D^\alpha, \Delta_k f] \Delta_\ell g(x) &=  2^k \int_{\BBR^2} m^{\ell,+}_{C_{\alpha}}(\xi_2, \xi_1) \tilde{\tilde{\psi}}_k(\xi_1) \widehat{ \Delta_k f}(\xi_1)  \widehat{\Delta_{\ell,+} g}(\xi_2) e^{ 2 \pi i x (\xi_1+\xi_2)} d \xi_1 d \xi_2 \label{eq:expr:comm} \\
&+ 2^k \int_{\BBR^2} m^{\ell,-}_{C_{\alpha}}(\xi_2, \xi_1) \tilde{\tilde{\psi}}_k(\xi_1) \widehat{ \Delta_k f}(\xi_1)  \widehat{\Delta_{\ell, -} g}(\xi_2) e^{ 2 \pi i x (\xi_1+\xi_2)} d \xi_1 d \xi_2. \nonumber
\end{align}

We can now perform a Fourier series decomposition of $m^{\ell,+}_{C_{\alpha}}(\xi_2, \xi_1)$ on $[2^{\ell-2}, 2^{\ell+2}] \times [-2^{\ell-3}, 2^{\ell-3}]$. Notice that this is a double Fourier series expansion, involving both variables $\xi_1$ and $\xi_2$. Indeed, we obtain
\begin{align*}
m^{\ell,+}_{C_{\alpha}}(\xi_2, \xi_1) = \sum_{L_1,L_2 \in \BBZ}C^{\ell,+}_{L_1, L_2} e^{2\pi i \frac{L_1}{2^{\ell-2}} \xi_1} e^{2\pi i \frac{L_2}{2^{\ell+2}} \xi_2},
\end{align*}
where the Fourier coefficients are described by 
\begin{align*}
C^{\ell,+}_{L_1, L_2} & = \tilde C \frac{1}{2^{\ell} \, 2^\ell} \int_{[-2^{\ell-3}, 2^{\ell-3}] \times [2^{\ell-2}, 2^{\ell+2}]}  \int_0^1 (\xi'_2 + t \xi'_1)^{\alpha-1}dt \: e^{-2\pi i \frac{L_1}{2^{\ell-2}} \xi'_1} e^{-2\pi i \frac{L_2}{2^{\ell+2}} \xi'_2} d \xi_1' \, d \xi_2' \\
& = 2^{\ell (\alpha-1)}  \tilde C  \int_{[-{1 \over 8}, {1 \over 8}] \times [{1 \over 4}, 4]}  \int_0^1 (\xi'_2 + t \xi'_1)^{\alpha-1}dt \: e^{-2\pi i 4 L_1  \xi'_1} e^{-2\pi i \frac{L_2}{4} \xi'_2} d \xi_1' \, d \xi_2' \\
& := 2^{\ell (\alpha-1)} C^+_{L_1, L_2}.
\end{align*}

Similarly, a double Fourier series decomposition can be applied to $m^{\ell,-}_{C_{\alpha}}(\xi_2, \xi_1)$ on $[-2^{\ell+2}, -2^{\ell-2}] \times [-2^{\ell-3}, 2^{\ell-3}]$ so that
\begin{equation*}
m^{\ell,-}_{C_{\alpha}}(\xi_2, \xi_1) = \sum_{L_1,L_2 \in \BBZ}C^{\ell,-}_{L_1, L_2} e^{2\pi i \frac{L_1}{2^{\ell-2}} \xi_1} e^{2\pi i \frac{L_2}{2^{\ell+2}} \xi_2},
\end{equation*}
where the Fourier coefficients can be expressed as
\begin{align*}
C^{\ell,-}_{L_1, L_2} & = 2^{\ell (\alpha-1)}  \tilde C  \int_{[-{1 \over 8}, {1 \over 8}] \times [-4, -{1 \over 4}]}  \int_0^1 (\xi'_2 + t \xi'_1)^{\alpha-1}dt \: e^{-2\pi i 4 L_1  \xi'_1} e^{-2\pi i \frac{L_2}{4} \xi'_2} d \xi_1' \, d \xi_2' := 2^{\ell(\alpha-1)}C^-_{L_1,L_2}.
\end{align*}

A straightforward, but important observation is the fact that the new coefficients $C^+_{L_1, L_2}$ and $C^-_{L_1,L_2}$ do not depend on the parameter $\ell$; moreover, due to the fact that $\int_0^1 (\xi'_2 + t \xi'_1)^{\alpha-1}dt$ is smooth for $(\xi_2', \xi_1') \in [{1 \over 4}, 4] \times [-{1 \over 8}, {1 \over 8}] \cup [-4, -{1 \over 4}] \times [-{1 \over 8}, {1 \over 8}]$, we also deduce their fast decay:
\begin{align} \label{strategy:comm_Fourier_coef_decay}
|C^+_{L_1, L_2}|, |C^-_{L_1, L_2}| \lesssim_{M} \frac{1}{(1+|L_1| + |L_2|)^{M}},
\end{align}
for any $M>0$.

Recalling \eqref{eq:expr:comm}, $[D^\alpha, \Delta_k f] \Delta_\ell g$ becomes
\begin{equation}
\label{eq:commutaor:superpositiom:no:norms}
\begin{aligned}
& \sum_{L_1,L_2} C^+_{L_1, L_2}  2^{\ell (\alpha-1)} 2^k  \big( \tilde{\tilde{\Delta}}_{k, {L_1 \over 2^{\ell-3}}} \Delta_{k} f \big)(x)    \big( \tilde \Delta_{\ell, +, {L_2 \over 2^{\ell+2}}} \Delta_{\ell} g \big)(x) \\
 + &\sum_{L_1,L_2} C^-_{L_1, L_2}  2^{\ell (\alpha-1)} 2^k  \big( \tilde{\tilde{\Delta}}_{k, {L_1 \over 2^{\ell-3}}} \Delta_{k}  f \big)(x)    \big(  \tilde \Delta_{\ell, -, {L_2 \over 2^{\ell+2}}} \Delta_{\ell} g \big)(x) = I_{+} + I_{-}.
\end{aligned}
\end{equation}
In other words, the commutator $[D^\alpha, \Delta_k f] \Delta_\ell g$ emerges as a superposition of products of modulated Littlewood-Paley projections, times $2^{\ell (\alpha-1)} 2^k$. 

So if $\tau\leq \min(1, p)$,
\begin{equation}
\label{eq:commutator:superposition}
\begin{aligned}
\|I_{\pm}\|_p^\tau & \lesssim 
\sum_{L_1,L_2} \big|C^{\pm}_{L_1, L_2}\big|^\tau   2^{\ell (\alpha-1) \tau} 2^{k \tau} \| \tilde{\tilde{\Delta}}_{k, {L_1 \over 2^{\ell-3}}} \Delta_{k} f \|_{p_1}^\tau \|\tilde \Delta_{\ell, \pm, {L_2 \over 2^{\ell+2}}} \Delta_{\ell} g  \|_{p_2}^\tau \\
& \lesssim  \sum_{L_1,L_2} |C^{\pm}_{L_1, L_2}|^{\tau}   2^{\ell (\alpha-1) \tau} 2^{k \tau} \| \Delta_{k} f \|_{p_1}^\tau \| \Delta_{\ell} g \|_{p_2}^\tau. 
\end{aligned}
\end{equation}
The fast decay of the $C^{\pm}_{L_1, L_2}$ coefficients from \eqref{strategy:comm_Fourier_coef_decay} and the estimate \eqref{eq:direction:LP:same} imply \eqref{result:commutator:1}. 

Following the same ideas, one can provide a direct proof for \eqref{result:Leibniz:Whitney} without invoking the commutator estimate \eqref{result:commutator:1}. For example, one can perform a double Fourier series decomposition for the smooth function
\[
|\xi_1+\xi_2|^\alpha \tilde \varphi_{\ell}(\xi_1) \tilde{\psi}_{\ell, \pm}(\xi_2)
\] 
on $[-2^{\ell-3}, 2^{\ell-3}] \times [2^{\ell-2}, 2^{\ell+2}]$ and $[-2^{\ell-3}, 2^{\ell-3}] \times [-2^{\ell+2}, -2^{\ell-2}]$ respectively.
\end{enumerate}
\end{proof}

\begin{remark} \label{remark:Bourgain-Li:generalizations}
We would like to draw attention to a certain component in the above argument: the Fourier series decomposition \eqref{eq:commutaor:superpositiom:no:norms} (and its consequence \eqref{eq:commutator:superposition}) will be as important as the localized estimates \eqref{result:commutator:1} in the treatment of generic flags. More concretely, $\|\Delta_k f\|_{p_1}$ can be replaced by $\| D^\alpha (\Delta_k f_1\cdot S_k f_2)\|_t$ with $t >0$ (possibly $t<1$) in the study of a flag $D^\beta(D^\alpha (\Delta_k f_1\cdot S_k f_2) \Delta_{\ell} g)$. The Fourier series decomposition reduces the original estimate to superpositions of subtree estimates, thus tensorizing a generic flag into flags of lower complexity. However, quantitative estimates in the spirit of \eqref{result:commutator:1} for $\| D^\alpha (\Delta_k f_1\cdot S_k f_2)\|_t$ will be needed in order to conclude by summing up the scales $k$ and $\ell$, as before. 
\end{remark}

\begin{remark}\label{remark:higher:dim} 
(1)
In higher dimensions, when $\xi_1, \xi_2 \in \BBR^d$, we cannot use \eqref{def:commutator} anymore; instead, we need to directly handle
\begin{equation}
\label{eq:comm:higher}
|\xi_1+\xi_2|^{\alpha}- |\xi_2|^\alpha = \int_0^1 \xi_1 \cdot (t \xi_1+ \xi_2) |t \xi_1+ \xi_2|^{\alpha -2} dt.
\end{equation}
in the off-diagonal region $\{(\xi_1, \xi_2) \in \BBR^{2d}: |\xi_2| \gg |\xi_1|\}$.
So when we want to estimate
\begin{equation} \label{eq:symbol_annulus}
\int_{\BBR^{2d}} \big(|\xi_1+\xi_2|^\alpha -|\xi_2|^\alpha \big) \widehat{\Delta_k f}(\xi_1)  \widehat{\Delta_\ell g}(\xi_2) e^{ 2 \pi i x \cdot (\xi_1+\xi_2)} d \xi_1 d \xi_2,
\end{equation}
under the assumption that $|\xi_1| \sim 2^k \ll 2^\ell \sim |\xi_2|$, and perform a double Fourier series decomposition for
\[
\big(|\xi_1+\xi_2|^\alpha -|\xi_2|^\alpha \big) \tilde{\psi}_k(\xi_1)  \tilde{\psi}_\ell(\xi_2),
\]
we need to further localize the symbol above onto ``Whitney boxes/rectangles'' of sizes $ \sim 2^k \times 2^\ell$.

This requires a technical (although standard) decomposition of the region $\{ \xi \in \BBR^d : \xi \neq 0\}$ as a finite\footnote{The numbers of cones depends on the dimension.} collection $\mathfrak{C}$ of directional cones; moreover, to each $\mathfrak{c} \in \mathfrak{C}$ we associate a collection of Whitney cubes $\{ Q_{k, \mathfrak{c}}\}_{k \in \BBZ}$ covering the conical regions and having the property that
\[
\text{sidelength}(Q_{k, \mathfrak{c}}) \sim 2^k \sim \text{dist}(Q_{k, \mathfrak{c}}, 0).
\]
In this way, for any $k \in \BBZ$, we have
\[
\psi_{k}(\xi)= \sum_{\mathfrak{c} \in \mathfrak{C}} \tilde\psi_{k, \mathfrak{c}}(\xi) \psi_{k}(\xi),
\]
which should bring back to mind the decomposition performed in \eqref{symbol_undirectional_WC}-\eqref{eq:doubleFS:3}.

So the symbol $|\xi_1+\xi_2|^\alpha -|\xi_2|^\alpha$, initially localized through Littlewood-Paley projections $\tilde \psi_{k}(\xi_1)$ and $\tilde \psi_{\ell}(\xi_2)$, now becomes
\begin{equation*}
\big(|\xi_1+\xi_2|^\alpha -|\xi_2|^\alpha \big) \tilde \psi_{k}(\xi_1) \tilde \psi_{\ell}(\xi_2) = \sum_{\mathfrak{c}_1, \mathfrak{c}_2 \in \mathfrak{C}}\big(|\xi_1+\xi_2|^\alpha -|\xi_2|^\alpha \big) \tilde \psi_{k,\mathfrak{c}_1}(\xi_1) \tilde \psi_{\ell,\mathfrak{c}_2}(\xi_2) \tilde \psi_{k}(\xi_1) \tilde \psi_{\ell}(\xi_2).
\end{equation*}

We then use a double Fourier series decomposition on each Whitney rectangle $Q_{k,\mathfrak{c}_1} \times Q_{\ell,\mathfrak{c}_2}$ in order to capture the interaction between two different scales. We have 
\begin{align*}
\big( |\xi_1+\xi_2|^{\alpha}- |\xi_2|^\alpha \big)  \tilde \psi_{k,\mathfrak{c}_1}(\xi_1) \tilde \psi_{\ell,\mathfrak{c}_2}(\xi_2) = 2^{\ell(\alpha-1)} 2^k \sum_{L_1, L_2 \in \BBZ^d} C_{L_1, L_2}^{k, \ell,\mathfrak{c}_1,\mathfrak{c}_2} e^{2\pi i \frac{L_1}{2^{k}} \xi_1} e^{2\pi i \frac{L_2}{2^{\ell}} \xi_2},
\end{align*}
where
\begin{equation}
\label{eq:higher:dim:Fourier:coeff}
C_{L_1, L_2}^{k, \ell,\mathfrak{c}_1,\mathfrak{c}_2} = \int_{\{ |\xi'_1|\sim 1\} \cap c_1 }  \int_{\{ |\xi'_2|\sim 1\} \cap c_2 }  \tilde \psi_{0,\mathfrak{c}_1}(\xi'_1) \tilde \psi_{0,\mathfrak{c}_2}(\xi'_2) \big( \int_0^1 \xi_1' \cdot (2^{k- \ell} t \xi_1' + \xi_2') |2^{k- \ell} t \xi_1' + \xi_2'|^{\alpha-2} dt \big) \, e^{-2 \pi i L_1 \cdot \xi_1'} e^{-2 \pi i L_2 \cdot \xi_2'} d \xi_1' d \xi_2'.
\end{equation}
In contrast to the one-dimensional case, the renormalized Fourier coefficients do depend on the scales $2^k$ and $2^\ell$. However, since $2^{k- \ell} \ll 1$ and the multiplier 
\begin{equation}
\label{eq:comm:multip:higher}
\tilde \psi_{0,\mathfrak{c}_1}(\xi'_1) \tilde \psi_{0,\mathfrak{c}_2}(\xi'_2) \int_0^1 \xi_1' \cdot (2^{k- \ell} t \xi_1' + \xi_2') |2^{k- \ell} t \xi_1' + \xi_2'|^{\alpha-2} dt 
\end{equation}
is smooth on its support included in $\{ (\xi_1', \xi_2') \in \BBR^{2d} : |\xi'_1|\sim 1, |\xi'_2|\sim 1 \} \cap \mathfrak{c}_1 \times \mathfrak{c}_2$, we do have, \emph{uniformly in $k$ and $\ell$},
\begin{equation}
\label{eq:decay:cond:higher}
|C_{L_1, L_2}^{k, \ell, \mathfrak{c}_1, \mathfrak{c}_2}|\lesssim_{M} \frac{1}{(1+|L_1| + |L_2|)^{M}},
\end{equation}
for any $M>0$. Indeed, this is a consequence of the uniform boundedness of the derivatives of the multiplier in \eqref{eq:comm:multip:higher}. Since the above estimate holds for any tensor product of directional cones $\mathfrak{c}_1 \times \mathfrak{c}_2$ with $\mathfrak{c}_1, \mathfrak{c}_2 \in \mathfrak{C}$ and there are only finitely many directional cones in $\mathfrak{C}$, we obtain the desired estimate for \eqref{eq:symbol_annulus}.~\\
(2)
In the diagonal region $\{(\xi_1,\xi_2) \in \BBR^{2d}: |\xi_1| \sim |\xi_2| \}$, $|\xi_1+\xi_2|^\alpha$ is not supported away from the origin, we perform a Fourier series decomposition of $|\xi_1+\xi_2|^\alpha \phi_\ell(\xi_1+\xi_2)$ on the cube $[-2^{\ell+3}, 2^{\ell+3}]^d \subseteq \BBR^d$ and notice that the corresponding Fourier coefficients only have limited decay:
\begin{equation}
\label{eq:limited:decay:higher:dim}
|C_L|\lesssim \frac{2^{\ell \alpha}}{\left( 1+|L| \right)^{\alpha+d}}. 
\end{equation}

The way the Fourier series are being performed is the only point in the proof which is different in higher dimensions. More precisely, the Fourier series decomposition requires a smooth localization of the symbol to Whitney rectangles within the region $\{ (\xi_1, \xi_2) : |\xi_1| \sim 2^k, |\xi_2| \sim 2^\ell \}$, and the renormalized Fourier coefficients do depend on the scales $k$ and $\ell$, although they are still well-behaved thanks to \eqref{eq:decay:cond:higher}. In order to avoid unnecessary complications, in what follows we will focus on the one dimensional case. 
\end{remark}

\smallskip 

\subsection{Some contrasting aspects with the Coifman-Meyer multiplier approach}~\\
\label{sec:differences:Coifman-Meyer}
Alternatively, one might want to use the boundedness of Coifman-Meyer multipliers to deduce the Leibniz rule \eqref{Leib_one_paraproduct:strategy}. The Littlewood-Paley decompositions $f =\sum_{k \in \BBZ} \Delta_k f$, $g=\sum_{\ell \in \BBZ} \Delta_\ell g$ will again be grouped into ``paraproducts'', with possibly altered projections:
\begin{align*}
f(x)g(x) & = \sum_{k \ll \ell} \Delta_k f (x) \Delta_\ell g(x) + \sum_{\ell \ll k} \Delta_k f (x) \Delta_\ell g(x) + \sum_{ |k - \ell| \leq 3} \Delta_k f (x) \Delta_\ell g(x)\\
& = \sum_{\ell \in \BBZ} S_\ell f(x) \Delta_\ell g(x) + \sum_{k \in \BBZ} \Delta_k f(x) S_k g(x) + \sum_{k \in \BBZ}   \Delta_k f (x) \tilde \Delta_k g(x).
\end{align*}
Then $D^\alpha(f \cdot g)$ becomes
\begin{align*}
D^\alpha( \sum_{\ell \in \BBZ} S_\ell f \cdot  \Delta_\ell g) + D^\alpha ( \sum_{k \in \BBZ} \Delta_k f \cdot S_k g) + D^\alpha (\sum_{k \in \BBZ}   \Delta_k f \cdot \tilde \Delta_k g);
\end{align*}
the first two terms being symmetric, it suffices to understand the first and the third one. Now we notice the following:
\\
\noindent
\textbf{(1)} the first term can be expressed as
\[
D^\alpha( \sum_{\ell \in \BBZ} S_\ell f \cdot  \Delta_\ell g)= \int_{\BBR^2} \sum_{\ell \in \BBZ} \varphi(2^\ell \xi) \psi(2^\ell \eta) \frac{|\xi+\eta|^\alpha}{|\eta|^\alpha} \hat f(\xi) \widehat{ D^\alpha g}(\eta)  e^{ 2 \pi i x(\xi+\eta)} d \xi d \eta;
\]
where $$\ds m_1(\xi, \eta):= \sum_{\ell \in \BBZ} \varphi(2^\ell \xi) \psi(2^\ell \eta) \frac{|\xi+\eta|^\alpha}{|\eta|^\alpha}$$ is a Coifman-Meyer symbol, satisfying \eqref{eq:Mikhlin:symbol}; this is because $m_1$ is morally constant on Whitney cubes $[-2^{\ell-2}, 2^{\ell-2}] \times [2^\ell, 2^{\ell+1}]$.
As a consequence, $$\ds D^\alpha( \sum_{\ell \in \BBZ} S_\ell f \cdot  \Delta_\ell g)(x) = T_{m_1}(f, D^\alpha g)(x),$$ where $T_{m_1}$ is the Coifman-Meyer multiplier associated to the Mikhlin symbol $m_1(\xi, \eta)$.
\\
\noindent
\textbf{(2)}
In a similar way, thanks to \eqref{eq:Fourier:dec:diagonal}, $\ds  D^\alpha (\sum_{k \in \BBZ}   \Delta_k f \cdot \tilde \Delta_k g)$ can be represented as a superposition of Coifman-Meyer-type multipliers
\[
D^\alpha (\sum_{k \in \BBZ}   \Delta_k f \cdot \tilde \Delta_k g)= \sum_{\tilde L \in \BBZ} C_{\tilde L} T_{m_{\tilde L}}(f, D^\alpha g),
\]
where 
\[
 m_{\tilde L}(\xi, \eta):= \sum_{\ell \in \BBZ} \psi(2^{-\ell} \xi) e^{2 \pi i \frac{\tilde L}{2^\ell} \xi} \tilde \psi(2^{-\ell} \eta) e^{2 \pi i \frac{\tilde L}{2^\ell} \eta}  2^{\ell \alpha} |\eta|^{-\alpha}
\]
is a symbol singular only at the origin, but depending explicitly on $\tilde L$. Since $\|T_{ m_{\tilde L}}\|_{L^{p_1} \times L^{p_2} \to L^p}$ can be shown to depend at most logarithmically\footnote{Indeed, the frequency modulation $e^{2 \pi i \tilde L \cdot}$ has an $\tilde L$-shifting effect in space (at every scale) and invoking the boundedness of shifted square functions produces the desired result; see \cite{graf-Leibniz_rules} or \cite{multilinear_harmonic}.} on $1+|\tilde L|$ whenever $1/{p_1}+1/{p_2}=1/p$, with $1<p_1, p_2\leq \infty, 1/2 < p<\infty$, \eqref{Leib_one_paraproduct:strategy} follows for $1<p_1, p_2\leq \infty, 1/2 < p<\infty$ such that $$\sum_{\tilde L} |C_{\tilde L}|^{\min(1, p)}$$ is summable, i.e. for $\frac{1}{1+\alpha}<p<\infty$.
\vskip .05in
We conclude with a brief comparison between the Coifman-Meyer multiplier approach and the Bourgain-Li approach presented above in Section \ref{Bourgain-Li_hilow}. 
\begin{enumerate}[label=(\roman*), leftmargin=*]
\item the Bourgain-Li approach can deal with endpoints for which Coifman-Meyer multipliers fail to be bounded; this was the original framework in which it was introduced -- the $L^\infty \times L^\infty \to L^\infty$ endpoint;
\item \label{obs:shared:Deriv} in the Bourgain-Li approach, which involves summation over the scales, the derivatives are allocated jointly to the functions $f$ and $g$: for small scales, both $f$ and $g$ get no derivatives; in contrast, for large scales $f$ is being attributed $\alpha$ derivatives and $g$ picks up $\epsilon$ derivatives for some $\epsilon \in (0, \alpha)$. Interpolation eventually yields, for some $0< \theta<1$,
\begin{equation}
\label{eq:Leib:comm:better}
\| D^{\alpha} (f g) - D^\alpha f \cdot g - f \cdot D^\alpha g\|_{L^p} \lesssim \big( \| f\|_{\dot B^\alpha_{p_1, \infty}} \| g\|_{\dot B^0_{p_2, \infty}}   \big)^\theta \, \big(   \| f\|_{\dot B^0_{p_1, \infty}} \| g\|_{\dot B^\alpha_{p_2, \infty}}   \big)^{1-\theta},
\end{equation}
which is a sharper estimate than \eqref{Leib_one_paraproduct:strategy};
\item 
in the Coifman-Meyer multiplier approach, the derivative will always be attached to the function with higher oscillation, as one can see from the identity $\ds D^\alpha( \sum_{\ell \in \BBZ} S_\ell f \cdot  \Delta_\ell g)(x) = T_{m_1}(f, D^\alpha g)(x)$; as a consequence, one can obtain ``off-diagonal'' $L^p$ estimates for the Leibniz rule, in the sense that \eqref{Leib_one_paraproduct:strategy} can be replaced by 
\begin{align}\label{Leib_one_paraproduct:soff:diagonal}
\| D^{\alpha} (f g) \|_{L^p} \lesssim \|D^{\alpha}f\|_{L^{p_1}}\|g\|_{L^{p_2}} + \|f\|_{L^{q_1}}\|D^{\alpha} g \|_{L^{q_2}},
\end{align}
where $\frac{1}{p}=\frac{1}{p_1}+\frac{1}{p_2}=\frac{1}{q_1}+\frac{1}{q_2}$, and $1<p_1, p_2, q_1, q_2 \leq \infty$, $\frac{1}{2}< p < \infty$. Interestingly, such off-diagonal estimates proved to be useful in certain applications -- for example \cite{KatoPonceCommut}, \cite{kenig1993_vvLeibniz}.
\item the Coifman-Meyer approach yields another type of improvement for the ``commutator'' in \eqref{eq:Leib:comm:better}:
\[
\| D^{\alpha} (f g) - D^\alpha f \cdot g - f \cdot D^\alpha g\|_{L^p} \lesssim \|D^{\alpha_1}f\|_{L^{p_1}}\|D^{\alpha_2} g\|_{L^{p_2}},
\]
for any $0<\alpha_1, \alpha_2<\alpha$ with $\alpha_1 + \alpha_2=\alpha$. Because of \ref{obs:shared:Deriv} and the role played by interpolation, without substantial modifications, the Bourgain-Li approach cannot recover such a result. 
\end{enumerate}

Our methodology for proving estimates for generic flag Leibniz rules and other Leibniz-type estimates can be perceived as a generalization of the Bourgain-Li approach presented above, as hinted in Remark \ref{remark:Bourgain-Li:generalizations}.

\section{A five-linear flag: one-parameter case}
\label{sec:5lin:1param}

Before dealing with the generic result, we consider a simpler example: $D^\beta( D^\alpha(f_1 \cdot f_2) \cdot f_3  \cdot D^\gamma( f_4 \cdot f_5)  )$, which is a one-parameter version of \eqref{bi_Muscalu}. This will allow us to emphasize the main ideas without getting too technical. We will assume that $d=1$, although the argument remains identical in higher dimensions.

We will prove, for $\alpha, \beta, \gamma \geq 0$ that
\begin{align} \label{eq:one:param_M}
&\|D^{\beta}\big(D^{\alpha}(f_1 f_2) f_3 D^{\gamma}(f_4 f_5)\big) \|_{L^{r}}  \\
 \lesssim 
& \|D^{\alpha+\beta}f_1\|_{L^{p_1}} \|f_2\|_{L^{p_2 }} \|f_3\|_{L^{ p_3}}\|D^{\gamma}f_4\|_{L^{p_4}}\|f_5\|_{L^{p_5}} + \| f_1\|_{L^{p_1}} \|D^{\alpha+\beta} f_2\|_{L^{p_2 }} \|f_3\|_{L^{p_3 }} \|D^{\gamma}f_4\|_{L^{p_4 }}\|f_5\|_{L^{p_5 }} \nonumber\\
+&\|D^{\alpha+\beta}f_1\|_{L^{p_1 }} \|f_2\|_{L^{p_2}} \|f_3\|_{L^{p_3}}\|f_4\|_{L^{p_4 }}\|D^{\gamma}f_5\|_{L^{p_5}} + \text{ other similar terms}, \nonumber
\end{align}
where $1 \leq p_1, \ldots, p_5 \leq \infty$, $1/5 \leq r \leq \infty$ and
\begin{equation}
\label{eq:cond:5lin:musc:1param}
\frac{1}{r}=\frac{1}{p_1}+\ldots+\frac{1}{p_5}, \quad \frac{1}{r}<1+\beta, \quad \frac{1}{p_{1,2}}:=\frac{1}{p_1}+\frac{1}{p_2} < 1+\alpha, \quad \frac{1}{p_{4,5}}:=\frac{1}{p_4}+\frac{1}{p_5} < 1+\gamma.
\end{equation}
Whenever $\alpha, \beta$ or $\gamma \in 2 \BBZ$, the constraint \eqref{eq:cond:5lin:musc:1param} can be removed.

Our approach for the above Leibniz rule relies on an iterative argument that in particular makes use of frequency-localized estimates for $D^{\alpha}(f_1 f_2)$ and $D^{\gamma}(f_4 f_5)$, in the spirit of Lemma \ref{lemma:multipliers:simpleParaproduct}. The present proof is not as systematic as the one in Section \ref{generic_induction}, although many elements are contained in the treatment of this particular example.

We use Littlewood-Paley projections to decompose the functions in frequency into dyadic pieces
\[
f_l=\sum_{k_l} \Delta_{k_l} f_l, \qquad \text{for all    } 1 \leq l \leq 5,
\]
for which the derivation becomes equivalent to multiplication: $D^\alpha \big(  \Delta_{k_l} f_l  \big) \sim 2^{\alpha  k_l}   \Delta_{k_l} f_l$. That means that we need to estimate
\begin{align}
\label{5lin:freq:rep}
\sum_{k_1, \ldots, k_5 } \int_{\BBR^5} |\xi_1+ \ldots + \xi_5|^\beta \cdot |\xi_1+\xi_2|^\alpha \cdot |\xi_4+\xi_5|^\gamma \widehat{\Delta_{k_1} f_1}(\xi_1) \cdot \ldots \cdot \widehat{\Delta_{k_5} f_5}(\xi_5) e^{2 \pi i x(\xi_1+\ldots+\xi_5)} d \xi_1 \ldots d\xi_5.
\end{align}

\[
\begin{forest}
    my treeSep
    [, label={above:$D^\beta$}
                       [, label={left: $D^\alpha$}
                                    [, label={below: $ \Delta_{k_1} f_1$}]
                                    [, label={below: $\Delta_{k_2} f_2$}]
                           ]
                        [, label={below:$\Delta_{k_3} f_3$}]
                            [, label={right: $D^\gamma$}
                                    [, label={below: $\Delta_{k_4} f_4$}]
                                    [, label={below: $\Delta_{k_5} f_5$}]
                             ]
]
\end{forest} 
\]

As explained in Sections \ref{strategy} and \ref{sec:Section2}, the frequency space will be split in various conical regions, producing in this way the classical paraproduct decomposition. Restrictions of \eqref{5lin:freq:rep} to each of these regions need to be analyzed, and we will see that the leading derivatives will be distributed among two functions: the highest oscillating functions and an auxiliary one.

Due to the structure of the present flag, we will need to consider several conical regions in frequency: 
\begin{enumerate}[label=(\Roman*), leftmargin=*]
\item \label{caseI:k_1} the region where $f_1$ is the fastest oscillating function:
\begin{equation} \label{Sec3:off-diagonal:1}
R_1=\{ (\xi_1, \ldots, \xi_5) :  |\xi_1| \gg |\xi_2|, \ldots, |\xi_5|    \}.
\end{equation}
This region is symmetric to those where $f_2, f_4$ or $f_5$ oscillate much faster than the remaining functions.
\item \label{caseII:k_3} the region
\begin{equation} \label{Sec3:off-diagonal:3}
R_3=\{ (\xi_1, \ldots, \xi_5) :  |\xi_3| \gg |\xi_1|, |\xi_2|, |\xi_4|, |\xi_5|    \},
\end{equation}
where $f_3$ is the fastest oscillating function.

\item \label{caseIII:multi} ``diagonal'' regions 
\begin{equation} \label{Sec3:diagonal}
%\tilde R_{1, 3}=\{ (\xi_1, \ldots, \xi_5) :  |\xi_1| \sim  |\xi_3| \geq |\xi_2|, |\xi_4|, |\xi_5|  \},
\tilde R_{l_1, l_2}=\{ (\xi_1, \ldots, \xi_5) :  |\xi_{l_1}| \sim  |\xi_{l_2}| \geq |\xi_{l'}| \ \ \text{for} \ \ l' \neq l_1, l_2 \},
\end{equation}
for $l_1 \neq l_2$.
%and similarly defined $\tilde R_{1, 4}$ and $R_{1,2}$, etc. 
In this situation, at least two of the functions oscillate at comparable high rates, which might cause $|\xi_1+\ldots+\xi_5|^\beta$ to become more singular than in the previous cases.
\end{enumerate}

\medskip

\subsection{\ref{caseI:k_1}: study of the conical region $R_1$}~\\
In \eqref{5lin:freq:rep}, we restrict the summation over $k_1 \ll k_2, \ldots, k_5$: we are in the situation when $2^{k_1} \sim |\xi_1| \ll |\xi_l| \sim 2^{k_l} $ for all $2 \leq l \leq 5$. In this case, $$2^{k_1-1} <|\xi_1+ \ldots+ \xi_5| < 2^{k_1+1},$$
which we write in short $|\xi_1+ \ldots+ \xi_5| \sim 2^{k_1}$.

Following the principle introduced by Bourgain and Li, in this scenario we would like to approximate $|\xi_1+ \ldots+ \xi_5|^\beta$ by $|\xi_1|^\beta$, and use to good advantage the better-behaving commutator $|\xi_1+ \ldots+ \xi_5|^\beta -|\xi_1|^\beta$. Since a $D^\alpha$ derivative also acts on $f_1$, we take an intermediate step in which we approximate $|\xi_1+ \ldots+ \xi_5|^\beta$ by $|\xi_1+\xi_2|^\beta$. This is consistent with the overall approach, since $|\xi_1+\xi_2| \sim 2^{k_1}$ as well. Thus the symbol corresponding to this particular Leibniz rule, which appears in \eqref{5lin:freq:rep}, breaks down as 
\begin{align*}
\label{eq:multipl:I}
&\frac{|\xi_1+ \ldots+ \xi_5|^\beta  - |\xi_1+\xi_2|^\beta}{\xi_3+\xi_4+\xi_5} \cdot \xi_3 \cdot |\xi_1+\xi_2|^\alpha \cdot |\xi_4+\xi_5|^\gamma + \frac{|\xi_1+ \ldots+ \xi_5|^\beta  - |\xi_1+\xi_2|^\beta}{\xi_3+\xi_4+\xi_5} \cdot (\xi_4+\xi_5) \cdot |\xi_1+\xi_2|^\alpha \cdot |\xi_4+\xi_5|^\gamma \\
& \qquad  + \frac{|\xi_1+\xi_2|^{\alpha+\beta}  - |\xi_1|^{\alpha+\beta}}{\xi_2} \cdot \xi_2 \cdot |\xi_4+\xi_5|^\gamma +  |\xi_1|^{\alpha+\beta} \cdot |\xi_4+\xi_5|^\gamma:= m_{I_A}+m_{I_B}+m_{I_C}+m_{I_D},
\end{align*}
which means that \eqref{5lin:freq:rep} restricted to the frequency region $R_1$ \eqref{Sec3:off-diagonal:1} reads as
\begin{align*}
&\sum_{k_2, \ldots, k_5 \ll k_1 } T_{m_{I_A}}(\Delta_{k_1}f_1, \ldots, \Delta_{k_5}f_5)(x)+ \sum_{k_2, \ldots, k_5  \ll k_1} T_{m_{I_B}}(\Delta_{k_1} f_1, \ldots,\Delta_{k_5} f_5)(x) \\
& + \sum_{k_2, \ldots, k_5 \ll k_1} T_{m_{I_C}}(\Delta_{k_1} f_1, \ldots, \Delta_{k_5} f_5)(x)+ \sum_{k_2, \ldots, k_5 \ll k_1 } T_{m_{I_D}}(\Delta_{k_1} f_1, \ldots, \Delta_{k_5} f_5)(x) := I_A+I_B+I_C+I_D.
\end{align*}
This corresponds to the first step of the strategy presented in the introduction: the \emph{splitting of the root symbol}. In Section \ref{generic_induction}, the splitting will be different: the emphasis will be put on the subtree structures obtained by removing the tree root corresponding to $D^\beta$. Here instead we track the root derivatives $D^\beta$ as they descend towards the leaves represented by the functions $f_1, \ldots, f_5$.

We will see that, due to the multiplier's shape, in the study of $I_A$ the functions $f_1$ and $f_3$ will play a special role and the $D^\beta$ derivatives will be shared among them; for $I_B$, it will be $f_1$ and one of $f_4$ or $f_5$ (an extra paraproduct decomposition will be used here, which will also determine the distribution of derivatives); for $I_C$, $f_1$ and $f_2$; and finally, for $I_D$, summing up all the scales will require a change in the order of summation.

\medskip

\begin{itemize}[ leftmargin=*]
\item[$I_A$)] In estimating the $I_A$ term, we will only need the Littlewood-Paley information for the functions $f_1$ and $f_3$; as a consequence, we can sum over $k_2, k_4, k_5 \ll k_1$ and focus on the five-linear operator
\begin{align*}
T_{m_{I_A}}^{k_1, k_3}(f_1, \ldots, f_5)(x):=\int_{\BBR^5} & \frac{|\xi_1+ \ldots+ \xi_5|^\beta  - |\xi_1+\xi_2|^\beta}{\xi_3+\xi_4+\xi_5} \cdot \xi_3 \cdot |\xi_1+\xi_2|^\alpha \cdot |\xi_4+\xi_5|^\gamma \\  & \quad \widehat{\Delta_{k_1} f_1}(\xi_1) \cdot \widehat{S_{ k_1} f_2}(\xi_2) \cdot \widehat{\Delta_{k_3} f_3}(\xi_3) \cdot \widehat{S_{k_1} f_4}(\xi_4) \cdot \widehat{S_{ k_1} f_5}(\xi_5) e^{2 \pi i x \left(\xi_1+\ldots+\xi_5 \right)} d \, \xi_1 \ldots d \, \xi_5.
\end{align*}

Next, we use \emph{Fourier series decompositions for the symbols} involved -- the second step of our strategy. For this we need to place ourselves in a suitable situation, and in particular to smoothly restrict the symbols to intervals where the Fourier series decomposition can be implemented.

First we look at $m_{C_\beta}(\xi_1+\xi_2, \xi_3+\xi_4+\xi_5)$ which, according to definition \eqref{def:commutator}, consists of  
\[
 \frac{|\xi_1+ \ldots+ \xi_5|^\beta  - |\xi_1+\xi_2|^\beta}{\xi_3+\xi_4+\xi_5} = \beta \int_0^1 |\xi_1+\xi_2+ t(\xi_3+\xi_4+\xi_5)|^{\beta-2} \big( \xi_1+\xi_2+ t(\xi_3+\xi_4+\xi_5) \big)  dt,
\]
and further restrict it to $\big[ 2^{k_1}, 2^{k_1+1} \big] \times [- 2^{k_1-1}, 2^{k_1-1} ]$ and $\big[ -2^{k_1+1}, -2^{k_1} \big] \times [- 2^{k_1-1}, 2^{k_1-1} ]$, respectively:
\[
 m^{k_1, \pm}_{C_\beta}(\xi_1+\xi_2, \xi_3+\xi_4+\xi_5):= m_{C_\beta}(\xi_1+\xi_2, \xi_3+\xi_4+\xi_5) \tilde \psi_{k_1, \pm}(\xi_1+\xi_2) \tilde \varphi_{k_1}(\xi_3+\xi_4+\xi_5).
\]
Then we proceed with a double Fourier series decomposition of $m^{k_1, \pm}_{C_\beta}$ on $\pm \big[ 2^{k_1}, 2^{k_1+1} \big] \times [- 2^{k_1-1}, 2^{k_1-1} ]$:
\begin{align}
\label{eq:FS:5lin:comm}
m^{k_1, \pm}_{C_\beta}(\xi_1+\xi_2, \xi_3+\xi_4+\xi_5)= \sum_{L_1, L_2 \in \BBZ} C_{L_1, L_2}^{\pm} 2^{k_1(\beta-1)} e^{2 \pi i  L_1 \frac{\xi_1+\xi_2}{2^{k_1}}} \, e^{2 \pi i  L_2 \frac{\xi_3+\xi_4+\xi_5}{2^{k_1}}}.
\end{align}

Similarly, we perform a Fourier series decomposition of $|\xi_1+\xi_2|^\alpha$ on the same interval $\pm [2^{k_1}, 2^{k_1+1}]$:
\begin{equation}
\label{eq:FS:I_A:2}
|\xi_1+\xi_2|^{\alpha} \tilde \psi_{k_1, \pm}(\xi_1+\xi_2)= \sum_{\tilde L \in \BBZ} C_{\tilde L}^{\pm} 2^{k_1 \alpha}  e^{\frac{2 \pi i  \tilde L \left( \xi_1+\xi_2\right)}{2^{k_1}}},
\end{equation}
and notice that the coefficients\footnote{Notice that the \emph{Fourier coefficients} consist of $C_{L_1, L_2}^{\pm} 2^{k_1(\beta-1)}$ and $C_{\tilde L}^{\pm} 2^{k_1 \alpha}$, respectively.} $C_{L_1,L_2}^{\pm}$ and $C_{\tilde L}^{\pm}$ (which depend on $\beta$ and $\alpha$, but not on $k_1$) decay fast enough: see \eqref{strategy:comm_Fourier_coef_decay}.

Thanks to these Fourier series decompositions of the symbols, $T_{m_{I_A}}(f_1, \ldots, f_5)$ becomes a superposition of tensorized operators of the form:
\begin{align*}
& \sum_{L_1, L_2 \in \BBZ} \sum_{ \tilde L \in \BBZ} C_{L_1, L_2}^{\pm} C_{\tilde L}^{\pm} \int_{\BBR^5} 2^{k_1(\beta-1)}  2^{k_3}  2^{k_1 \alpha}  \widehat{\Delta_{k_1, \pm} f_1}(\xi_1)  e^{\frac{2 \pi i  (L_1+\tilde L) \xi_1}{2^{k_1}}} \cdot \widehat{S_{ k_1} f_2}(\xi_2)  e^{\frac{2 \pi i  (L_1+\tilde L) \xi_2}{2^{k_1}}} \\ & \cdot \widehat{\Delta_{k_3} f_3}(\xi_3) \tilde \psi_{k_3}(\xi_3)  {\xi_3 \over 2^{k_3}}  e^{\frac{2 \pi i  L_2  \xi_3}{2^{k_1}}} \cdot |\xi_4 +\xi_5|^\gamma \cdot \widehat{S_{k_1} f_4}(\xi_4)  e^{\frac{2 \pi i  L_2 \xi_4}{2^{k_1}}} \cdot \widehat{S_{ k_1} f_5}(\xi_5)  e^{\frac{2 \pi i  L_2 \xi_5}{2^{k_1}}} e^{2 \pi i x \left(\xi_1+\ldots+\xi_5 \right)} d \, \xi_1 \ldots d \, \xi_5   \\
:= & \sum_{L_1, L_2 \in \BBZ} \sum_{ \tilde L \in \BBZ} C_{L_1, L_2}^{\pm} C_{\tilde L}^{\pm}  2^{k_1(\beta-1)}  2^{k_3}  2^{k_1 \alpha}  (\Delta_{{k_1},\pm,\frac{L_1+\tilde L}{2^{k_1}}} f_1)(x) \cdot (S_{{k_1},\frac{L_1+\tilde L}{2^{k_1}}} f_2)(x) \cdot (\tilde{\tilde \Delta}_{{k_3}, \frac{L_2}{2^{k_1}}} f_3)(x) \\ & \quad \cdot D^\gamma(S_{k_1,  \frac{L_2}{2^{k_1}}} f_4 \cdot S_{k_1,  \frac{L_2}{2^{k_1}}} f_5) (x).
\end{align*}

This step is precisely the \emph{tensorization into subtrees} part of our strategy.

Now we simply notice that, for $0<\tau \leq \min (1, r)$, we can use H\"older's inequality with $ \frac{1}{p_1}+\frac{1}{p_2}+ \frac{1}{p_3}+ \frac{1}{p_{4, 5}}=\frac{1}{r}$ for each of the tensorized structures:
\begin{align*}
\|T_{m_{I_A}}^{k_1, k_3}(f_1, \ldots, f_5)\|_r^\tau \lesssim \sum_{L_1, L_2 \in \BBZ} \sum_{ \tilde L \in \BBZ} |C_{L_1, L_2}^{\pm}|^\tau |C_{\tilde L}^{\pm}|^\tau   2^{k_1(\beta-1) \tau}  2^{k_3 \tau}  2^{k_1 \alpha \tau} \big\| \Delta_{{k_1}, \pm, \frac{L_1+\tilde L}{2^{k_1}}} f_1\big\|_{p_1}^\tau \big\|S_{{k_1},\frac{L_1+\tilde L}{2^{k_1}}} f_2 \big\|_{p_2}^\tau  & \\ \big \|\tilde {\tilde \Delta}_{{k_3}, \frac{L_2}{2^{k_1}}} \Delta_{k_3} f_3 \big\|_{p_3}^\tau  \big \| D^\gamma(S_{k_1,  \frac{L_2}{2^{k_1}}} f_4 \cdot S_{k_1,  \frac{L_2}{2^{k_1}}} f_5) \big \|_{p_{4,5}}^\tau&.
\end{align*}

For the last term, we invoke the bilinear Leibniz rule (in this particular case, the unified result that first appeared in \cite{OhWu})
\begin{align*}
\| D^\gamma(S_{k_1,  \frac{L_2}{2^{k_1}}} f_4  \cdot S_{k_1,  \frac{L_2}{2^{k_1}}}  f_5)\|_{p_{4,5}} & \lesssim \|D^\gamma S_{k_1,  \frac{L_2}{2^{k_1}}} f_4\|_{p_4} \cdot \| S_{k_1,  \frac{L_2}{2^{k_1}}}  f_5\|_{p_5} +  \| S_{k_1,  \frac{L_2}{2^{k_1}}} f_4 \|_{p_4} \cdot \| D^\gamma  S_{k_1,  \frac{L_2}{2^{k_1}}} f_5\|_{p_5},
\end{align*}
which holds true whenever $\frac{1}{p_{4,5}}=\frac{1}{p_4}+\frac{1}{p_5}$, $1\leq p_4, p_5 \leq \infty$, and $\frac{1}{p_{4,5}}<1+\gamma$.

Now the estimates \eqref{eq:direction:LP:same} and \eqref{eq:trivial:modulation} (and implicitly the fact that $1 \leq p_i \leq \infty$ for all $1 \leq i \leq 5$) imply that 
\begin{align*}
\|I_A \|_{r}^{\tau}& \lesssim \sum_{k_3< k_1} 2^{k_1(\beta-1) \tau}  2^{k_3 \p}  2^{k_1 \alpha \p}  \|\Delta_{{k_1}} f_1\|_{p_1}^{\p}  \|f_2\|_{p_2}^{\p}  \| \Delta_{{k_3}} f_3\|_{p_3}^{\tau} \big( \|D^\gamma f_4\|_{p_4} \cdot \| f_5\|_{p_5} +  \| f_4\|_{p_4} \cdot \| D^\gamma f_5\|_{p_5} \big)^\tau.
\end{align*}

So we are left with proving the estimate
\begin{align}
\label{eq:I_A:Lp:aim}
\sum_{k_3< k_1} 2^{k_1(\beta-1) \p}  2^{k_3 \p}  2^{k_1 \alpha \p}  \|\Delta_{{k_1}} f_1\|_{p_1}^{\p}  \|\Delta_{{k_3}} f_3\|_{p_3}^{\tau} \lesssim \|D^{\alpha+\beta} f_1\|_{p_1}^{\p} \cdot \|f_3\|_{p_3}^{\p}+ \|D^{\alpha} f_1\|_{p_1}^{\p} \cdot \|D^{\beta}f_3\|_{p_3}^{\p}.
\end{align}

Just like in Section \ref{Bourgain-Li_hilow},
\begin{align*}
\sum_{k_3< k_1} 2^{k_1(\beta-1) \p}  2^{k_3 \p}  2^{k_1  \alpha \p}  \|\Delta_{{k_1}} f_1\|_{p_1}^{\p}  \|\Delta_{{k_3}} f_3\|_{p_3}^\tau &  \lesssim \sum_{k_1} \min \big( 2^{k_1  \beta \p} \|D^\alpha  f_1\|_{\dot B_{p_1, \infty}^0}^{\p}   \| f_3\|_{\dot B_{p_3, \infty}^0}^{\p},  2^{- k_1  \epsilon \p} \|D^\alpha  f_1\|_{\dot B_{p_1, \infty}^{\beta}}^{\p}   \| f_3\|_{\dot B_{p_3, \infty}^{\epsilon}}^{\p} \big) \\
&\lesssim  \big( \|D^\alpha  f_1\|_{\dot B_{p_1, \infty}^0}^{\p}   \| f_3\|_{\dot B_{p_3, \infty}^0}^{\p}\big)^{\epsilon \over {\beta+\epsilon}}  \big( \|D^\alpha  f_1\|_{\dot B_{p_1, \infty}^{\beta}}^{\p}   \| f_3\|_{\dot B_{p_3, \infty}^{\epsilon}}^{\p}  \big)^{\beta \over {\beta+\epsilon}} \\
&\lesssim  \big( \|D^\alpha  f_1\|_{\dot B_{p_1, \infty}^0}^{\p}   \| f_3\|_{\dot B_{p_3, \infty}^\beta}^{\p} \big)^{\epsilon \over {\beta+\epsilon}}  \big( \|D^\alpha  f_1\|_{\dot B_{p_1, \infty}^{\beta}}^{\p}   \| f_3\|_{\dot B_{p_3, \infty}^{0}}^{\p}  \big)^{\beta \over {\beta+\epsilon}}.
\end{align*}

Using Young's inequality and standard properties of Besov norms -- more specifically \eqref{eq:obs:besov}, we deduce that the expression above is bounded by 
\[
\big(  \|D^\alpha  f_1\|_{p_1}  \|D^\beta f_3\|_{p_3} +  \|D^{\alpha +\beta} f_1\|_{p_1}  \| f_3\|_{p_3} \big)^{\p}.
\]

\medskip
\item[$I_B)$] In frequency we are still restricted to the region $R_1$ \eqref{Sec3:off-diagonal:1}, but the shape of the multiplier $m_{I_B}$ suggests the important role played by $\xi_4+\xi_5$ in the current situation, which requires an extra conical decomposition. We assume without loss of generality that we are restricted to the region
\[
R_{1, 5}=\{ (\xi_1, \ldots, \xi_5) :  |\xi_1| \gg |\xi_2|, \ldots, |\xi_5| \text{    and   } |\xi_4| \leq |\xi_5|  \}.
\]
In consequence, we can sum in \eqref{5lin:freq:rep} over $k_2, k_3 \ll k_1$ and $k_4 \leq k_5$ and focus our attention on 
\begin{align*}
T_{m_{I_B}}^{k_1, k_5}(f_1, \ldots, f_5)(x):=\int_{\BBR^5} & \frac{|\xi_1+ \ldots+ \xi_5|^\beta  - |\xi_1+\xi_2|^\beta}{\xi_3+\xi_4+\xi_5} \cdot (\xi_4+\xi_5) \cdot |\xi_1+\xi_2|^\alpha \cdot |\xi_4+\xi_5|^\gamma \\  & \widehat{\Delta_{k_1} f_1}(\xi_1) \cdot \widehat{S_{ k_1} f_2}(\xi_2) \cdot \widehat{S_{k_1} f_3}(\xi_3) \cdot \widehat{\Delta_{\leq k_5} f_4}(\xi_4) \cdot \widehat{\Delta_{k_5} f_5}(\xi_5) e^{2 \pi i x \left(\xi_1+\ldots+\xi_5 \right)} d \, \xi_1 \ldots d \, \xi_5.
\end{align*}

As before, we perform a double Fourier series decomposition of $m_{C_\beta}^{k_1, \pm}$ on $\pm \big[ 2^{k_1}, 2^{k_1+1} \big] \times [- 2^{k_1-1}, 2^{k_1-1} ]$, and similarly to \eqref{eq:FS:I_A:2}, we decompose $|\xi_1+\xi_2|^\alpha$ on $\pm \big[ 2^{k_1}, 2^{k_1+1} \big]$; in both cases, the coefficients have arbitrary decay.

If we denote by $d$ the classical derivative on $\BBR$, we have 
\begin{align*}
\|T_{m_{I_A}}^{k_1, k_5}(f_1, \ldots, f_5)\|_r^\tau \lesssim \sum_{L_1, L_2 \in \BBZ} \sum_{ \tilde L \in \BBZ} |C_{L_1, L_2}^{\pm}|^\tau |C^{\pm}_{\tilde L}|^\tau   2^{k_1(\beta-1) \tau}  2^{k_1 \alpha \tau}  \big\| \Delta_{{k_1}, \pm, \frac{L_1+\tilde L}{2^{k_1}}} f_1\big\|_{p_1}^\tau \big\|S_{{k_1},\frac{L_1+\tilde L}{2^{k_1}}} f_2 \big\|_{p_2}^\tau \big \|S_{k_1, \frac{L_2}{2^{k_1}}} f_3 \big\|_{p_3}^\tau & \\ \cdot \big \|  D^{\gamma} \circ d (\Delta_{\leq k_5 ,  \frac{L_2}{2^{k_1}}} f_4 \cdot \Delta_{k_5,  \frac{L_2}{2^{k_1}}} f_5) \big \|_{p_{4,5}}^\tau&.
\end{align*}

We again use the induction hypothesis for estimating the last term; however, $f_4$ and $f_5$ are already well localized in frequency, so we can invoke \eqref{result:Leibniz:Whitney} and \eqref{result:local:diag}:
\begin{align*}
 \big \| D^{\gamma}\circ d (\Delta_{\leq k_5 ,  \frac{L_2}{2^{k_1}}} f_4 \cdot \Delta_{k_5,  \frac{L_2}{2^{k_1}}} f_5) \big \|_{p_{4,5}} %=  \big \| D^{\gamma} (\Delta_{\leq k_5} f_4 \cdot \Delta_{k_5} f_5) \big \|_{p_{4,5}}
% &\lesssim 2^{k_5} \| \Delta_{\leq k_5 ,  \frac{L_2}{2^{k_1}}} f_4  \|_{p_4} \|  D^\gamma (\Delta_{k_5,  \frac{L_2}{2^{k_1}}} f_5) \|_{p_5} \\
% &\lesssim  \|\Delta_{\leq k_5 ,  \frac{L_2}{2^{k_1}}}f_4\|_{p_4} \|  \Delta_{k_5,  \frac{L_2}{2^{k_1}}}  D^\gamma f_5 \|_{p_5} 
 & \lesssim  2^{k_5}\|f_4\|_{p_4} \|  \Delta_{k_5}  D^\gamma f_5 \|_{p_5}
\end{align*}
for $1 \leq p_4, p_5 \leq \infty$, $\frac{1}{1+\gamma} < p_{4,5} \leq \infty$. The last inequality follows from \eqref{eq:trivial:modulation} and Young's convolution inequality.

With these considerations, and invoking again \eqref{eq:trivial:modulation} and \eqref{eq:direction:LP:same}, we can simply focus on bounding 
\begin{align*}
\|I_B  \|_{r}^{\p} &\lesssim \sum_{k_5< k_1} 2^{k_1(\beta-1) \p}  2^{k_5 \p}  2^{k_1  \alpha \p}  \|\Delta_{{k_1}} f_1\|_{p_1}^{\p}  \| f_2\|_{p_2}^{\p}  \| f_3\|_{p_3}^{\p} \| f_4  \|_{p_4}^{\p}    \| \Delta_{k_5} f_5\|_{p_5}^{\p} \\
& \lesssim  \| f_2\|_{p_2}^{\p}  \| f_3\|_{p_3}^{\p} \| f_4 \|_{p_4}^{\p} \big( \sum_{k_5< k_1} 2^{k_1(\beta-1) \p}  2^{k_5  \p}   \|\Delta_{{k_1}} D^\alpha f_1\|_{p_1}^{\p}  \| \Delta_{k_5} D^\gamma f_5\|_{p_5}^{\p} \big).
\end{align*}

As before, we make appear the Besov norms, and we optimize in the $2^{k_1}$ parameter: the expression in the last display involving $k_5$ and $k_1$ is bounded above by 
\begin{align*}
\sum_{k_1} \min & \big( 2^{k_1  \beta \p} \|D^\alpha  f_1\|_{\dot B_{p_1, \infty}^0}^{\p}   \|D^\gamma f_5\|_{\dot B_{p_5, \infty}^0}^{\p},  2^{- k_1  \epsilon \p} \|D^\alpha  f_1\|_{\dot B_{p_1, \infty}^{\beta}}^{\p}   \| D^\gamma f_5\|_{\dot B_{p_5, \infty}^{\epsilon}}^{\p} \big) \\
&\lesssim  \big( \|D^\alpha  f_1\|_{\dot B_{p_1, \infty}^0}^{\p}   \|D^\gamma f_5\|_{\dot B_{p_5, \infty}^0}^{\p}\big)^{\epsilon \over {\beta+\epsilon}}  \big( \|D^\alpha  f_1\|_{\dot B_{p_1, \infty}^{\beta}}^{\p}   \|D^\gamma f_5\|_{\dot B_{p_5, \infty}^{\epsilon}}^{\p}  \big)^{\beta \over {\beta+\epsilon}} \\
&\lesssim  \big( \|D^\alpha  f_1\|_{\dot B_{p_1, \infty}^0}^{\p}   \| D^\gamma f_5\|_{\dot B_{p_5, \infty}^\beta}^{\p} \big)^{\epsilon \over {\beta+\epsilon}}  \big( \|D^\alpha  f_1\|_{\dot B_{p_1, \infty}^{\beta}}^{\p}   \| D^\gamma f_5\|_{\dot B_{p_5, \infty}^{0}}^{\p}  \big)^{\beta \over {\beta+\epsilon}}.
\end{align*}

Notice that the $D^\beta$ derivatives are distributed between the $D^\alpha f_1$ and $D^\gamma f_5$ functions.
\vskip .2cm

\item[$I_C$)] In this case, only the Littlewood-Paley information for $f_1$ and $f_2$ will be needed, so we sum over $k_3, k_4, k_5 \ll k_1$ in \eqref{5lin:freq:rep}, focusing on  
\begin{align*}
T_{m_{I_C}}^{k_1, k_2}(f_1, \ldots, f_5)(x):=\int_{\BBR^5} & \frac{|\xi_1+\xi_2|^{\alpha+\beta}  - |\xi_1|^{\alpha+\beta}}{\xi_2} \cdot \xi_2 \cdot |\xi_4+\xi_5|^\gamma \quad \widehat{\Delta_{k_1} f_1}(\xi_1) \cdot \widehat{\Delta_{k_2} f_2}(\xi_2) \\
& \cdot \widehat{S_{k_1} f_3}(\xi_3) \cdot \widehat{S_{k_1} f_4}(\xi_4) \cdot \widehat{S_{ k_1} f_5}(\xi_5) e^{2 \pi i x \left(\xi_1+\ldots+\xi_5 \right)} d \, \xi_1 \ldots d \, \xi_5.
\end{align*}
This factorizes straightway into 
\begin{align*}
\Big( \int_{\BBR^2} & \frac{|\xi_1+\xi_2|^{\alpha+\beta}  - |\xi_1|^{\alpha+\beta}}{\xi_2} \cdot \xi_2 \cdot \widehat{\Delta_{k_1} f_1}(\xi_1) \cdot \widehat{\Delta_{k_2} f_2}(\xi_2) e^{2 \pi i x \left(\xi_1+\xi_2 \right)} d \, \xi_1  \xi_2 \Big) \cdot S_{k_1} f_3(x) \cdot D^\gamma( S_{k_1} f_4 \cdot S_{k_1} f_5)(x).
\end{align*}

The first term corresponds to $m_{C_{\alpha+\beta}}(\xi_1, \xi_2)$, a commutator symbol as in \eqref{def:commutator}, so we simply estimate it as in \eqref{eq:commutator:superposition}. For the remaining terms, we use the bilinear Leibniz rule result and the boundedness of the $S_{k_1}$ operator to deduce
\begin{align*}
\|I_C\|_{r}^{\p} \lesssim \Big(\sum_{k_2 < k_1} 2^{k_1(\alpha+\beta-1) \p} 2^{k_2 \p} \|\Delta_{{k_1}} f_1\|_{p_1}^{\p}  \|\Delta_{{k_2}} f_2\|_{p_2}^{\p}  \Big) \cdot \|f_3\|_{p_3}^{\p} \big( \|D^\gamma f_4\|_{p_4} \cdot \| f_5\|_{p_5} +  \| f_4\|_{p_4} \cdot \| D^\gamma f_5\|_{p_5} \big)^{\p}.
\end{align*}

The summation over $k_2<k_1$ is reduced as before to 
\begin{align*}
\sum_{k_1} \min \big( 2^{k_1 (\alpha+ \beta) \p} \| f_1\|_{\dot B_{p_1, \infty}^0}^{\p}   \| f_2\|_{\dot B_{p_2, \infty}^0}^{\p},  &  2^{- k_1  \epsilon \p} \| f_1\|_{\dot B_{p_1, \infty}^{\alpha+\beta}}^{\p}   \|  f_2\|_{\dot B_{p_2, \infty}^{\epsilon}}^{\p} \big) \lesssim \ldots  \\
&\lesssim  \big( \| f_1\|_{\dot B_{p_1, \infty}^0}^{\p}   \| f_2\|_{\dot B_{p_2, \infty}^{\alpha+\beta}}^{\p} \big)^{\epsilon \over {\alpha+\beta+\epsilon}}  \big( \| f_1\|_{\dot B_{p_1, \infty}^{\alpha+\beta}}^{\p}   \| f_2\|_{\dot B_{p_2, \infty}^{0}}^{\p}  \big)^{{\alpha+\beta} \over {\alpha+\beta+\epsilon}}.
\end{align*}

\item[$I_D$)] For this last term, we need to deal with $T_{m_{I_D}}^{k_1}(f_1, \ldots, f_5)$, defined by
\begin{align*}
&\quad \int_{\BBR^5}|\xi_1|^{\alpha+\beta} \cdot |\xi_4+\xi_5|^\gamma \quad \widehat{\Delta_{k_1} f_1}(\xi_1) \cdot \widehat{S_{k_1} f_2}(\xi_2) \cdot \widehat{S_{k_1} f_3}(\xi_3) \cdot \widehat{S_{k_1} f_4}(\xi_4) \cdot \widehat{S_{ k_1} f_5}(\xi_5) e^{2 \pi i x \left(\xi_1+\ldots+\xi_5 \right)} d \, \xi_1 \ldots d \, \xi_5 \\
&= (\Delta_{k_1} D^{\alpha+\beta} f_1)(x)  S_{k_1} f_2(x)  S_{k_1} f_3(x)  D^\gamma(S_{k_1} f_4 \cdot S_{k_1} f_5)(x).
\end{align*}

The idea is to write each $S_{k_1} f_l$, for $2 \leq l \leq 5$, as 
\[
S_{k_1} f_l(x) %= f_l(x)- \big( \sum_{k_l > k_1-3}  \Delta_{k_l} f_l (x)\big):
= f_l(x)- \Delta_{\succ k_1}f_l(x).
\]
Then $T_{m_{I_D}}^{k_1}(f_1, \ldots, f_5)$ becomes, once we sum in $k_1$,
\begin{align*}
&\sum_{k_1}(\Delta_{k_1} D^{\alpha+\beta} f_1)(x)  ( f_2(x) - \Delta_{\succ k_1} f_2)(x)  ( f_3(x) - \Delta_{\succ k_1} f_3)(x)  D^\gamma(( f_4 - \Delta_{\succ k_1} f_4) \cdot ( f_5 - \Delta_{\succ k_1} f_5)(x) \\
= &  (D^{\alpha+\beta} f_1)(x) \cdot f_2(x)\cdot f_3(x) \cdot D^\gamma(f_4 \cdot f_5)(x) - \sum_{k_1} (\Delta_{k_1} D^{\alpha+\beta} f_1)(x) (\Delta_{\succ k_1} f_2)(x) f_3(x)  D^\gamma(f_4 \cdot f_5)(x) \\
-&  \sum_{k_1} (\Delta_{k_1} D^{\alpha+\beta} f_1)(x)  f_2(x) f_3(x)  D^\gamma( (\Delta_{\succ k_1} f_4) \cdot f_5)(x)    + \text{similar terms}.
\end{align*}

The first term is bounded thanks to H\"older's inequality and the boundedness of the bilinear Leibniz rule for $D^\gamma(f_4 \cdot f_5)$. For the remaining terms, the $D^\beta$ derivatives will be shared between $D^\alpha \Delta_{k_1} f_1$ and $\Delta_{\succ k_1} f_l$, for some $l \in \{ 2, 3, 4, 5 \}$. If there is more than one function with associated projections $\Delta_{\succ k_1}$, we simply use one of them.

Two situations become apparent: 1) when the function associated to the maximal scale (in this case $f_1$ which corresponds to $k_1$) and the function on which $\Delta_{\succ k_1}$ acts are in the same subtree, and  2) when the functions are in different subtrees. 

The term $\sum_{k_1} (\Delta_{k_1} D^{\alpha+\beta} f_1) (\Delta_{\succ k_1} f_2) f_3  D^\gamma(f_4 \cdot f_5)$ corresponds to the first situation ($f_1$ and $f_2$ are contained in the same subtree associated to the Leibniz rule $D^\alpha(f_1 \cdot f_2)$), and $\sum_{k_1} (\Delta_{k_1} D^{\alpha+\beta} f_1)  f_2 f_3  D^\gamma( (\Delta_{\succ k_1} f_4) \cdot f_5)$ to the second one ($f_1$ and $f_4$ are leaves in different subtrees).

In order to estimate $\sum_{k_1} (\Delta_{k_1} D^{\alpha+\beta} f_1)(x) (\Delta_{\succ  k_1} f_2)(x) f_3(x)  D^\gamma(f_4 \cdot f_5)(x)$, we start with the observation that
\begin{equation}
\label{eq:switch} 
\| \Delta_{\succ  k_1} f_2\|_{p_2} \leq \min( \| f_2\|_{p_2}, \, 2^{- \epsilon k_1}\| f_2\|_{\dot B_{p_2, \infty}^\epsilon}).
\end{equation}
Then we use H\"older's inequality and the bilinear Leibniz rule \eqref{Leib_one_paraproduct:strategy} for the $ D^\gamma(f_4 \cdot f_5)$ part, to again, reduce ourselves to estimating
\begin{align*}
&\sum_{k_1} 2^{k_1 \beta} \|\Delta_{k_1} f_1\|_{p_1}^{\p}  \|\Delta_{\succ  k_1} f_2 \|_{p_2}^{\p} \lesssim  \sum_{k_1} \min  \big( 2^{k_1  \beta \p} \|D^\alpha  f_1\|_{\dot B_{p_1, \infty}^0}^{\p}   \| f_2\|_{p_2}^{\p},  2^{- k_1  \epsilon \p} \|D^\alpha  f_1\|_{\dot B_{p_1, \infty}^{\beta}}^{\p}   \| f_2\|_{\dot B_{p_2, \infty}^{\epsilon}}^{\p} \big) \\
&\lesssim  \big( \|D^\alpha  f_1\|_{\dot B_{p_1, \infty}^0}^{\p}   \| f_2\|_{p_2}^{\p}\big)^{\epsilon \over {\beta+\epsilon}}  \big( \|D^\alpha  f_1\|_{\dot B_{p_1, \infty}^{\beta}}^{\p}   \|f_2\|_{\dot B_{p_2, \infty}^{\epsilon}}^{\p}  \big)^{\beta \over {\beta+\epsilon}} \lesssim  \big( \|D^\alpha  f_1\|_{\dot B_{p_1, \infty}^0}^{\p}   \| f_2\|_{\dot B_{p_2, \infty}^\beta}^{\p} \big)^{\epsilon \over {\beta+\epsilon}}  \big( \|D^\alpha  f_1\|_{\dot B_{p_1, \infty}^{\beta}}^{\p}   \| f_2\|_{p_2}^{\p}  \big)^{\beta \over {\beta+\epsilon}}.
\end{align*}
Of course, the $D^{\alpha+\beta}$ derivatives could be shared between $f_1$ and $f_2$.

In the second situation, we start by applying H\"older to obtain 
\begin{align}
\label{eq:I_D:change}
\sum_{k_1} 2^{k_1 \beta \tau} \|D^\alpha \Delta_{k_1} f_1\|_{p_1}^\tau \|f_2\|_{p_2}^\tau \|f_3\|_{p_3}^\tau \|D^\gamma( (\Delta_{\succ  k_1} f_4) \cdot f_5)\|_{p_{4,5}}^\tau.
\end{align}
Instead of directly applying the bilinear Leibniz rule \eqref{Leib_one_paraproduct:strategy} to $D^\gamma( (\Delta_{\succ  k_1} f_4) \cdot f_5)$ , we perform an extra paraproduct decomposition. This is because  we want the terms appearing in the Leibniz rule to coincide with those obtained by regular composition : we should not have terms such as $\|f_1\|_{p_1} \|D^\beta f_4\|_{p_4}\|D^\gamma f_5\|_{p_5}$ appearing. Hence we write 
\begin{align*}
D^\gamma( (\Delta_{\succ  k_1} f_4) \cdot f_5)= \sum_{k_4} D^\gamma( (\Delta_{k_4}\Delta_{\succ  k_1} f_4) \cdot (\Delta_{\leq k_4} f_5))+  \sum_{k_5} D^\gamma( (S_{k_5}\Delta_{\succ  k_1} f_4) \cdot (\Delta_{ k_5} f_5)).
\end{align*}
Now notice that $\Delta_{k_4}\Delta_{\succ  k_1} \neq 0$ only if $k_4 \succ k_1$, and similarly $S_{k_5}\Delta_{\succ  k_1} \neq 0$ only if $k_5 \succ k_1$. So in fact 
\[
\sum_{k_4} D^\gamma( (\Delta_{k_4}\Delta_{\succ  k_1} f_4) \cdot (\Delta_{\leq k_4} f_5))= \sum_{k_4 \succ k_1} D^\gamma( (\Delta_{k_4}\Delta_{\succ  k_1} f_4) \cdot (\Delta_{\leq k_4} f_5))
\]
and 
\[
 \sum_{k_5} D^\gamma( (S_{k_5}\Delta_{\succ  k_1} f_4) \cdot (\Delta_{ k_5} f_5))=  \sum_{k_5 \succ k_1} D^\gamma( (S_{k_5}\Delta_{\succ  k_1} f_4) \cdot (\Delta_{ k_5} f_5)).
\]
Overall we get, thanks to \eqref{result:Leibniz:Whitney} and \eqref{result:local:diag}, 
\begin{align*}
\|D^\gamma( (\Delta_{\succ  k_1} f_4) \cdot f_5)\|_{p_{4,5}}^\tau \lesssim  & \sum_{k_4 \succ k_1} \|\Delta_{k_4}\Delta_{\succ  k_1} D^\gamma f_4\|_{p_4}^\tau \|\Delta_{\leq k_4} f_5\|_{p_5}^\tau +  \sum_{k_5 \succ k_1} \|S_{k_5}\Delta_{\succ  k_1} f_4 \|_{p_4}^\tau \|\Delta_{ k_5} D^\gamma f_5\|_{p_5}^\tau \\
\lesssim &  \sum_{k_4 \succ k_1} \|\Delta_{k_4} D^\gamma f_4\|_{p_4}^\tau \| f_5\|_{p_5}^\tau +  \sum_{k_5 \succ k_1} \|\Delta_{\succ  k_1} f_4 \|_{p_4}^\tau \|\Delta_{ k_5} D^\gamma f_5\|_{p_5}^\tau. 
\end{align*}
So that \eqref{eq:I_D:change} is reduced to 
\begin{align*}
&\|f_2\|_{p_2}^\tau \|f_3\|_{p_3}^\tau \| f_5\|_{p_5}^\tau  \big( \sum_{k_4 \succ k_1} 2^{k_1 \beta \tau} \|D^\alpha \Delta_{k_1} f_1\|_{p_1}^\tau  \|\Delta_{k_4} D^\gamma f_4\|_{p_4}^\tau \big) + \|f_2\|_{p_2}^\tau \|f_3\|_{p_3}^\tau \| f_4\|_{p_4}^\tau  \big( \sum_{k_5 \succ k_1} 2^{k_1 \beta \tau} \|D^\alpha \Delta_{k_1} f_1\|_{p_1}^\tau  \|\Delta_{k_5} D^\gamma f_5\|_{p_5}^\tau \big). 
\end{align*}

Due to symmetry, we only look at the term $\sum_{k_4 \succ k_1} 2^{k_1 \beta \tau} \|D^\alpha \Delta_{k_1} f_1\|_{p_1}^\tau  \|\Delta_{k_4} D^\gamma f_4\|_{p_4}^\tau$, which can be estimated by
\begin{align*}
&\sum_{k_4} \min \big(  2^{k_4 \beta \tau} \|D^\alpha f_1\|_{\dot{B}^0_{p_1, \infty}}^{\tau} \| D^\gamma f_4\|_{\dot{B}^0_{p_4, \infty}}^{\tau}, 2^{-k_4 \epsilon \tau }  \|D^\alpha f_1\|_{\dot{B}^\epsilon_{p_1, \infty}}^{\tau} \| D^\gamma f_4\|_{\dot{B}^\beta_{p_4, \infty}}^{\tau}   \big) \\
&\lesssim \big(    \|D^\alpha f_1\|_{\dot{B}^0_{p_1, \infty}}^{\tau} \| D^\gamma f_4\|_{\dot{B}^0_{p_4, \infty}}^{\tau}  \big)^{\epsilon \over {\beta+\epsilon}}  \big(   \|D^\alpha f_1\|_{\dot{B}^\epsilon_{p_1, \infty}}^{\tau} \| D^\gamma f_4\|_{\dot{B}^\beta_{p_4, \infty}}^{\tau}  \big)^{\beta \over {\beta+\epsilon}} \\
&\lesssim \big(    \|D^\alpha f_1\|_{\dot{B}^\beta_{p_1, \infty}}^{\tau} \| D^\gamma f_4\|_{\dot{B}^0_{p_4, \infty}}^{\tau}  \big)^{\epsilon \over {\beta+\epsilon}}  \big(   \|D^\alpha f_1\|_{\dot{B}^0_{p_1, \infty}}^{\tau} \| D^\gamma f_4\|_{\dot{B}^\beta_{p_4, \infty}}^{\tau}  \big)^{\beta \over {\beta+\epsilon}}. 
\end{align*}
Of course, this is bounded by $\big(\|D^{\alpha+\beta} f_1\|_{p_1}  \| D^\gamma f_4\|_{p_4}+\|D^{\alpha} f_1\|_{p_1}  \| D^{\beta+\gamma} f_4\|_{p_4} \big)^\tau$; we point out that the distribution of derivatives follows the same law as the composition of Leibniz rules, except that now input functions in any $L^p$ spaces, with $1 \leq p \leq \infty$, are admissible.
\end{itemize}

This exhausts the possible cases corresponding to the restriction to the frequency conical region $R_1$ \eqref{Sec3:off-diagonal:1}.

\bigskip

\subsection{\ref{caseII:k_3}: study of the conical region $R_3$}~\\
In this case, we restrict our attention to the frequency region $R_3$ \eqref{Sec3:off-diagonal:3}. %, where $|\xi_1|, |\xi_2|, |\xi_4|, \xi_5| \ll |\xi_3|$. 
Here we split the multiplier associated to the flag in %\eqref{eq:MuscaluFlag}
\eqref{5lin:freq:rep} into 
 \begin{align*}
&\frac{ |\xi_1+\ldots +\xi_5|^\beta- |\xi_3|^\beta}{\xi_1+\xi_2 +\xi_4 +\xi_5} \cdot  (\xi_1+\xi_2) \cdot |\xi_1+\xi_2|^\alpha  \cdot |\xi_4+\xi_5|^\gamma + \frac{ |\xi_1+\ldots +\xi_5|^\beta- |\xi_3|^\beta}{\xi_1+\xi_2 +\xi_4 +\xi_5} \cdot |\xi_1+\xi_2|^\alpha    \cdot   (\xi_4+\xi_5) \cdot |\xi_4+\xi_5|^\gamma  \\
& + |\xi_3|^\beta \cdot |\xi_1+\xi_2|^\alpha \cdot |\xi_4+\xi_5|^\gamma := m_{II_A}+ m_{II_B}+ m_{II_C}.
  \end{align*}
As before, these are combined with Littlewood-Paley projections to obtain that \eqref{5lin:freq:rep} restricted to $R_3$ \eqref{Sec3:off-diagonal:3} equals
\begin{align*}
&\sum_{k_1, \ldots, k_5 \ll k_3 } T_{m_{II_A}}(\Delta_{k_1}f_1, \ldots, \Delta_{k_5}f_5)(x)+ \sum_{k_1, \ldots, k_5 \ll k_3 } T_{m_{II_B}}(\Delta_{k_1} f_1, \ldots,\Delta_{k_5} f_5)(x) \\
 + &\sum_{k_1, \ldots, k_5 \ll k_3 } T_{m_{II_C}}(\Delta_{k_1} f_1, \ldots, \Delta_{k_5} f_5)(x):= II_A+II_B+II_C.
\end{align*}

The multipliers $m_{II_A}$ and $m_{II_B}$ are symmetric, so it will suffice to study $II_A$ and $II_C$.

\begin{itemize}[leftmargin=*]
\item[$II_A$)] While $f_3$ is the fastest oscillating function, the shape of $m_{II_A}$ indicates that one of $f_1$ or $f_2$ will also be involved in the scale-by-scale analysis. To decide which, an additional paraproduct decomposition concerning the $\xi_1$ and $\xi_2$ variables is necessary -- for simplicity we assume $|\xi_2| \leq |\xi_1|$.

After summing over $k_4, k_5 \ll k_3$ and $k_2 \leq k_1$ in \eqref{5lin:freq:rep}, we need to analyze 
\begin{align*}
T_{m_{II_A}}^{k_3, k_1}(f_1, \ldots, f_5)(x):=\int_{\BBR^5} & \frac{|\xi_1+ \ldots+ \xi_5|^\beta  - |\xi_3|^\beta}{\xi_1+\xi_2+\xi_4+\xi_5}  \cdot  (\xi_1+\xi_2) \cdot |\xi_1+\xi_2|^\alpha  \cdot |\xi_4+\xi_5|^\gamma \\  &  \widehat{\Delta_{k_1} f_1}(\xi_1) \cdot \widehat{\Delta_{\leq k_1} f_2}(\xi_2) \cdot \widehat{\Delta_{k_3} f_3}(\xi_3) \cdot \widehat{S_{k_3} f_4}(\xi_4) \cdot \widehat{S_{ k_3} f_5}(\xi_5) e^{2 \pi i x \left(\xi_1+\ldots+\xi_5 \right)} d \, \xi_1 \ldots d \, \xi_5.
\end{align*}

We decompose into Fourier series the symbol 
\[
m_{C_\beta}^{k_3, \pm}(\xi_3, \xi_1+\xi_2+\xi_4+\xi_5)= \frac{|\xi_1+ \ldots+ \xi_5|^\beta  - |\xi_3|^\beta}{\xi_1+\xi_2+\xi_4+\xi_5} \tilde \psi_{k_3, \pm}(\xi_3) \tilde \varphi_{k_3}(\xi_1+\xi_2+\xi_4+\xi_5),
\]  
obtaining arbitrary decay for the Fourier coefficients. This allows to express $T_{m_{II_A}}^{k_3, k_1}(f_1, \ldots, f_5)(x)$ as a sum of 
\begin{align*}
\sum_{L_1, L_2 \in \BBZ} C_{L_1, L_2}^{\pm} 2^{k_3(\beta-1)} D^{\alpha} \circ d \big( \Delta_{k_1, \frac{L_2}{2^{k_3}}} f_1 \cdot \Delta_{\leq k_1, \frac{L_2}{2^{k_3}}} f_2 \big)(x) \cdot \Delta_{k_3,\pm, \frac{L_1}{2^{k_3}}} f_3(x) \cdot D^{\gamma} \big( S_{k_3, \frac{L_2}{2^{k_3}}} f_4 \cdot S_{k_3, \frac{L_2}{2^{k_3}}} f_5 \big)(x).
\end{align*}
We will appeal shortly to the boundedness of the bilinear Leibniz rule \eqref{Leib_one_paraproduct:strategy}, and its frequency-localized versions \eqref{result:Leibniz:Whitney} and \eqref{result:local:diag},
\[
\|D^{\alpha}\circ d \big( \Delta_{k_1, \frac{L_2}{2^{k_3}}} f_1 \cdot \Delta_{\leq k_1, \frac{L_2}{2^{k_3}}} f_2 \big)\|_{p_{1,2}} %\lesssim 2^{k_1} \|\Delta_{k_1, \frac{L_2}{2^{k_3}}} D^\alpha f_1 \|_{p_1}  \|S_{k_1, \frac{L_2}{2^{k_3}}}f_2  \|_{p_2} 
\lesssim 2^{k_1}\|\Delta_{k_1} D^\alpha f_1 \|_{p_1}  \|f_2  \|_{p_2},
\]
which is true as long as $\frac{1}{p_{1,2}}=\frac{1}{p_1}+\frac{1}{p_2}$, $1\leq p_1, p_2 \leq \infty$, and $\frac{1}{p_{1,2}}<1+\alpha$. 

Thus we have, for $0<\tau \leq \min(1, r)$, 
\begin{align*}
\| II_A  \|_r^{\tau} & \lesssim \sum_{L_1, L_2 \in \BBZ} |C_{L_1, L_2}^{\pm}|^\tau \sum_{k_1< k_3} 2^{k_3(\beta-1) \p}  2^{k_1 \p}  \| \Delta_{k_1} D^\alpha f_1 \|_{p_1}^{\p}  \|f_2  \|_{p_2}^{\p}  \|\Delta_{k_3, \pm} f_3\|_{p_3}^{\p} \|D^{\gamma} \big( S_{k_3, \frac{L_2}{2^{k_3}}} f_4 \cdot S_{k_3, \frac{L_2}{2^{k_3}}} f_5 \big)\|_{p_{4,5}}^{\p} \\
&\lesssim  \sum_{L_1, L_2 \in \BBZ} |C^{\pm}_{L_1, L_2}|^\tau  \sum_{k_1< k_3} 2^{k_3(\beta-1) \p}  2^{k_1 \p}  \| \Delta_{k_1} D^\alpha f_1 \|_{p_1}^{\p}  \| f_2  \|_{p_2}^{\p}  \|\Delta_{k_3} f_3\|_{p_3}^{\p} \\
& \qquad \cdot  \big(  \|S_{k_3, \frac{L_2}{2^{k_3}}} D^{\gamma} f_4 \|_{p_4} \| S_{k_3, \frac{L_2}{2^{k_3}}} f_5\|_{p_5} +  \|S_{k_3, \frac{L_2}{2^{k_3}}} f_4 \|_{p_4} \| S_{k_3, \frac{L_2}{2^{k_3}}} D^{\gamma} f_5\|_{p_5} \big)^\tau.
\end{align*}

Due to the fast decay of the coefficients $|C^{\pm}_{L_1, L_2}|$, and the boundedness of the $\Delta_{ \leq k_1, a}$ and $S_{k_3, a}$ operators (with norms independent of $k_1$, $k_3$ or $a$), we are left with summing
\begin{align*}
\sum_{k_1< k_3} 2^{k_3(\beta-1) \p}  2^{k_1 \p}  \|\Delta_{{k_1}} D^\alpha f_1\|_{p_1}^{\p}  \|\Delta_{{k_3}} f_3\|_{p_3}^{\tau} &  \lesssim \sum_{k_3} \min \big( 2^{k_3  \beta \p} \|D^\alpha  f_1\|_{\dot B_{p_1, \infty}^0}^{\p}   \| f_3\|_{\dot B_{p_3, \infty}^0}^{\p},  2^{- k_3  \epsilon \p} \|D^\alpha  f_1\|_{\dot B_{p_1, \infty}^{\epsilon}}^{\p}   \| f_3\|_{\dot B_{p_3, \infty}^{\beta}}^{\p} \big) \\
&\lesssim  \big( \|D^\alpha  f_1\|_{\dot B_{p_1, \infty}^\beta}^{\p}   \| f_3\|_{\dot B_{p_3, \infty}^0}^{\p} \big)^{\epsilon \over {\beta+\epsilon}}  \big( \|D^\alpha  f_1\|_{\dot B_{p_1, \infty}^{0}}^{\p}   \| f_3\|_{\dot B_{p_3, \infty}^{\beta}}^{\p}  \big)^{\beta \over {\beta+\epsilon}}.
\end{align*}

\medskip

\item[$II_C)$] Now we look at $II_C$, which for a fixed $k_3 \in \BBZ$ and after summing in $k_1, k_2, k_4, k_5 \ll k_3$ in \eqref{5lin:freq:rep}, corresponds to the operator 
\begin{equation}
\label{eq:T:m_3}
D^\alpha(S_{k_3} f_1 \cdot S_{k_3} f_2) \cdot \Delta_{k_3}D^\beta f_3 \cdot D^\gamma(S_{k_3} f_4 \cdot S_{k_3} f_5).
\end{equation}
The derivatives $D^\beta$ hit the fastest oscillating function, i.e. $f_3$. We still need to sum over $k_3$, so although the operator in \eqref{eq:T:m_3} is tensorized, it will not be trivial to estimate. Here we switch the order of summation again as in $I_D$, and the operator in \eqref{eq:T:m_3} becomes
\begin{equation}
\label{eq:T:m_3:switch}
D^\alpha(  (f_1- \Delta_{\succ k_3} f_1) \cdot (f_2- \Delta_{\succ k_3} f_2)) \cdot \Delta_{k_3}D^\beta f_3 \cdot D^\gamma((f_4- \Delta_{\succ k_3} f_4) \cdot (f_5- \Delta_{\succ k_3} f_5)).
\end{equation}

We use the linearity of the derivation operators to sum over $k_3$ the expressions in \eqref{eq:T:m_3:switch}; we have
\begin{align*}
II_C=&D^\alpha( f_1 \cdot  f_2) \cdot D^\beta f_3 \cdot D^\gamma( f_4 \cdot  f_5) - \sum_{k_3} D^\alpha( \Delta_{\succ k_3} f_1 \cdot  f_2) \cdot \Delta_{k_3}D^\beta f_3 \cdot D^\gamma( f_4 \cdot  f_5) \\
&- \sum_{k_3} D^\alpha( f_1 \cdot  \Delta_{\succ k_3}  f_2) \cdot \Delta_{k_3}D^\beta f_3 \cdot D^\gamma( f_4 \cdot  f_5)+ \text{ similar term}
\end{align*}

The first term can be easily bounded in $L^r$ thanks to H\"older's inequality and the paraproduct Leibniz rule \eqref{Leib_one_paraproduct:strategy}. For the other terms, it will be sufficient to use one of the functions on which $\Delta_{\succ k_3}$ acts; and if it acts on several, we pick one of them, which will contribute to the summation in $k_3$. For flag Leibniz rules, a special attention is required by the distribution of the derivatives; for that reason, when we are compelled to use a function belonging to a different subtree than $f_3$, an intermediate step is necessary to make sure that the $D^\beta$ derivatives will be distributed according to the composition laws. It will be enough to treat the second term, since the remaining ones are very similar. It will be bounded, in $\|\cdot \|_r^{\p}$, by
\begin{align*}
&\sum_{k_3} \|D^\alpha( \Delta_{\succ k_3} f_1 \cdot  f_2)\|_{p_{1, 2}}^{\p} \|\Delta_{k_3} D^\beta f_3\|_{p_3}^{\p}  \|D^\gamma( \Delta_{\succ k_3} f_4 \cdot  f_5)\|_{p_{4,5}}^{\p} \\
&\lesssim \sum_{k_3}  \|D^\alpha( \Delta_{\succ k_3} f_1 \cdot  f_2)\|_{p_{1, 2}}^{\p}  \|\Delta_{k_3} D^\beta f_3\|_{p_3}^{\p} \big(  \|D^\gamma f_4 \|_{p_4} \cdot \| f_5\|_{p_5} + \| f_4\|_{p_4} \cdot  \|D^\gamma  f_5 \|_{p_5}  \big)^{\p}.
\end{align*}

We take a closer look at the factor $D^\alpha( \Delta_{\succ k_3} f_1 \cdot  f_2)$, which is equal to
\begin{align*}
\sum_{k_1} D^\alpha( \Delta_{k_1} \Delta_{\succ k_3} f_1 \cdot \Delta_{\leq k_1}  f_2)(x)+  \sum_{k_2} D^\alpha( S_{k_2} \Delta_{\succ k_3} f_1 \cdot \Delta_{ k_2}  f_2)(x).
\end{align*}
We notice that the only way $ \Delta_{k_1} \Delta_{\succ k_3} \neq 0$ is if $k_1\succ k_3$, and similarly, $S_{k_2} \Delta_{\succ k_3} \neq 0$ only if $k_2 \succ k_3$. Thus the expression above becomes
\begin{align*}
\sum_{k_1 \succ k_3} D^\alpha( \Delta_{k_1} \Delta_{\succ k_3} f_1 \cdot \Delta_{\leq k_1}  f_2)(x)+  \sum_{k_2 \succ k_3} D^\alpha( S_{k_2} \Delta_{\succ k_3} f_1 \cdot \Delta_{ k_2}  f_2)(x).
\end{align*}
Because of \eqref{result:local:diag} and \eqref{result:Leibniz:Whitney}, which indicate that derivatives tend to move towards higher oscillating functions,
\begin{align*}
\|D^\alpha( \Delta_{\succ k_3} f_1 \cdot  f_2)\|_{p_{1, 2}}^{\p} & \lesssim \sum_{k_1 \succ k_3} \|D^\alpha( \Delta_{k_1} \Delta_{\succ k_3} f_1 \cdot \Delta_{\leq k_1}  f_2)\|_{p_{1,2}}^{\p} + \sum_{k_2 \succ k_3} \|D^\alpha( S_{k_2} \Delta_{\succ k_3} f_1 \cdot \Delta_{ k_2}  f_2)\|_{p_{1,2}}^{\p} \\
%&  \lesssim \sum_{k_1 \geq k_3}  \|D^\alpha( \Delta_{k_1} \Delta_{\succ k_3} f_1)\|_{p_1}^\tau \|f_2\|_{p_2}^\tau + \sum_{k_2 \geq k_3}  \|S_{k_2} \Delta_{\succ k_3} f_1\|_{p_1}^\tau \|\Delta_{ k_2} D^\alpha f_2\|_{p_2}^\tau \\
&\lesssim \sum_{k_1 \succ k_3}  \|\Delta_{k_1} D^\alpha f_1\|_{p_1}^\tau \| f_2\|_{p_2}^\tau + \sum_{k_2 \succ k_3}  \| f_1\|_{p_1}^\tau \|\Delta_{ k_2} D^\alpha  f_2\|_{p_2}^\tau.
\end{align*}
Hence $\|II_C  \|_{r}^{\p}$ is bounded by the sum of several similar terms of the form 
\[
\sum_{k_1 \succ k_3}  \|\Delta_{k_1} D^\alpha f_1\|_{p_1}^\tau \| f_2\|_{p_2}^\tau  2^{k_3 \beta \tau} \|\Delta_{k_3} f_3\|_{p_3}^{\p}  \|D^\gamma f_4 \|_{p_4}^{\p}  \| f_5\|_{p_5}^\tau. 
\]

We can put the functions $f_2, f_4$ and $f_5$ aside (in the one-parameter case, there will always be two functions involved in this type of summation), so we are left with
\begin{align*}
\sum_{k_1} & \min \big(  2^{k_1 \beta \tau} \|D^\alpha f_1\|_{\dot{B}^0_{p_1, \infty}}^{\tau} \| f_3\|_{\dot{B}^0_{p_3, \infty}}^{\tau}, 2^{-k_1 \epsilon \tau }  \|D^\alpha f_1\|_{\dot{B}^\beta_{p_1, \infty}}^{\tau} \| f_3\|_{\dot{B}^\epsilon_{p_3, \infty}}^{\tau}   \big) \\
&\lesssim \big(    \|D^\alpha f_1\|_{\dot{B}^0_{p_1, \infty}}^{\tau} \|  f_3\|_{\dot{B}^\beta_{p_3, \infty}}^{\tau}  \big)^{\epsilon \over {\beta+\epsilon}}  \big(   \|D^\alpha f_1\|_{\dot{B}^\beta_{p_1, \infty}}^{\tau} \| f_3\|_{\dot{B}^0_{p_3, \infty}}^{\tau}  \big)^{\beta \over {\beta+\epsilon}}. 
\end{align*} 
\end{itemize}
\vskip .4cm

\subsection{\ref{caseIII:multi}: study of ``diagonal'' conical regions}~\\ 
Here we restrict the operator in \eqref{5lin:freq:rep} to the frequency region \eqref{Sec3:diagonal}. 
%where $|\xi_{l_1}|\sim |\xi_{l_2}|$ for some $l_1 \neq l_2 \in \{ 1, \ldots, 5  \}$ and $|\xi_l|\leq |\xi_{l_1}|$ for all $l \in  \{ 1, \ldots, 5  \}$. 
Due to the structure of the flag and its symmetries, it will be enough to investigate the cases $l_1=1, l_2=3$ (representing $III_A$) and $l_1=1, l_2=5$ (case $III_B$).

In this situation, we will not make use of commutators; instead, the Littlewood-Paley information of the functions $f_{l_1}$ and $f_{l_2}$ will be sufficient for estimating the summation of the various scales in \eqref{5lin:freq:rep}.

\begin{itemize}[leftmargin=*]
\item[$III_A)$] In this case, the main contribution will come from the functions $f_1$ and $f_3$. We can sum in \eqref{5lin:freq:rep} over $k_2, k_4, k_5 \leq k_1  \sim k_3$ to get
\begin{align*}
T_{m_{III_A}}^{k_1}(f_1, \ldots, f_5)(x):=\int_{\BBR^5} & |\xi_1+ \ldots + \xi_5|^\beta \cdot |\xi_1+\xi_2|^\alpha \cdot |\xi_4+\xi_5|^\gamma \cdot \widehat{\Delta_{k_1} f_1}(\xi_1) \cdot \widehat{\Delta_{\leq k_1} f_2}(\xi_2)  \\  & \quad \cdot \widehat{\Delta_{k_1} f_3}(\xi_3) \cdot \widehat{\Delta_{\leq k_1} f_4}(\xi_4) \cdot \widehat{\Delta_{\leq k_1} f_5}(\xi_5) e^{2 \pi i x \left(\xi_1+\ldots+\xi_5 \right)} d \, \xi_1 \ldots d \, \xi_5.
\end{align*}

Given the assumptions on the scales, we have that $ |\xi_1+ \ldots + \xi_5|\leq C 2^{k_1}$ and $ |\xi_1+ \xi_2| \leq  2^{k_1+1}$. Hence\footnote{Here we might need to assume a certain amount of separation between the frequency pieces, which is easy to obtain by a sparsification argument that only introduces $O(1)$ new terms.} we use Fourier series on $[- 2^{k_1}, 2^{k_1}]$ to tensorize and decompose the $|\xi_1+ \ldots + \xi_5|^\beta \tilde \varphi_{k_1}(\xi_1+ \ldots + \xi_5)$ symbol; the Fourier coefficients will only have limited decay, but that is still okay. $T_{m_{III_A}}^{k_1}(f_1, \ldots, f_5)(x)$ becomes
\begin{align*}
\sum_{L \in \BBZ} C_L 2^{k_1 \beta} D^\alpha(\Delta_{k_1, \frac{L}{2^{k_1}}} f_1 \cdot \Delta_{\leq k_1, \frac{L}{2^{k_1}}} f_2)(x) \cdot  \Delta_{k_1, \frac{L}{2^{k_1}}} f_3 (x) \cdot D^\gamma(\Delta_{\leq k_1, \frac{L}{2^{k_1}}} f_4 \cdot \Delta_{\leq k_1, \frac{L}{2^{k_1}}} f_5)(x),
\end{align*}
where $|C_L| \lesssim (1+|L|)^{-(1+\beta)}$. So for $\tau \leq \min(1, r)$ with $\frac{1}{1+\beta} < \tau$, we have 
\begin{align*}
\| III_A  \|_r^\tau \lesssim \sum_{k_1} \sum_{L \in \BBZ} |C_L|^\tau 2^{k_1 \beta \tau } \|D^\alpha(\Delta_{k_1, \frac{L}{2^{k_1}}} f_1 \cdot \Delta_{\leq k_1, \frac{L}{2^{k_1}}} f_2)\|_{p_{1,2}}^\tau \| \Delta_{k_1, \frac{L}{2^{k_1}}} f_3 \|_{p_3}^\tau \|D^\gamma(\Delta_{\leq k_1, \frac{L}{2^{k_1}}} f_4 \cdot \Delta_{\leq k_1, \frac{L}{2^{k_1}}} f_5)\|_{p_{4,5}}^\tau.
\end{align*}

For $p_{1, 2}, p_{4,5}$ so that $\frac{1}{1+\alpha}< p_{1,2}=\frac{p_1 p_2}{p_1+p_2} \leq \infty$, $\frac{1}{1+\gamma}< p_{4,5}=\frac{p_4 p_5}{p_4+p_5} \leq \infty$, we further deduce
\begin{align*}
\| III_A  \|_r^\tau &\lesssim  \sum_{k_1} \sum_{L \in \BBZ} C_L^\tau 2^{k_1 \beta \tau } \|\Delta_{k_1, \frac{L}{2^{k_1}}} D^\alpha f_1 \|_{p_1}^\tau \|f_2\|_{p_{2}}^\tau \| \Delta_{k_1, \frac{L}{2^{k_1}}} f_3 \|_{p_3}^\tau \\
& \qquad \qquad\cdot  \big(  \|D^\gamma(\Delta_{ \leq k_1, \frac{L}{2^{k_1}}} f_4) \|_{p_4} \|\Delta_{\leq k_1, \frac{L}{2^{k_1}}} f_5\|_{p_{5}} +   \|\Delta_{\leq k_1, \frac{L}{2^{k_1}}} f_4\|_{p_4}  \| D^\gamma (\Delta_{\leq k_1, \frac{L}{2^{k_1}}} f_5)\|_{p_5} \big)^\tau \\
&\lesssim  \sum_{k_1}  \|\Delta_{k_1} D^\alpha f_1 \|_{p_1}^\tau \| f_2\|_{p_{2}}^\tau \| \Delta_{k_1} f_3 \|_{p_3}^\tau  \big(  \|D^\gamma f_4 \|_{p_4} \|f_5\|_{p_{5}} +   \| f_4\|_{p_4}  \| D^\gamma  f_5\|_{p_5} \big)^\tau.
\end{align*}

Summing in $k_1$ is by now a formality:
\begin{align*}
 \sum_{k_1}  2^{k_1 \beta \p}  \|\Delta_{k_1} D^{\alpha} f_1\|_{p_1}^{\p}    \|\Delta_{ k_1}  f_3\|_{p_3}^{\p} &\lesssim \sum_{k_1} \min \big( 2^{k_1  \beta \p} \|D^\alpha  f_1\|_{\dot B_{p_1, \infty}^0}^{\p}   \| f_3\|_{\dot B_{p_3, \infty}^0}^{\p},  2^{- k_1  \beta \p} \|D^\alpha  f_1\|_{\dot B_{p_1, \infty}^{\beta}}^{\p}   \| f_3\|_{\dot B_{p_3, \infty}^{\beta}}^{\p} \big) \\
&\lesssim  \big( \|D^\alpha  f_1\|_{\dot B_{p_1, \infty}^0}^{\p}   \| f_3\|_{\dot B_{p_3, \infty}^\beta}^{\p} \big)^{1 \over 2}  \big( \|D^\alpha  f_1\|_{\dot B_{p_1, \infty}^{\beta}}^{\p}   \| f_3\|_{\dot B_{p_3, \infty}^{0}}^{\p}  \big)^{1\over 2}.
\end{align*}
\medskip
\item[$III_B)$] Now we are in the situation when $|\xi_2|, |\xi_3|, |\xi_4| \leq |\xi_1| \sim |\xi_5|$. Accordingly, we sum over $k_2, k_3, k_4 \leq k_1 \sim k_5$ in \eqref{5lin:freq:rep} to obtain
\begin{align*}
T_{m_{III_B}}^{k_1}(f_1, \ldots, f_5)(x):=\int_{\BBR^5} & |\xi_1+ \ldots + \xi_5|^\beta \cdot |\xi_1+\xi_2|^\alpha \cdot |\xi_4+\xi_5|^\gamma \widehat{\Delta_{k_1} f_1}(\xi_1) \cdot \widehat{\Delta_{\leq k_1} f_2}(\xi_2)  \\  & \quad \cdot \widehat{\Delta_{\leq k_1} f_3}(\xi_3) \cdot \widehat{\Delta_{\leq k_1} f_4}(\xi_4) \cdot \widehat{\Delta_{ k_1} f_5}(\xi_5) e^{2 \pi i x \left(\xi_1+\ldots+\xi_5 \right)} d \, \xi_1 \ldots d \, \xi_5.
\end{align*}

As before, we smoothly localize and use a Fourier series expansion for $|\xi_1+ \ldots + \xi_5|^\beta$ on $[-2^{k_1}, 2^{k_1}]$, with limited decay of the Fourier coefficients. This allows to rewrite
\begin{align*}
T_{m_{III_B}}^{k_1}(f_1, \ldots, f_5)(x)= \sum_{L \in \BBZ} C_L 2^{k_1 \beta} D^\alpha(\Delta_{k_1, \frac{L}{2^{k_1}}} f_1 \cdot \Delta_{\leq k_1, \frac{L}{2^{k_1}}} f_2)(x) \cdot \Delta_{\leq k_1, \frac{L}{2^{k_1}}} f_3 (x) \cdot D^\gamma(\Delta_{\leq k_1, \frac{L}{2^{k_1}}} f_4 \cdot \Delta_{k_1, \frac{L}{2^{k_1}}} f_5)(x),
\end{align*}
with $|C_L| \lesssim (1+|L|)^{-(1+\beta)}$. As before, 
\begin{align*}
\| III_B  \|_{r}^{\p} &\lesssim  \| f_2\|_{p_2}^{\p}   \|  f_3\|_{p_3}^{\p}   \|  f_4\|_{p_4}^{\p}  \big( \|D^\alpha  f_1\|_{\dot B_{p_1, \infty}^0}^{\p}   \| D^\gamma f_5\|_{\dot B_{p_5, \infty}^\beta}^{\p} \big)^{1 \over 2}  \big( \|D^\alpha  f_1\|_{\dot B_{p_1, \infty}^{\beta}}^{\p}   \|D^\gamma f_5\|_{\dot B_{p_5, \infty}^{0}}^{\p}  \big)^{1\over 2}.
\end{align*}
\end{itemize}

This completes the proof of \eqref{eq:one:param_M}, which, we recall, follows also from the flag paraproducts' boundedness in \cite{flag_paraproducts} -- except for certain endpoints.

\bigskip

\section{A five-linear flag: bi-parameter case}
\label{sec:2param:5linflag}

Now we move to the bi-parameter version of \eqref{eq:one:param_M}, which reads as
\begin{align}
\label{thm:5lin:biparam}
&\| D_{(1)}^{\beta_1}D_{(2)}^{\beta_2}( D_{(1)}^{\alpha_1}D_{(2)}^{\alpha_2}(f_1 \cdot f_2) \cdot f_3  \cdot D_{(1)}^{\gamma_1}D_{(2)}^{\gamma_2}( f_4 \cdot f_5) ) \|_{L^pL^q} \\
& \lesssim \| D_{(1)}^{\alpha_1+\beta_1}D_{(2)}^{\alpha_2+\beta_2} f_1 \|_{L^{p_1}L^{q_1}} \| f_2 \|_{L^{p_2}L^{q_2}} \|f_3\|_{L^{p_3}L^{q_3}} \|  D_{(1)}^{\gamma_1}D_{(2)}^{\gamma_2} f_4 \|_{L^{p_4}L^{q_4}}  \| f_5 \|_{L^{p_5}L^{q_5}} + \text{other similar terms.} \nonumber
\end{align}
This is the main motivation of our work, since bi-parameter equivalents of \cite{flag_paraproducts} are yet to be proved. For simplicity, we work with functions defined on $\BBR \times \BBR$, although all the results remain valid for functions on $\BBR^{d_1} \times \BBR^{d_2}$. In order for \eqref{thm:5lin:biparam} to hold, the conditions on the Lebesgue exponents, already presented in the Introduction, are
\[
\begin{array}{ c c c }
 \frac{1}{q}< 1+\beta_2, & \frac{1}{q_{1,2}}:= \frac{1}{q_1}+\frac{1}{q_2}  < 1+\alpha_2, & \frac{1}{q_{4,5}}:=\frac{1}{q_4}+\frac{1}{q_5}  <  1+\gamma_2,\\ 
  \frac{1}{p}< \min( 1+\beta_1, 1+\beta_2 ), \qquad  &  \frac{1}{p_{1,2}}:= \frac{1}{p_1}+\frac{1}{p_2}  <  \min( 1+\alpha_1, 1+\alpha_2), \qquad &  \frac{1}{p_{4,5}}:= \frac{1}{p_4}+\frac{1}{p_5}  < \min(1+\gamma_1, 1+\gamma_2).
\end{array}
\]
The asymmetry on the Lebesgue exponents associated to the first and second variables is a consequence of the mixed norm condition. The constraints on $p$ and $q$ are imposed by slower decaying conditions on the associated Fourier coefficients when treating the ``diagonal case'': see Section \ref{sec:2:param:5:flag:several:max}.

Although conceptually the method employed for proving \eqref{thm:5lin:biparam} will be similar to that presented in the previous section, technical aspects specific\footnote{These difficulties will be nowhere near as laborious as the technical aspects typical of multi-parameter singular integrals.} to multi-parameter problems will appear; in particular, the number of cases that need to be considered is significantly higher. That is partly due to the asymmetry of the objects: $f_2$ might not be involved in the summation of the flag acting in the first variable, but we cannot put it aside because it might be involved in the summation of the flag acting in the second variable.

As before in Section \ref{sec:5lin:1param}, in this particular situation we will not systematically apply the inductive procedure presented later in Section \ref{subsection_bi_leibniz}; however, the differences are minor.

In the bi-parameter setting, the multiplier associated to the flag \eqref{thm:5lin:biparam} is 
\begin{equation}
\label{eq:flag:5:bi}
\big( |\xi_1+\ldots+\xi_5|^{\beta_1} |\xi_1+\xi_2|^{\alpha_1} |\xi_4+\xi_5|^{\gamma_1}  \big) \cdot \big( |\eta_1+\ldots+\eta_5|^{\beta_2} |\eta_1+\eta_2|^{\alpha_2} |\eta_4+\eta_5|^{\gamma_2}  \big).
\end{equation}

As before in Section \ref{sec:5lin:1param}, the frequency regions will be decomposed into cones, allowing to determine the fastest oscillating functions in each parameter, and (up to a point) the distribution of derivatives. Depending on the structure of the associated flag/rooted tree and on the conical regions in each parameter, the multiplier will be split into several pieces. However, in each parameter, the decomposition is carried out as in the previous section.

Our aim is to estimate
\begin{align}
\label{5lin:freq:rep:bi}
&\sum_{k_1, \ldots, k_5 } \sum_{m_1, \ldots, m_5 } \int_{\BBR^{10}}\big( |\xi_1+\ldots+\xi_5|^{\beta_1} |\xi_1+\xi_2|^{\alpha_1} |\xi_4+\xi_5|^{\gamma_1}  \big) \cdot \big( |\eta_1+\ldots+\eta_5|^{\beta_2} |\eta_1+\eta_2|^{\alpha_2} |\eta_4+\eta_5|^{\gamma_2}  \big)  \ii F (\Delta_{k_1}^{(1)}\Delta_{m_1}^{(2)} f_1)(\xi_1, \eta_1) \\
& \cdot \ii F (\Delta_{k_2}^{(1)}\Delta_{m_2}^{(2)} f_2)(\xi_2, \eta_2) \cdot \ldots \cdot  \ii F (\Delta_{k_5}^{(1)}\Delta_{m_5}^{(2)} f_5)(\xi_5, \eta_5) e^{2 \pi i x(\xi_1+\ldots+\xi_5)} e^{2 \pi i y(\eta_1+\ldots+\eta_5)}  d \xi_1 \ldots d\xi_5  d \eta_1 \ldots d\eta_5 \nonumber
\end{align}
in the mixed (quasi-)norm  $\| \cdot \|_{L^pL^q}$. Throughout the section, $\tau$ will denote
\[
\tau:= \min(1, p, q),
\]
which renders $\| \cdot \|_{L^pL^q}^{\tau}$ subadditive.

Next, we will consider various operators that arise from restricting our attention to conical frequency regions $\{ |\xi_1| \gg |\xi_2|, \ldots, |\xi_5|, |\eta_1| \gg |\eta_1|, \ldots, |\eta_5|  \}$, etc. This corresponds to a bi-parameter paraproduct decomposition and allows to split \eqref{5lin:freq:rep:bi} into pieces that will be independently estimated. We will restrict our attention to certain typical conical regions, as the remaining cases follow from similar arguments.

\subsection{Study of the ``off-diagonal'' conical region $R_1^{(1)} \times R_1^{(2)}$}
\label{sec:5:flag:bi:param:k_1:m_1} 
When we are in the region where $|\xi_1|$ is the largest frequency variable in the first component and $|\eta_1|$ is the largest one in the second component, we want to approximate $|\xi_1+\ldots+\xi_5|^{\beta_1}$ by $|\xi_1|^{\beta_1}$ (and $|\eta_1+\ldots+\eta_5|^{\beta_2}$ by $|\eta_1|^{\beta_1}$). As before in Section \ref{sec:5lin:1param}, $\xi_1$ is connected to $\xi_2$ by the $D^{\alpha_1}_{(1)}$ derivative, so several steps are necessary.

We will have 
\begin{align}
\label{eq:decomposition:1st:coord}
|\xi_1+\ldots+\xi_5|^{\beta_1} = |\xi_1+\ldots+\xi_5|^{\beta_1}- |\xi_1+\xi_2|^{\beta_1} + |\xi_1+\xi_2|^{\beta_1}- |\xi_1|^{\beta_1}+|\xi_1|^{\beta_1} &\\
 = \frac{|\xi_1+\ldots+\xi_5|^{\beta_1}- |\xi_1+\xi_2|^{\beta_1}}{\xi_3+\xi_4+\xi_5 } \cdot \xi_3 + \frac{|\xi_1+\ldots+\xi_5|^{\beta_1}- |\xi_1+\xi_2|^{\beta_1}}{\xi_3+\xi_4+\xi_5 } \cdot (\xi_4+\xi_5)&+\big(|\xi_1+\xi_2|^{\beta_1} - |\xi_1|^{\beta_1}\big)+|\xi_1|^{\beta_1} \nonumber
\end{align}

With this decomposition, the part of the multiplier from \eqref{eq:flag:5:bi} acting on the first coordinate writes as
\begin{align*}
& \frac{|\xi_1+\ldots+\xi_5|^{\beta_1}- |\xi_1+\xi_2|^{\beta_1}}{\xi_3+\xi_4+\xi_5 } \cdot \xi_3 \cdot |\xi_1+\xi_2|^{\alpha_1} |\xi_4+\xi_5|^{\gamma_1} + \frac{|\xi_1+\ldots+\xi_5|^{\beta_1}- |\xi_1+\xi_2|^{\beta_1}}{\xi_3+\xi_4+\xi_5 } \cdot (\xi_4+\xi_5) \cdot |\xi_1+\xi_2|^{\alpha_1} |\xi_4+\xi_5|^{\gamma_1} \\
&+ \frac{|\xi_1+\xi_2|^{\beta_1+\alpha_1}- |\xi_1|^{\beta_1+\alpha_1}}{\xi_2}\cdot \xi_2 \cdot |\xi_4+\xi_5|^{\gamma_1} + |\xi_1|^{\beta_1+\alpha_1} \cdot |\xi_4+\xi_5|^{\gamma_1}:=m_{I_A^{(1)}}+m_{I_B^{(1)}}+m_{I_C^{(1)}}+m_{I_D^{(1)}}.
\end{align*}

The above formulas indicate the presence of three commutators and one multiplier that will require the change in the order of summation. For $m_{I_A^{(1)}}$, the variables involved in the commutator are $\xi_1+\xi_2$ and $\xi_3$, for $m_{I_B^{(1)}}$ - $\xi_1+\xi_2$ and $\xi_4+\xi_5$, and for $m_{I_C^{(1)}}$ it will be $\xi_1$ and $\xi_2$ that will participate in the summation.

In the second coordinate, the multiplier will split similarly into four terms:

\begin{align*}
& \frac{|\eta_1+\ldots+\eta_5|^{\beta_2}- |\eta_1+\eta_2|^{\beta_2}}{\eta_3+\eta_4+\eta_5 } \cdot \eta_3 \cdot |\eta_1+\eta_2|^{\alpha_2} |\eta_4+\eta_5|^{\gamma_2} + \frac{|\eta_1+\ldots+\eta_5|^{\beta_2}- |\eta_1+\eta_2|^{\beta_2}}{\eta_3+\eta_4+\eta_5 } \cdot (\eta_4+\eta_5) \cdot |\eta_1+\eta_2|^{\alpha_2} |\eta_4+\eta_5|^{\gamma_2} \\
&+ \frac{|\eta_1+\eta_2|^{\beta_2+\alpha_2}- |\eta_1|^{\beta_2+\alpha_2}}{\eta_2}\cdot \eta_2 \cdot |\eta_4+\eta_5|^{\gamma_2} + |\eta_1|^{\beta_2+\alpha_2} \cdot |\eta_4+\eta_5|^{\gamma_2}:=m_{I_A^{(2)}}+m_{I_B^{(2)}}+m_{I_C^{(2)}}+m_{I_D^{(2)}}.
\end{align*}

When the symbol is restricted to the conical region $R_1^{(1)} \times R_1^{(2)}$, \eqref{5lin:freq:rep:bi} breaks down as 
{\fontsize{9}{9}\begin{align*}
&\sum_{\substack {k_2, \ldots, k_5 \ll k_1  \\ m_2, \ldots, m_5 \ll m_1}} T_{I_A^{(1)}\times I_A^{(2)}}(\Delta_{k_1}^{(1)}\Delta_{m_1}^{(2)} f_1, \ldots, \Delta_{k_5}^{(1)}\Delta_{m_5}^{(2)} f_5)(x, y) + \sum_{\substack {k_2, \ldots, k_5 \ll k_1  \\ m_2, \ldots, m_5 \ll m_1}} T_{I_A^{(1)}\times I_B^{(2)}}(\Delta_{k_1}^{(1)}\Delta_{m_1}^{(2)} f_1, \ldots, \Delta_{k_5}^{(1)}\Delta_{m_5}^{(2)} f_5)(x, y)+ \\
&\sum_{\substack {k_2, \ldots, k_5 \ll k_1  \\ m_2, \ldots, m_5 \ll m_1}} T_{I_A^{(1)}\times I_C^{(2)}}(\Delta_{k_1}^{(1)}\Delta_{m_1}^{(2)} f_1, \ldots, \Delta_{k_5}^{(1)}\Delta_{m_5}^{(2)} f_5)(x, y)+ \ldots+ \sum_{\substack {k_2, \ldots, k_5 \ll k_1  \\ m_2, \ldots, m_5 \ll m_1}} T_{I_D^{(1)}\times I_D^{(2)}}(\Delta_{k_1}^{(1)}\Delta_{m_1}^{(2)} f_1, \ldots, \Delta_{k_5}^{(1)}\Delta_{m_5}^{(2)} f_5)(x, y)  \\
&:= I_A^{(1)}\times I_A^{(2)} + I_A^{(1)}\times I_B^{(2)} + I_A^{(1)}\times I_C^{(2)}+ \ldots+ I_D^{(1)}\times I_D^{(2)}.
\end{align*}}
In total, we have $16$ cases, many of which are similar; so we will only treat some of them, as explained below. Due to the structure of the symbols, $I_A^{(1)}\times I_B^{(2)}$ presents novel attributes, and will be discussed in more detail.

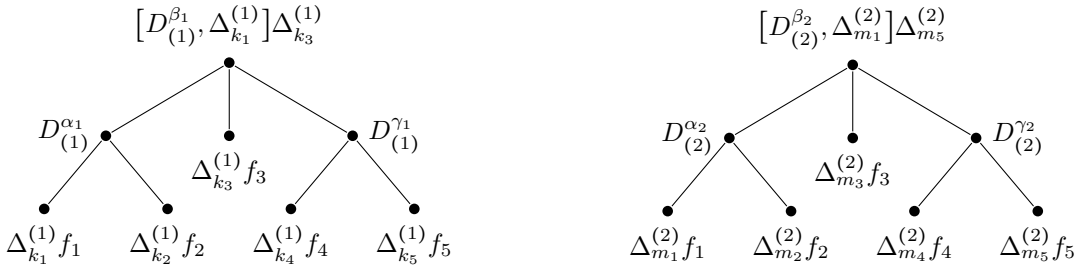
\begin{itemize}[leftmargin=.3cm]
\item $I_A^{(1)}\times I_B^{(2)})$

\bigskip
\[
\begin{array}{ccc}
 \vcenter{\hbox{\begin{forest}
    my treeSep
    [, label={above:$\big[D_{(1)}^{\beta_1}, \Delta^{(1)}_{k_1}\big] \Delta^{(1)}_{k_3} $}
                       [, label={left: $D_{(1)}^{\alpha_1}$}
                                    [, label={below: $ \Delta^{(1)}_{k_1} f_1$}]
                                    [, label={below: $\Delta^{(1)}_{k_2} f_2$}]
                           ]
                        [, label={below:$\Delta^{(1)}_{k_3} f_3$}]
                            [, label={right: $D_{(1)}^{\gamma_1}$}
                                    [, label={below: $\Delta^{(1)}_{k_4} f_4$}]
                                    [, label={below: $\Delta^{(1)}_{k_5} f_5$}]
                             ]
]
\end{forest}}} & \quad\quad \quad \quad&  \vcenter{\hbox{\begin{forest}
    my treeSep
    [, label={above:$\big[D_{(2)}^{\beta_2}, \Delta^{(2)}_{m_1}\big] \Delta^{(2)}_{m_5} $}
                       [, label={left: $D_{(2)}^{\alpha_2}$}
                                    [, label={below: $ \Delta^{(2)}_{m_1} f_1$}]
                                    [, label={below: $\Delta^{(2)}_{m_2} f_2$}]
                           ]
                        [, label={below:$\Delta^{(2)}_{m_3} f_3$}]
                            [, label={right: $D_{(2)}^{\gamma_2}$}
                                    [, label={below: $\Delta^{(2)}_{m_4} f_4$}]
                                    [, label={below: $\Delta^{(2)}_{m_5} f_5$}]
                             ]
]
\end{forest}}}
\end{array}
\]

The multiplier suggested by the trees above is $m_{I_A^{(1)}}(\xi_1, \ldots, \xi_5) \cdot m_{I_B^{(2)}}(\eta_1, \ldots, \eta_5)$, which corresponds to
{\fontsize{9}{9}\begin{align*}
&\Big( \frac{|\xi_1+\ldots+\xi_5|^{\beta_1}- |\xi_1+\xi_2|^{\beta_1}}{\xi_3+\xi_4+\xi_5} \cdot \xi_3 \cdot |\xi_1+\xi_2|^{\alpha_1} |\xi_4+\xi_5|^{\gamma_1} \Big) \cdot \Big( \frac{|\eta_1+\ldots+\eta_5|^{\beta_2}- |\eta_1+\eta_2|^{\beta_2}}{\eta_3+\eta_4+\eta_5} \cdot (\eta_4+\eta_5) \cdot |\eta_1+\eta_2|^{\alpha_2} |\eta_4+\eta_5|^{\gamma_2} \Big). 
\end{align*}}
In the first variable the functions involved in the summation over the scales are $f_1$ and $f_3$, so the functions $f_2,f_4$ and $f_5$ do not contribute to this process. However, in the second variable, the shape of the commutator indicates that $f_1$ and one of the functions $f_4$ or $f_5$ will play a prominent role; by restricting the symbol further to the region where $|\eta_4| \leq |\eta_5|$, we know the functions $f_1$ and $f_5$ will take part in summing up the scales, in the second parameter. We carefully group the terms contributing to $I_A^{(1)}\times I_B^{(2)}$ by summing over $k_2, k_4, k_5 \ll k_1$ and over $m_2, m_3 \ll m_1, m_4 \leq m_5$:
{\fontsize{9}{9}\begin{align*}
&T_{m_{I_A^{(1)}}, m_{I_B^{(2)}}}^{k_1, k_3; m_1, m_5}(f_1, \ldots, f_5)(x, y) = \int_{\BBR^{10}} m_{I_A^{(1)}}(\xi_1, \ldots, \xi_5) \cdot m_{I_B^{(2)}}(\eta_1, \ldots, \eta_5) \ii F(\Delta_{k_1}^{(1)} \Delta_{m_1}^{(2)} f_1 ) (\xi_1, \eta_1) \cdot \ii F(S_{k_1}^{(1)} S_{m_1}^{(2)} f_2)(\xi_2, \eta_2)  \\
&\cdot  \ii F(\Delta_{k_3}^{(1)} S_{m_1}^{(2)} f_3 )(\xi_3, \eta_3) \cdot  \ii F( S_{k_1}^{(1)} \Delta^{(2)}_{\leq m_5} f_4)(\xi_4, \eta_4) \cdot  \ii F(S_{k_1}^{(1)} \Delta^{(2)}_{m_5} f_5)(\xi_5, \eta_5) \cdot  e^{ 2 \pi ix \left(\xi_1+\ldots+\xi_5 \right)}  e^{ 2 \pi iy \left(\eta_1+\ldots+\eta_5 \right)} d \, \xi_1 \ldots d \, \xi_5 d \, \eta_1 \ldots d \, \eta_5.
\end{align*}}
As in Section \ref{sec:5lin:1param}, we restrict the symbol 
\[
\frac{|\xi_1+\ldots+\xi_5|^{\beta_1}- |\xi_1+\xi_2|^{\beta_1}}{\xi_3+\xi_4+\xi_5}= m_{C_{\beta_1}}(\xi_1+\xi_2, \xi_3+\xi_4+\xi_5)
\]
to $\pm [2^{k_1}, 2^{k_1+1}] \times [-2^{k_1-1}, 2^{k_1-1}]$ and denote its localized version by $m_{C_{\beta_1}}^{k_1, \pm}(\xi_1+ \xi_2, \xi_3+\xi_4+\xi_5)$, on which we perform a double Fourier series decomposition:
\[
m_{C_{\beta_1}}^{k_1, \pm}(\xi_1+\xi_2, \xi_3+\xi_4+\xi_5)= \sum_{L_1, L_2} C_{L_1, L_2}^{\pm} 2^{k_1(\beta_1-1)} e^{ 2 \pi i (\xi_1+\xi_2) \frac{L_1}{2^{k_1}}}  e^{ 2 \pi i (\xi_3+\xi_4+\xi_5) \frac{L_2}{2^{k_1}}}.
\]

Similarly, let $m^{m_1, \pm}_{C_{\beta_2}}(\eta_1+\eta_2, \eta_3+\eta_4+\eta_5)$ denote the symbol $m_{C_{\beta_2}}(\eta_1+\eta_2, \eta_3+\eta_4+\eta_5)$ localized on $\pm [2^{m_1}, 2^{m_1+1}] \times [-2^{m_1-1}, 2^{m_1-1}]$, where
\[
m_{C_{\beta_2}}(\eta_1+\eta_2, \eta_3+\eta_4+\eta_5) := \frac{|\eta_1+\ldots+\eta_5|^{\beta_2}- |\eta_1+\eta_2|^{\beta_2}}{\eta_3+\eta_4+\eta_5}.
\]
The double Fourier series expansion on $m^{m_1, \pm}_{C_{\beta_2}}$ yields
\[
m^{m_1, \pm}_{C_{\beta_2}}(\eta_1+\eta_2, \eta_3+\eta_4+\eta_5)= \sum_{\tilde L_1, \tilde L_2} C^{\pm}_{\tilde L_1, \tilde L_2} 2^{m_1(\beta_2-1)} e^{ 2 \pi i (\eta_1+\eta_2) \frac{\tilde L_1}{2^{m_1}}}  e^{ 2 \pi i (\eta_3+\eta_4+\eta_5) \frac{\tilde L_2}{2^{m_1}}}.
\]
In both cases, the renormalized Fourier coefficients have arbitrary decay. This implies that 
\begin{align*}
\|I_A^{(1)}\times I_B^{(2)}\|_{L^pL^q}^\tau \leq  & \sum_{L_1, L_2} |C_{L_1, L_2}^{\pm}|^\tau  \sum_{\tilde L_1, \tilde L_2} |C^{\pm}_{\tilde L_1, \tilde L_2}|^\tau \big\| D^{\alpha_1}_{(1)} D^{\alpha_2}_{(2)} (  \Delta_{k_1, \pm, \frac{L_1}{2^{k_1}}}^{(1)} \Delta_{m_1, \pm, \frac{\tilde L_1}{2^{m_1}}}^{(2)}  f_1 \cdot S_{k_1, \frac{L_1}{2^{k_1}}}^{(1)} S_{m_1, \frac{\tilde L_1}{2^{m_1}}}^{(2)} f_2 ) \big\|_{p_{1, 2}}^\tau \\
& \big\|   \Delta_{k_3, \frac{L_2}{2^{k_1}}}^{(1)} S_{m_1, \frac{\tilde L_2}{2^{m_1}}}^{(2)}  f_3  \big\|_{L^{p_3}L^{q_3}}^\tau \big\| D^{\gamma_1}_{(1)} D^{\gamma_2}_{(2)} (  S_{k_1, \frac{L_2}{2^{k_1}}}^{(1)} \Delta_{\leq m_5, \frac{\tilde L_2}{2^{m_1}}}^{(2)}  f_4 \cdot S_{k_1, \frac{L_2}{2^{k_1}}}^{(1)} \Delta_{m_5, \frac{\tilde L_2}{2^{m_1}}}^{(2)} f_5 ) \big\|_{p_{4, 5}}^\tau.
\end{align*}

Now we use the boundedness of flag paraproducts of lower complexity: more concretely,
the biparameter variant\footnote{The general statement will be presented in Proposition \ref{prop:biparameter_induction}.} of Lemma \ref{lemma:multipliers:simpleParaproduct} which describes the localized version of Oh and Wu \cite{OhWu} with the additional observation that $D^{\gamma_2}_{(2)}$ derivatives will be attached to $f_5$:
\begin{align*}
\big\| D^{\gamma_1}_{(1)} D^{\gamma_2}_{(2)} (  S_{k_1, \frac{L_2}{2^{k_1}}}^{(1)} \Delta_{\leq m_5, \frac{\tilde L_2}{2^{m_1}}}^{(2)}  f_4 \cdot S_{k_1, \frac{L_2}{2^{k_1}}}^{(1)} \Delta_{m_5, \frac{\tilde L_2}{2^{m_1}}}^{(2)} f_5 ) \big\|_{L^{p_{4, 5}}L^{q_{4,5}}} \lesssim & \|D^{\gamma_1}_{(1)} f_4\|_{L^{p_4}L^{q_4}}  \| \Delta_{m_5}^{(2)} D^{\gamma_2}_{(2)} f_5    \|_{L^{p_5}L^{q_5}} \\ & + \| f_4\|_{L^{p_4}L^{q_4}}  \| \Delta_{m_5}^{(2)} D^{\gamma_1}_{(1)} D^{\gamma_2}_{(2)} f_5    \|_{L^{p_5}L^{q_5}}. 
\end{align*}
Here, we need to assume that $p_{4,5}>\frac{1}{1+\gamma_1}$, while $q_{4,5}$ can be any Lebesgue exponent $\geq \frac{1}{2}$. Similarly,\footnote{Here we use implicitly \eqref{eq:direction:LP:same} and \eqref{eq:trivial:modulation}} 
\[
 \big\| D^{\alpha_1}_{(1)} D^{\alpha_2}_{(2)} (  \Delta_{k_1, \pm,  \frac{L_1}{2^{k_1}}}^{(1)} \Delta_{m_1, \pm, \frac{\tilde L_1}{2^{m_1}}}^{(2)}  f_1 \cdot S_{k_1, \frac{L_1}{2^{k_1}}}^{(1)} S_{m_1, \frac{\tilde L_1}{2^{m_1}}}^{(2)} f_2 ) \big\|_{L^{p_{1, 2}}} \lesssim \| \Delta_{k_1}^{(1)} \Delta_{m_1}^{(2)}   D^{\alpha_1}_{(1)} D^{\alpha_2}_{(2)} f_1\|_{L^{p_1}L^{q_1}} \, \|f_2\|_{L^{p_2}L^{q_2}},
 \]
with $p_{1,2}, q_{1, 2} \geq {1 \over 2}$.
 
For simplicity, we denote $(F_4, F_5)$ either of the couples $(D_{(1)}^{\gamma_1} f_4,  f_5 )$ or $( f_4, D_{(1)}^{\gamma_1} f_5)$, and $F_1:= D^{\alpha_1}_{(1)} D^{\alpha_2}_{(2)} f_1$.

Since the renormalized Fourier coefficients $C_{L_1, L_2}^{\pm}$ and $C_{\tilde L_1, \tilde L_2}^{\pm}$ are summable and all the Fourier projections $P_{k, a}^{(j)}$ are bounded on $L^{p_j}L^{q_j}$ spaces with $1 \leq p_j, q_j \leq \infty$, we are left with summing 
\begin{align}
\label{eq:red:k_1m_1k_3m_5:red}
& \sum_{k_3<k_1} \sum_{m_5<m_1} 2^{k_1(\beta_1-1) \tau} 2^{k_3 \tau} 2^{m_1(\beta_2-1) \tau} 2^{m_5 \tau} \| \Delta_{k_1}^{(1)} \Delta_{m_1}^{(2)} F_1 \|_{L^{p_1}L^{q_1}}^\tau \| f_2   \|_{L^{p_2}L^{q_2}}^\tau  \| \Delta_{k_3}^{(1)} f_3   \|_{L^{p_3}L^{q_3}}^\tau   \| F_4\|_{L^{p_4}L^{q_4}}^\tau  \|\Delta^{(2)}_{m_5} D^{\gamma_2}_{(2)} F_5  \|_{L^{p_5}L^{q_5}}^\tau. 
\end{align}

Since $f_2$ and $F_4$ have no contribution in the summation, we put them on the side and estimate what is left of \eqref{eq:red:k_1m_1k_3m_5:red} as
\begin{align}
\label{eq:red:k_1m_1k_3m_5:red:min}
\sum_{k_1, m_1} \min \big(  &2^{k_1\beta_1 \tau} 2^{m_1\beta_2 \tau} \|F_1\|_{\dot B^0_{p_1, \infty} \dot B^0_{q_1, \infty}}^\tau  \|f_3\|_{\dot B^0_{p_3, \infty} L^{q_3}}^\tau  \| D_{(2)}^{\gamma_2}F_5\|_{L^{p_5} \dot B^0_{q_5, \infty}}^\tau, \\ & 2^{k_1\beta_1 \tau} 2^{ -m_1\epsilon2 \tau} \|F_1\|^\tau_{\dot B^0_{p_1, \infty} \dot B^{\beta_2}_{q_1, \infty}}  \|f_3\|^\tau_{\dot B^0_{p_3, \infty} L^{q_3}}  \| D_{(2)}^{\gamma_2}F_5\|^\tau_{L^{p_5} \dot B^{\epsilon_2}_{q_5, \infty}}, \nonumber \\
& 2^{-k_1\epsilon_1 \tau} 2^{m_1\beta_2 \tau} \|F_1\|^\tau_{\dot B^{\beta_1}_{p_1, \infty} \dot B^0_{q_1, \infty}}  \|f_3\|^\tau_{\dot B^{\epsilon_1}_{p_3, \infty} L^{q_3}}  \| D_{(2)}^{\gamma_2}F_5\|^\tau_{L^{p_5} \dot B^0_{q_5, \infty}}, \nonumber \\& 2^{-k_1\epsilon_1 \tau} 2^{-m_1\epsilon_2 \tau} \|F_1\|^\tau_{\dot B^{\beta_1}_{p_1, \infty} \dot B^{\beta_2}_{q_1, \infty}}  \|f_3\|^\tau_{\dot B^{\epsilon_1}_{p_3, \infty} L^{q_3}}  \| D_{(2)}^{\gamma_2}F_5\|^\tau_{L^{p_5} \dot B^{\epsilon_2}_{q_5, \infty}} \big). \nonumber
\end{align}

The aim here is to obtain both positive and negative powers of $2^{k_1}$ and $2^{m_1}$ respectively, which will allow to sum both the small and the large scales. Now the key observation is the following:
\begin{align}
\label{eq:optim:2param}
\sum_{k_1, m_1} \min( a_{k_1} a_{m_1} A, a_{k_1} b_{m_1} B, b_{k_1} a_{m_1} C, b_{k_1} b_{m_1} D) \leq \sum_{k_1} \min \big( a_{k_1} \sum_{m_1}   \min(a_{m_1} A, b_{m_1} B),  b_{k_1} \sum_{m_1}  \min(a_{m_1} C, b_{m_1} D ) \big). 
\end{align}

We have seen before that
\[
 \sum_{m_1}   \min(2^{a m_1} A, 2^{- b m_1} B) \leq A^{b \over a+b} B^{a \over a+b},
\]
so we have
\begin{align}
\label{eq:optim:2param:2}
&\sum_{k_1, m_1} \min( 2^{a_1 k_1} 2^{a_2 m_1} A, 2^{a_1 k_1} 2^{-b_2 m_1} B, 2^{-b_1  k_1} 2^{a_2 m_1} C, 2^{-b_1 k_1} 2^{-b_2 m_1} D) \nonumber \\
&\leq \sum_{k_1} \min(  2^{a_1  k_1}  A^{b_2 \over {a_2+b_2}} B^{a_2 \over {a_2+b_2}}, 2^{- b_1 k_1}  C^{b_2 \over {a_2+b_2}} D^{a_2 \over {a_2+b_2}}) \nonumber \\
&\leq  \big( A^{b_2 \over {a_2+b_2}} B^{a_2 \over {a_2+b_2}} \big)^{b_1 \over {a_1+b_1}} \,  \big(  C^{b_2 \over {a_2+b_2}} D^{a_2 \over {a_2+b_2}}\big)^{a_1 \over{a_1+b_1}}.
\end{align}

Up to this point, the argument is similar to Lemma 4.5 in OhWu \cite{OhWu}. It remains however to justify that we have the correct terms and the correct powers. We have obtained that $\|I_A^{(1)}\times I_B^{(2)}\|_{L^pL^q}$ is bounded by a geometric average of
\begin{align*}
&\| f_2   \|_{L^{p_2}L^{q_2}} ,  \quad  \|F_1\|_{\dot B^0_{p_1, \infty} \dot B^0_{q_1, \infty}}  \|f_3\|_{\dot B^0_{p_3, \infty} L^{q_3}}  \| D_{(2)}^{\gamma_2}F_5\|_{L^{p_5} \dot B^0_{q_5, \infty}}, \quad \|F_1\|_{\dot B^0_{p_1, \infty} \dot B^{\beta_2}_{q_1, \infty}}  \|f_3\|_{\dot B^0_{p_3, \infty} L^{q_3}}  \| D_{(2)}^{\gamma_2}F_5\|_{L^{p_5} \dot B^{\epsilon_2}_{q_5, \infty}}, \\
& \| F_4\|_{L^{p_4}L^{q_4}}, \quad \|F_1\|_{\dot B^{\beta_1}_{p_1, \infty} \dot B^0_{q_1, \infty}}  \|f_3\|_{\dot B^{\epsilon_1}_{p_3, \infty} L^{q_3}}  \| D_{(2)}^{\gamma_2}F_5\|_{L^{p_5} \dot B^0_{q_5, \infty}}, \quad \|F_1\|_{\dot B^{\beta_1}_{p_1, \infty} \dot B^{\beta_2}_{q_1, \infty}}  \|f_3\|_{\dot B^{\epsilon_1}_{p_3, \infty} L^{q_3}}  \| D_{(2)}^{\gamma_2}F_5\|_{L^{p_5} \dot B^{\epsilon_2}_{q_5, \infty}}. \label{eq:fin:bi:5}
\end{align*}

In our case, $a_1=\beta_1, a_2=\beta_2, b_1=\epsilon_1, b_2=\epsilon_2$ and
\begin{align*}
&A=\|F_1\|_{\dot B^0_{p_1, \infty} \dot B^0_{q_1, \infty}}  \|f_3\|_{\dot B^0_{p_3, \infty} L^{q_3}}  \| D_{(2)}^{\gamma_2}F_5\|_{L^{p_5} \dot B^0_{q_5, \infty}}, \quad  B=\|F_1\|_{\dot B^0_{p_1, \infty} \dot B^{\beta_2}_{q_1, \infty}}  \|f_3\|_{\dot B^0_{p_3, \infty} L^{q_3}}  \| D_{(2)}^{\gamma_2}F_5\|_{L^{p_5} \dot B^{\epsilon_2}_{q_5, \infty}}\\
&C= \|F_1\|_{\dot B^{\beta_1}_{p_1, \infty} \dot B^0_{q_1, \infty}}  \|f_3\|_{\dot B^{\epsilon_1}_{p_3, \infty} L^{q_3}}  \| D_{(2)}^{\gamma_2}F_5\|_{L^{p_5} \dot B^0_{q_5, \infty}}, \quad D= \|F_1\|_{\dot B^{\beta_1}_{p_1, \infty} \dot B^{\beta_2}_{q_1, \infty}}  \|f_3\|_{\dot B^{\epsilon_1}_{p_3, \infty} L^{q_3}}  \| D_{(2)}^{\gamma_2}F_5\|_{L^{p_5} \dot B^{\epsilon_2}_{q_5, \infty}}.
\end{align*}

Now we need to understand what happens to the $\epsilon_1$ and $\epsilon_2$ derivatives encoded in the mixed Besov and Lebesgue norms. First, we look at the interaction between $A$ and $B$; $ \|f_3\|_{\dot B^0_{p_3, \infty} L^{q_3}}$ remains unchanged, and similar to the one-parameter Leibniz rule from Section \ref{Bourgain-Li_hilow}, 
\begin{align*}
&\big( \|F_1\|_{\dot B^0_{p_1, \infty} \dot B^0_{q_1, \infty}} \,  \| D_{(2)}^{\gamma_2}F_5\|_{L^{p_5} \dot B^0_{q_5, \infty}} \big)^{\epsilon_2 \over{\beta_2+\epsilon_2}} \cdot \big( \|F_1\|_{\dot B^0_{p_1, \infty} \dot B^{\beta_2}_{q_1, \infty}} \, \| D_{(2)}^{\gamma_2}F_5\|_{L^{p_5} \dot B^{\epsilon_2}_{q_5, \infty}}  \big)^{\beta_2 \over{\beta_2 + \epsilon_2}} \\
&\lesssim \big( \|F_1\|_{\dot B^0_{p_1, \infty} \dot B^0_{q_1, \infty}} \,  \| D_{(2)}^{\gamma_2}F_5\|_{L^{p_5} \dot B^{\beta_2}_{q_5, \infty}} \big)^{\epsilon_2 \over{\beta_2+\epsilon_2}} \cdot \big( \|F_1\|_{\dot B^0_{p_1, \infty} \dot B^{\beta_2}_{q_1, \infty}} \, \| D_{(2)}^{\gamma_2}F_5\|_{L^{p_5} \dot B^{0}_{q_5, \infty}}  \big)^{\beta_2 \over{\beta_2 + \epsilon_2}}.
\end{align*}

For the $C$ and $D$ interaction, initially $ \|f_3\|_{\dot B^{\epsilon_1}_{p_3, \infty} L^{q_3}}$ remains unchanged and the Besov norm interpolation \eqref{eq:interpolation:Besov:epsilon} yields
\begin{align*}
&\big( \|F_1\|_{\dot B^{\beta_1}_{p_1, \infty} \dot B^0_{q_1, \infty}} \,  \| D_{(2)}^{\gamma_2}F_5\|_{L^{p_5} \dot B^0_{q_5, \infty}} \big)^{\epsilon_2 \over{\beta_2+\epsilon_2}} \cdot \big( \|F_1\|_{\dot B^{\beta_1}_{p_1, \infty} \dot B^{\beta_2}_{q_1, \infty}} \, \| D_{(2)}^{\gamma_2}F_5\|_{L^{p_5} \dot B^{\epsilon_2}_{q_5, \infty}}  \big)^{\beta_2 \over{\beta_2 + \epsilon_2}} \\
&\lesssim \big( \|F_1\|_{\dot B^{\beta_1}_{p_1, \infty} \dot B^0_{q_1, \infty}} \,  \| D_{(2)}^{\gamma_2}F_5\|_{L^{p_5} \dot B^{\beta_2}_{q_5, \infty}} \big)^{\epsilon_2 \over{\beta_2+\epsilon_2}} \cdot \big( \|F_1\|_{\dot B^{\beta_1}_{p_1, \infty} \dot B^{\beta_2}_{q_1, \infty}} \, \| D_{(2)}^{\gamma_2}F_5\|_{L^{p_5} \dot B^{0}_{q_5, \infty}}  \big)^{\beta_2 \over{\beta_2 + \epsilon_2}}.
\end{align*}

So \eqref{eq:optim:2param:2}, for our particular choice of $A, \ldots, D$, $a_1, \ldots, b_2$, is bounded above by
\begin{align*}
&\Big(  \big( \|f_3\|_{\dot B^0_{p_3, \infty} L^{q_3}} \, \|F_1\|_{\dot B^0_{p_1, \infty} \dot B^0_{q_1, \infty}} \,  \| D_{(2)}^{\gamma_2}F_5\|_{L^{p_5} \dot B^{\beta_2}_{q_5, \infty}} \big)^{\epsilon_2 \over{\beta_2+\epsilon_2}} \cdot \big(\|f_3\|_{\dot B^0_{p_3, \infty} L^{q_3}} \, \|F_1\|_{\dot B^0_{p_1, \infty} \dot B^{\beta_2}_{q_1, \infty}} \, \| D_{(2)}^{\gamma_2}F_5\|_{L^{p_5} \dot B^{0}_{q_5, \infty}}  \big)^{\beta_2 \over{\beta_2 + \epsilon_2}} \Big)^{\epsilon_1 \over {\beta_1+\epsilon_1}} \\
& \cdot \Big(  \big(  \|f_3\|_{\dot B^{\epsilon_1}_{p_3, \infty} L^{q_3}}\, \|F_1\|_{\dot B^{\beta_1}_{p_1, \infty} \dot B^0_{q_1, \infty}} \,  \| D_{(2)}^{\gamma_2}F_5\|_{L^{p_5} \dot B^{\beta_2}_{q_5, \infty}} \big)^{\epsilon_2 \over{\beta_2+\epsilon_2}} \cdot \big(  \|f_3\|_{\dot B^{\epsilon_1}_{p_3, \infty} L^{q_3}} \, \|F_1\|_{\dot B^{\beta_1}_{p_1, \infty} \dot B^{\beta_2}_{q_1, \infty}} \, \| D_{(2)}^{\gamma_2}F_5\|_{L^{p_5} \dot B^{0}_{q_5, \infty}}  \big)^{\beta_2 \over{\beta_2 + \epsilon_2}}   \Big)^{\beta_1 \over {\beta_1+\epsilon_1}}
\end{align*}
Notice that we removed the $\epsilon_2$ parameter from the Besov norms in the second variable, and that the derivatives have been redistributed thanks to the interpolation result \eqref{eq:interpolation:Besov}. Now we regroup the terms, and the expression above becomes
\begin{align*}
&\Big(  \big( \|f_3\|_{\dot B^0_{p_3, \infty} L^{q_3}} \, \|F_1\|_{\dot B^0_{p_1, \infty} \dot B^0_{q_1, \infty}} \,  \| D_{(2)}^{\gamma_2}F_5\|_{L^{p_5} \dot B^{\beta_2}_{q_5, \infty}} \big)^{\epsilon_1 \over{\beta_1+\epsilon_1}} \cdot\big(  \|f_3\|_{\dot B^{\epsilon_1}_{p_3, \infty} L^{q_3}}\, \|F_1\|_{\dot B^{\beta_1}_{p_1, \infty} \dot B^0_{q_1, \infty}} \,  \| D_{(2)}^{\gamma_2}F_5\|_{L^{p_5} \dot B^{\beta_2}_{q_5, \infty}} \big)^{\beta_1 \over{\beta_1 + \epsilon_1}} \Big)^{\epsilon_2 \over {\beta_2+\epsilon_2}} \\
& \cdot \Big(  \big(  \|f_3\|_{\dot B^{\epsilon_1}_{p_3, \infty} L^{q_3}}\, \|F_1\|_{\dot B^{\beta_1}_{p_1, \infty} \dot B^0_{q_1, \infty}} \,  \| D_{(2)}^{\gamma_2}F_5\|_{L^{p_5} \dot B^{\beta_2}_{q_5, \infty}} \big)^{\epsilon_1 \over{\beta_1+\epsilon_1}} \cdot \big(  \|f_3\|_{\dot B^{\epsilon_1}_{p_3, \infty} L^{q_3}} \, \|F_1\|_{\dot B^{\beta_1}_{p_1, \infty} \dot B^{\beta_2}_{q_1, \infty}} \, \| D_{(2)}^{\gamma_2}F_5\|_{L^{p_5} \dot B^{0}_{q_5, \infty}}  \big)^{\beta_1 \over{\beta_1 + \epsilon_1}} \Big)^{\beta_2 \over {\beta_2+\epsilon_2}}.
\end{align*}

On the first line, $\| D_{(2)}^{\gamma_2}F_5\|_{L^{p_5} \dot B^{\beta_2}_{q_5, \infty}}$ remains unchanged and 
\begin{align*}
& \big( \|f_3\|_{\dot B^0_{p_3, \infty} L^{q_3}} \, \|F_1\|_{\dot B^0_{p_1, \infty} \dot B^0_{q_1, \infty}}\big)^{\epsilon_1 \over{\beta_1+\epsilon_1}} \cdot\big(  \|f_3\|_{\dot B^{\epsilon_1}_{p_3, \infty} L^{q_3}}\, \|F_1\|_{\dot B^{\beta_1}_{p_1, \infty} \dot B^0_{q_1, \infty}}  \big)^{\beta_1 \over{\beta_1 + \epsilon_1}}  \\
 &\lesssim  \big( \|f_3\|_{\dot B^{\beta_1}_{p_3, \infty} L^{q_3}} \, \|F_1\|_{\dot B^0_{p_1, \infty} \dot B^0_{q_1, \infty}}\big)^{\epsilon_1 \over{\beta_1+\epsilon_1}} \cdot\big(  \|f_3\|_{\dot B^{0}_{p_3, \infty} L^{q_3}}\, \|F_1\|_{\dot B^{\beta_1}_{p_1, \infty} \dot B^0_{q_1, \infty}}  \big)^{\beta_1 \over{\beta_1 + \epsilon_1}}. 
\end{align*}

The second line can be estimated similarly; thus \eqref{eq:optim:2param:2} is bounded above by
\begin{align*}
&\Big(  \big( \|f_3\|_{\dot B^{\beta_1}_{p_3, \infty} L^{q_3}} \, \|F_1\|_{\dot B^0_{p_1, \infty} \dot B^0_{q_1, \infty}} \,  \| D_{(2)}^{\gamma_2}F_5\|_{L^{p_5} \dot B^{\beta_2}_{q_5, \infty}} \big)^{\epsilon_1 \over{\beta_1+\epsilon_1}} \cdot\big(  \|f_3\|_{\dot B^{0}_{p_3, \infty} L^{q_3}}\, \|F_1\|_{\dot B^{\beta_1}_{p_1, \infty} \dot B^0_{q_1, \infty}} \,  \| D_{(2)}^{\gamma_2}F_5\|_{L^{p_5} \dot B^{\beta_2}_{q_5, \infty}} \big)^{\beta_1 \over{\beta_1 + \epsilon_1}} \Big)^{\epsilon_2 \over {\beta_2+\epsilon_2}} \\
& \cdot \Big(  \big(  \|f_3\|_{\dot B^{0}_{p_3, \infty} L^{q_3}}\, \|F_1\|_{\dot B^{\beta_1}_{p_1, \infty} \dot B^0_{q_1, \infty}} \,  \| D_{(2)}^{\gamma_2}F_5\|_{L^{p_5} \dot B^{\beta_2}_{q_5, \infty}} \big)^{\epsilon_1 \over{\beta_1+\epsilon_1}} \cdot \big(  \|f_3\|_{\dot B^{0}_{p_3, \infty} L^{q_3}} \, \|F_1\|_{\dot B^{\beta_1}_{p_1, \infty} \dot B^{\beta_2}_{q_1, \infty}} \, \| D_{(2)}^{\gamma_2}F_5\|_{L^{p_5} \dot B^{0}_{q_5, \infty}}  \big)^{\beta_1 \over{\beta_1 + \epsilon_1}} \Big)^{\beta_2 \over {\beta_2+\epsilon_2}}
\end{align*}
which can be further estimated by
\begin{align*}
&\|f_3\|_{\dot B^{\beta_1}_{p_3, \infty} L^{q_3}} \, \|F_1\|_{\dot B^0_{p_1, \infty} \dot B^0_{q_1, \infty}} \,  \| D_{(2)}^{\gamma_2}F_5\|_{L^{p_5} \dot B^{\beta_2}_{q_5, \infty}} + \|f_3\|_{\dot B^{0}_{p_3, \infty} L^{q_3}}\, \|F_1\|_{\dot B^{\beta_1}_{p_1, \infty} \dot B^0_{q_1, \infty}} \,  \| D_{(2)}^{\gamma_2}F_5\|_{L^{p_5} \dot B^{\beta_2}_{q_5, \infty}} \\
& +   \|f_3\|_{\dot B^{0}_{p_3, \infty} L^{q_3}}\, \|F_1\|_{\dot B^{\beta_1}_{p_1, \infty} \dot B^0_{q_1, \infty}} \,  \| D_{(2)}^{\gamma_2}F_5\|_{L^{p_5} \dot B^{\beta_2}_{q_5, \infty}}  +\|f_3\|_{\dot B^{0}_{p_3, \infty} L^{q_3}} \, \|F_1\|_{\dot B^{\beta_1}_{p_1, \infty} \dot B^{\beta_2}_{q_1, \infty}} \, \| D_{(2)}^{\gamma_2}F_5\|_{L^{p_5} \dot B^{0}_{q_5, \infty}}. 
\end{align*}

\medskip

\item $I_A^{(1)}\times I_A^{(2)})$ This case is simpler than the previous one, which is why we will only briefly present the arguments. The multiplier $m_{I_A^{(1)}}(\xi_1, \ldots, \xi_5) \cdot m_{I_A^{(2)}}(\eta_1, \ldots, \eta_5)$ is given by 
\begin{align*}
&\Big( \frac{|\xi_1+\ldots+\xi_5|^{\beta_1}- |\xi_1+\xi_2|^{\beta_1}}{\xi_3+\xi_4+\xi_5} \cdot \xi_3 \cdot |\xi_1+\xi_2|^{\alpha_1} |\xi_4+\xi_5|^{\gamma_1} \Big) \cdot \Big( \frac{|\eta_1+\ldots+\eta_5|^{\beta_2}- |\eta_1+\eta_2|^{\beta_2}}{\eta_3+\eta_4+\eta_5} \cdot \eta_3 \cdot |\eta_1+\eta_2|^{\alpha_2} |\eta_4+\eta_5|^{\gamma_2} \Big). 
\end{align*}
After summing over $k_2, k_4, k_5 \ll k_1$ and $m_2, m_4, m_5 \ll m_1$, the associated operator from \eqref{5lin:freq:rep:bi} becomes
\begin{align*}
&T_{m_{I_A^{(1)}}, m_{I_A^{(2)}}}^{k_1, k_3; m_1, m_3}(f_1, \ldots, f_5)(x, y) = \int_{\BBR^{10}} m_{I_A^{(1)}}(\xi_1, \ldots, \xi_5) \cdot m_{I_A^{(2)}}(\eta_1, \ldots, \eta_5) \ii F(\Delta_{k_1}^{(1)} \Delta_{m_1}^{(2)} f_1 ) (\xi_1, \eta_1) \cdot \ii F(S_{k_1}^{(1)} S_{m_1}^{(2)} f_2)(\xi_2, \eta_2)  \\
&\cdot  \ii F(\Delta_{k_3}^{(1)} \Delta_{m_3}^{(2)} f_3 )(\xi_3, \eta_3) \cdot  \ii F( S_{k_1}^{(1)} S_{m_1}^{(2)} f_4)(\xi_4, \eta_4) \cdot  \ii F(S_{k_1}^{(1)} S_{m_1}^{(2)} f_5)(\xi_5, \eta_5) \cdot  e^{ 2 \pi ix \left(\xi_1+\ldots+\xi_5 \right)}  e^{ 2 \pi iy \left(\eta_1+\ldots+\eta_5 \right)} d \, \xi_1 \ldots d \, \xi_5 d \, \eta_1 \ldots d \, \eta_5.
\end{align*}

Notice that the operator is symmetric in the first and second parameter, in the sense that the conical frequency regions are described by similar inequalities and the multipliers are similar. We continue with the usual Fourier series decomposition, which allows to regard $T_{m_{I_A^{(1)}}, m_{I_A^{(2)}}}(f_1, \ldots, f_5)$ as a superposition of (modulated) terms of the form 
\begin{align*}
2^{k_1(\beta_1-1)} 2^{k_3}\, 2^{m_1(\beta_2-1)} 2^{m_3} & \big( \Delta_{k_1}^{(1)} \Delta_{m_1}^{(2)}  D^{\alpha_1}_{(1)} D^{\alpha_2}_{(2)} f_1 \big)(x, y) \cdot S_{k_1}^{(1)} S_{m_1}^{(2)} f_2(x, y) \\
& \cdot  \Delta_{k_3}^{(1)} \Delta_{m_3}^{(2)} f_3  (x, y) \cdot D^{\gamma_1}_{(1)} D^{\gamma_2}_{(2)} \big( S_{k_1}^{(1)} S_{m_1}^{(2)} f_4 \cdot S_{k_1}^{(1)} S_{m_1}^{(2)} f_5 \big)(x, y).
\end{align*}

When each of these terms is estimated in $\| \cdot \|_{L^pL^q}^\tau$, they are bounded above by
\begin{align*}
& \, 2^{k_1(\beta_1-1) \tau} 2^{k_3 \tau} \, 2^{m_1(\beta_2-1) \tau} 2^{m_3 \tau} \|\Delta_{k_1}^{(1)} \Delta_{m_1}^{(2)} D^{\alpha_1}_{(1)} D^{\alpha_2}_{(2)} f_1\|_{L^{p_1}L^{q_1}}^\tau \, \|S_{k_1}^{(1)} S_{m_1}^{(2)} f_2\|_{L^{p_2}L^{q_2}}^\tau \\
& \qquad \|  \Delta_{k_3}^{(1)} \Delta_{m_3}^{(2)} f_3\|_{L^{p_3}L^{q_3}}^\tau \, \| D^{\gamma_1}_{(1)} D^{\gamma_2}_{(2)} \big( S_{k_1}^{(1)} S_{m_1}^{(2)} f_4 \cdot S_{k_1}^{(1)} S_{m_1}^{(2)} f_5 \big)\|_{L^{p_{4,5}}L^{q_{4,5}}}^\tau \\
\lesssim  & 2^{k_1(\beta_1-1) \tau} 2^{k_3 \tau} \, 2^{m_1(\beta_2-1) \tau} 2^{m_3 \tau}  \|\Delta_{k_1}^{(1)} \Delta_{m_1}^{(2)}  D^{\alpha_1}_{(1)} D^{\alpha_2}_{(2)}  f_1\|_{L^{p_1}L^{q_1}}^\tau \,  \|  \Delta_{k_3}^{(1)} \Delta_{m_3}^{(2)} f_3\|_{L^{p_3}L^{q_3}}^\tau \\
& \qquad \| f_2\|_{L^{p_2}L^{q_2}}^\tau \big(  \| D^{\gamma_1}_{(1)} D^{\gamma_2}_{(2)}  f_4\|_{L^{p_4}L^{q_4}} \,  \| f_5\|_{L^{p_5}L^{q_5}} + \ldots+ \| f_4\|_{L^{p_4}L^{q_4}} \,  \| D^{\gamma_1}_{(1)} D^{\gamma_2}_{(2)}  f_5\|_{L^{p_5}L^{q_5}} \big)^\tau
\end{align*}

The summation over $k_3< k_1, m_3< m_1$ will only affect the first line in the term above, and it can be estimated by
\begin{align*}
\sum_{k_1, m_1} \min \big(  &2^{k_1\beta_1 \tau} 2^{m_1\beta_2 \tau} \|D^{\alpha_1}_{(1)} D^{\alpha_2}_{(2)}  f_1\|_{\dot B^0_{p_1, \infty} \dot B^0_{q_1, \infty}}^\tau  \|f_3\|_{\dot B^0_{p_3, \infty} \dot B^0_{q_3, \infty}}^\tau , 2^{k_1\beta_1 \tau} 2^{ -m_1\epsilon_2 \tau} \|D^{\alpha_1}_{(1)} D^{\alpha_2}_{(2)}  f_1\|^\tau_{\dot B^0_{p_1, \infty} \dot B^{\beta_2}_{q_1, \infty}}  \|f_3\|^\tau_{\dot B^0_{p_3, \infty} \dot B^{\epsilon_2}_{p_3, \infty}}, \\
& 2^{-k_1\epsilon_1 \tau} 2^{m_1\beta_2 \tau} \|D^{\alpha_1}_{(1)} D^{\alpha_2}_{(2)} f_1\|^\tau_{\dot B^{\beta_1}_{p_1, \infty} \dot B^0_{q_1, \infty}}  \|f_3\|^\tau_{\dot B^{\epsilon_1}_{p_3, \infty} \dot B^0_{p_3, \infty}}, 2^{-k_1\epsilon_1 \tau} 2^{-m_1\epsilon_2 \tau} \|D^{\alpha_1}_{(1)} D^{\alpha_2}_{(2)}  f_1\|^\tau_{\dot B^{\beta_1}_{p_1, \infty} \dot B^{\beta_2}_{q_1, \infty}}  \|f_3\|^\tau_{\dot B^{\epsilon_1}_{p_3, \infty} \dot B^{\epsilon_2}_{q_3, \infty}} \big).
\end{align*}

The procedure described in treating $I_A^{(1)} \times I_B^{(2)}$ will eventually yield that this is majorized by
\begin{align*}
& \big(  \|D^{\alpha_1}_{(1)} D^{\alpha_2}_{(2)} f_1 \|_{\dot B^0_{p_1, \infty} \dot B^0_{q_1, \infty}}^\tau \|f_3\|_{\dot B^{\beta_1}_{p_3, \infty} \dot B^{\beta_2}_{q_3, \infty}}^\tau \big)^{{\epsilon_1 \over{\beta_1+\epsilon_1}} \cdot {\epsilon_2 \over {\beta_2+\epsilon_2}}}  \cdot\big( \|D^{\alpha_1}_{(1)} D^{\alpha_2}_{(2)} f_1 \|_{\dot B^{\beta_1}_{p_1, \infty} \dot B^0_{q_1, \infty}} \|f_3\|_{\dot B^0_{p_3, \infty} \dot B^{\beta_2}_{q_3, \infty}}^\tau\big)^{{\beta_1 \over{\beta_1 + \epsilon_1}} \cdot {\epsilon_2 \over {\beta_2+\epsilon_2}}} \\
& \cdot  \big(\|D^{\alpha_1}_{(1)} D^{\alpha_2}_{(2)} f_1 \|_{\dot B^{0}_{p_1, \infty} \dot B^{\beta_2}_{q_1, \infty}} \|f_3\|_{\dot B^{\beta_1}_{p_3, \infty} \dot B^{0}_{q_3, \infty}}^\tau\big)^{{\epsilon_1 \over{\beta_1+\epsilon_1}} \cdot {\beta_2 \over {\beta_2+\epsilon_2}}} \cdot \big( \|D^{\alpha_1}_{(1)} D^{\alpha_2}_{(2)} f_1 \|_{\dot B^{\beta_1}_{p_1, \infty} \dot B^{\beta_2}_{q_1, \infty}}  \|f_3\|_{\dot B^{0}_{p_3, \infty} \dot B^{0}_{q_3, \infty}}^\tau  \big)^{{\beta_1 \over{\beta_1 + \epsilon_1}} \cdot {\beta_2 \over {\beta_2+\epsilon_2}}}.
\end{align*}

\medskip

\item $I_A^{(1)} \times I_C^{(2)})$ We recall that the symbol $m_{I_A^{(1)}}(\xi_1, \ldots, \xi_5) \cdot m_{I_C^{(2)}}(\eta_1, \ldots, \eta_5)$ is
\begin{align*}
&\Big( \frac{|\xi_1+\ldots+\xi_5|^{\beta_1}- |\xi_1+\xi_2|^{\beta_1}}{\xi_3+\xi_4+\xi_5} \cdot \xi_3 \cdot |\xi_1+\xi_2|^{\alpha_1} |\xi_4+\xi_5|^{\gamma_1} \Big) \cdot \Big( \frac{|\eta_1+\eta_2|^{\alpha_2+\beta_2}- |\eta_1|^{\alpha_2+\beta_2}}{\eta_2} \cdot \eta_2 \cdot |\eta_4+\eta_5|^{\gamma_2} \Big)
\end{align*}
and thus the associated operator, obtained after a suitable regrouping of the terms, is
\begin{align*}
&T_{m_{I_A^{(1)}}, m_{I_C^{(2)}}}(f_1, \ldots, f_5)(x, y) = \int_{\BBR^{10}} m_{I_A^{(1)}}(\xi_1, \ldots, \xi_5) \cdot m_{I_C^{(2)}}(\eta_1, \ldots, \eta_5) \ii F(\Delta_{k_1}^{(1)} \Delta_{m_1}^{(2)} f_1 ) (\xi_1, \eta_1) \cdot \ii F(S_{k_1}^{(1)} \Delta_{m_2}^{(2)} f_2)(\xi_2, \eta_2)  \\
&\cdot  \ii F(\Delta_{k_3}^{(1)} S_{m_1}^{(2)} f_3 )(\xi_3, \eta_3) \cdot  \ii F( S_{k_1}^{(1)} S^{(2)}_{m_1} f_4)(\xi_4, \eta_4) \cdot  \ii F(S_{k_1}^{(1)} S^{(2)}_{m_1} f_5)(\xi_5, \eta_5) \cdot  e^{ 2 \pi ix \left(\xi_1+\ldots+\xi_5 \right)}  e^{ 2 \pi iy \left(\eta_1+\ldots+\eta_5 \right)} d \, \xi_1 \ldots d \, \xi_5 d \, \eta_1 \ldots d \, \eta_5.
\end{align*}

Thanks to the usual Fourier series decomposition, the commutator estimates and result on lower complexity flags localized in frequency, we are left with summing 
\begin{align*}
& \sum_{k_3< k_1} \sum_{m_2<m_1} 2^{k_1(\beta_1-1) \tau} 2^{k_3 \tau} 2^{k_1 \alpha_1 \tau}\, 2^{m_1(\alpha_2+\beta_2-1) \tau} 2^{m_2 \tau} \|\Delta_{k_1}^{(1)} \Delta_{m_1}^{(2)} f_1\|_{L^{p_1}L^{q_1}}^\tau \, \|S_{k_1}^{(1)} \Delta_{m_2}^{(2)} f_2\|_{L^{p_2}L^{q_2}}^\tau \\
& \qquad \qquad \|  \Delta_{k_3}^{(1)} S_{m_1}^{(2)} f_3\|_{L^{p_3}L^{q_3}}^\tau \, \| D^{\gamma_1}_{(1)} D^{\gamma_2}_{(2)} \big( S_{k_1}^{(1)} S_{m_1}^{(2)} f_4 \cdot S_{k_1}^{(1)} S_{m_1}^{(2)} f_5 \big)\|_{L^{p_{4,5}}L^{q_{4,5}}}^\tau \\
&\lesssim \sum_{k_3< k_1} \sum_{m_2<m_1} 2^{k_1(\beta_1-1) \tau} 2^{k_3 \tau} \, 2^{m_1(\alpha_2+\beta_2-1) \tau} 2^{m_2 \tau}  \|\Delta_{k_1}^{(1)} \Delta_{m_1}^{(2)} D^{\alpha_1}_{(1)} f_1\|_{L^{p_1}L^{q_1}}^\tau \, \| \Delta_{m_2}^{(2)} f_2\|_{L^{p_2}L^{q_2}}^\tau  \|  \Delta_{k_3}^{(1)}  f_3\|_{L^{p_3}L^{q_3}}^\tau  \\
&\quad \cdot \Big( \|  D^{\gamma_1}_{(1)} D^{\gamma_2}_{(2)} f_4\|_{L^{p_4}L^{q_4}} \, \| f_5\|_{L^{p_5}L^{q_5}} + \ldots + \|f_4\|_{L^{p_4}L^{q_4}} \, \|  D^{\gamma_1}_{(1)} D^{\gamma_2}_{(2)} f_5\|_{L^{p_5}L^{q_5}}  \Big)^\tau.
\end{align*}

Setting aside the term involving  the $D^{\gamma_1}_{(1)} D^{\gamma_2}_{(2)}$ derivatives applied to (projections of) the functions $f_4$ and $f_5$, the term above is bounded by
\begin{align*}
\sum_{k_1, m_1} \min \big(  &2^{k_1\beta_1 \tau} 2^{m_1(\alpha_2+\beta_2) \tau} \|D^{\alpha_1}_{(1)} f_1\|_{\dot B^0_{p_1, \infty} \dot B^0_{q_1, \infty}}^\tau  \| f_2\|_{L^{p_2} \dot B^0_{q_2, \infty}}^\tau  \|f_3\|_{\dot B^0_{p_3, \infty} L^{q_3}}^\tau, \\
&2^{k_1\beta_1 \tau} 2^{ -m_1 \epsilon_2 \tau} \|D^{\alpha_1}_{(1)} f_1\|^\tau_{\dot B^0_{p_1, \infty} \dot B^{\alpha_2+\beta_2}_{q_1, \infty}}  \| f_2\|_{L^{p_2} \dot B^{\epsilon_2}_{q_2, \infty}}^\tau  \|f_3\|^\tau_{\dot B^0_{p_3, \infty} L^{q_3}}, \\
& 2^{-k_1\epsilon_1 \tau} 2^{m_1(\alpha_2+\beta_2) \tau} \|D^{\alpha_1}_{(1)} f_1\|^\tau_{\dot B^{\beta_1}_{p_1, \infty} \dot B^0_{q_1, \infty}}  \| f_2\|_{L^{p_2} \dot B^0_{q_2, \infty}}^\tau  \|f_3\|^\tau_{\dot B^{\epsilon_1}_{p_3, \infty} L^{q_3}} , \\
& 2^{-k_1\epsilon_1 \tau} 2^{-m_1\epsilon_2 \tau} \| D^{\alpha_1}_{(1)} f_1\|^\tau_{\dot B^{\beta_1}_{p_1, \infty} \dot B^{\alpha_2+\beta_2}_{q_1, \infty}}  \| f_2\|_{L^{p_2} \dot B^{\epsilon_2}_{q_2, \infty}}^\tau \|f_3\|^\tau_{\dot B^{\epsilon_1}_{p_3, \infty} L^{q_3}} \big).
\end{align*}

Eventually, we get
\begin{align*}
& \big(  \|D^{\alpha_1}_{(1)} f_1\|_{\dot B^0_{p_1, \infty} \dot B^0_{q_1, \infty}}^\tau  \| f_2\|_{L^{p_2} \dot B^{\alpha_2+\beta_2}_{q_2, \infty}}^\tau  \|f_3\|_{\dot B^{\beta_1}_{p_3, \infty} L^{q_3}}^\tau \big)^{{\epsilon_1 \over{\beta_1+\epsilon_1}} \cdot {\epsilon_2 \over {\beta_2+\epsilon_2}}}  \cdot\big( \|D^{\alpha_1}_{(1)} f_1\|_{\dot B^{\beta_1}_{p_1, \infty} \dot B^0_{q_1, \infty}}^\tau  \| f_2\|_{L^{p_2} \dot B^{\alpha_2+\beta_2}_{q_2, \infty}}^\tau  \|f_3\|_{\dot B^0_{p_3, \infty} L^{q_3}}^\tau\big)^{{\beta_1 \over{\beta_1 + \epsilon_1}} \cdot {\epsilon_2 \over {\beta_2+\epsilon_2}}} \\
& \cdot  \big( \|D^{\alpha_1}_{(1)} f_1\|_{\dot B^0_{p_1, \infty} \dot B^{\alpha_2+\beta_2}_{q_1, \infty}}^\tau  \| f_2\|_{L^{p_2} \dot B^0_{q_2, \infty}}^\tau  \|f_3\|_{\dot B^{\beta_1}_{p_3, \infty} L^{q_3}}^\tau\big)^{{\epsilon_1 \over{\beta_1+\epsilon_1}} \cdot {\beta_2 \over {\beta_2+\epsilon_2}}} \cdot \big( \|D^{\alpha_1}_{(1)} f_1\|_{\dot B^{\beta_1}_{p_1, \infty} \dot B^{\alpha_2+\beta_2}_{q_1, \infty}}^\tau  \| f_2\|_{L^{p_2} \dot B^0_{q_2, \infty}}^\tau  \|f_3\|_{\dot B^0_{p_3, \infty} L^{q_3}}^\tau  \big)^{{\beta_1 \over{\beta_1 + \epsilon_1}} \cdot {\beta_2 \over {\beta_2+\epsilon_2}}}.
\end{align*}

\medskip

\item $I_A^{(1)} \times I_D^{(2)})$ The symbol $m_{I_A^{(1)}}(\xi_1, \ldots, \xi_5) \cdot m_{I_D^{(2)}}(\eta_1, \ldots, \eta_5)$ is
\begin{align*}
&\Big( \frac{|\xi_1+\ldots+\xi_5|^{\beta_1}- |\xi_1+\xi_2|^{\beta_1}}{\xi_3+\xi_4+\xi_5} \cdot \xi_3 \cdot |\xi_1+\xi_2|^{\alpha_1} |\xi_4+\xi_5|^{\gamma_1} \Big) \cdot \big( |\eta_1|^{\alpha_2+\beta_2} \cdot |\eta_4+\eta_5|^{\gamma_2} \big) 
\end{align*}
and the operator associated to it, obtained after summing in $k_2, k_4, k_5 \ll k_1$ and $m_2, m_3, m_4, m_5 \ll m_1$, is
\begin{align*}
&T_{m_{I_A^{(1)}}, m_{I_D^{(2)}}}^{k_1, k_3; m_1}(f_1, \ldots, f_5)(x, y) = \int_{\BBR^{10}} m_{I_A^{(1)}}(\xi_1, \ldots, \xi_5) \cdot m_{I_D^{(2)}}(\eta_1, \ldots, \eta_5) \ii F(\Delta_{k_1}^{(1)} \Delta_{m_1}^{(2)} f_1 ) (\xi_1, \eta_1) \cdot \ii F(S_{k_1}^{(1)} S_{m_1}^{(2)} f_2)(\xi_2, \eta_2)  \\
&\cdot  \ii F(\Delta_{k_3}^{(1)} S_{m_1}^{(2)} f_3 )(\xi_3, \eta_3) \cdot  \ii F( S_{k_1}^{(1)} S^{(2)}_{m_1} f_4)(\xi_4, \eta_4) \cdot  \ii F(S_{k_1}^{(1)} S^{(2)}_{m_1} f_5)(\xi_5, \eta_5) \cdot  e^{ 2 \pi ix \left(\xi_1+\ldots+\xi_5 \right)}  e^{ 2 \pi iy \left(\eta_1+\ldots+\eta_5 \right)} d \, \xi_1 \ldots d \, \xi_5 d \, \eta_1 \ldots d \, \eta_5.
\end{align*}

In the second parameter, we need to switch the order of summation: each $S_{m_1}^{(2)} f_l$ will be written as 
\[
S_{m_1}^{(2)} f_l= f_l - \Delta_{\succ   m_1}^{(2)} f_l,\footnote{As in the one-parameter setting, $\displaystyle \Delta^{(j)}_{\succ m} f_l := \sum_{\tilde{m} > m-3} \Delta^{(j)}_{\tilde{m}} f_l$, and has the same properties as $\Delta_{\succ m} f_l$ defined in \eqref{projection_<>}.} 
\qquad \text{ for all } 2 \leq l \leq 5.
\]
Then $T_{m_{I_A^{(1)}}, m_{I_D^{(2)}}}^{k_1, k_3; m_1}(f_1, \ldots, f_5)$ becomes
\begin{align}
&\int_{\BBR^{10}} m_{I_A^{(1)}}(\xi_1, \ldots, \xi_5) \cdot m_{I_D^{(2)}}(\eta_1, \ldots, \eta_5) \ii F(\Delta_{k_1}^{(1)} \Delta_{m_1}^{(2)} f_1 ) (\xi_1, \eta_1) \cdot \ii F(S_{k_1}^{(1)}  f_2)(\xi_2, \eta_2) \label{eq:bi:param:I_AxI_D:1} \\
&\cdot  \ii F(\Delta_{k_3}^{(1)}  f_3 )(\xi_3, \eta_3) \cdot  \ii F( S_{k_1}^{(1)}f_4)(\xi_4, \eta_4) \cdot  \ii F(S_{k_1}^{(1)} )(\xi_5, \eta_5) \cdot  e^{ 2 \pi ix \left(\xi_1+\ldots+\xi_5 \right)}  e^{ 2 \pi iy \left(\eta_1+\ldots+\eta_5 \right)} d \, \xi \nonumber \\
& -  \int_{\BBR^{10}} m_{I_A^{(1)}}(\xi_1, \ldots, \xi_5) \cdot m_{I_D^{(2)}}(\eta_1, \ldots, \eta_5) \ii F(\Delta_{k_1}^{(1)} \Delta_{m_1}^{(2)} f_1 ) (\xi_1, \eta_1) \cdot \ii F(S_{k_1}^{(1)} \Delta_{\succ   m_1}^{(2)} f_2)(\xi_2, \eta_2) \label{eq:bi:param:I_AxI_D:2}  \\
&\cdot  \ii F(\Delta_{k_3}^{(1)} f_3 )(\xi_3, \eta_3) \cdot  \ii F( S_{k_1}^{(1)}  f_4)(\xi_4, \eta_4) \cdot  \ii F(S_{k_1}^{(1)}  f_5)(\xi_5, \eta_5) \cdot  e^{ 2 \pi ix \left(\xi_1+\ldots+\xi_5 \right)}  e^{ 2 \pi iy \left(\eta_1+\ldots+\eta_5 \right)} d \, \xi \nonumber \\
&  -  \int_{\BBR^{10}} m_{I_A^{(1)}}(\xi_1, \ldots, \xi_5) \cdot m_{I_D^{(2)}}(\eta_1, \ldots, \eta_5) \ii F(\Delta_{k_1}^{(1)} \Delta_{m_1}^{(2)} f_1 ) (\xi_1, \eta_1) \cdot \ii F(S_{k_1}^{(1)} f_2)(\xi_2, \eta_2) \label{eq:bi:param:I_AxI_D:3}  \\
&\cdot  \ii F(\Delta_{k_3}^{(1)} f_3 )(\xi_3, \eta_3) \cdot  \ii F( S_{k_1}^{(1)} \Delta_{\succ   m_1}^{(2)}  f_4)(\xi_4, \eta_4) \cdot  \ii F(S_{k_1}^{(1)}  f_5)(\xi_5, \eta_5) \cdot  e^{ 2 \pi ix \left(\xi_1+\ldots+\xi_5 \right)}  e^{ 2 \pi iy \left(\eta_1+\ldots+\eta_5 \right)} d \, \xi \nonumber \\
& + \text{   similar terms}. \nonumber
\end{align}

If we sum in $m_1$, the first term becomes 
\begin{align*}
&\int_{\BBR^{10}} m_{I_A^{(1)}}(\xi_1, \ldots, \xi_5) \cdot |\eta_4 + \eta_5|^{\gamma_2} \ii F(\Delta_{k_1}^{(1)} D^{\alpha_2+\beta_2}_{(2)} f_1 ) (\xi_1, \eta_1) \cdot \ii F(S_{k_1}^{(1)}  f_2)(\xi_2, \eta_2)  \cdot\ii F(\Delta_{k_3}^{(1)}  f_3 )(\xi_3, \eta_3) \\
& \qquad  \cdot  \ii F( S_{k_1}^{(1)}f_4)(\xi_4, \eta_4) \cdot  \ii F(S_{k_1}^{(1)} f_5)(\xi_5, \eta_5) \cdot  e^{ 2 \pi ix \left(\xi_1+\ldots+\xi_5 \right)}  e^{ 2 \pi iy \left(\eta_1+\ldots+\eta_5 \right)} d \, \xi.
\end{align*}
So besides the $D^{\gamma_2}_{(2)}$ derivatives acting on $f_4$ and $f_5$ (thus an object of lower complexity), we have fundamentally a one-parameter Leibniz rule. We use Fourier series for decomposing the symbol $m_{I_A^{(1)}}(\xi_1, \ldots, \xi_5)$, and the boundedness of \eqref{eq:bi:param:I_AxI_D:1} in $\| \cdot  \|_{L^pL^q}^\tau$ is a consequence of
\begin{align*}
&\sum_{k_3<k_1} 2^{k_1(\beta_1-1) \tau} 2^{k_3 \tau} 2^{k_1 \alpha_1 \tau} \|\Delta_{k_1}^{(1)}  D^{\alpha_2+\beta_2}_{(2)} f_1\|_{L^{p_1}L^{q_1}}^\tau \, \|S_{k_1}^{(1)}f_2\|_{L^{p_2}L^{q_2}}^\tau  \Delta_{k_3}^{(1)} f_3\|_{L^{p_3}L^{q_3}}^\tau \, \| D^{\gamma_1}_{(1)} D^{\gamma_2}_{(2)} \big( S_{k_1}^{(1)}  f_4 \cdot S_{k_1}^{(1)} f_5 \big)\|_{L^{p_{4,5}}L^{q_{4,5}}}^\tau \\
&\lesssim \sum_{k_3< k_1} 2^{k_1(\beta_1-1) \tau} 2^{k_3 \tau}  \|\Delta_{k_1}^{(1)} D^{\alpha_1}_{(1)}  D^{\alpha_2+\beta_2}_{(2)} f_1\|_{L^{p_1}L^{q_1}}^\tau \, \|f_2\|_{L^{p_2}L^{q_2}}^\tau  \|  \Delta_{k_3}^{(1)}  f_3\|_{L^{p_3}L^{q_3}}^\tau  \\
&\quad \cdot \Big( \|  D^{\gamma_1}_{(1)} D^{\gamma_2}_{(2)} f_4\|_{L^{p_4}L^{q_4}} \, \| f_5\|_{L^{p_5}L^{q_5}} + \ldots + \|f_4\|_{L^{p_4}L^{q_4}} \, \|  D^{\gamma_1}_{(1)} D^{\gamma_2}_{(2)} f_5\|_{L^{p_5}L^{q_5}}  \Big)^\tau
\end{align*}

Only the functions $f_1$ and $f_3$ will participate in the summation, and the optimization resembles the one-parameter case discussed in Section \ref{Bourgain-Li_hilow}:
\begin{align*}
&\sum_{k_1}  \min \big(  2^{k_1\beta_1 \tau}  \|D^{\alpha_1}_{(1)}D^{\alpha_2+\beta_2}_{(2)} f_1\|_{\dot B^0_{p_1, \infty} L^{q_1}}^\tau \|f_3\|_{\dot B^0_{p_3, \infty} L^{q_3}}^\tau  ,  2^{-k_1\epsilon_1 \tau} \|D^{\alpha_1}_{(1)} D^{\alpha_2+\beta_2}_{(2)} f_1\|^\tau_{\dot B^{\beta_1}_{p_1, \infty} L^{q_1}} \|f_3\|^\tau_{\dot B^{\epsilon_1}_{p_3, \infty} L^{q_3}}  \big) \\
&\lesssim  \big( \|D^{\alpha_1}_{(1)}D^{\alpha_2+\beta_2}_{(2)} f_1\|_{\dot B^0_{p_1, \infty} L^{q_1}}^\tau \|f_3\|_{\dot B^{\beta_1}_{p_3, \infty} L^{q_3}}^\tau  \big)^{\epsilon_1 \over{\beta_1+\epsilon_1}} \cdot \big(   \|D^{\alpha_1}_{(1)}D^{\alpha_2+\beta_2}_{(2)} f_1\|_{\dot B^{\beta_1}_{p_1, \infty} L^{q_1}}^\tau \|f_3\|_{\dot B^0_{p_3, \infty} L^{q_3}}^\tau  \big)^{\beta_1 \over{\beta_1+\epsilon_1}}.
\end{align*}

We return to \eqref{eq:bi:param:I_AxI_D:2}, for which the summation in $m_1$ is performed outside the $\|\cdot \|^\tau_{L^pL^q}$ quasi-norms:
\begin{align*}
&\sum_{k_3<k_1} \sum_{m_1} 2^{k_1(\beta_1-1) \tau} 2^{k_3 \tau} 2^{k_1 \alpha_1 \tau} 2^{m_1 (\alpha_2+\beta_2) \tau} \|\Delta_{k_1}^{(1)} \Delta_{m_1}^{(2)} f_1\|_{L^{p_1}L^{q_1}}^\tau \, \|S_{k_1}^{(1)} \Delta_{\succ  m_1}^{(2)} f_2\|_{L^{p_2}L^{q_2}}^\tau \\
& \qquad \qquad \|  \Delta_{k_3}^{(1)} f_3\|_{L^{p_3}L^{q_3}}^\tau \, \| D^{\gamma_1}_{(1)} D^{\gamma_2}_{(2)} \big( S_{k_1}^{(1)}  f_4 \cdot S_{k_1}^{(1)} f_5 \big)\|_{L^{p_{4,5}}L^{q_{4,5}}}^\tau \\
&\lesssim \sum_{k_3< k_1}  \sum_{m_1} 2^{k_1(\beta_1-1) \tau} 2^{k_3 \tau} 2^{m_1 (\alpha_2+\beta_2) \tau} \|\Delta_{k_1}^{(1)} D^{\alpha_1}_{(1)}  D^{\alpha_2+\beta_2}_{(2)} f_1\|_{L^{p_1}L^{q_1}}^\tau \, \|\Delta_{\succ  m_1}^{(2)} f_2\|_{L^{p_2}L^{q_2}}^\tau  \|  \Delta_{k_3}^{(1)}  f_3\|_{L^{p_3}L^{q_3}}^\tau  \\
&\quad \cdot \Big( \|  D^{\gamma_1}_{(1)} D^{\gamma_2}_{(2)} f_4\|_{L^{p_4}L^{q_4}} \, \| f_5\|_{L^{p_5}L^{q_5}} + \ldots + \|f_4\|_{L^{p_4}L^{q_4}} \, \|  D^{\gamma_1}_{(1)} D^{\gamma_2}_{(2)} f_5\|_{L^{p_5}L^{q_5}}  \Big)^\tau.
\end{align*}

Using the observation that
\[
\|\Delta_{\succ  m_1}^{(2)} f_2\|_{L^{p_2}L^{q_2}} \leq \min\big( \|f_2\|_{L^{p_2}L^{q_2}} , 2^{- m_1 \epsilon} \|f_2\|_{L^{p_2} \dot B^{\epsilon}_{q_2, \infty}} \big),
\]
we notice that the term on the first line is majorized by
{\fontsize{9}{9}\begin{align*}
\sum_{k_1, m_1}  \min \big(  &2^{k_1\beta_1 \tau} 2^{m_1(\alpha_2+\beta_2) \tau} \|f_1\|_{\dot B^0_{p_1, \infty} \dot B^0_{q_1, \infty}}^\tau  \| f_2\|_{L^{p_2} L^{q_2}}^\tau  \|f_3\|_{\dot B^0_{p_3, \infty} L^{q_3}}^\tau  , 2^{k_1\beta_1 \tau} 2^{ -m_1 \epsilon_2 \tau} \|f_1\|^\tau_{\dot B^0_{p_1, \infty} \dot B^{\alpha_2+\beta_2}_{q_1, \infty}}  \| f_2\|_{L^{p_2} \dot B^{\epsilon_2}_{q_2, \infty}}^\tau  \|f_3\|^\tau_{\dot B^0_{p_3, \infty} L^{q_3}}, \\
& 2^{-k_1\epsilon_1 \tau} 2^{m_1(\alpha_2+\beta_2) \tau} \|f_1\|^\tau_{\dot B^{\beta_1}_{p_1, \infty} \dot B^0_{q_1, \infty}}  \| f_2\|_{L^{p_2} L^{q_2}}^\tau  \|f_3\|^\tau_{\dot B^{\epsilon_1}_{p_3, \infty} L^{q_3}} , 2^{-k_1\epsilon_1 \tau} 2^{-m_1\epsilon_2 \tau} \|f_1\|^\tau_{\dot B^{\beta_1}_{p_1, \infty} \dot B^{\alpha_2+\beta_2}_{q_1, \infty}}  \| f_2\|_{L^{p_2} \dot B^{\epsilon_2}_{q_2, \infty}}^\tau \|f_3\|^\tau_{\dot B^{\epsilon_1}_{p_3, \infty} L^{q_3}} \big).
\end{align*}}

This will eventually produce the expected term -- the computations follow the usual pattern.

% be bounded by
%\begin{align*}
%& \big( \|f_1\|^\tau_{\dot B^{0}_{p_1, \infty} \dot B^{0}_{q_1, \infty}}  \| f_2\|_{L^{p_2} \dot B^{\alpha_2+\beta_2}_{q_2, \infty}}^\tau \|f_3\|^\tau_{\dot B^{\beta_1}_{p_3, \infty} L^{q_3}} \big)^{{\epsilon_1 \over{\beta_1+\epsilon_1}} \cdot {\epsilon_2 \over {\beta_2+\epsilon_2}}}  \cdot\big( \|f_1\|^\tau_{\dot B^{\beta_1}_{p_1, \infty} \dot B^{0}_{q_1, \infty}}  \| f_2\|_{L^{p_2} \dot B^{\alpha_2+\beta_2}_{q_2, \infty}}^\tau \|f_3\|^\tau_{\dot B^{0}_{p_3, \infty} L^{q_3}} \big)^{{\beta_1 \over{\beta_1 + \epsilon_1}} \cdot {\epsilon_2 \over {\beta_2+\epsilon_2}}} \\
%& \cdot  \big(\|f_1\|^\tau_{\dot B^{0}_{p_1, \infty} \dot B^{\alpha_2+\beta_2}_{q_1, \infty}}  \| f_2\|_{L^{p_2}L^{q_2}}^\tau \|f_3\|^\tau_{\dot B^{\beta_1}_{p_3, \infty} L^{q_3}} \big)^{{\epsilon_1 \over{\beta_1+\epsilon_1}} \cdot {\beta_2 \over {\beta_2+\epsilon_2}}} \cdot \big( \|f_1\|^\tau_{\dot B^{\beta_1}_{p_1, \infty} \dot B^{\alpha_2+\beta_2}_{q_1, \infty}}  \| f_2\|_{L^{p_2} L^{q_2}}^\tau \|f_3\|^\tau_{\dot B^{0}_{p_3, \infty} L^{q_3}} \big)^{{\beta_1 \over{\beta_1 + \epsilon_1}} \cdot {\beta_2 \over {\beta_2+\epsilon_2}}}.
%\end{align*}

Finally, we take a quick look at \eqref{eq:bi:param:I_AxI_D:3} as well, since the structure of the subtree associated to $D^{\gamma_2}_{(2)} (f_4 \cdot f_5)$ becomes a part of the analysis. Without loss of generality (since the other case is similar), we further restrict the symbol to the conical region $|\eta_4| \leq |\eta_5|$. Using the usual Fourier series expansion in the first variable in order to tensorize the symbol, we are led ultimately to estimating\footnote{Several similar terms need to be considered.}
{\fontsize{9}{9}
\begin{align*}
\sum_{\substack{k_3 \ll k_1 \\ m_1, m_5}} 2^{(\beta_1-1)k_1 \tau} 2^{k_3 \tau} 2^{m_1(\alpha_2+\beta_2)} \|\Delta_{k_1}^{(1)} \Delta_{m_1}^{(2)} D^{\alpha_1}_{(1)} f_1 \|_{L^{p_1}L^{q_1}}^\tau \|f_2\|_{L^{p_2}L^{q_2}}^\tau  \| \Delta_{k_3}^{(1)} f_3\|_{L^{p_3}L^{q_3}}^\tau  \|  D_{\gamma_2}^{(2)}\left(S_{k_1}^{(1)} \Delta_{\succ  m_1}^{(2)} \Delta_{\leq m_5}^{(2)} D^{\gamma_1}_{(1)} f_4 S_{k_1}^{(1)} \Delta_{m_5}^{(2)} f_5\right)\|_{L^{p_{4,5}}L^{q_{4,5}}}^\tau. & 
%= \sum_{\substack{k_3 \ll k_1 \\ m_1 \leq m_5}} 2^{(\beta_1-1)k_1 \tau} 2^{k_3 \tau} 2^{m_1(\alpha_2+\beta_2)} \|\Delta_{k_1}^{(1)} \Delta_{m_1}^{(2)} D^{\alpha_1}_{(1)} f_1 \|_{L^{p_1}L^{q_1}}^\tau \|f_2\|_{L^{p_2}L^{q_2}}^\tau  \| \Delta_{k_3}^{(1)} f_3\|_{L^{p_3}L^{q_3}}^\tau  \|  D_{\gamma_2}^{(2)}\left(S_{k_1}^{(1)} \Delta_{\succ  m_1}^{(2)} \Delta_{\leq m_5}^{(2)} D^{\gamma_1}_{(1)} f_4 S_{k_1}^{(1)} \Delta_{m_5}^{(2)} f_5\right)\|_{L^{p_{4,5}}L^{q_{4,5}}}^\tau & 
% \| S_{k_1}^{(1)} \Delta_{\succ  m_1}^{(2)} S_{k_5}^{(2)}   D^{\gamma_1}_{(1)} f_4\|_{L^{p_4}L^{q_4}}^\tau    \| S_{k_1}^{(1)}   \Delta_{m_5}^{(2)}  D^{\gamma_2}_{(2)} f_5\|_{L^{p_5}L^{q_5}}^\tau.  &
\end{align*}}
The only way $\Delta_{\succ  m_1}^{(2)} S_{k_5}^{(2)}$ is non-zero is  if $m_5 \succ m_1$ so that we can restrict the above expression to the sum over $m_5 \succ m_1$. We also invoke the localized Leibniz rule: 
\begin{equation*}
 \|  D_{\gamma_2}^{(2)}\left(S_{k_1}^{(1)} \Delta_{\succ  m_1}^{(2)} \Delta_{\leq m_5}^{(2)} D^{\gamma_1}_{(1)} f_4 \cdot S_{k_1}^{(1)} \Delta_{m_5}^{(2)} f_5\right)\|_{L^{p_{4,5}}L^{q_{4,5}}} \lesssim \| S_{k_1}^{(1)} \Delta_{\succ  m_1}^{(2)}    D^{\gamma_1}_{(1)} f_4\|_{L^{p_4}L^{q_4}}   \| S_{k_1}^{(1)}   \Delta_{m_5}^{(2)}  D^{\gamma_2}_{(2)} f_5\|_{L^{p_5}L^{q_5}}.
\end{equation*}

This means that the term above can be further estimated by
\begin{align*}
\sum_{\substack{k_3 \ll k_1 \\ m_5 \succ  m_1}} 2^{(\beta_1-1)k_1 \tau} 2^{k_3 \tau} 2^{m_1(\alpha_2+\beta_2)} \|\Delta_{k_1}^{(1)} \Delta_{m_1}^{(2)} D^{\alpha_1}_{(1)} f_1 \|_{L^{p_1}L^{q_1}}^\tau \|f_2\|_{L^{p_2}L^{q_2}}^\tau  \| \Delta_{k_3}^{(1)} f_3\|_{L^{p_3}L^{q_3}}^\tau  \|  D^{\gamma_1}_{(1)} f_4\|_{L^{p_4}L^{q_4}}^\tau    \|  \Delta_{m_5}^{(2)}  D^{\gamma_2}_{(2)} f_5\|_{L^{p_5}L^{q_5}}^\tau.
\end{align*}
We put $f_2$ and $f_4$ aside, and what is left will be bounded (through the usual process) by
{\fontsize{9}{9}\begin{align*}
& \big( \|D^{\alpha_1}_{(1)}  f_1\|^\tau_{\dot B^{0}_{p_1, \infty} \dot B^{\alpha_2+\beta_2}_{q_1, \infty}} \|f_3\|^\tau_{\dot B^{\beta_1}_{p_3, \infty} L^{q_3}}  \| D^{\gamma_2}_{(2)} f_5\|_{L^{p_5} \dot B^{0}_{q_5, \infty}}^\tau \big)^{{\epsilon_1 \over{\beta_1+\epsilon_1}} \cdot {\epsilon_2 \over {\beta_2+\epsilon_2}}}  \cdot\big( \|D^{\alpha_1}_{(1)}  f_1\|^\tau_{\dot B^{\beta_1}_{p_1, \infty} \dot B^{\alpha_2+\beta_2}_{q_1, \infty}} \|f_3\|^\tau_{\dot B^{0}_{p_3, \infty} L^{q_3}}  \| D^{\gamma_2}_{(2)} f_5\|_{L^{p_5} \dot B^{0}_{q_5, \infty}}^\tau  \big)^{{\beta_1 \over{\beta_1 + \epsilon_1}} \cdot {\epsilon_2 \over {\beta_2+\epsilon_2}}} \\
& \cdot  \big(\|D^{\alpha_1}_{(1)}  f_1\|^\tau_{\dot B^{0}_{p_1, \infty} \dot B^{0}_{q_1, \infty}}  \|f_3\|^\tau_{\dot B^{\beta_1}_{p_3, \infty} L^{q_3}}    \| D^{\gamma_2}_{(2)} f_5\|_{L^{p_5} \dot B^{\alpha_2+\beta_2}_{q_5, \infty}}^\tau  \big)^{{\epsilon_1 \over{\beta_1+\epsilon_1}} \cdot {\beta_2 \over {\beta_2+\epsilon_2}}} \cdot \big( \|D^{\alpha_1}_{(1)}  f_1\|^\tau_{\dot B^{\beta_1}_{p_1, \infty} \dot B^{0}_{q_1, \infty}} \|f_3\|^\tau_{\dot B^{0}_{p_3, \infty} L^{q_3}}  \| D^{\gamma_2}_{(2)} f_5\|_{L^{p_5} \dot B^{\alpha_2+\beta_2}_{q_5, \infty}}^\tau  \big)^{{\beta_1 \over{\beta_1 + \epsilon_1}} \cdot {\beta_2 \over {\beta_2+\epsilon_2}}}.
\end{align*}}
The remaining terms  can be treated in a similar way.

\medskip
\item $I_B^{(1)} \times I_A^{(2)})$ The case $I_B^{(1)} \times I_A^{(2)}$ is symmetric to $I_A^{(1)} \times I_B^{(2)}$.
\medskip
\item $I_B^{(1)}\times I_B^{(2)})$ The frequency symbol $m_{I_B^{(1)}}(\xi_1, \ldots, \xi_5) \cdot m_{I_B^{(2)}}(\eta_1, \ldots, \eta_5)$ is described by
\begin{align*}
\Big( \frac{|\xi_1+\ldots+\xi_5|^{\beta_1}- |\xi_1+\xi_2|^{\beta_1}}{\xi_3+\xi_4+\xi_5} \cdot (\xi_4+\xi_5) \cdot |\xi_1+\xi_2|^{\alpha_1} |\xi_4+\xi_5|^{\gamma_1} \Big) \cdot \Big( \frac{|\eta_1+\ldots+\eta_5|^{\beta_2}- |\eta_1+\eta_2|^{\beta_2}}{\eta_3+\eta_4+\eta_5} \cdot (\eta_4+\eta_5) \cdot |\eta_1+\eta_2|^{\alpha_2} |\eta_4+\eta_5|^{\gamma_2} \Big). 
\end{align*}

The ``low scales'' correspond to $\xi_4+\xi_5$ in the first parameter and to $\eta_4+\eta_5$ in the second one; for the purpose of deciding which functions will be involved in the optimization part, we need to decide which of $|\xi_4|$ and $|\xi_5|$ is larger (similarly for $|\eta_4|$ and $|\eta_5|$). We assume that $|\xi_5| \leq |\xi_4|$ and $|\eta_4| \leq |\eta_5|$, which is one of the more convoluted situations. 

The associated operator of interest is
\begin{align*}
&T_{m_{I_B^{(1)}}, m_{I_B^{(2)}}}^{k_1, k_4; m_1, m_5}(f_1, \ldots, f_5)(x, y) = \int_{\BBR^{10}} m_{I_B^{(1)}}(\xi_1, \ldots, \xi_5) \cdot m_{I_B^{(2)}}(\eta_1, \ldots, \eta_5) \ii F(\Delta_{k_1}^{(1)} \Delta_{m_1}^{(2)} f_1 ) (\xi_1, \eta_1) \cdot \ii F(S_{k_1}^{(1)} S_{m_1}^{(2)} f_2)(\xi_2, \eta_2)  \\
&\cdot  \ii F(S_{k_1}^{(1)} S_{m_1}^{(2)} f_3 )(\xi_3, \eta_3) \cdot  \ii F( \Delta_{k_4}^{(1)} \Delta^{(2)}_{\leq m_5} f_4)(\xi_4, \eta_4) \cdot  \ii F(\Delta_{\leq k_4}^{(1)} \Delta^{(2)}_{m_5} f_5)(\xi_5, \eta_5) \cdot  e^{ 2 \pi ix \left(\xi_1+\ldots+\xi_5 \right)}  e^{ 2 \pi iy \left(\eta_1+\ldots+\eta_5 \right)} d \, \xi_1 \ldots d \, \xi_5 d \, \eta_1 \ldots d \, \eta_5.
\end{align*}

The usual double Fourier series decomposition of the frequency-localized commutator symbols $m_{c_{\beta_1}}(\xi_1+\xi_2, \xi_3+\xi_4+\xi_5)$ and $m_{c_{\beta_2}}(\eta_1+\eta_2, \eta_3+\eta_4+\eta_5)$ allows to reduce the problem concerning the boundedness of $T_{m_{I_B^{(1)}}, m_{I_B^{(2)}}}^{k_1, k_4; m_1, m_5}(f_1, \ldots, f_5)$ to subtrees . Using results concerning bi-parameter paraproducts (flags of lower complexity), we deduce
\begin{align*}
&\|T_{m_{I_B^{(1)}}, m_{I_B^{(2)}}}^{k_1, k_4; m_1, m_5} (f_1, \ldots, f_5)\|_{L^pL^q}^\tau \lesssim 2^{k_1(\beta_1-1) \tau} 2^{k_4 \tau} 2^{k_1 \alpha_1 \tau} 2^{k_4 \gamma_1 \tau} 2^{m_1(\beta_2-1) \tau} 2^{m_5 \tau} 2^{m_1 \alpha_2 \tau} 2^{m_5 \gamma_2 \tau} \\
& \qquad \cdot \| \Delta_{k_1}^{(1)} \Delta_{m_1}^{(2)} f_1  \|_{L^{p_1} L^{q_1}}^\tau \,  \| S_{k_1}^{(1)} S_{m_1}^{(2)} f_2 \|_{L^{p_2} L^{q_2}}^\tau \,  \| S_{k_1}^{(1)} S_{m_1}^{(2)}  f_3 \|_{L^{p_3} L^{q_3}}^\tau \,  \| \Delta_{k_4}^{(1)} \Delta^{(2)}_{\leq m_5} f_4 \|_{L^{p_4} L^{q_4}}^\tau \, \| \Delta_{\leq k_4}^{(1)} \Delta^{(2)}_{m_5} f_5 \|_{L^{p_5} L^{q_5}}^\tau \\
&\lesssim 2^{k_1(\beta_1-1) \tau} 2^{k_4 \tau}  2^{m_1(\beta_2-1) \tau} 2^{m_5 \tau}  \| \Delta_{k_1}^{(1)} \Delta_{m_1}^{(2)} D^{\alpha_1}_{(1)} D^{\alpha_2}_{(2)} f_1  \|_{L^{p_1} L^{q_1}}^\tau \,  \|f_2 \|_{L^{p_2} L^{q_2}}^\tau \,  \| f_3 \|_{L^{p_3} L^{q_3}}^\tau \,  \| \Delta_{k_4}^{(1)}  D^{\gamma_1}_{(1)} f_4 \|_{L^{p_4} L^{q_4}}^\tau \, \| \Delta^{(2)}_{m_5} D^{\gamma_2}_{(1)} f_5 \|_{L^{p_5} L^{q_5}}^\tau.
\end{align*}

Summing now in $k_4 < k_1$, $m_5<m_1$, we obtain the desired upper bound.
\medskip
\item $I_B^{(1)} \times I_C^{(2)})$ The symbol symbol $m_{I_B^{(1)}}(\xi_1, \ldots, \xi_5) \cdot m_{I_C^{(2)}}(\eta_1, \ldots, \eta_5)$ is
\begin{align*}
& \Big( \frac{|\xi_1+\ldots+\xi_5|^{\beta_1}- |\xi_1+\xi_2|^{\beta_1}}{\xi_3+\xi_4+\xi_5} \cdot (\xi_4+\xi_5) \cdot |\xi_1+\xi_2|^{\alpha_1} |\xi_4+\xi_5|^{\gamma_1} \Big)  \cdot \Big( \frac{|\eta_1+\eta_2|^{\alpha_2+\beta_2}- |\eta_1|^{\alpha_2+\beta_2}}{\eta_2} \cdot \eta_2 \cdot |\eta_4+\eta_5|^{\gamma_2} \Big).
\end{align*}
We want to study the multiplier
\begin{align*}
&T_{m_{I_B^{(1)}}, m_{I_C^{(2)}}}^{k_1, k_4; m_1, m_2}(f_1, \ldots, f_5)(x, y) = \int_{\BBR^{10}} m_{I_B^{(1)}}(\xi_1, \ldots, \xi_5) \cdot m_{I_C^{(2)}}(\eta_1, \ldots, \eta_5) \ii F(\Delta_{k_1}^{(1)} \Delta_{m_1}^{(2)} f_1 ) (\xi_1, \eta_1) \cdot \ii F(S_{k_1}^{(1)} \Delta_{m_2}^{(2)} f_2)(\xi_2, \eta_2)  \\
&\cdot  \ii F(S_{k_1}^{(1)} S_{m_1}^{(2)} f_3 )(\xi_3, \eta_3) \cdot  \ii F( \Delta_{k_4}^{(1)} S_{m_1}^{(2)} f_4)(\xi_4, \eta_4) \cdot  \ii F(\Delta_{\leq k_4}^{(1)} S_{m_1}^{(2)} f_5)(\xi_5, \eta_5) \cdot  e^{ 2 \pi ix \left(\xi_1+\ldots+\xi_5 \right)}  e^{ 2 \pi iy \left(\eta_1+\ldots+\eta_5 \right)} d \, \xi_1 \ldots d \, \xi_5 d \, \eta_1 \ldots d \, \eta_5.
\end{align*}
As before, we assume without loss of generality that $|\xi_5| \leq |\xi_4|$. After invoking the Fourier series decomposition, $T_{m_{I_B^{(1)}}, m_{I_C^{(2)}}}^{k_1, k_4; m_1, m_2}(f_1, \ldots, f_5)$ tensorizes as superpositions of the form
{\fontsize{8.5}{8.5}
\begin{align*}
&2^{k_1(\beta_1-1)} 2^{k_4} 2^{m_1(\alpha_2+\beta_2-1)} 2^{m_2} \, \big\|D^{\alpha_1}_{(1)} \left(\Delta_{k_1}^{(1)} \Delta_{m_1}^{(2)} f_1S_{k_1}^{(1)} \Delta_{m_2}^{(2)} f_2\right)\big\|_{L^{p_{1,2}}L^{q_{1,2}}}\|S_{k_1}^{(1)} S_{m_1}^{(2)} f_3\|_{L^{p_3}L^{q_3}} \big\|D^{\gamma_1}_{(1)}D^{\gamma_2}_{(2)} \left(  \Delta_{k_4}^{(1)}S_{m_1}^{(2)}  f_4  \Delta_{\leq k_4}^{(1)} S_{m_1}^{(2)} f_5   \right)\big\|_{L^{p_{4,5}}L^{q_{4,5}}}.
%& \qquad \cdot S_{k_1}^{(1)} S_{m_1}^{(2)} f_3(x, y) \cdot D^{\gamma_2}_{(2)} \big(  \Delta_{k_4}^{(1)}S_{m_1}^{(2)} D^{\gamma_1}_{(1)} f_4 \cdot S_{k_4}^{(1)} S_{m_1}^{(2)} f_5   \big)(x, y).
\end{align*}}

These imply that $\|I_B^{(1)} \times I_C^{(2)}\|_{L^p L^q}^\tau$ is bounded by
\begin{align*}
&\sum_{k_4< k_1} \sum_{m_2<m_1} 2^{k_1(\beta_1-1) \tau} 2^{k_4 \tau} 2^{m_1(\alpha_2+\beta_2-1) \tau} 2^{m_2 \tau} \| \Delta_{k_1}^{(1)} \Delta_{m_1}^{(2)} D^{\alpha_1}_{(1)} f_1  \|_{L^{p_1}L^{q_1}}^\tau  \| \Delta_{m_2}^{(2)} f_2 \|_{L^{p_2}L^{q_2}}^\tau  \| f_3  \|_{L^{p_3}L^{q_3}}^\tau  \\
& \qquad\big(   \|  \Delta_{k_4}^{(1)} D^{\gamma_1}_{(1)} D^{\gamma_2}_{(2)} f_4 \|_{L^{p_1}L^{q_1}}   \| f_5 \|_{L^{p_5}L^{q_5}} +  \|  \Delta_{k_4}^{(1)} D^{\gamma_1}_{(1)} f_4 \|_{L^{p_4}L^{q_4}}  \| D^{\gamma_2}_{(2)} f_5 \|_{L^{p_5}L^{q_5}} \big)^\tau.
\end{align*}

If $(F_4, F_5)$ denotes either of the couples $(D^{\gamma_1}_{(1)} D^{\gamma_2}_{(2)} f_4, f_5)$ or $(D^{\gamma_1}_{(1)} f_4, D^{\gamma_2}_{(2)} f_5)$, we are left with bounding 
{\fontsize{9}{9}\begin{align*}
\sum_{k_1, m_1} \min \big(  &2^{k_1\beta_1 \tau} 2^{m_1(\alpha_2+\beta_2) \tau} \|f_1\|_{\dot B^0_{p_1, \infty} \dot B^0_{q_1, \infty}}^\tau  \| f_2\|_{L^{p_2} \dot B^0_{q_2, \infty}}^\tau  \|F_4\|_{\dot B^0_{p_4, \infty} L^{q_4}}^\tau  , 2^{k_1\beta_1 \tau} 2^{ -m_1 \epsilon_2 \tau} \|f_1\|^\tau_{\dot B^0_{p_1, \infty} \dot B^{\alpha_2+\beta_2}_{q_1, \infty}}  \| f_2\|_{L^{p_2} \dot B^{\epsilon_2}_{q_2, \infty}}^\tau  \|F_4\|^\tau_{\dot B^0_{p_4, \infty} L^{q_4}}, \\
& 2^{-k_1\epsilon_1 \tau} 2^{m_1(\alpha_2+\beta_2) \tau} \|f_1\|^\tau_{\dot B^{\beta_1}_{p_1, \infty} \dot B^0_{q_1, \infty}}  \| f_2\|_{L^{p_2} \dot B^0_{q_2, \infty}}^\tau  \|F_4\|^\tau_{\dot B^{\epsilon_1}_{p_4, \infty} L^{q_4}} , 2^{-k_1\epsilon_1 \tau} 2^{-m_1\epsilon_2 \tau} \|f_1\|^\tau_{\dot B^{\beta_1}_{p_1, \infty} \dot B^{\alpha_2+\beta_2}_{q_1, \infty}}  \| f_2\|_{L^{p_2} \dot B^{\epsilon_2}_{q_2, \infty}}^\tau \|F_4\|^\tau_{\dot B^{\epsilon_1}_{p_4,\infty} L^{q_4}} \big).
\end{align*}}
Per usual, interpolation and regrouping of the terms produces a desired upper bound.

\medskip
\item $I_D^{(1)} \times I_D^{(2)})$  We want to estimate the operator 
\begin{align*}
&T_{m_{I_D^{(1)}}, m_{I_D^{(2)}}}^{k_1; m_1}(f_1, \ldots, f_5)(x, y) = \int_{\BBR^{10}} m_{I_D^{(1)}}(\xi_1, \ldots, \xi_5) \cdot m_{I_D^{(2)}}(\eta_1, \ldots, \eta_5) \ii F(\Delta_{k_1}^{(1)} \Delta_{m_1}^{(2)} f_1 ) (\xi_1, \eta_1) \cdot \ii F(S_{k_1}^{(1)} S_{m_1}^{(2)} f_2)(\xi_2, \eta_2)  \\
&\cdot  \ii F(S_{k_1}^{(1)} S_{m_1}^{(2)} f_3 )(\xi_3, \eta_3) \cdot  \ii F( S_{k_1}^{(1)} S_{m_1}^{(2)} f_4)(\xi_4, \eta_4) \cdot  \ii F(S_{k_1}^{(1)} S_{m_1}^{(2)} f_5)(\xi_5, \eta_5) \cdot  e^{ 2 \pi ix \left(\xi_1+\ldots+\xi_5 \right)}  e^{ 2 \pi iy \left(\eta_1+\ldots+\eta_5 \right)} d \, \xi_1 \ldots d \, \xi_5 d \, \eta_1 \ldots d \, \eta_5
\end{align*}
of symbol $m_{I_D^{(1)}}(\xi_1, \ldots, \xi_5) \cdot m_{I_D^{(2)}}(\eta_1, \ldots, \eta_5)$, which writes as
\begin{align*}
|\xi_1|^{\alpha_1+\beta_1} \cdot |\xi_4+\xi_5|^{\gamma_1} \cdot |\eta_1|^{\alpha_2+\beta_2} \cdot |\eta_4+\eta_5|^{\gamma_2}. 
\end{align*}

In fact, $T_{m_{I_D^{(1)}}, m_{I_D^{(2)}}}^{k_1; m_1}(f_1, \ldots, f_5)$ is equal to 
\begin{equation}
\label{eq:I_D;I_D}
\Delta_{k_1}^{(1)} \Delta_{m_1}^{(2)}  D^{\alpha_1+\beta_1}_{(1)} D^{\alpha_2+\beta_2}_{(2)} f_1(x, y) \cdot S_{k_1}^{(1)} S_{m_1}^{(2)} f_2(x, y)\cdot S_{k_1}^{(1)} S_{m_1}^{(2)} f_3(x, y ) \cdot D^{\gamma_1}_{(1)} D^{\gamma_2}_{(2)} \big( S_{k_1}^{(1)}S_{m_1}^{(2)}  f_4  \cdot S_{k_1}^{(1)} S_{m_1}^{(2)} f_5 \big)(x, y).
\end{equation}

Next, we write every $S_{k_1}^{(1)}S_{m_1}^{(2)}  F$ as
\[
S_{k_1}^{(1)}S_{m_1}^{(2)} F= F - \Delta_{\succ  k_1}^{(1)}F - \Delta_{\succ  m_1}^{(2)}  F +  \Delta_{\succ  k_1}^{(1)}  \Delta_{\succ  m_1}^{(2)} F,
\]
for $F$ being any of the functions $f_2, f_3, f_4$ or $f_5$. When plugging this in \eqref{eq:I_D;I_D}, we obtain four types of terms.
\begin{itemize}
\item[-] The first will simply produce 
\begin{equation}
\label{eq:ID:ID:no:proj}
\Delta_{k_1}^{(1)} \Delta_{m_1}^{(2)}  D^{\alpha_1+\beta_1}_{(1)} D^{\alpha_2+\beta_2}_{(2)} f_1(x, y) \cdot f_2(x, y)\cdot f_3(x, y ) \cdot D^{\gamma_1}_{(1)} D^{\gamma_2}_{(2)} \big( f_4  \cdot   f_5 \big)(x, y)
\end{equation}
and the summation in $k_1$ and $m_1$ -- that needs to be performed before taking the $\| \cdot \|_{L^pL^q}$ norms -- yields
\[
D^{\alpha_1+\beta_1}_{(1)} D^{\alpha_2+\beta_2}_{(2)} f_1(x, y) \cdot f_2(x, y)\cdot f_3(x, y ) \cdot D^{\gamma_1}_{(1)} D^{\gamma_2}_{(2)} \big( f_4  \cdot  f_5 \big)(x, y),
\]
for which we invoke a mixed-norm H\"older's inequality.
\item[-] There is at most one function ``hit'' by the $\Delta_{\succ  k_1}^{(1)}$ projection, and all the functions are unaffected by $\Delta_{\succ  m_1}^{(2)}$: for example
\[
\Delta_{k_1}^{(1)} \Delta_{m_1}^{(2)}  D^{\alpha_1+\beta_1}_{(1)} D^{\alpha_2+\beta_2}_{(2)} f_1(x, y) \Delta_{\succ  k_1}^{(1)} f_2(x, y)\cdot f_3(x, y ) \cdot D^{\gamma_1}_{(1)} D^{\gamma_2}_{(2)} \big( f_4  \cdot f_5 \big)(x, y)
\]
or the more involved 
\[
\Delta_{k_1}^{(1)} \Delta_{m_1}^{(2)}  D^{\alpha_1+\beta_1}_{(1)} D^{\alpha_2+\beta_2}_{(2)} f_1(x, y) f_2(x, y)\cdot f_3(x, y ) \cdot D^{\gamma_1}_{(1)} D^{\gamma_2}_{(2)} \big( f_4  \cdot \Delta_{\succ  k_1}^{(1)}  f_5 \big)(x, y)
\]
which requires a further cone decomposition in the subtree corresponding to $D^{\gamma_1}_{(1)} D^{\gamma_2}_{(2)} \big( f_4  \cdot f_5 \big)$.

In either case, we first sum in $m_1$ to obtain the ``full function'' $\Delta_{k_1}^{(1)} D^{\alpha_1+\beta_1}_{(1)} D^{\alpha_2+\beta_2}_{(2)} f_1(x, y)$, and from there on we continue as in the one-parameter situation $I_D$.

\item[-] None of the functions $f_2, f_3, f_4, f_5$ are affected by $\Delta_{\succ  k_1}^{(1)}$, but at most one of them is hit by $\Delta_{\succ  m_1}^{(2)}$; this situation is symmetric to the previous one.

\item[-] At least one of the functions is hit by $\Delta_{\succ  k_1}^{(1)}$, and at least one (possibly a different one) by $\Delta_{\succ  m_1}^{(2)}$; say for example that we have 
\begin{equation}
\label{eq:bi:p:last}
 2^{k_1(\alpha_1+\beta_1)} 2^{m_1(\alpha_2+\beta_2)}  \Delta_{k_1}^{(1)} \Delta_{m_1}^{(2)} f_1(x, y) f_2(x, y)\cdot \Delta_{\succ  m_1}^{(2)} f_3(x, y ) \cdot D^{\gamma_1}_{(1)} D^{\gamma_2}_{(2)} \big(\Delta_{\succ  k_1}^{(1)}   f_4  \cdot f_5 \big)(x, y).
\end{equation}
Then we can write $\Delta_{\succ  m_1}^{(2)} f_3(x, y )$ as $\sum_{m_3} \Delta_{m_3}^{(2)} \Delta_{\succ  m_1}^{(2)} f_3(x, y)$ and notice that the only non-zero terms correspond to $m_3 \succ m_1 $. Similarly, we re-decompose 
\[
D^{\gamma_1}_{(1)} D^{\gamma_2}_{(2)} \big(\Delta_{\succ  k_1}^{(1)}   f_4  \cdot f_5 \big)(x, y)= \sum_{k_4, k_5} D^{\gamma_1}_{(1)} D^{\gamma_2}_{(2)} \big(\Delta_{\succ  k_1}^{(1)} \Delta_{k_4}^{(1)}  f_4  \cdot  \Delta_{k_5}^{(1)}  f_5 \big)(x, y),
\]
which we further restrict to the region $|\xi_5| \leq |\xi_4|$; then we need to sum 
\[
\sum_{k_4 \geq  k_5} D^{\gamma_1}_{(1)} D^{\gamma_2}_{(2)} \big(\Delta_{\succ  k_1}^{(1)} \Delta_{k_4}^{(1)}  f_4  \cdot  \Delta_{k_5}^{(1)}  f_5 \big)(x, y),
\]
which can be reduced to 
\[
\sum_{\substack{k_4: k_4 \succ k_1}} D^{\gamma_1}_{(1)} D^{\gamma_2}_{(2)} \big(\Delta_{\succ  k_1}^{(1)} \Delta_{k_4}^{(1)}  f_4  \cdot  S_{k_4}^{(1)}  f_5 \big)(x, y).
\]
All these produce 
\begin{align*}
\sum_{\substack{ k_4 \succ k_1 \\ m_3 \succ m_1}}  2^{k_1\beta_1 \tau} 2^{m_1 \beta_2 \tau} \| \Delta_{k_1}^{(1)} \Delta_{m_1}^{(2)} f_1\|_{L^{p_1}L^{q_1}}^\tau \|f_2\|_{L^{p_2}L^{q_2}}^\tau  \|\Delta_{m_3}^{(2)}  f_3\|_{L^{p_3}L^{q_3}}^\tau \| D^{\gamma_1}_{(1)} D^{\gamma_2}_{(2)} \big( \Delta_{k_4}^{(1)}  f_4 \cdot S_{k_4}^{(1)} f_5 \big)\|_{L^{p_{4,5}}L^{q_{4,5}}}^\tau,
\end{align*}
which is by now a usual estimate. 
\end{itemize}

\item The remaining terms,  although not perfectly identical to the ones discussed above, can be treated in a similar way; the details are left to the reader.

\end{itemize}

\medskip

\subsection{Study of ``diagonal'' conical regions}\label{sec:2:param:5:flag:several:max} 
Since the strategy used is the same, we will not repeat the computations. We emphasize however that this is the situation where the Fourier coefficients corresponding to $|\xi_1+\ldots+\xi_5|^{\beta_1}$ (or to $|\eta_1+\ldots+\eta_5|^{\beta_2}$), localized to suitable frequency intervals, will only have limited decay. This forces the conditions
\begin{equation}
\label{eq:mixed:bi:cond}
p>\max \big ({1 \over {1+\beta_1}}, {1 \over {1+\beta_2}} \big), \qquad q>{1 \over {1+\beta_2}}
\end{equation}
on the Lebesgue exponents of the target space, as we will shortly see.

To take an example, we assume that $|\xi_1| \sim |\xi_5|, |\xi_2|, |\xi_3|, |\xi_4| \ll |\xi_1|$, and in the second parameter $|\eta_1| \sim |\eta_2|, |\eta_3|, |\eta_4|, |\eta_5| \ll |\eta_1|$. Then \eqref{5lin:freq:rep:bi} will be (morally) replaced by
\begin{align}
\label{5lin:2comp:sc:bi}
&\sum_{k, m } \int_{\BBR^{10}}\big( |\xi_1+\ldots+\xi_5|^{\beta_1} |\xi_1+\xi_2|^{\alpha_1} |\xi_4+\xi_5|^{\gamma_1}  \big) \cdot \big( |\eta_1+\ldots+\eta_5|^{\beta_2} |\eta_1+\eta_2|^{\alpha_2} |\eta_4+\eta_5|^{\gamma_2}  \big) \ii F (\Delta_{k}^{(1)}\Delta_{m}^{(2)} f_1)(\xi_1, \eta_1) \\
& \cdot \ii F (S_k^{(1)}\Delta_{m}^{(2)} f_2)(\xi_2, \eta_2) \cdot \ldots \cdot  \ii F (\Delta_{k}^{(1)} S_{m}^{(2)} f_5)(\xi_5, \eta_5) e^{2 \pi i x(\xi_1+\ldots+\xi_5)} e^{2 \pi i y(\eta_1+\ldots+\eta_5)} d \xi_1 \ldots d\xi_5 d \eta_1 \ldots d\eta_5 \nonumber
\end{align}

Since 
\[
|\xi_1+\ldots+\xi_5| \leq C 2^{k}, \qquad |\eta_1+\ldots+\eta_5| \leq C 2^{m},
\]
as discussed in Section \ref{Bourgain-Li_hilow} -- equation \eqref{eq:Fourier:dec:diagonal}, we have that 
\[
|\xi_1+\ldots+\xi_5|^{\beta_1} \phi(2^{-k} (\xi_1+\ldots+\xi_5))= \sum_{L \in \BBZ} C_{L} 2^{k \beta_1} e^{2\pi i \frac{ L}{2^{k}} (\xi_1+\ldots+\xi_5)},
\]
where 
\begin{equation*}
\label{eq:limited:decay}
|C_{L}| \lesssim \frac{1}{(1+| L|)^{1+\beta_1}}.
\end{equation*}
Similarly,
 \[
|\eta_1+\ldots+\eta_5|^{\beta_2} \phi(2^{-m} (\eta_1+\ldots+\eta_5))= \sum_{\tilde L \in \BBZ} C_{\tilde L} 2^{m \beta_2} e^{2\pi i \frac{\tilde L}{2^{m}} (\eta_1+\ldots+\eta_5)},
\]
where 
\begin{equation*}
|C_{\tilde L}| \lesssim \frac{1}{(1+|\tilde L|)^{1+\beta_2}}.
\end{equation*}

So \eqref{5lin:2comp:sc:bi} becomes a sum over $k, m \in \BBZ$ of terms of the form
\begin{align} \label{5lin:cond:exp:1}
2^{k \beta_1} 2^{m \beta_2}  \sum_{L, \tilde L} C_L C_{\tilde L}  D^{\alpha_1}_{(1)}D^{\alpha_2}_{(2)} \big( \Delta_{k}^{(1)} \Delta_{m}^{(2)} f_1 \cdot S_{k}^{(1)} \Delta_{m}^{(2)} f_2   \big)(x+{ L \over {2^k}}, y +  {\tilde L \over {2^m}} ) S_{k}^{(1)} S_{m}^{(2)} f_3(x+{ L \over {2^k}}, y +  {\tilde L \over {2^m}} ) & \\
\cdot D^{\gamma_1}_{(1)}D^{\gamma_2}_{(2)} \big( S_{k}^{(1)} S_{m}^{(2)} f_4 \cdot \Delta_{k}^{(1)} S_{m}^{(2)} f_5   \big)(x+{ L \over {2^k}}, y +  {\tilde L \over {2^m}} ) \nonumber. &
\end{align}

Notice that in this context, we prefer to write
 \[
  D^{\alpha_1}_{(1)}D^{\alpha_2}_{(2)} \big( \Delta_{k, { L \over {2^k}}}^{(1)} \Delta_{m, {\tilde L \over {2^m}}}^{(2)} f_1 \cdot S_{k, { L \over {2^k}}}^{(1)} \Delta_{m, {\tilde L \over {2^m}}}^{(2)} f_2   \big)(x, y)
\]  
as
\[
D^{\alpha_1}_{(1)}D^{\alpha_2}_{(2)} \big( \Delta_{k}^{(1)} \Delta_{m}^{(2)} f_1 \cdot S_{k}^{(1)} \Delta_{m}^{(2)} f_2   \big)(x+{ L \over {2^k}}, y +  {\tilde L \over {2^m}} ).
\]
This allows us to first integrate in $y$ without taking into account the effect of the modulation in the $L$ parameter, which only acts in the first variable.

Once this clarified, we return to \eqref{5lin:cond:exp:1} and further write it as
\[
2^{k \beta_1} 2^{m \beta_2}   \sum_{L, \tilde L \in \BBZ} C_L C_{\tilde L} F^{k, m}_{L, \tilde L}(x,y):= 2^{k \beta_1} 2^{m \beta_2}   \sum_{L, \tilde L \in \BBZ} C_L C_{\tilde L} F^{k, m} (x+{ L \over {2^k}}, y +  {\tilde L \over {2^m}} ) .
\]

The delicate point about the constraints on $p$ and $q$ appears here. Previously in  Section \ref{sec:5:flag:bi:param:k_1:m_1}, the Fourier coefficients had arbitrary decay and we used the estimate
\[
\big\| \sum_{L, \tilde L \in \BBZ} C_L C_{\tilde L} F_{L, \tilde L}^{k, m}    \big\|_{L^pL^q}^\tau \leq \sum_{L, \tilde L \in \BBZ} |C_L|^\tau  |C_{\tilde L}|^\tau  \big\| F_{L, \tilde L}^{k, m}    \big\|_{L^pL^q}^\tau; 
\]
however, this would require that
\[
p, q>\max \big ({1 \over {1+\beta_1}}, {1 \over {1+\beta_2}} \big),
\]
which is stronger than the announced \eqref{eq:mixed:bi:cond}.

Instead, we use that, for $q_0 \leq \min(1, q)$, $\|\cdot  \|_q^{q_0}$ is subadditive, and thus 
\begin{align}
\label{trick:summation:bi}
\big\| \sum_{L, \tilde L \in \BBZ} C_L C_{\tilde L} F_{L, \tilde L}^{k, m}    \big\|_{L^pL^q}^\tau & \lesssim \sum_{L \in \BBZ} |C_L|^\tau \Big( \int_{\BBR} \Big( \sum_{\tilde L \in \BBZ} |C_{\tilde L}|^{q_0} \big( \int_{\BBR} \big| F^{k, m}(x+{ L \over {2^k}}, y +  {\tilde L \over {2^m}} )  \big|^q dy \big)^{{{q_0} \over q}}        \Big)^{{p \over {q_0}}}  dx  \Big)^{\tau \over p} \\
&\lesssim \sum_{L \in \BBZ} |C_L|^\tau \Big( \int_{\BBR} \Big( \sum_{\tilde L \in \BBZ} |C_{\tilde L}|^{q_0} \big( \int_{\BBR} \big| F^{k, m}(x, y)  \big|^q dy \big)^{{{q_0} \over q}}   \Big)^{{p \over {q_0}}}  dx  \Big)^{\tau \over p}. \nonumber
\end{align}

At this point, it is important to notice that  
\begin{align*}
\|F^{k, m}(x, \cdot) \|_{L^q_y} \lesssim \big \|  D^{\alpha_1}_{(1)}D^{\alpha_2}_{(2)} \big( \Delta_{k}^{(1)} \Delta_{m}^{(2)} f_1 \cdot S_{k}^{(1)} \Delta_{m}^{(2)} f_2   \big )(x, \cdot )  \big\|_{L^{q_{1,2}}_y} \, \big\| S_{k}^{(1)} S_{m}^{(2)} f_3   \big\|_{L^{q_3}_y} & \\ \big \| D^{\gamma_1}_{(1)}D^{\gamma_2}_{(2)} \big( S_{k}^{(1)} S_{m}^{(2)} f_4 \cdot \Delta_{k}^{(1)} S_{m}^{(2)} f_5   \big)(x, \cdot )  \big\|_{L^{q_{4,5}}_y} . &
\end{align*}

So provided that
\[
\sum_{L \in \BBZ} |C_L|^\tau < \infty, \qquad \sum_{\tilde L \in \BBZ} |C_{\tilde L}|^{q_0}< \infty,
\]
which amounts to conditions \eqref{eq:mixed:bi:cond} holding, we have that \eqref{5lin:2comp:sc:bi}, estimated in $\| \cdot \|_{L^pL^q}^\tau$ is bounded above by
\begin{align*}
\sum_{k, m} 2^{k \beta_1 \tau} 2^{m \beta_2 \tau}  \big \|  D^{\alpha_1}_{(1)}D^{\alpha_2}_{(2)} \big( \Delta_{k}^{(1)} \Delta_{m}^{(2)} f_1 \cdot S_{k}^{(1)} \Delta_{m}^{(2)} f_2   \big )\big\|_{L^{p_{1,2}}L^{q_{1,2}}}^\tau \, \big\| S_{k}^{(1)} S_{m}^{(2)} f_3   \big\|_{L^{p_3}L^{q_3}}^\tau & \\ \big \| D^{\gamma_1}_{(1)}D^{\gamma_2}_{(2)} \big( S_{k}^{(1)} S_{m}^{(2)} f_4 \cdot \Delta_{k}^{(1)} S_{m}^{(2)} f_5   \big)\big\|_{L^{p_{4,5}}L^{q_{4,5}}}^\tau. 
\end{align*}

From here on the argument follows the usual strategy: using the boundedness of the lower complexity flag paraproducts -- in this case the mixed norm estimates for frequency-localized bi-parameter paraproducts, we obtain that the expression above is further bounded by
\begin{align*}
\|f_3\|_{L^{p_3}L^{q_3}}^\tau  \|D^{\gamma_2}_{(2)} f_4\|_{L^{p_4}L^{q_4}}^\tau  \Big( \big\|  D^{\alpha_1}_{(1)}D^{\alpha_2}_{(2)} f_1   \big\|_{\dot B^0_{p_1, \infty} \dot B^0_{q_1, \infty}} \|f_2\|_{L^{p_2} \dot B^{\beta_2}_{q_1, \infty} } \big\|  D^{\gamma_1}_{(1)} f_5  \big\|_{\dot B^{\beta_1}_{p_5, \infty} L^{q_5} } \Big)^{\tau \over 4} & \\
\Big( \big\|  D^{\alpha_1}_{(1)}D^{\alpha_2}_{(2)} f_1   \big\|_{\dot B^{\beta_1}_{p_1, \infty} \dot B^0_{q_1, \infty}} \|f_2\|_{L^{p_2} \dot B^{\beta_2}_{q_1, \infty} } \big\|  D^{\gamma_1}_{(1)} f_5  \big\|_{\dot B^{0}_{p_5, \infty} L^{q_5} } \Big)^{\tau \over 4} \Big( \big\|  D^{\alpha_1}_{(1)}D^{\alpha_2}_{(2)} f_1   \big\|_{\dot B^0_{p_1, \infty} \dot B^{\beta_2}_{q_1, \infty}} \|f_2\|_{L^{p_2} \dot B^{0}_{q_1, \infty} } \big\|  D^{\gamma_1}_{(1)} f_5  \big\|_{\dot B^{\beta_1}_{p_5, \infty} L^{q_5} } \Big)^{\tau \over 4} &\\
\Big( \big\|  D^{\alpha_1}_{(1)}D^{\alpha_2}_{(2)} f_1   \big\|_{\dot B^{\beta_1}_{p_1, \infty} \dot B^{\beta_2}_{q_1, \infty}} \|f_2\|_{L^{p_2} \dot B^{0}_{q_1, \infty} } \big\|  D^{\gamma_1}_{(1)} f_5  \big\|_{\dot B^{0}_{p_5, \infty} L^{q_5} } \Big)^{\tau \over 4} &
\end{align*}
+ a similar term, in which the $D^{\gamma_2}_{(2)}$ derivatives acts on the function $f_5$.

The remaining cases can be treated in a similar way; the main idea to bear in mind is that a lower decay of the Fourier coefficients requires a regrouping of the information as in \eqref{trick:summation:bi}, so that the derivatives acting on the exterior variables will not affect the Lebesgue exponents corresponding to inner variables.

\section{Generic flag: an inductive argument}\label{generic_induction}

In this section we provide an inductive argument -- based on the complexity of the rooted tree -- that allows to prove the general result of Theorem \ref{thm:main}. Our approach integrates many of the ideas already presented: one starts by decomposing the frequency space into cones,\footnote{This is the usual paraproduct decomposition.} and then further into Whitney rectangles; if the cone is so that the output variable is away form the origin, the symbol smoothly restricted to the Whitney rectangles/cubes will be split as ``commutator" $+$ ``derivative acting on a lower number of functions''. Next, a  Fourier series expansion on each Whitney cube/rectangle is used in order to tensorize the information contained in the root symbol, obtaining in this way similar objects associated to rooted trees of lower complexity. From here, one proceeds as in Section \ref{Bourgain-Li_hilow}, although in case of multi-parameter flag Leibniz rules one needs to track more carefully the distribution of derivatives, encoded in various types of mixed Lebesgue and Besov norms.

We will address the various difficulties one at a time. First, in Section \ref{sec:one:param:generic:ind} we present the inductive argument in the one-parameter case, with emphasis on the splitting of the root symbol (depending on the type of cone we are looking at), and the necessary inductive statements that allow to reduce the complexity of the rooted tree. In Section \ref{subsection_bi_leibniz}, the bi-parameter case is presented; the process of splitting the root symbol, already used in the previous section, needs to be performed in each parameter separately, which increases the number of cases to be considered. Similarly, we will have a variety of necessary inductive statements, depending on the tree structures and the configurations of functions that appear in the summation over the scales\footnote{As in Sections \ref{Bourgain-Li_hilow}, \ref{sec:5lin:1param} and \ref{sec:2param:5linflag}, there will always be two functions involved in the summation over the scales step -- in each parameter.} step. Once acquainted with the splitting of the root symbols and the reduction of the tree's complexity when several parameters are involved, it remains to check that the end result -- now expressed as a geometric mean of mixed Lebesgue and Besov norms -- indeed corresponds to the desired distribution of derivatives. This last step is carried out in Section \ref{generic_optim_interp}.

In what follows, our analysis will be performed in dimension one; as discussed in Remark \ref{remark:higher:dim}, the employed strategy is easily adaptable to higher dimensions. In the one-parameter case, presented in Section \ref{sec:one:param:generic:ind}, we will assume the target space norm $\| \cdot\|_{L^r}$ to be subadditive: when $r<1$, the subadditivity is achieved by considering $\| \cdot\|_{L^r}^\tau$ with $\tau \leq \min(1, r)$. Similarly, in the mixed-norm multi-parameter case we would need to work with $\| \cdot\|_{L^{\vec r}}^\tau$ with $\tau \leq \min(1, r^1, \ldots, r^N)$ in order to obtain subadditivity; the more involved conditions on the Lebesgue exponents expressed in \eqref{cond:thm:2} of Theorem \ref{thm:main} require a more careful analysis, which was detailed in Section \ref{sec:2:param:5:flag:several:max}. In an attempt to remove unnecessary technicalities burdening the notation, we will also assume in Section \ref{subsection_bi_leibniz} that $\| \cdot\|_{L^{\vec r}}$ is subadditive.

\subsection{One-parameter flag Leibniz rule}\label{sec:one:param:generic:ind} 
We provide a proof for the one-parameter Leibniz rule corresponding to an arbitrary $n$-linear flag in dimension one using an inductive argument. In what follows, we use the notation introduced in Section \ref{strategy}.
Let $\mathcal{G}$ be a tree of arbitrary complexity. 
Due to the paraproduct decomposition described in Section \ref{sec:differences:Coifman-Meyer}, the frequency space is split into conical regions, which are generically of two types:
\begin{equation}
\label{eq:cone:ums}
R_{l_0}:=\{ (\xi_1, \ldots, \xi_n) : |\xi_{l_0}| \gg |\xi_l|  \text{    for all  } 1 \leq l \neq l_0 \leq n  \}
\end{equation}
and 
\begin{equation}
\label{eq:cone:nonums}
\tilde R_{l_1, l_2}:=\{ (\xi_1, \ldots, \xi_n) :  |\xi_{l_1}| \sim |\xi_{l_2}| \geq  |\xi_l|  \text{    for all  } 1 \leq l  \leq n  \}.
\end{equation}

%With the Littlewood-Paley decomposition on each function, the operator we want to bound writes as
%$$
%T_{\mathcal{G}}(f_1, \ldots, f_n) = \sum_{k_1,\ldots, k_n}T_{\mathcal{G}}(\Delta_{k_1}f_1, \ldots, \Delta_{k_n}f_n). 
%$$
We will introduce the maps $\mathfrak{M}$ and $\mathfrak{m}$ defined on the collection of conical regions:
\begin{align}\label{index:major}
\mathfrak{M}(R) := & \{1 \leq l \leq n: (\xi_1, \ldots, \xi_n) \in R,\ \ |\xi_l| \sim \max_{1 \leq l' \leq n} |\xi_{l'}| \}, \\
\mathfrak{m}(R) := & \{1 \leq l \leq n: (\xi_1, \ldots, \xi_n) \in R,\ \ |\xi_l| \ll \max_{1 \leq l' \leq n} |\xi_{l'}| \}, \nonumber
\end{align}
where $R$ is a conical region of the form \eqref{eq:cone:ums} or \eqref{eq:cone:nonums}.
The definition of conical regions thus implies that
$$
\mathcal{L}(\mathfrak{r}_{\mathcal{G}})= \{ 1, \ldots, n \} = \mathfrak{M}(R) \cup \mathfrak{m}(R).
$$

Since the frequency space can be decomposed into finitely many such regions, it suffices to derive the same bound for our multilinear expression localized on a fixed conical region in frequency. Let us denote by $T^R_{\mathcal{G}}$ the multilinear operator smoothly restricted to a cone $R$. Then for $k \in \BBZ$ and $L \in \BBR$, we define the projection operators $P_k$ and $P_{k, L}$\footnote{We recall that $P_{k, L}$ is simply a frequency modulation of $P_k$ -- see \eqref{eq:def:mod:proj}.}:
% that allow further localizations on Whitney cubes:
\begin{equation*}
P_k(l) := 
\begin{cases}
\Delta_k  \ \ \text{if} \ \  l \in \mathfrak{M}(R) \\
S_k  \ \  \ \text{if} \ \  l \in \mathfrak{m}(R),
\end{cases} \qquad \text{and} \qquad
P_{k,L}(l) := 
\begin{cases}
\Delta_{k,L}  \ \ \text{if} \ \  l \in \mathfrak{M}(R) \\
S_{k,L}  \ \ \ \text{if} \ \  l \in \mathfrak{m}(R).
\end{cases}
\end{equation*}
The projection operators themselves depend on the conical region $R$; this will be omitted from the notation, but it should be implicit in the analysis. 

If we look at the cone $R_{l_0}$ described in \eqref{eq:cone:ums}, we notice that $\mathfrak{M}(R_{l_0})= \{l_0\}$ and $\mathfrak{m}(R_{l_0})= \{1, \ldots, n\} \setminus \{l_0\}$. Moreover, it can be represented as a union of Whitney cubes:
\begin{equation}\label{local:whitney_cube_1max}
R_{l_0}= \bigcup_{k_{l_0}}R^{\pm}_{k_{l_0}}:= \bigcup_{k_{l_0}}  \{(\xi_1, \ldots, \xi_n): \xi_{l_0} \sim \pm 2^{k_{l_0}}, |\xi_l| \ll 2^{k_{l_0}} \ \ \text{for} \ \ l \neq l_0 \},
\end{equation}
which corresponds exactly to the projection operators applied to the leaves: 
$$\Delta_{k_{l_0}, \pm}f_{l_0} \ \ \text{and}  \ \ S_{k_{l_0}}f_l \ \ \text{for} \ \ l \neq l_0.$$ 
On the fixed conical region $R_{l_0}$ (as defined in \eqref{local:whitney_cube_1max}), we will refer to $k_{l_0}$ as $k_{\max}$ since it is naturally associated to the variable $\xi_{l_0}$ and the function $f_{l_0}$. Notice that $R_{l_0}$ becomes the union of all Whitney cubes at scale $k_{\max}$.

We can therefore denote by 
$$T_{\mathcal{G}}^{R}\left((P_{k_{\max}}f_l)_{1 \leq l \leq n}\right)$$
the multi-linear expression $T_{\mathcal{G}}(f_1, \ldots, f_n)$ localized on the union of Whitney cubes at scale $k_{\max}$ in the conical region $R$. %Furthermore, since a cone is fully determined by $k_{\max}$, $\mathfrak{M}(k_{\max})$ and $\mathfrak{m}(k_{\max})$,  
$T_{\mathcal{G}}(f_1, \ldots, f_n)$ restricted to the entire conical region can then be represented as
\begin{equation}\label{leibniz_cone}
\sum_{k_{\max} \in \mathbb{Z}}T^{R}_{\mathcal{G}}\left((P_{k_{\max}} f_l)_{1 \leq l \leq n}\right).
%T_{\mathcal{G}}\left((\Delta_{k_{\max}}f_{l})_{\substack{\\ l \in \mathfrak{M}}}, (S_{k_{\max}} f_l)_{\substack{ \\ l \in \mathfrak{m}}}\right).
\end{equation}

In our inductive process, operators associated to subtrees of $\mathcal{G}$ will play an important role; hence, for any vertex $v$, we denote by $\mathcal{G}^v$ the subtree of $\mathcal G$ rooted in $v$. We will need to consider also paraproduct decompositions on these subtrees. For any non-leaf vertex $v \in \mathcal{V}$, which becomes the root of the subtree $\mathcal{G}^v$, the new conical region associated to the subtree will be clarified and the corresponding Whitney cubes will be specified by $k_{\max}(v)$. 
 For simplicity of notation, we will use the abbreviation $k_{\max} = k_{\max}(\mathfrak{r}_{\mathcal{G}})$ to denote the maximal scale involved in the definition of $T_{\mathcal{G}}$ restricted to a certain cone. 

We observe that when the operator $T_{\mathcal{G}}$ is localized on a conical region $R$ such that $\mathfrak{M}(R) \cap \mathcal{L}(v) \neq \emptyset$ for some $v \in \mathcal{V}$, then $T_{\mathcal{G}^v}$ -- the operator associated to the subtree $\mathcal{G}^v$ -- is also automatically restricted to the conical region 
\begin{equation*}
R(v) := \{(\xi_l)_{l \in \mathcal{L}(v)}: |\xi_{l}| \gg |\xi_{l'}| \ \  \text{for} \ \ l \in  \mathcal{L}(v) \cap \mathfrak{M}(R),\ \  l' \in \mathcal{L}(v)\setminus \mathfrak{M}(R)  \}.
\end{equation*} 
Such localization on the subtree imposed by the conical decomposition for the original tree will be repetitively used in our inductive process.

The following notation will be useful in the formulation of induction. For any vertex $v \in \mathcal{V}$ which generates the subtree $\mathcal{G}^v$, define $\mathcal{V}^v$ to be the set of non-leaf vertices associated with $\mathcal{G}^v$. If $v$ is a leaf, then $\mathcal{G}^v = \{v\}$ and $\mathcal{V}^v = \emptyset$.

We recall that the properties \ref{deriv:distrib:i} and \ref{deriv:distrib:ii} from Section \ref{introduction} were important in describing the derivative distribution function $\delta: \mathcal V \to \mathcal L_{\mathcal G}$, which has to agree with the composition law. In order to better understand the behavior of $\delta$ when restricted to subtrees, we need to define first the collection of non-leaf vertices in the path of $f_l$ (for $1 \leq l \leq n$) to the root $v$ in the subtree $\mathcal{G}^v$:
\begin{align*}
\mathcal{V}^{v}_l:= &  \{w: w \in \mathcal{V}^{v}, l \in \mathcal{L}(w) \}, 
\end{align*}
and also its complement with respect to the subtree $\mathcal{G}^{v}$:
\begin{align*}
 (\mathcal{V}^v_l)^c:= &  \mathcal{V}^v \setminus \mathcal{V}^v_l.
\end{align*}
With some abuse of notation, if $v= f_l$ is a leaf, then $ \mathcal{V}^{v}_l = \emptyset$. 

We include a figure to illustrate the notation: in the first figure, the path highlighted in red represents the path from the vertex $v$ to the leaf $l$ and $\mathcal{V}^v_{l}$ indeed corresponds to the collection of the non-leaf vertices along the red path, namely $\{\bar{v}_1, \bar{v}_2, \ldots, \bar{v}_M\}$, as indicated in Figure \ref{fig:paths}(\subref{fig:path}). 

We define the common vertices shared by the paths of $f_{l_1}$ and $f_{l_2}$ by
\begin{align*}
 \mathcal{V}^{v}_{l_1,l_2} := & \mathcal{V}^{v}_{l_1} \cap \mathcal{V}^{v}_{l_2},
\end{align*}
which can be represented as an ordered set starting from the root $v$ and ending with the vertex denoted by $v^{l_1,l_2}$\footnote{The vertex $v^{l_1,l_2}$ represents the last common ancestor of $l_1$ and $l_2$ in $\mathcal{G}$.}:
$$
\mathcal{V}^v_{l_1,l_2} = \{v=: \tilde{v}_1, \tilde{v}_2, \ldots , v^{l_1,l_2} =: \tilde{v}_M \}.
$$
in the sense that the latter element is a direct descendant of the former, which further implies that
\begin{equation} \label{leaf_sequence}
\mathcal{L}(v^{l_1,l_2}) = \mathcal{L}(\tilde{v}_M) \subseteq \mathcal{L}(\tilde{v}_{M-1}) \subseteq \ldots \subseteq \mathcal{L}(\tilde{v}_2) \subseteq \mathcal{L}(\tilde{v}_1) = \mathcal{L}(v).
\end{equation}
Moreover, the definitions of the common path and of $v^{l_1,l_2}$ indicate the existence of two vertices $w^{l_1}, w^{l_2}$ such that 
%\\
%\noindent
%(i)
$w^{l_1}, w^{l_2}$ are direct descendants of $v^{l_1,l_2}$ with $w^{l_1} \neq w^{l_2}$ and the subtree stemming from $w^{l_i}$, denoted by $\mathcal{G}^{w^{l_i}}$, contains $l_i$ as its leaf (for $i=1, 2$). Equivalently, $l_1 \in \mathcal{L}(w^{l_1})$ and $l_2 \in \mathcal{L}(w^{l_2})$. 

We clarify the notation through Figure \ref{fig:paths}(\subref{fig:common:path}) -- the common path from $v$ to $l_1$ and to $l_2$ is marked in red and $\mathcal{V}^{v}_{l_1,l_2}$ is the collection of all the non-leaf vertices along this path. The subtrees $\mathcal{G}^{w^{l_1}}$ and $\mathcal{G}^{w^{l_1}}$ are highlighted in blue and green respectively. 

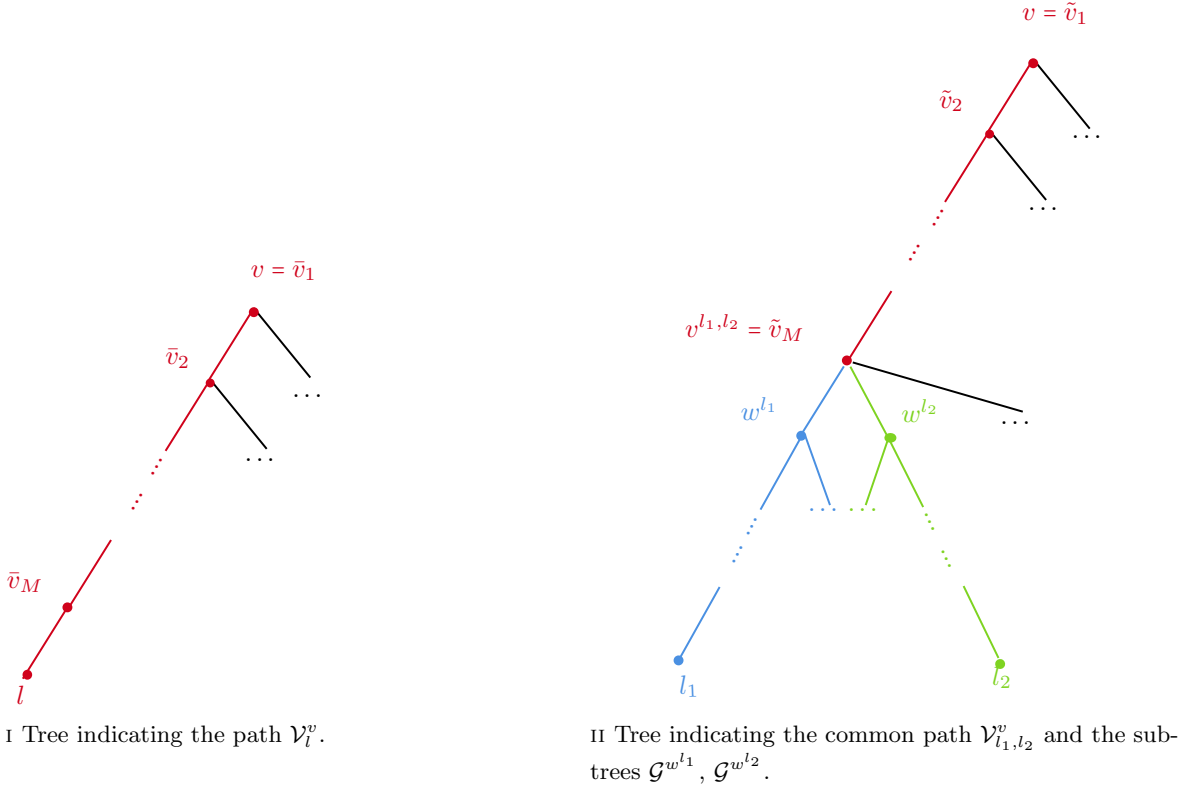
\begin{figure}[!htbp]
\begin{subfigure}[t]{.45\textwidth}
  \centering

\tikzset{every picture/.style={line width=0.75pt}} %set default line width to 0.75pt        

\begin{tikzpicture}[x=0.75pt,y=0.75pt,yscale=-1,xscale=1]
%uncomment if require: \path (0,1804); %set diagram left start at 0, and has height of 1804

%Straight Lines [id:da7533286274823682] 
\draw [color={rgb, 255:red, 208; green, 2; blue, 27 }  ,draw opacity=1 ]   (178,1288.4) -- (155.5,1324.4) ;
%Straight Lines [id:da8705290142494344] 
\draw [color={rgb, 255:red, 208; green, 2; blue, 27 }  ,draw opacity=1 ]   (156,1324.4) -- (133.5,1360.4) ;
%Straight Lines [id:da6701812732736683] 
\draw    (178,1288.4) -- (206.5,1323.4) ;
%Straight Lines [id:da3965368331032163] 
\draw    (156,1324.4) -- (184.5,1359.4) ;
%Straight Lines [id:da11323970713265696] 
\draw [color={rgb, 255:red, 208; green, 2; blue, 27 }  ,draw opacity=1 ]   (84,1439.4) -- (61.5,1475.4) ;
%Straight Lines [id:da06694590733397232] 
\draw [color={rgb, 255:red, 208; green, 2; blue, 27 }  ,draw opacity=1 ]   (106,1405.4) -- (83.5,1441.4) ;
%Shape: Circle [id:dp14693861124650143] 
\draw  [color={rgb, 255:red, 208; green, 2; blue, 27 }  ,draw opacity=1 ][fill={rgb, 255:red, 208; green, 2; blue, 27 }  ,fill opacity=1 ] (176,1290.4) .. controls (176,1289.3) and (176.9,1288.4) .. (178,1288.4) .. controls (179.1,1288.4) and (180,1289.3) .. (180,1290.4) .. controls (180,1291.5) and (179.1,1292.4) .. (178,1292.4) .. controls (176.9,1292.4) and (176,1291.5) .. (176,1290.4) -- cycle ;
%Shape: Circle [id:dp16497950303081355] 
\draw  [color={rgb, 255:red, 208; green, 2; blue, 27 }  ,draw opacity=1 ][fill={rgb, 255:red, 208; green, 2; blue, 27 }  ,fill opacity=1 ] (157.75,1326.15) .. controls (157.75,1325.18) and (156.97,1324.4) .. (156,1324.4) .. controls (155.03,1324.4) and (154.25,1325.18) .. (154.25,1326.15) .. controls (154.25,1327.12) and (155.03,1327.9) .. (156,1327.9) .. controls (156.97,1327.9) and (157.75,1327.12) .. (157.75,1326.15) -- cycle ;
%Shape: Circle [id:dp3089570604243318] 
\draw  [color={rgb, 255:red, 208; green, 2; blue, 27 }  ,draw opacity=1 ][fill={rgb, 255:red, 208; green, 2; blue, 27 }  ,fill opacity=1 ] (82,1439.4) .. controls (82,1438.3) and (82.9,1437.4) .. (84,1437.4) .. controls (85.1,1437.4) and (86,1438.3) .. (86,1439.4) .. controls (86,1440.5) and (85.1,1441.4) .. (84,1441.4) .. controls (82.9,1441.4) and (82,1440.5) .. (82,1439.4) -- cycle ;
%Shape: Ellipse [id:dp945842759067461] 
\draw  [color={rgb, 255:red, 208; green, 2; blue, 27 }  ,draw opacity=1 ][fill={rgb, 255:red, 208; green, 2; blue, 27 }  ,fill opacity=1 ] (61.86,1473.4) .. controls (61.86,1472.3) and (62.72,1471.4) .. (63.79,1471.4) .. controls (64.85,1471.4) and (65.71,1472.3) .. (65.71,1473.4) .. controls (65.71,1474.5) and (64.85,1475.4) .. (63.79,1475.4) .. controls (62.72,1475.4) and (61.86,1474.5) .. (61.86,1473.4) -- cycle ;

% Text Node
\draw (175,1263.4) node [anchor=north west][inner sep=0.75pt]  [color={rgb, 255:red, 208; green, 2; blue, 27 }  ,opacity=1 ]  {$v=\bar{v}_{1}$};
% Text Node
\draw (132,1307.4) node [anchor=north west][inner sep=0.75pt]  [color={rgb, 255:red, 208; green, 2; blue, 27 }  ,opacity=1 ]  {$\bar{v}_{2}$};
% Text Node
\draw (129.85,1360.87) node [anchor=north west][inner sep=0.75pt]  [color={rgb, 255:red, 208; green, 2; blue, 27 }  ,opacity=1 ,rotate=-30.04]  {$\vdots $};
% Text Node
\draw (119.85,1378.87) node [anchor=north west][inner sep=0.75pt]  [color={rgb, 255:red, 208; green, 2; blue, 27 }  ,opacity=1 ,rotate=-30.04]  {$\vdots $};
% Text Node
\draw (196,1330.4) node [anchor=north west][inner sep=0.75pt]    {$\dotsc $};
% Text Node
\draw (172,1362.4) node [anchor=north west][inner sep=0.75pt]    {$\dotsc $};
% Text Node
\draw (56.75,1477.4) node [anchor=north west][inner sep=0.75pt]  [color={rgb, 255:red, 208; green, 2; blue, 27 }  ,opacity=1 ]  {$l$};
% Text Node
\draw (52,1420.4) node [anchor=north west][inner sep=0.75pt]  [color={rgb, 255:red, 208; green, 2; blue, 27 }  ,opacity=1 ]  {$\bar{v}_{M}$};
\end{tikzpicture}
\caption{Tree indicating the path $\mathcal{V}^{v}_{l}$.}
\label{fig:path}
\end{subfigure}
 \hfill    
\begin{subfigure}[t]{.45\textwidth}
  \centering  
  \tikzset{every picture/.style={line width=0.75pt}} %set default line width to 0.75pt        

\begin{tikzpicture}[x=0.75pt,y=0.75pt,yscale=-1,xscale=1]
%uncomment if require: \path (0,1951); %set diagram left start at 0, and has height of 1951

%Straight Lines [id:da022740603705615148] 
\draw [color={rgb, 255:red, 208; green, 2; blue, 27 }  ,draw opacity=1 ]   (331,1547.4) -- (308.5,1583.4) ;
%Straight Lines [id:da8588710305825634] 
\draw [color={rgb, 255:red, 208; green, 2; blue, 27 }  ,draw opacity=1 ]   (309,1583.4) -- (286.5,1619.4) ;
%Straight Lines [id:da8240374752125333] 
\draw    (331,1547.4) -- (359.5,1582.4) ;
%Straight Lines [id:da355777835950318] 
\draw    (309,1583.4) -- (337.5,1618.4) ;
%Straight Lines [id:da8877323894090071] 
\draw [color={rgb, 255:red, 74; green, 144; blue, 226 }  ,draw opacity=1 ]   (172.84,1813.75) -- (152.12,1850.8) ;
%Straight Lines [id:da7167894444209647] 
\draw [color={rgb, 255:red, 208; green, 2; blue, 27 }  ,draw opacity=1 ]   (259.5,1664.4) -- (237,1700.4) ;
%Shape: Circle [id:dp8618385940750272] 
\draw  [color={rgb, 255:red, 208; green, 2; blue, 27 }  ,draw opacity=1 ][fill={rgb, 255:red, 208; green, 2; blue, 27 }  ,fill opacity=1 ] (329,1549.4) .. controls (329,1548.3) and (329.9,1547.4) .. (331,1547.4) .. controls (332.1,1547.4) and (333,1548.3) .. (333,1549.4) .. controls (333,1550.5) and (332.1,1551.4) .. (331,1551.4) .. controls (329.9,1551.4) and (329,1550.5) .. (329,1549.4) -- cycle ;
%Shape: Circle [id:dp8199985229830428] 
\draw  [color={rgb, 255:red, 208; green, 2; blue, 27 }  ,draw opacity=1 ][fill={rgb, 255:red, 208; green, 2; blue, 27 }  ,fill opacity=1 ] (310.75,1585.15) .. controls (310.75,1584.18) and (309.97,1583.4) .. (309,1583.4) .. controls (308.03,1583.4) and (307.25,1584.18) .. (307.25,1585.15) .. controls (307.25,1586.12) and (308.03,1586.9) .. (309,1586.9) .. controls (309.97,1586.9) and (310.75,1586.12) .. (310.75,1585.15) -- cycle ;
%Shape: Circle [id:dp07981806878293463] 
\draw  [color={rgb, 255:red, 208; green, 2; blue, 27 }  ,draw opacity=1 ][fill={rgb, 255:red, 208; green, 2; blue, 27 }  ,fill opacity=1 ] (235,1699.4) .. controls (235,1698.3) and (235.9,1697.4) .. (237,1697.4) .. controls (238.1,1697.4) and (239,1698.3) .. (239,1699.4) .. controls (239,1700.5) and (238.1,1701.4) .. (237,1701.4) .. controls (235.9,1701.4) and (235,1700.5) .. (235,1699.4) -- cycle ;
%Straight Lines [id:da6126810805786542] 
\draw [color={rgb, 255:red, 74; green, 144; blue, 226 }  ,draw opacity=1 ]   (235.5,1702.4) -- (213,1738.4) ;
%Shape: Circle [id:dp6390265108783926] 
\draw  [color={rgb, 255:red, 74; green, 144; blue, 226 }  ,draw opacity=1 ][fill={rgb, 255:red, 74; green, 144; blue, 226 }  ,fill opacity=1 ] (212,1737.4) .. controls (212,1736.3) and (212.9,1735.4) .. (214,1735.4) .. controls (215.1,1735.4) and (216,1736.3) .. (216,1737.4) .. controls (216,1738.5) and (215.1,1739.4) .. (214,1739.4) .. controls (212.9,1739.4) and (212,1738.5) .. (212,1737.4) -- cycle ;
%Straight Lines [id:da3523439895990118] 
\draw [color={rgb, 255:red, 126; green, 211; blue, 33 }  ,draw opacity=1 ]   (238.77,1702.55) -- (259.02,1740.6) ;
%Shape: Ellipse [id:dp47390484026368007] 
\draw  [color={rgb, 255:red, 126; green, 211; blue, 33 }  ,draw opacity=1 ][fill={rgb, 255:red, 126; green, 211; blue, 33 }  ,fill opacity=1 ] (257.49,1740.03) .. controls (256.36,1739.41) and (256.1,1738.2) .. (256.92,1737.33) .. controls (257.74,1736.47) and (259.32,1736.27) .. (260.45,1736.9) .. controls (261.58,1737.52) and (261.84,1738.73) .. (261.02,1739.6) .. controls (260.2,1740.46) and (258.62,1740.66) .. (257.49,1740.03) -- cycle ;
%Shape: Circle [id:dp16003213979472464] 
\draw  [color={rgb, 255:red, 74; green, 144; blue, 226 }  ,draw opacity=1 ][fill={rgb, 255:red, 74; green, 144; blue, 226 }  ,fill opacity=1 ] (150.13,1850.9) .. controls (150.07,1849.8) and (150.92,1848.86) .. (152.03,1848.8) .. controls (153.13,1848.75) and (154.07,1849.6) .. (154.12,1850.7) .. controls (154.17,1851.81) and (153.32,1852.75) .. (152.22,1852.8) .. controls (151.12,1852.85) and (150.18,1852) .. (150.13,1850.9) -- cycle ;
%Straight Lines [id:da06133736272487389] 
\draw [color={rgb, 255:red, 126; green, 211; blue, 33 }  ,draw opacity=1 ]   (295.88,1813.27) -- (313.48,1849.57) ;
%Shape: Circle [id:dp9470227723826441] 
\draw  [color={rgb, 255:red, 126; green, 211; blue, 33 }  ,draw opacity=1 ][fill={rgb, 255:red, 126; green, 211; blue, 33 }  ,fill opacity=1 ] (313.45,1854.42) .. controls (312.45,1853.94) and (312.04,1852.75) .. (312.52,1851.75) .. controls (313,1850.76) and (314.19,1850.34) .. (315.19,1850.82) .. controls (316.18,1851.3) and (316.6,1852.5) .. (316.12,1853.49) .. controls (315.64,1854.49) and (314.44,1854.9) .. (313.45,1854.42) -- cycle ;
%Straight Lines [id:da32774901771321163] 
\draw    (240,1700.4) -- (325.73,1725.25) ;
%Straight Lines [id:da345879165551885] 
\draw [color={rgb, 255:red, 74; green, 144; blue, 226 }  ,draw opacity=1 ]   (214.17,1737.64) -- (193.45,1774.69) ;
%Straight Lines [id:da5992142422009107] 
\draw [color={rgb, 255:red, 74; green, 144; blue, 226 }  ,draw opacity=1 ]   (216,1737.4) -- (228.5,1772.4) ;
%Straight Lines [id:da21471032467362183] 
\draw [color={rgb, 255:red, 126; green, 211; blue, 33 }  ,draw opacity=1 ]   (257.49,1740.03) -- (246.5,1772.4) ;
%Straight Lines [id:da5210178766411758] 
\draw [color={rgb, 255:red, 126; green, 211; blue, 33 }  ,draw opacity=1 ]   (259.02,1740.6) -- (275.5,1773.4) ;

% Text Node
\draw (324,1517.4) node [anchor=north west][inner sep=0.75pt]  [color={rgb, 255:red, 208; green, 2; blue, 27 }  ,opacity=1 ]  {$v=\tilde{v}_{1}$};
% Text Node
\draw (282,1561.4) node [anchor=north west][inner sep=0.75pt]  [color={rgb, 255:red, 208; green, 2; blue, 27 }  ,opacity=1 ]  {$\tilde{v}_{2}$};
% Text Node
\draw (283.85,1619.87) node [anchor=north west][inner sep=0.75pt]  [color={rgb, 255:red, 208; green, 2; blue, 27 }  ,opacity=1 ,rotate=-30.04]  {$\vdots $};
% Text Node
\draw (271.85,1637.87) node [anchor=north west][inner sep=0.75pt]  [color={rgb, 255:red, 208; green, 2; blue, 27 }  ,opacity=1 ,rotate=-30.04]  {$\vdots $};
% Text Node
\draw (349,1584.4) node [anchor=north west][inner sep=0.75pt]    {$\dotsc $};
% Text Node
\draw (313,1728.4) node [anchor=north west][inner sep=0.75pt]    {$\dotsc $};
% Text Node
\draw (143.68,1855.27) node [anchor=north west][inner sep=0.75pt]  [color={rgb, 255:red, 74; green, 144; blue, 226 }  ,opacity=1 ,rotate=-357.2]  {$ \begin{array}{l}
l_{1}\\
\end{array}$};
% Text Node
\draw (154,1674.4) node [anchor=north west][inner sep=0.75pt]  [color={rgb, 255:red, 208; green, 2; blue, 27 }  ,opacity=1 ]  {$v^{l_{1} ,l_{2}} =\tilde{v}_{M}$};
% Text Node
\draw (182,1716.4) node [anchor=north west][inner sep=0.75pt]  [color={rgb, 255:red, 74; green, 144; blue, 226 }  ,opacity=1 ]  {$w^{l_{1}}$};
% Text Node
\draw (263,1717.4) node [anchor=north west][inner sep=0.75pt]  [color={rgb, 255:red, 126; green, 211; blue, 33 }  ,opacity=1 ]  {$w^{l_{2}}$};
% Text Node
\draw (189.5,1775.59) node [anchor=north west][inner sep=0.75pt]  [color={rgb, 255:red, 74; green, 144; blue, 226 }  ,opacity=1 ,rotate=-27.24]  {$\vdots $};
% Text Node
\draw (182.34,1790.06) node [anchor=north west][inner sep=0.75pt]  [color={rgb, 255:red, 74; green, 144; blue, 226 }  ,opacity=1 ,rotate=-27.24]  {$\vdots $};
% Text Node
\draw (273.23,1776.18) node [anchor=north west][inner sep=0.75pt]  [color={rgb, 255:red, 126; green, 211; blue, 33 }  ,opacity=1 ,rotate=-335.18]  {$\vdots $};
% Text Node
\draw (281.68,1794.5) node [anchor=north west][inner sep=0.75pt]  [color={rgb, 255:red, 126; green, 211; blue, 33 }  ,opacity=1 ,rotate=-335.18]  {$\vdots $};
% Text Node
\draw (308.71,1852.33) node [anchor=north west][inner sep=0.75pt]  [color={rgb, 255:red, 126; green, 211; blue, 33 }  ,opacity=1 ,rotate=-359.32]  {$l_{2}$};
% Text Node
\draw (327,1620.4) node [anchor=north west][inner sep=0.75pt]    {$\dotsc $};
% Text Node
\draw (216,1772.4) node [anchor=north west][inner sep=0.75pt]  [color={rgb, 255:red, 74; green, 144; blue, 226 }  ,opacity=1 ]  {$\dotsc $};
% Text Node
\draw (236,1772.4) node [anchor=north west][inner sep=0.75pt]  [color={rgb, 255:red, 126; green, 211; blue, 33 }  ,opacity=1 ]  {$\dotsc $};
\end{tikzpicture}

\caption{Tree indicating the common path $\mathcal{V}^{v}_{l_1,l_2}$ and the subtrees $\mathcal{G}^{w^{l_1}}$, $\mathcal{G}^{w^{l_2}}$.}
\label{fig:common:path}
\end{subfigure}
\caption{Figures illustrating the notation.}
\label{fig:paths}
\end{figure}

Define the sum of derivatives in the downward path from the vertex $v$ to the leaf $f_{l}$ by
\begin{equation*}
\beta(v, l):= \sum_{w \in \mathcal{V}^{v}_{l}} \beta^w.
\end{equation*}
If $v = f_l$ is a leaf, then $\beta(v,l)  = 0$. The property \ref{deriv:distrib:ii} of the derivative distribution function can also be written as follows: if $\delta(v) = l$ for some $v \in \mathcal{V}$, then following the notation in Figure \ref{fig:paths}(\subref{fig:path}), $\delta(\bar{v}_2) = \ldots = \delta(\bar{v}_M) = l$. Equivalently, $\delta$ restricted to the non-leaf vertices $\mathcal{V}^v$ of the subtree $\mathcal{G}^v$ satisfies
 \begin{equation} \label{delta:dominating}
\left(\restr{\delta}{\mathcal{V}^v}\right)^{-1}(l) = \beta(v, l).
 \end{equation}
For any subset of non-leaf vertices $\mathcal{W} \subseteq \mathcal{V}$, we can denote by $\mathcal{D}(\mathcal{W})$ the collection of the maps $\restr{\delta}{\mathcal{W}}$ defined on $\mathcal{W}$ satisfying the conditions \ref{deriv:distrib:i} and \ref{deriv:distrib:ii}. On the complement $(\mathcal{V}^{v}_{l})^c$, we thus have the condition that $\delta_{(\mathcal{V}^{v}_{l})^c} \in \mathcal{D}((\mathcal{V}^{v}_{l})^c)$. We will consider $\restr{\delta}{\mathcal{W}} \in \mathcal{D}(\mathcal{W})$ as a default condition for the restricted map $\restr{\delta}{\mathcal{W}}$ with $\mathcal{W} \subseteq \mathcal{V}$, which will be omitted oftentimes for the simplicity of notation.

With the notation set, we are now ready to present the proof of Theorem \ref{thm:main}, in the one-parameter case. Instead of performing the induction solely on the conclusion statement, we will proceed with multiple inductive statements that are useful in deducing \eqref{induction_result}. Let $R$ be a conical region of the form \eqref{eq:cone:ums} or \eqref{eq:cone:nonums}. Without loss of generality, we will focus on $T^{R}_{\mathcal{G}}\left((P_{k_{\max}}f_l)_{1 \leq l \leq n}\right)$ and abbreviate as $T_{\mathcal{G}}\left((P_{k_{\max}}f_l)_{1 \leq l \leq n}\right)$. The inductive statements, which can be seen as a generalization of \eqref{result:Leibniz:Whitney} in Lemma \ref{lemma:multipliers:simpleParaproduct}, are the following:

\begin{proposition}\label{prop:1param:ind:statements}
Suppose that all the Lebesgue exponents in the inductive statement satisfy the condition described in Theorem \ref{thm:main}, and that $T_{\mathcal G}$ is restricted to a cone denoted by R. %the maximal scale $k_{\max}$ and the set $\mathfrak{M}(R)$.
\begin{enumerate}
\item \label{prop:whitney:1}
Suppose that $l_0 \in \mathfrak{M}(R)$. Then 
\begin{equation} \label{induction_whitney_cube_statement}
 \left\|T_{\mathcal{G}}\left((P_{k_{\max}}f_l)_{1 \leq l \leq n}\right)\right\|_{r} \lesssim 2^{k_{\max}\cdot\beta(\mathfrak{r}_{\mathcal{G}}, l_0)}\|\Delta_{k_{\max}}f_{l_0}\|_{p_{l_0}}\sum_{\restr{\delta}{(\mathcal{V}_{l_0}^{\mathfrak{r}_{\mathcal{G}}})^c}}\prod_{\substack{ \\ l \neq l_0}}\|D^{\delta^{-1}(l)}f_{l}\|_{p_l};
\end{equation}
\item \label{prop:whitney:2}
Suppose that $l_1, l_2 \in \mathfrak{M}(R)$ with $l_1 \neq l_2$. Then
\begin{equation} \label{induction_whitney_cube_statement-2max}
\left\|T_{\mathcal{G}}\left((P_{k_{\max}}f_l)_{1 \leq l \leq n}\right)\right\|_{r} \lesssim 2^{k_{\max}\cdot\beta(\mathfrak{r}_{\mathcal{G}}, v^{l_1, l_2})}\|\Delta_{k_{\max}}D^{\beta(w^{l_1}, l_1)}f_{l_1}\|_{p_{l_1}}\|\Delta_{k_{\max}}D^{\beta(w^{l_2}, l_2)}f_{l_2}\|_{p_{l_2}}\sum_{\restr{\delta}{(\mathcal{V}_{l_1}^{\mathfrak{r}_{\mathcal{G}}})^c \setminus \mathcal{V}^{w^{l_2}}_{l_2}}}\prod_{\substack{ \\ l \neq l_1, l_2}}\|D^{\delta^{-1}(l)}f_{l}\|_{p_l};
\end{equation}
\item \label{prop:whitney:3}
\begin{equation}\label{induction_cone_statement}
\big\|\sum_{k_{\max}}T_{\mathcal{G}}\left((P_{k_{\max}} f_l)_{1\leq l \leq n}\right)\big\|_{r} \lesssim \sum_{\delta}\prod_{l=1}^n\|D^{\delta^{-1}(l)}f_l\|_{p_l}.
\end{equation}
\end{enumerate}
\end{proposition}
\begin{remark}
\textbf{(i)}
Induction statements \eqref{prop:whitney:1} and \eqref{prop:whitney:2} describe the estimates for the multi-linear expression localized to a union of Whitney cubes at scale $k_{\max}$ in the cone $R$, depending on the configuration of $\mathfrak{M}(R)$. In particular, \eqref{induction_whitney_cube_statement-2max} corresponds to the case when the Whitney cubes are located in a cone $\tilde R_{l_1, l_2}$ of type \eqref{eq:cone:nonums}; \eqref{induction_whitney_cube_statement} on the other hand holds both for cones of type \eqref{eq:cone:ums} or \eqref{eq:cone:nonums}. Since \eqref{prop:whitney:2} describes a special case of \eqref{prop:whitney:1}, it is not surprising to observe that the expression on the right hand side of \eqref{induction_whitney_cube_statement-2max} can be majorized by the right hand side of \eqref{induction_whitney_cube_statement}.\\
\textbf{(ii)}
We would like to emphasize that \eqref{induction_cone_statement} for all possible paraproduct regions implies precisely the conclusion of Theorem  \ref{thm:main}:
\begin{equation}\label{leibniz_global}
\left\|T_{\mathcal{G}}\left((f_l)_{1 \leq l \leq n}\right)\right\|_{r} \lesssim \sum_{\delta}\prod_{l=1}^n\|D^{\delta^{-1}(l)}f_l\|_{p_l}.
\end{equation}
\end{remark}
\begin{proof}
As mention earlier, we will only focus on the case $r \geq 1$. We proceed by induction on the tree structure: the base case corresponds to trees of complexity 1, thus
the base cases for \eqref{induction_whitney_cube_statement} and \eqref{induction_whitney_cube_statement-2max} are verified by Lemma \ref{lemma:multipliers:simpleParaproduct}.\footnote{Although the lemma indicates the bi-linear case, it can be easily extended to the $n$-linear setting} The base case for \eqref{induction_cone_statement} is simply the Leibniz rule corresponding to paraproducts and thus is proven to be true -- see Section \ref{Bourgain-Li_hilow}. We would like to prove \eqref{induction_whitney_cube_statement}, \eqref{induction_whitney_cube_statement-2max} and \eqref{induction_cone_statement} for a tree of complexity $\mathcal{C}$ assuming that \eqref{induction_whitney_cube_statement}, \eqref{induction_whitney_cube_statement-2max} and \eqref{induction_cone_statement} hold for any tree of \textit{any} lower complexity(that is, of complexity $\mathcal{C}-1, \ldots, 1$). 
\begin{enumerate}[leftmargin=*]
\item
We first verify the inductive statement \eqref{induction_whitney_cube_statement}. Denote by $(v_i)_{i=1}^{n_1}$ the vertices of depth 1 in $\mathcal{G}$, and let $m_{\beta^{\mathfrak{r}_{\mathcal{G}}}}$ be the symbol defined by 
$$
m_{\beta^{\mathfrak{r}_{\mathcal{G}}}}(\sum_{l=1}^n \xi_l) := |\sum_{l =1}^n \xi_l|^{\beta^{\mathfrak{r}_{\mathcal{G}}}}.
$$
The multilinear expression $ T_{\mathcal{G}}$ can be rewritten as
\begin{align} \label{induction_whitney_op}
& T_{\mathcal{G}}\left((P_{k_{\max}}f_l)_{1 \leq l \leq n}\right) = T_{m^{k_{\max}}_{\beta^{\mathfrak{r}_{\mathcal{G}}}}}\Big(T_{\mathcal{G}^{v_1}} \big((P_{k_{\max}} f_l)_{\substack{ \\ l \in \mathcal{L}(v_1)}}\big), \ldots, T_{\mathcal{G}^{v_{n_1}}} \big((P_{k_{\max}} f_l)_{\substack{ \\ l \in \mathcal{L}(v_{n_1}) }}\big)  \Big),
%= & T_{m_{\beta^{\mathfrak{r}_{\mathcal{G}}}} \cdot \chi_R}\left(\left(T_{\mathcal{G}^{v_1}} \left((\Delta_{k_{\max}}f_{l})_{\substack{\\ l \in \mathcal{L}(v_i) \cap \mathfrak{M}(k_{\max})}}, (S_{k_{\max}} f_l)_{\substack{ \\ l \in \mathcal{L}(v_i) \cap \mathfrak{M}(k_{\max})}}\right)  \right)_{i=1}^{n_1}\right).
%\int |\sum_{l =1}^n \xi_l|^{\beta_{\mathfrak{r}_{\mathcal{G}}}} \prod_{i=1}^{n_1}T_{\mathcal{G}^{v_i}}\left((\Delta_{k_{\max}(v_i)}f_{l})_{\substack{\\ l \in \mathfrak{M}(v_i)}}, (S_{k_{\max}(v_i)} f_l)_{\substack{ \\ l \in \mathfrak{m}(v_i)}}\right)
\end{align} 
where $T_m$ generically denotes a multilinear operator associated to a symbol $m$.

We recall that the paraproduct decomposition yields the localization to the region $R_{k_{\max}}$ defined by 
\begin{equation} \label{definition_localization}
R_{k_{\max}}:= \{(\xi_1, \ldots, \xi_n): |\xi_l| \sim 2^{k_{\max}} \ \ \text{for} \ \  l \in \mathfrak{M}(R) \ \  \text{and} \ \ |\xi_l| \ll 2^{k_{\max}} \ \ \text{for} \ \  l \in \mathfrak{m}(R) \},
\end{equation}
%We smoothly restrict the root symbol $m_{\beta^{\mathfrak{r}_{\mathcal{G}}}}$ to $R_{k_{\max}}$ and denote the localized symbol as $m^{k_{\max}}_{\beta^{\mathfrak{r}_{\mathcal{G}}}}$. As a consequence, we can rewrite %Then $T_{\mathcal{G}}$ localized to the region $R_{\max}$ can be expressed as
%There is some abuse of notation in the expression above because $\mathcal{L}(v_i) \cap \mathfrak{M}(k_{\max})$ might be empty. 
on which one has $|\sum_{l=1}^n \xi_l| \lesssim n 2^{k_{\max}}$. We can thus smoothly restrict the symbol $m_{\beta^{\mathfrak{r}_{\mathcal{G}}}}(\sum_{l=1}^n \xi_l)$ to the interval $[-n2^{k_{\max}}, n2^{k_{\max}}]$ and denote it by $m^{k_{\max}}_{\beta^{\mathfrak{r}_{\mathcal{G}}}}(\sum_{l=1}^n \xi_l)$. We perform a Fourier series decomposition on the symbol as in Section \ref{Bourgain-Li_hilow} (Fourier series decomposition for the ``diagonal" term):
\begin{align*}
m^{k_{\max}}_{\beta^{\mathfrak{r}_{\mathcal{G}}}} (\sum_{l=1}^n\xi_l)= (n2^{k_{\max}})^{ \beta^{\mathfrak{r}_{\mathcal{G}}}} \sum_{L \in \mathbb{Z}}C_L e^{2\pi i \frac{L}{n2^{k_{\max}}} \sum_{l=1}^n \xi_l},
\end{align*}
where the renormalized Fourier coefficients satisfy the decaying condition
\begin{align}\label{fourier_coeff_decay_induction_1parameter}
|C_L| \lesssim \frac{1}{(1+|L|)^{1+\beta^{\mathfrak{r}_{\mathcal{G}}}}}.
\end{align}
As a result, one can rewrite (\ref{induction_whitney_op}) (up to a constant depending implicitly on $n$ and $\beta^{\mathfrak{r}_{\mathcal{G}}}$) as
\begin{align} \label{induction_whitney_fourier}
\sum_{L \in \mathbb{Z}} C_L 2^{k_{\max}\cdot\beta^{\mathfrak{r}_{\mathcal{G}}}} \prod_{i=1}^{n_1}T_{\mathcal{G}^{v_i}} \Big((P_{k_{\max}, \frac{L}{n2^{k_{\max}}}}f_{l})_{\substack{\\ l \in \mathcal{L}(v_i)}} \Big),
\end{align}
which can be estimated by
\begin{align}  \label{induction_whitney_fourier_0}
\|(\ref{induction_whitney_fourier})\|_{r} \lesssim \sum_{L \in \mathbb{Z}} |C_L| 2^{k_{\max}\cdot\beta^{\mathfrak{r}_{\mathcal{G}}}} \prod_{i=1}^{n_1} \Big\|T_{\mathcal{G}^{v_i}} \big((P_{k_{\max}, \frac{L}{n2^{k_{\max}}}}f_{l})_{\substack{\\ l \in \mathcal{L}(v_i)}} \big) \Big\|_{p_{v_i}}.
\end{align}
We observe that for each $1 \leq i \leq n_1$, the following identity holds:
\begin{equation} \label{identity_translation}
T_{\mathcal{G}^{v_i}} \big((P_{k_{\max}, \frac{L}{n2^{k_{\max}}}}f_{l})_{\substack{\\ l \in \mathcal{L}(v_i)}} \big)(x) = T_{\mathcal{G}^{v_i}} \big((P_{k_{\max}}f_{l})_{\substack{\\ l \in \mathcal{L}(v_i)}} \big)(x+\frac{L}{n2^{k_{\max}}})
\end{equation}
and the translation invariance of the measure yields
\begin{equation*}
\big\|T_{\mathcal{G}^{v_i}} \big((P_{k_{\max}, \frac{L}{n2^{k_{\max}}}}f_{l})_{\substack{\\ l \in \mathcal{L}(v_i)}} \big) \big\|_{p_{v_i}} = \big\|T_{\mathcal{G}^{v_i}} \big((P_{k_{\max}}f_{l})_{\substack{\\ l \in \mathcal{L}(v_i)}} \big)\big\|_{p_{v_i}}.
\end{equation*}
Due to the decay of the Fourier coefficients \eqref{fourier_coeff_decay_induction_1parameter},\footnote{When $r <1$, we would use subaddtivity to deduce
\begin{equation*}
\|(\ref{induction_whitney_fourier})\|_r^r \lesssim \sum_{L \in \mathbb{Z}} |C_L|^r \cdot 2^{k_{\max}\beta^{\mathfrak{r}_{\mathcal{G}}}r} \prod_{i=1}^{n_1} \Big\|T_{\mathcal{G}^{v_i}} \big((P_{k_{\max}, \frac{L}{n2^{k_{\max}}}}f_{l})_{\substack{\\ l \in \mathcal{L}(v_i)}} \big) \Big\|_{p_{v_i}}^r =  \sum_{L \in \mathbb{Z}} |C_L|^r \cdot 2^{k_{\max}\beta^{\mathfrak{r}_{\mathcal{G}}}r} \prod_{i=1}^{n_1} \big\|T_{\mathcal{G}^{v_i}} \big((P_{k_{\max}}f_{l})_{\substack{\\ l \in \mathcal{L}(v_i)}} \big) \big\|_{p_{v_i}}^r,
\end{equation*}
where the Fourier coefficients satisfy the decay condition (\ref{fourier_coeff_decay_induction_1parameter}). It is natural to impose the condition $$r(1+\beta^{\mathfrak{r}_{\mathcal{G}}})>1.$$ For the same reason, for any $v \in \mathcal{L}(\mathfrak{r}_{\mathcal{G}})$, there is an associated differential operator $D^{\beta^v}$ whose Fourier series decomposition yields Fourier coefficients with limited decay, thus imposing the condition on the Lebesgue exponent 
$$
p_v(1+\beta^v) >1.
$$
} (\ref{induction_whitney_fourier_0}) is majorized by
\begin{equation}\label{induction_whitney_total}
 2^{k_{\max}\cdot\beta^{\mathfrak{r}_{\mathcal{G}}}} \prod_{i=1}^{n_1} \Big\|T_{\mathcal{G}^{v_i}} \big((P_{k_{\max}}f_{l})_{\substack{\\ l \in \mathcal{L}(v_i) }}  \big) \Big\|_{p_{v_i}}.
\end{equation}

A simple observation is that there exists some $1 \leq i_0 \leq n_1$ such that $\mathcal{L}(v_{i_0}) \cap \mathfrak{M}(R) \neq \emptyset$, in which case the subtree $\mathcal{G}^{v_{i_0}}$ is automatically restricted to a conical region and $k_{\max} = k_{\max}(v_{i_0})$ specifies the union of Whitney cubes at scale $k_{\max}$ in such a cone. Assume without loss of generality that $i_0 = 1$. We now apply the inductive hypothesis (\ref{induction_whitney_cube_statement}):
\begin{align} \label{induction_whitney_major}
\big\|T_{\mathcal{G}^{v_1}} \big((P_{k_{\max}}f_{l})_{\substack{\\ l \in \mathcal{L}(v_1)}}  \big)\Big\|_{{p}_{v_1}} %= & \Big\|T_{\mathcal{G}^{v_1}} \big((P_{k_{\max}(v_1)}f_{l})_{\substack{\\ l \in \mathcal{L}(v_1)}} \big)\Big\|_{{p}_{v_1}}  \nonumber \\
 \lesssim & 2^{k_{\max}\cdot\beta(v_1, l_0)}\|\Delta_{k_{\max}}f_{l_0}\|_{p_{l_0}}\sum_{\restr{\delta}{(\mathcal{V}_{l_0}^{v_1})^c}}\prod_{\substack{l \in \mathcal{L}(v_1) \\ l \neq l_0}}\|D^{\delta^{-1}(l)}f_{l}\|_{p_l}.
\end{align}
Meanwhile, we can invoke the inductive hypothesis (\ref{induction_cone_statement}) and thus (\ref{leibniz_global}) to deduce that for $i \neq 1$,
\begin{align} \label{induction_whitney_minor}
& \big\|T_{\mathcal{G}^{v_i}} \big((P_{k_{\max}}f_{l})_{\substack{\\ l \in \mathcal{L}(v_i) }} \big)\big\|_{p_{v_i}}\lesssim  \sum_{\restr{\delta}{\mathcal{V}^{v_i}}}\prod_{l \in \mathcal{L}(v_i)}\|D^{\delta^{-1}(l)}f_l\|_{p_l}.
\end{align}
Combining the estimates (\ref{induction_whitney_major}) and (\ref{induction_whitney_minor}), we derive the following desired estimate for (\ref{induction_whitney_total}):
\begin{align*}
2^{k_{\max}\cdot \beta(\mathfrak{r}_{\mathcal{G}}, l_0)}\|\Delta_{k_{\max}}f_{l_0}\|_{p_{l_0}}\sum_{\restr{\delta}{(\mathcal{V}_{l_0}^{\mathfrak{r}_{\mathcal{G}}})^c}}\prod_{\substack{ \\ l \neq l_0}}\|D^{\delta^{-1}(l)}f_{l}\|_{p_l}.
\end{align*}
\item 
To show the second inductive statement (\ref{induction_whitney_cube_statement-2max}), we use the Fourier series decomposition for the root symbol applied in the proof of (\ref{induction_whitney_cube_statement}) and obtain (\ref{induction_whitney_fourier}), whose $L^r$ norm can be estimated by (\ref{induction_whitney_total}). There are two possibilities with respect to the positions of $l_1, l_2$:
\begin{enumerate}[leftmargin=*]
\item 
$l_1, l_2 \in \mathcal{L}(v_{i_0})$ for some $1 \leq i_0 \leq n_1$ (or equivalently $l_1, l_2$ are leaves of the same subtree rooted in $v_{i_0}$, a direct descendant of the root $r_{\mathcal G}$). We now invoke the inductive hypothesis (\ref{induction_whitney_cube_statement-2max}) on $T_{\mathcal{G}^{v_{i_0}}}$:
{\fontsize{8}{8}
\begin{align}\label{induction_cube_2max_proof}
\big\|T_{\mathcal{G}^{v_{i_0}}}\big((P_{k_{\max}}f_l)_{l \in \mathcal{L}(v_{i_0})}\big)\big\|_{p_{v_{i_0}}} \lesssim 2^{k_{\max}\cdot\beta(v_{i_0}, v^{l_1, l_2})}\|\Delta_{k_{\max}}D^{\beta(w^{l_1}, l_1)}f_{l_1}\|_{p_{l_1}}\|\Delta_{k_{\max}}D^{\beta(w^{l_2}, l_2)}f_{l_2}\|_{p_{l_2}}\sum_{\restr{\delta}{(\mathcal{V}_{l_1}^{v_{i_0}})^c \setminus \mathcal{V}^{w^{l_2}}_{l_2}}}\prod_{\substack{ l \in \mathcal{L}(v_{i_0}) \\ l \neq l_1, l_2}}\|D^{\delta^{-1}(l)}f_{l}\|_{p_l}. 
\end{align}}
We can also apply the inductive hypothesis (\ref{leibniz_global}) to derive (\ref{induction_whitney_minor}) for $i \neq i_0$. Combining (\ref{induction_cube_2max_proof}) and (\ref{induction_whitney_minor}), we conclude with (\ref{induction_whitney_cube_statement-2max}).

\item 
$l_1 \in \mathcal{L}(v_{i_1})$ and $l_2 \in \mathcal{L}(v_{i_2})$ for $1\leq i_1, i_2 \leq n_1$ and $i_1 \neq i_2$ (or equivalently $l_1, l_2$ are leaves of two different subtrees stemming from direct descendants of the root). In this case, the common path is
$
\mathcal{V}^{\mathfrak{r}_{\mathcal{G}}}_{l_1,l_2} = \{\mathfrak{r}_{\mathcal{G}} \},
$
which means that  
\begin{equation} \label{end_common_precise}
v^{l_1,l_2} = \mathfrak{r}_{\mathcal{G}}.
\end{equation}
Moreover, the assumption about the positions of $l_1$ and $l_2$ gives the precise vertices $w^{l_1}$ and $w^{l_2}$:
\begin{equation}\label{branch_out_root}
w^{l_1} = v_{i_1}  \ \ \text{and} \ \  w^{l_2} = v_{i_2}.
\end{equation}
The inductive hypothesis (\ref{induction_whitney_cube_statement}) can be applied to 
$$
T_{\mathcal{G}^{v_{i_1}}}\big((P_{k_{\max}}f_l)_{l \in \mathcal{L}(v_{i_1})}\big) \ \  \text{and} \ \ T_{\mathcal{G}^{v_{i_2}}}\big((P_{k_{\max}}f_l)_{l \in \mathcal{L}(v_{i_2})}\big).
$$
In particular, 
\begin{align*}
\big\|T_{\mathcal{G}^{v_{i_1}}}\big((P_{k_{\max}}f_l)_{l \in \mathcal{L}(v_{i_1})}\big) \big\|_{p_{v_{i_1}}} \lesssim & 2^{k_{\max}\cdot\beta(v_{i_1}, l_1)}\|\Delta_{k_{\max}}f_{l_1}\|_{p_{l_1}}\sum_{\restr{\delta}{(\mathcal{V}_{l_1}^{v_{i_1}})^c}}\prod_{\substack{l \in \mathcal{L}(v_{i_1}) \\ l \neq l_1}}\|D^{\delta^{-1}(l)}f_{l}\|_{p_l} \\
\lesssim & \|\Delta_{k_{\max}}D^{\beta(v_{i_1}, l_1)}f_{l_1}\|_{p_{l_1}}\sum_{\restr{\delta}{(\mathcal{V}_{l_1}^{v_{i_1}})^c}}\prod_{\substack{l \in \mathcal{L}(v_{i_1}) \\ l \neq l_1}}\|D^{\delta^{-1}(l)}f_{l}\|_{p_l},
\end{align*}
where the last inequality follows from (\ref{eq:obs:besov}). A similar reasoning gives 
\begin{align*}
\big\|T_{\mathcal{G}^{v_{i_2}}}\big((P_{k_{\max}}f_l)_{l \in \mathcal{L}(v_{i_2})}\big) \big\|_{p_{v_{i_2}}} \lesssim &  \|\Delta_{k_{\max}}D^{\beta(v_{i_2}, l_2)}f_{l_2}\|_{p_{l_2}}\sum_{\restr{\delta}{(\mathcal{V}_{l_2}^{v_{i_2}})^c}}\prod_{\substack{l \in \mathcal{L}(v_{i_2}) \\ l \neq l_2}}\|D^{\delta^{-1}(l)}f_{l}\|_{p_l}.
\end{align*}
For $i\neq i_1, i_2$, we use the estimate (\ref{induction_whitney_minor}) implied from the inductive hypothesis (\ref{leibniz_global}). As a consequence, (\ref{induction_whitney_total}) can be majorized by
\begin{align*}
2^{k_{\max}\beta^{\mathfrak{r}_{\mathcal{G}}}}\|\Delta_{k_{\max}}D^{\beta(v_{i_1}, l_1)}f_{l_1}\|_{p_{l_1}}  \|\Delta_{k_{\max}}D^{\beta(v_{i_2}, l_2)}f_{l_2}\|_{p_{l_2}}\sum_{\restr{\delta}{(\mathcal{V}_{l_1}^{\mathfrak{r}_{\mathcal{G}}})^c \setminus \mathcal{V}^{v_{i_2}}_{l_2}}}\prod_{\substack{ \\ l \neq l_1, l_2}}\|D^{\delta^{-1}(l)}f_{l}\|_{p_l},
\end{align*}
which agrees with (\ref{induction_whitney_cube_statement-2max}) due to the interpretation of notations (\ref{end_common_precise}) and (\ref{branch_out_root}).
\end{enumerate}
\vskip .1in
\item
We will now prove the third inductive statement corresponding to \eqref{induction_cone_statement}; in this case we need to take into account the more precise structure of the conical region $T_{\mathcal G}$ is restricted to.\begin{enumerate}[leftmargin=*]
\item
\textbf{Case 1: The conical region $R$ is of the form \eqref{eq:cone:ums}.} \\
 Suppose that %the unique maximal scale $k_{\max} = k_{l_0}$, or equivalently
$$\mathfrak{M}(R) = \{ l_0\}, \ \  \mathfrak{m}(R)= \{1 \leq l \leq n: l \neq l_0\}, 
$$ and $l_0 \in \mathcal{L}(v_{i_0})$ for some $1 \leq i_0 \leq n_1$. 
Then $\displaystyle \sum_{k_{\max}} T_{\mathcal{G}}\big((P_{k_{\max}} f_l)_{1 \leq l \leq n}\big)$, or more precisely  $$ \sum_{k_{\max}} T^{R}_{\mathcal{G}}\big((P_{k_{\max}} f_l)_{1 \leq l \leq n}\big)$$ concerns the frequency space localized to the conical region \eqref{eq:cone:ums}, on which we apply the \textit{splitting of the root symbol} step introduced in Section \ref{Bourgain-Li_hilow} and used in Section 3.1:
\begin{align}
%m^{k_{l_0}}_{\beta^{\mathfrak{r}_{\mathcal{G}}}}
|\sum_{l =1}^n\xi_l|^{\beta^{\mathfrak{r}_{\mathcal{G}}}}%\cdot \chi_{R}(\xi_1, \ldots, \xi_n)%= & \frac{\displaystyle |\sum_{l =1}^n\xi_l|^{\beta^{\mathfrak{r}_{\mathcal{G}}}} - |\sum_{l \in \mathcal{L}(v_{i_0})} \xi_l|^{\beta^{\mathfrak{r}_{\mathcal{G}}}}}{\displaystyle \sum_{l \nin \mathcal{L}(v_{i_0})} \xi_l} \cdot \sum_{i \neq i_0}\sum_{l \in \mathcal{L}(v_i)}\xi_l +  |\sum_{l \in \mathcal{L}(v_{i_0})} \xi_l|^{\beta^{\mathfrak{r}_{\mathcal{G}}}} \nonumber \\
= & \underbrace{\sum_{i \neq i_0}   m_{C_\beta^{\mathfrak{r_{\mathcal{G}}}}}\Big(\sum_{l \in \mathcal{L}(v_{i_0})}\xi_l, \sum_{l \nin \mathcal{L}(v_{i_0})}\xi_l\Big)\cdot \sum_{l \in \mathcal{L}(v_i)}\xi_l}_{\mathcal{E}_1} +  \underbrace{|\sum_{l \in \mathcal{L}(v_{i_0})} \xi_l|^{\beta^{\mathfrak{r}_{\mathcal{G}}}}}_{\mathcal{E}_2},\label{symbol_1parameter_uni_maximal} 
\end{align}
where 
\begin{align*}
m^{}_{C_\beta^{\mathfrak{r_{\mathcal{G}}}}}\left(\tilde{\xi}_1, \tilde{\xi}_2\right) := \frac{\displaystyle |\tilde{\xi}_1 + \tilde{\xi}_2|^{\beta^{\mathfrak{r}_{\mathcal{G}}}} - |\tilde{\xi}_1|^{\beta^{\mathfrak{r}_{\mathcal{G}}}}}{\displaystyle \tilde{\xi}_2}.
\end{align*}
\\
\noindent
\textbf{Estimate of $\mathcal{E}_1$:}
\vskip .1 in
The symbol denoted by $\mathcal{E}_1$ generates a commutator whose treatment builds on the approach described in Section \ref{Bourgain-Li_hilow}. In order to perform the double Fourier series decomposition on one Whitney cube at a time, we need to further decompose $R_{k_{l_0}}$ \eqref{local:whitney_cube_1max} -- the union of Whitney cubes at scale $k_{l_0}$. In particular, 
$$
R_{k_{l_0}} = R_{k_{l_0}}^+ \cup R_{k_{l_0}}^-,
$$
where
\begin{align*}
R_{k_{l_0}}^+ := \{(\xi_1, \ldots, \xi_n): \xi_{l_0} \sim 2^{k_{l_0}}, |\xi_l| \ll 2^{k_{l_0}}\}, \quad R_{k_{l_0}}^- := \{(\xi_1, \ldots, \xi_n): \xi_{l_0} \sim -2^{k_{l_0}}, |\xi_l| \ll 2^{k_{l_0}}\}.
\end{align*}
We restrict the symbol $m^{}_{C_\beta^{\mathfrak{r_{\mathcal{G}}}}}$ to $R_{k_{l_0}}^+$ and $R_{k_{l_0}}^-$ and denote them by $m^{k_{l_0},+}_{C_\beta^{\mathfrak{r_{\mathcal{G}}}}}$ and $m^{k_{l_0},-}_{C_\beta^{\mathfrak{r_{\mathcal{G}}}}}$ respectively. 
Then the double Fourier series decomposition yields
 \begin{align} 
  m^{k_{l_0},\pm}_{C_\beta^{\mathfrak{r_{\mathcal{G}}}}}\Big(\sum_{l \in \mathcal{L}(v_{i_0})}\xi_l, \sum_{l \nin \mathcal{L}(v_{i_0})}\xi_l\Big)= \sum_{L_1, L_2 \in \mathbb{Z}} C^{\pm}_{L_1, L_2} 2^{k_{l_0}(\beta^{\mathfrak{r}_{\mathcal{G}}}-1)} e^{2 \pi i L_1 \frac{\sum_{l \in \mathcal{L}(v_{i_0})}\xi_l} {2^{k_{l_0}}}} \, e^{2 \pi i L_2 \frac{\sum_{l \nin \mathcal{L}(v_{i_0})} \xi_l}{2^{k_{l_0}}}},\label{commutator_1parameter_fourier_decomp+} %\\
  %  m^{k_{l_0},-}_{C_\beta^{\mathfrak{r_{\mathcal{G}}}}}\Big(\sum_{l \in \mathcal{L}(v_{i_0})}\xi_l, \sum_{l \nin \mathcal{L}(v_{i_0})}\xi_l\Big)= \sum_{L_1, L_2 \in \mathbb{Z}} C^-_{L_1, L_2} 2^{k_{l_0}(\beta^{\mathfrak{r}_{\mathcal{G}}}-1)} e^{2 \pi i L_1 \frac{\sum_{l \in \mathcal{L}(v_{i_0})}\xi_l} {2^{k_{l_0}}}} \, e^{2 \pi i L_2 \frac{\sum_{l \nin \mathcal{L}(v_{i_0})} \xi_l}{2^{k_{l_0}}}}.\label{commutator_1parameter_fourier_decomp-}
 \end{align}
which is essentially the same as \eqref{eq:FS:5lin:comm} with a few natural adjustments: replacing $1$ by $l_0$ and $i_0$, $\xi_1 + \xi_2$ by $\sum_{l \in \mathcal{L}(v_{i_0})} \xi_l$ and $\xi_3 + \xi_4 + \xi_5$ by $\sum_{l \nin \mathcal{L}(v_{i_0})} \xi_l$. We remark that all the renormalized Fourier coefficients involved decay rapidly. 

Without loss of generality, assume that 
$$
\mathfrak{M}(R) = \{1\}, \ \ \mathfrak{m}(R)= \{1 \leq l \leq n: l \neq 1\}.
$$
We will use $k_1$ specifically instead of $k_{\max}$ and denote by $R_{k_1}$ the union of Whitney cubes at scale $k_1$ in the cone $R$.
For any fixed $L \in \mathbb{Z}$,
$$P_{k_{1},L}(1) = \Delta_{k_{1},L}, \ \ P_{k_1,L}(l) = S_{k_{1},L} \ \ \text{for} \ \  l \neq 1.$$

As a consequence of the above two steps, we can rewrite the symbol $\mathcal{E}_1$ as 
 $$\mathcal{E}_1 =: \sum_{i\neq i_0} \mathcal{E}_1^i.$$ We focus on $\mathcal{E}_1^2$ and consider the multiplier associated to it:
 \begin{align} \label{E_12:pre_loc}
 & \sum_{k_1}T_{\mathcal{E}_1^2}\Big(T_{\mathcal{G}^{v_1}} \big(\Delta_{k_1}f_1, (S_{k_{1}} f_l)_{\substack{ \\ 1 \neq l \in \mathcal{L}(v_1)}}\big), T_{\mathcal{G}^{v_{2}}} \big((S_{k_{1}} f_l)_{\substack{ \\ l \in \mathcal{L}(v_{2}) }}\big), \ldots, T_{\mathcal{G}^{v_{n_1}}} \big((S_{k_{1}} f_l)_{\substack{ \\ l \in \mathcal{L}(v_{n_1}) }}\big)  \Big).
 \end{align}
For the multilinear expression associated to the subtree $\mathcal{G}^{v_2}$, namely
 $$
 T_{\mathcal{G}^{v_2}}\left((S_{k_1}f_l)_{l \in \mathcal{L}(v_2)}\right),
 $$
we will further perform the paraproduct decomposition so as to focus on a fixed conical region denoted by $R(v_2)$.
Let $\mathfrak{M}(R(v_2))$ denote a subset of $\mathcal{L}(v_2)$ defined similarly to \eqref{index:major} and $\mathfrak{m}(R(v_2)) = \mathcal{L}(v_2) \setminus \mathfrak{M}(R(v_2))$. Furthermore, we define
%specified the set $\mathfrak{M}(k_{\max}(v_2))$, 
$$ \quad P_{k_{\max}(v_2)}(l) = \Delta_{k_{\max}(v_2)}, \ \text{if} \ l \in \mathcal{L}(v_2) \cap \mathfrak{M}(R(v_2)) \ \text{and} \ \ P_{k_{\max},(v_2)}(l) = S_{k_{\max}(v_2)}  \ \text{if} \  l \in \mathcal{L}(v_2) \cap \mathfrak{m}(R(v_2)).$$
The localized multilinear expression can then be written as
\begin{equation} \label{local:subtree}
 \sum_{k_{\max}(v_2)}T^{R(v_2)}_{\mathcal{G}^{v_2}}\big((P_{k_{\max}(v_2)}S_{k_1}f_{l})_{\substack{\\ l \in \mathcal{L}(v_2)}} \big).
\end{equation}

Since $\Delta_{k_{\max}(v_2)} S_{k_1} \nequiv 0$ if and only if $k_{\max(v_2)} \ll k_1$, we can restrict the sum in \eqref{local:subtree}:
 $$ \sum_{k_{\max}(v_2):k_{\max}(v_2) \ll k_1}T^{R(v_2)}_{\mathcal{G}^{v_2}}\big((P_{k_{\max}(v_2)}f_{l})_{\substack{\\ l \in \mathcal{L}(v_2)}} \big),$$
which abbreviates as 
 $$ \sum_{k_{\max}(v_2): k_{\max}(v_2) \ll k_1}T_{\mathcal{G}^{v_2}}\big((P_{k_{\max}(v_2)}f_{l})_{\substack{\\ l \in \mathcal{L}(v_2)}} \big).$$
 As a consequence, \eqref{E_12:pre_loc} takes the form
 \begin{align} \label{decomp:pararprod:minorcone}
%& \sum_{k_{\max}}T_{\mathcal{E}_1^2}\left(T_{\mathcal{G}^{v_1}} \left(\Delta_{k_1}f_1, (S_{k_{1}} f_l)_{\substack{ \\ 1 \neq l \in \mathcal{L}(v_1)}}\right), T_{\mathcal{G}^{v_{2}}} \left((S_{k_{1}} f_l)_{\substack{ \\ l \in \mathcal{L}(v_{2}) }}\right), \ldots, T_{\mathcal{G}^{v_{n_1}}} \left((S_{k_{1}} f_l)_{\substack{ \\ l \in \mathcal{L}(v_{n_1}) }}\right)  \right)\nonumber \\
  %= 
  &  \sum_{k_{\max}(v_2) \ll k_1}T_{\mathcal{E}_1^2}\Big(T_{\mathcal{G}^{v_1}} \big(\Delta_{k_1}f_1,(S_{k_{1}} f_l)_{\substack{ \\ 1\neq l \in \mathcal{L}(v_1)}}\big), T_{\mathcal{G}^{v_2}}\big((P_{k_{\max}(v_2)}f_{l})_{\substack{\\ l \in \mathcal{L}(v_2)}} \big), \ldots, T_{\mathcal{G}^{v_{n_1}}} \big((S_{k_{1}} f_l)_{\substack{ \\ l \in \mathcal{L}(v_{n_1}) }}\big)  \Big).
 \end{align}
By applying the Fourier series \eqref{commutator_1parameter_fourier_decomp+} to \eqref{decomp:pararprod:minorcone}, the latter becomes a sum of terms of the form
\begin{equation}
\label{commutator_fourier_series}
\begin{aligned}
 \sum_{L_1, L_2\in \mathbb{Z}}C_{L_1, L_2}^{\pm}\sum_{k_{\max}(v_2) \ll k_1} & 2^{k_{\max}(v_2)} 2^{k_1 (\beta^{\mathfrak{r}_{\mathcal{G}}}-1)}T_{\mathcal{G}^{v_1}}\Big(\Delta_{k_1,\pm, \frac{L_1}{2^{k_1}}}f_1,  (S_{k_1, \frac{L_1}{2^{k_1}}}f_l)_{\substack{l \in \mathcal{L}(v_1) \\ l \neq 1}})\Big) \\ & T_{\mathcal{G}^{v_2}}\Big((P_{k_{\max}(v_2), \frac{L_2}{2^{k_1}} }f_{l})_{\substack{\\ l \in \mathcal{L}(v_2)}} \Big)  \prod_{i=3}^{n_1} T_{\mathcal{G}^{v_{i}}}\Big((S_{k_1, \frac{L_2}{2^{k_1}}}f_l)_{l \in \mathcal{L}(v_i)}\Big)
\end{aligned}
\end{equation}

Notice that the multipliers generated by $\mathcal{E}_1^i$, for $i >2$, behave analogously. By an observation similar to \eqref{identity_translation} and thanks to the decay of the $C_{L_1, L_2}^{\pm}$ coefficients, the $L^r$ norm of \eqref{commutator_fourier_series} can be estimated by
\begin{equation}\label{induction_a}
\begin{aligned}
 \sum_{k_{\max}(v_2) \ll k_1} & 2^{k_{\max}(v_2)} 2^{k_1 \cdot (\beta^{\mathfrak{r}_{\mathcal{G}}}-1)} \Big\|T_{\mathcal{G}^{v_1}}\big(\Delta_{k_1,+}f_1,  (S_{k_1}f_l)_{\substack{l \in \mathcal{L}(v_1) \\ l \neq 1}})\big)\Big\|_{{p}_{v_1}} \\ & \qquad \Big\|T_{\mathcal{G}^{v_2}}\big((P_{k_{\max}(v_2)}f_{l})_{\substack{\\ l \in \mathcal{L}(v_2)}}\big)\Big\|_{{p}_{v_2}}\prod_{i=3}^{n_1} \left\|T_{\mathcal{G}^{v_{i}}}\left((S_{k_1}f_l)_{l \in \mathcal{L}(v_i)}\right) \right\|_{p_{v_i}}.
\end{aligned}
\end{equation}

Since $\mathcal{G}^{v_i}$ for $1\leq i \leq n_1$ are subtrees of lower complexity, we invoke the inductive hypothesis (\ref{induction_whitney_cube_statement}) and \eqref{eq:direction:LP:same}, so that
\begin{align}
& \Big\|T_{\mathcal{G}^{v_1}}\big(\Delta_{k_1,+}f_1,  (S_{k_1}f_l)_{\substack{l \in \mathcal{L}(v_1) \\ l \neq 1}})\big)\Big\|_{p_{v_1}} \lesssim \|\Delta_{k_1} D^{\beta(v_1, 1)}f_1\|_{p_1}\sum_{\restr{\delta}{(\mathcal{V}_1^{v_1})^c}}\prod_{\substack{l\in \mathcal{L}(v_1) \\ l \neq 1}}\|D^{\delta^{-1}(l)}f_{l}\|_{p_l} \label{inductive_hyp_1}
\end{align}
and
\begin{align}
 \Big\|T_{\mathcal{G}^{v_2}}\big((P_{k_{\max}(v_2)}f_{l})_{\substack{\\ l \in \mathcal{L}(v_2)}}\big)\Big\|_{p_{v_2}} \lesssim \|\Delta_{k_{\max}(v_2)}D^{\beta(v_2,l_0)}f_{l_0}\|_{\vec{p_{l_0}}}\sum_{\restr{\delta}{(\mathcal{V}_{l_0}^{v_2})^c}}\prod_{\substack{ \\ l \neq l_0}}\|D^{\delta^{-1}(l)}f_{l}\|_{p_l}. \label{inductive_hyp_minor}
\end{align}

Moreover, the inductive hypotheses (\ref{induction_cone_statement}) and thus (\ref{leibniz_global}) generate
\begin{align}\label{inductive_hyp_rest}
& \left\|T_{\mathcal{G}^{v_{i}}}\left((S_{k_1}f_l)_{l \in \mathcal{L}(v_i)}\right) \right\|_{p_{v_i}} 
 \lesssim \sum_{\substack{\restr{\delta}{\mathcal{V}^{v_i}}}}\prod_{l \in \mathcal{L}(v_i)}\|D^{\delta^{-1}(l)}f_{l}\|_{{p_{l}}} \ \ \text{for}\ \  i \neq 1,2. %\prod_{\substack{j \in \mathcal{J}_{n_1} \\j \nin \text{range}(\delta)}}\|f_j\|_{p_j} %+ \nonumber\\
 %&   \sum_{\substack{\restr{\delta}{(\beta^{i_{n_1}}_1)_{i_{n_1}}}: \\ j_{n_1-1}+1\nin \text{range}(\delta)}}\prod_{\beta_1^{i_{n_1}}}\|D^{\beta_1^{i_{n_1}}}f_{\delta(\beta_1^{i_{n_1}})}\|_{p_{\delta(\beta_1^{i_{n_1}})}}\prod_{\substack{j \in \mathcal{J}_{n_1} \\ j \nin \text{range}(\delta)\\ j \neq j_{n_1-1}+1}} \|f_j\|_{p_j} \|S_{k_1}f_{j_{n_1-1}+1}\|_{p_{j_{n_1-1}+1}}= a^{n_1}_1 + a^{n_1}_2.
%\|\Delta_{k_1}D^{\sum_{i_1}\beta_1^{i_1}}f_1\|_{p_1}\|f_2\|_{p_2} \ldots \|f_{j_1}\|_{p_{j_1}} + \|D^{\beta_1^2 + \beta_1^3+ \ldots} f_1\|_{p_1} \prod_{} \ldots 
\end{align}
Applying the estimates (\ref{inductive_hyp_1}), (\ref{inductive_hyp_minor}) and (\ref{inductive_hyp_rest}) to (\ref{induction_a}), we deduce that
\begin{align*}
%&\sum_{k_{\max}(v_2) \ll k_1} 2^{k_{\max}(v_2)}2^{k_1\cdot (\beta^{\mathfrak{r}_{\mathcal{G}}} - 1)}\left( \sum_{\restr{\delta}{(\beta^{i_1}_1)_{i_1}}}\|\Delta_{k_1}D^{\delta^{-1}(1)}f_1\|_{p_1}\prod_{\substack{j \in \mathcal{J}_1 \\ j \neq 1}}\|D^{\delta^{-1}(j)}f_{j}\|_{p_j}\right) \cdot \\
%&\ \  \quad \quad \quad \quad \quad \quad \quad \quad \quad \quad \left(\sum_{\substack{\restr{\delta}{(\beta_1^{i_2})_{i_2}}\\ }}\|\Delta_{k_{j_1+1}}D^{\delta^{-1}(j_1+1)}f_{j_1+1}\|_{p_{j_1+1}}\prod_{\substack{j \in \mathcal{J}_2 \\ j \neq j_1+1}}\|D^{\delta^{-1}(j)}f_{j}\|_{p_j} \right)\cdots \\
%& \ \  \quad \quad \quad \quad \quad \quad \quad \quad \quad \quad  \left(\sum_{\substack{\restr{\delta}{(\beta^{i_{n_1}}_1)_{i_{n_1}}}\\ }}\|S_{k_1}D^{\delta^{-1}(j_{n_1-1}+1)}f_{j_{n_1-1}+1}\|_{p_{j_{n_1-1}+1}}\prod_{\substack{j \in \mathcal{J}_{n_1} \\ j \neq j_{n_1-1}+1}}\|D^{\beta_1^{i_{n_1}}}f_{\delta(\beta_1^{i_{n_1}})}\|_{p_{\delta(\beta_1^{i_{n_1}})}} \right) \\
&\sum_{k_{\max}(v_2) \ll k_1}2^{k_{\max}(v_2)}2^{k_1\cdot (\beta^{\mathfrak{r}_{\mathcal{G}}} - 1)} \cdot \|\Delta_{k_1}D^{ \beta(v_1,1)}f_1\|_{p_1} \|\Delta_{k_{\max}(v_2)}D^{ \beta(v_2,l_0)}f_{l_0}\|_{{p_{l_0}}}  \sum_{\substack{\restr{\delta}{\mathcal{V}\setminus \{\mathfrak{r}_{\mathcal{G}}\} \setminus \mathcal{V}^{v_1}_{1} \setminus \mathcal{V}^{v_2}_{l_0}}}}\prod_{l \neq 1, l_0}\|D^{\delta^{-1}(l)}f_{l}\|_{p_l}.
%= &\sum_{\substack{\delta}}\sum_{k_1} \sum_{k_{j_1+1} \leq k_1}2^{k_{j_1+1}}2^{k_1(\beta^0_1 - 1)} \|\Delta_{k_1}D^{\delta^{-1}(1)}f_1\|_{p_1}\|\Delta_{k_{j_1+1}}D^{\delta^{-1}(j_1+1)}f_{j_1+1}\|_{p_{j_1+1}} \ldots \|S_{k_1}D^{\delta^{-1}(j_{n_1-1}+1)}f_{j_{n_1-1}+1}\|_{p_{j_{n_1-1}+1}} \\
%& \quad \quad \quad \quad \quad \cdot  \prod_{\substack{\\ j \neq 1,j_1,\ldots, j_{n_1}+1}}\|D^{\delta^{-1}(j)} f_{j}\|_{p_j}  \\
%\leq &\sum_{\substack{\delta}}\sup_{k_1}\|S_{k_1}D^{\delta^{-1}(j_2+1)}f_{j_2+1}\|_{p_{j_2+1}} \ldots \|S_{k_1}D^{\delta^{-1}(j_{n_1-1}+1)}f_{j_{n_1-1}+1}\|_{p_{j_{n_1-1}+1}} \cdot \prod_{\substack{\\ j \neq 1,j_1+1,\ldots,j_{n_1-1}+1}}\|D^{\delta^{-1}(j)} f_{j}\|_{p_j} \cdot \\ 
%& \sum_{k_1} \sum_{k_{j_1+1} \leq k_1}2^{k_{j_1+1}}2^{k_1(\beta^0_1 - 1)} \|\Delta_{k_1}D^{\delta^{-1}(1)}f_1\|_{p_1}\|\Delta_{k_{j_1+1}}D^{\delta^{-1}(j_1+1)}f_{j_1+1}\|_{p_{j_1+1}} %\\
%\leq &\sum_{\substack{\delta: \\ 1, j_1+1, j_{n_1-1}+1 \\ \in \text{range}(\delta)}} \|f_{j_2+1}\|_{p_{j_2+1}}\ldots \|D^{\delta^{-1}(j_{n_1-1}+1)}f_{j_{n_1-1}+1}\|_{p_{j_{n_1-1}+1}} \cdot \sum_{\substack{\beta^i: \\ \delta(\beta^i)\nin \{1,j_1+1,\ldots,j_{n_1-1}+1\}}}\|D^{\beta^i} f_{\delta(\beta^i)}\|_{p_{\delta(\beta^i)}} \prod_{j\nin \text{range}(\delta)}\|f_j\|_{p_j} \cdot \\ 
%& \sum_{k_1} \sum_{k_{j_1+1} \leq k_1}2^{k_{j_1+1}}2^{k_1(\beta^0_1 - 1)} \|\Delta_{k_1}D^{\delta^{-1}(1)}f_1\|_{p_1}\|\Delta_{k_{j_1+1}}D^{\delta^{-1}(j_1+1)}f_{j_1+1}\|_{p_{j_1+1}} \\
\end{align*}
We notice that
\begin{align*}
& \sum_{k_{\max}(v_2) \ll k_1}2^{k_{\max}(v_2)}2^{k_1\cdot (\beta^{\mathfrak{r}_{\mathcal{G}}} - 1)} \cdot \|\Delta_{k_1}D^{ \beta(v_1,1)}f_1\|_{p_1} \|\Delta_{k_{\max}(v_2)}D^{\beta(v_2,l_0)}f_{l_0}\|_{{p_{l_0}}} \\
= & \sum_{k_{\max}(v_2) \ll k_1}2^{k_{\max}(v_2)\cdot (1-\epsilon)}2^{k_1\cdot (\beta^{\mathfrak{r}_{\mathcal{G}}} - 1)}\|\Delta_{k_1}D^{\beta(v_1,1)}f_1\|_{p_1}2^{k_{\max}(v_2) \cdot \epsilon} \|\Delta_{k_{\max}(v_2)}D^{ \beta(v_2,l_0)}f_{l_0}\|_{{p_{l_0}}}\\
%\lesssim & \sum_{k_1} 2^{k_1(\beta^0_1 - \epsilon_{\beta^0_1})} \|\Delta_{k_1}D^{\delta^{-1}(1)}f_1\|_{p_1}  \|f_{j_1+1}\|_{\dot{B}^{\epsilon_{\beta^0_1}}_{p_{j_1+1},\infty}} \\
\end{align*}
for any $0 < \epsilon < \min (1,  \beta^{\mathfrak{r}_{\mathcal{G}}})$. We can then distribute the derivatives in various ways to bound the above expression by
\begin{equation}
\label{eq:besov:ill:1}
 \sum_{k_1}\min\Big (2^{k_1\beta^{\mathfrak{r}_{\mathcal{G}}}}\|D^{\beta(v_1,1)}f_1\|_{\dot{B}^0_{p_1,\infty}}\|D^{\beta(v_2,l_0)}f_{l_0}\|_{\dot{B}^{0}_{{p_{l_0}},\infty}}, 2^{-k_1\epsilon} \|D^{\beta(v_1,1)}f_1\|_{\dot{B}^{\beta^{\mathfrak{r}_{\mathcal{G}}}}_{p_1,\infty}}\|D^{\beta(v_2,l_0)}f_{l_0}\|_{\dot{B}^{\epsilon}_{{p_{l_0}},\infty}} \Big).
\end{equation}
By further optimizing in $k_1$, this becomes
\begin{equation}
\label{eq:besov:ill:2}
\big(\|D^{\beta(v_1,1)}f_1\|_{\dot{B}^{0}_{p_1,\infty}}\|D^{ \beta(v_2,l_0)}f_{l_0}\|_{\dot{B}^{0}_{{p_{l_0}},\infty}}\big)^{\frac{\epsilon}{\beta^{\mathfrak{r}_{\mathcal{G}}} + \epsilon}}\big(\|D^{\beta(v_1,1)}f_1\|_{\dot{B}^{\beta^{\mathfrak{r}_{\mathcal{G}}}}_{p_1,\infty}}\|D^{ \beta(v_2,l_0)}f_{l_0}\|_{\dot{B}^{\epsilon}_{p_{{l_0}},\infty}}\big)^{\frac{{\beta^{\mathfrak{r}_{\mathcal{G}}}}}{\beta^{\mathfrak{r}_{\mathcal{G}}} + \epsilon}}.
\end{equation}
Now we invoke the interpolation inequality for Besov norms \eqref{eq:interpolation:Besov:epsilon} -- which is essentially a redistribution of the derivatives -- to end up with
\begin{equation}
\label{eq:besov:ill:3}
\begin{aligned}
&\big(\|D^{\beta(v_1,1)}f_1\|_{\dot{B}^{0}_{p_1,\infty}}\|D^{\beta(v_2,l_0)}f_{l_0}\|_{\dot{B}^{\beta^{\mathfrak{r}_{\mathcal{G}}}}_{{p_{l_0}},\infty}}\big)^{\frac{\epsilon}{\beta^{\mathfrak{r}_{\mathcal{G}}} + \epsilon}}\big(\|D^{\beta(v_1,1)}f_1\|_{\dot{B}^{\beta^{\mathfrak{r}_{\mathcal{G}}}}_{p_1,\infty}}\|D^{ \beta(v_2,l_0)}f_{l_0}\|_{\dot{B}^{0}_{p_{{l_0}},\infty}}\big)^{\frac{{\beta^{\mathfrak{r}_{\mathcal{G}}}}}{\beta^{\mathfrak{r}_{\mathcal{G}}} + \epsilon}}\\
%\lesssim & \|D^{\delta^{-1}(1)}f_1\|_{\dot{B}^{0}_{p_1,\infty}}\|D^{\delta^{-1}(j_1+1)}f_{j_1+1}\|_{\dot{B}^{\beta^0_1}_{p_{j_1+1},\infty}} + \|D^{\delta^{-1}(1)}f_1\|_{\dot{B}^{\beta^0_1}_{p_1,\infty}}\|D^{\delta^{-1}(j_1+1)}f_{j_1+1}\|_{\dot{B}^{0}_{p_{j_1+1},\infty}} \\
\lesssim & \|D^{\beta^{\mathfrak{r}_{\mathcal{G}}} + \beta(v_1,1)}f_1\|_{p_1}\|D^{\beta(v_2,l_0)}f_{l_0}\|_{{p_{l_0}}} +  \|D^{ \beta(v_1,1)}f_1\|_{p_1}\|D^{\beta^{\mathfrak{r}_{\mathcal{G}}}+\beta(v_2,l_0)}f_{l_0}\|_{{p_{l_0}}}.
\end{aligned}
\end{equation}
\vskip .15in
\textbf{Estimate of $\mathcal{E}_2$:}
\vskip .1in
\noindent
The multiplier generated by the symbol $|\sum_{l \in \mathcal{L}(v_{i_0})} \xi_l|^{\beta^{\mathfrak{r}_{\mathcal{G}}}}$ localized on the conical region \eqref{eq:cone:ums}  with $l_0 = 1$ is
\begin{align} \label{induction_b}
\sum_{k_1}\Big(D^{\beta^{\mathfrak{r}_{\mathcal{G}}}}T_{\mathcal{G}^{v_1}}\big(\Delta_{k_1}f_1, (S_{k_1}f_l)_{\substack{l \in \mathcal{L}(v_1) \\ l \neq 1}}\big)\Big) \prod_{i=2}^{n_1}T_{\mathcal{G}^{v_i}} \big((S_{k_1} f_l)_{l \in \mathcal{L}(v_i)} \big).
%\cdots \|T_D^{(\beta_1^{i_{n_1}})_{i_{n_1}}}(\Delta_{\geq k_1}f_{j_{n_1-1}+1}\ldots f_{n}) \|_{t_{n_1}}
\end{align}
We first simplify \eqref{induction_b}  using the high-low switch technique discussed in Sections \ref{Bourgain-Li_hilow} and \ref{sec:5lin:1param}. Denote by $\tilde{\mathcal{G}}^{v_1}$ the tree having the same structure as $\mathcal{G}^{v_1}$ with the derivative $\beta^{v_1}$ replaced by $\beta^{v_1}+ \beta^{\mathfrak{r}_{\mathcal{G}}}$ so that \eqref{induction_b} can be rewritten as
\begin{align} \label{induction_b_simplified}
& \sum_{k_1}\big(T_{\tilde{\mathcal{G}}^{v_1}}\big(\Delta_{k_1}f_1, (S_{k_1}f_l)_{\substack{l \in \mathcal{L}(v_1) \\ l \neq 1}}\big)\big) \prod_{i\neq 1}T_{\mathcal{G}^{v_i}} \left((S_{k_1} f_l)_{l \in \mathcal{L}(v_i)} \right) 
\end{align}
We then perform a finer\footnote{Notice that as we perform this step, we also restrict ourselves to certain conical regions associated to each subtree $\mathcal{G}^{v_i}$, for $i \neq 1$.} paraproduct decompositions on the functions in the subtrees $\mathcal{G}^{v_i}$ for $i \neq 1$ so that \eqref{induction_b_simplified} can be written as a finite sum of terms with the following form:
\begin{align}\label{induction_b_LP}
& \sum_{k_1}\Big(T_{\tilde{\mathcal{G}}^{v_1}}\big(\Delta_{k_1}f_1, (S_{k_1}f_l)_{\substack{l \in \mathcal{L}(v_1) \\ l \neq 1}}\big)\Big) \prod_{i\neq 1}\sum_{k_{\max}(v_i)}T_{\mathcal{G}^{v_i}} \Big((S_{k_1}P_{k_{\max}(v_i)} f_l)_{\substack{ \\ l \in \mathcal{L}(v_i)}} \Big), \nonumber \\
= & \sum_{k_1}\Big(T_{\tilde{\mathcal{G}}^{v_1}}\big(\Delta_{k_1}f_1, (S_{k_1}f_l)_{\substack{l \in \mathcal{L}(v_1) \\ l \neq 1}}\big)\Big) \prod_{i\neq 1}\sum_{k_{\max}(v_i) \ll k_1}T_{\mathcal{G}^{v_i}} \left((P_{k_{\max}(v_i)} f_l)_{\substack{\\ l \in \mathcal{L}(v_i)}} \right).
\end{align}
We notice that the equation holds because the conical decomposition on the subtree $\mathcal{G}^{v_2}$ gives $\Delta_{k_{\max}(v_2)}$ for $l \in \mathfrak{M}(R(v_2))$ and $S_{k_1}\Delta_{k_{\max}(v_i)}  \neq\nvdash 0$ if and only if $k_{\max}(v_i) \ll k_1$.
We can now apply the high-low switch to swap the role of $k_1$ and $k_{\max}(v_i)$ and rewrite (\ref{induction_b_LP}) as
\begin{align*}
& \sum_{k_1}\big(T_{\tilde{\mathcal{G}}^{v_1}}\big(\Delta_{k_1}f_1, (S_{k_1}f_l)_{\substack{l \in \mathcal{L}(v_1) \\ l \neq 1}}\big)\big) \prod_{i\neq 1}\sum_{k_{\max}(v_i)}T_{\mathcal{G}^{v_i}} \big((P_{k_{\max}(v_i)} f_l)_{\substack{l \in \mathcal{L}(v_i) \\ }}\big)  \\
&- \sum_{k_{\max}(v_2) \succ k_1}\big(T_{\tilde{\mathcal{G}}^{v_1}}\big(\Delta_{k_1}f_1, (S_{k_1}f_l)_{\substack{l \in \mathcal{L}(v_1) \\ l \neq 1}}\big)\big)T_{\mathcal{G}^{v_2}} \big((P_{k_{\max}(v_2)} f_l)_{\substack{l \in \mathcal{L}(v_2) \\ }}\big) \cdot \prod_{i\neq 1,2}\sum_{k_{\max}(v_i)}T_{\mathcal{G}^{v_i}} \big((P_{k_{\max}(v_i)} f_l)_{\substack{l \in \mathcal{L}(v_i)}} \big)\\
&  \pm \text{similar terms} := I - II \pm \text{similar terms}.
%= & \sum_{k_1}\left(D^{\beta^{\mathfrak{r}_{\mathcal{G}}}}T_{\mathcal{G}^{v_1}}\left(\Delta_{k_1}f_1, (f_l - \Delta_{> k_1}f_l)_{\substack{l \in \mathcal{L}(v_1) \\ l \neq 1}}\right)\right) \prod_{i=2}^{n_1}T_{\mathcal{G}^{v_i}} \left((f_l-\Delta_{>k_1} f_l)_{l \in \mathcal{L}(v_i)} \right) \\
%=  & \sum_{k_1}\left(D^{\beta^{\mathfrak{r}_{\mathcal{G}}}}T_{\mathcal{G}^{v_1}}\left(\Delta_{k_1}f_1, (f_l)_{\substack{l \in \mathcal{L}(v_1) \\ l \neq 1}}\right)\right) \prod_{i=2}^{n_1}T_{\mathcal{G}^{v_i}} \left((f_l)_{l \in \mathcal{L}(v_i)} \right) - \\
%& \sum_{k_1}\left(D^{\beta^{\mathfrak{r}_{\mathcal{G}}}}T_{\mathcal{G}^{v_1}}\left(\Delta_{k_1}f_1, \Delta_{>k_1}f_2, (f_l)_{\substack{l \in \mathcal{L}(v_1) \\ l \neq 1,2}}\right)\right) \prod_{i=2}^{n_1}T_{\mathcal{G}^{v_i}} \left(( f_l)_{l \in \mathcal{L}(v_i)} \right) \pm \\
%& \sum_{k_1}\left(D^{\beta^{\mathfrak{r}_{\mathcal{G}}}}T_{\mathcal{G}^{v_1}}\left(\Delta_{k_1}f_1, (f_l)_{\substack{l \in \mathcal{L}(v_1) \\ l \neq 1}}\right)\right) \prod_{i=2}^{n_1}T_{\mathcal{G}^{v_i}} \left((\Delta_{>k_1} f_l)_{l \in \mathcal{L}(v_i)} \right) \pm \text{similar terms}
\end{align*}
The first term $I$ can be estimated by using inductive hypothesis (\ref{induction_cone_statement}) on
\begin{align*}
 \sum_{k_1}\big(T_{\tilde{\mathcal{G}}^{v_1}}\big(\Delta_{k_1}f_1, (S_{k_1}f_l)_{\substack{l \in \mathcal{L}(v_1) \\ l \neq 1}}\big)\big),
\end{align*}
and on 
\begin{align*}
\sum_{k_{\max}(v_i)}T_{\mathcal{G}^{v_i}} \big((P_{k_{\max}(v_i)} f_l)_{\substack{l \in \mathcal{L}(v_i)}} \big)
\end{align*}
for $i \neq 1$.
The second term $II$ requests a more careful treatment. We recall that
\begin{align*}
T_{\tilde{\mathcal{G}}^{v_1}}\big(\Delta_{k_1}f_1, (S_{k_1}f_l)_{\substack{l \in \mathcal{L}(v_1) \\ l \neq 1}}\big)= D^{\beta^{\mathfrak{r}_{\mathcal{G}}}}T_{\mathcal{G}^{v_1}}\big(\Delta_{k_1}f_1, (S_{k_1}f_l)_{\substack{l \in \mathcal{L}(v_1) \\ l \neq 1}}\big), 
\end{align*}
where $\tilde{\mathcal{G}}^{v_1}$ is a tree with the differential operator associated to the root $v_1$ being 
$$
D^{\beta^{v_1} + \beta^{\mathfrak{r}_{\mathcal{G}}}}.
$$
\begin{comment}
so that we apply the Fourier series decomposition on the symbol $|\sum_{l \in \mathcal{L}(v_1)} \xi_l|^{\beta^{\mathfrak{r}_{\mathcal{G}}}}$ with appropriate localizations to derive 
\begin{align*}
 |\sum_{l \in \mathcal{L}(v_1)} \xi_l|^{\beta^{\mathfrak{r}_{\mathcal{G}}}}= \sum_{L \in \mathbb{Z}}C_L 2^{k_1 \beta^{\mathfrak{r}_{\mathcal{G}}}}e^{2 \pi i \frac{L}{2^{k_1}}\sum_{l \in \mathcal{L}(v_1)}\xi_l}. 
\end{align*}
As a consequence,
\begin{align*}
T_{\tilde{\mathcal{G}}^{v_1}}\left(\Delta_{k_1}f_1, (S_{k_1}f_l)_{\substack{l \in \mathcal{L}(v_1) \\ l \neq 1}}\right) = \sum_{L \in \mathbb{Z}}C_L 2^{k_1 \beta^{\mathfrak{r}_{\mathcal{G}}}}\left(\Delta_{k_1, \frac{L}{2^{k_1}}}f_1, (S_{k_1, \frac{L}{2^{k_1}}}f_l)_{\substack{l \in \mathcal{L}(v_1) \\ l \neq 1}}\right), 
\end{align*}
which yields
\begin{align*}
\left\|T_{\tilde{\mathcal{G}}^{v_1}}\left(\Delta_{k_1, \frac{L}{2^{k_1}}}f_1, (S_{k_1, \frac{L}{2^{k_1}}}f_l)_{\substack{l \in \mathcal{L}(v_1) \\ l \neq 1}}\right)\right\|_{\vec{p^{v_1}}} \lesssim \sum_{L \in \mathbb{Z}} C_L \cdot 2^{k_1 \beta^{\mathfrak{r}_{\mathcal{G}}}}\left\|T_{\mathcal{G}^{v_1}}\left(\Delta_{k_1, \frac{L}{2^{k_1}}}f_1, (S_{k_1, \frac{L}{2^{k_1}}}f_l)_{\substack{l \in \mathcal{L}(v_1) \\ l \neq 1}}\right)  \right\|_{\vec{p^{v_1}}} 
\end{align*}
The estimate for the case when $L = 0$ is illustrative and we will focus on this simpler case from now on. 
\end{comment}

We invoke the inductive hypothesis \eqref{inductive_hyp_1}:
\begin{align*}
\big\|T_{\tilde{\mathcal{G}}^{v_1}}\big(\Delta_{k_1}f_1, (S_{k_1}f_l)_{\substack{l \in \mathcal{L}(v_1) \\ l \neq 1}}\big)  \big\|_{p_{v_1}} \lesssim 2^{k_1\beta^{\mathfrak{r}_{\mathcal{G}}}} \|\Delta_{k_1}D^{\beta(v_1,1)} f_1\|_{p_1}\sum_{\restr{\delta}{(\mathcal{V}_1^{v_1})^c}}\prod_{\substack{l\in \mathcal{L}(v_1) \\ l \neq 1}}\|D^{\delta^{-1}(l)}f_{l}\|_{p_l}.
\end{align*}
When combined with (\ref{inductive_hyp_minor}), we deduce that
\begin{align*}
%& \sum_{k_{\max}(v_2) > k_1}\left(T_{\tilde{\mathcal{G}}^{v_1}}\left(\Delta_{k_1}f_1, (S_{k_1}f_l)_{\substack{l \in \mathcal{L}(v_1) \\ l \neq 1}}\right)\right)T_{\mathcal{G}^{v_2}} \left((P_{k_{\max}(v_2)} f_l)_{\substack{l \in \mathcal{L}(v_2)}} \right) \cdot \prod_{i\neq 1,2}\sum_{k_{\max}(v_i)}T_{\mathcal{G}^{v_i}} \left((P_{k_{\max}(v_i)} f_l)_{\substack{l \in \mathcal{L}(v_i)}} \right) \\
\|II\|_{r}\lesssim &  \sum_{ k_{\max}(v_2) \succ k_1}2^{k_1 \beta^{\mathfrak{r}_{\mathcal{G}}}}\|\Delta_{k_1}D^{\beta(v_1,1)} f_1\|_{p_1}\|\Delta_{k_{\max}(v_2)}D^{ \beta(v_2,l_0)}f_{l_0}\|_{{p_{l_0}}}\sum_{\restr{\delta}{(\mathcal{V}_1^{v_1})^c \cup (\mathcal{V}_{l_0}^{v_2})^c}}\prod_{\substack{\\ l \neq 1, l_0}}\|D^{\delta^{-1}(l)}f_{l}\|_{p_l}.
\end{align*}
A similar computation specified in Section \ref{sec:5lin:1param} yields
\begin{align}
 &\sum_{k_{\max}(v_2) \succ k_1}2^{k_1 \beta^{\mathfrak{r}_{\mathcal{G}}}}\|\Delta_{k_1}D^{\beta(v_1,1)} f_1\|_{p_1}\|\Delta_{k_{\max}(v_2)}D^{ \beta(v_2,l_0)}f_{l_0}\|_{p_{l_0}} \lesssim \sum_{k_{\max}(v_2)}2^{k_{\max}(v_2) \beta^{\mathfrak{r}_{\mathcal{G}}}}\|D^{\beta(v_1,1)} f_1\|_{\dot{B}^0_{p_1}}\|D^{ \beta(v_2,l_0)}f_{l_0}\|_{\dot{B}^0_{p_{l_0}}}. \label{induction_hilow}%\\
 %\lesssim & \sum_{k_{\max}(v_2)}2^{k_{\max}(v_2) \beta^{\mathfrak{r}_{\mathcal{G}}}}\|D^{\beta(v_1,1)} f_1\|_{\dot{B}^0_{p_1}}\|D^{ \beta(v_2,l_0)}f_{l_0}\|_{\dot{B}^0_{p_{l_0}}}.\nonumber
 \end{align}
Meanwhile, (\ref{induction_hilow}) can also be estimated by
 \begin{align*}
 & \sum_{k_{\max}(v_2) \succ k_1}2^{k_1 (\beta^{\mathfrak{r}_{\mathcal{G}}}- \epsilon)}2^{-k_{\max}(v_2) \beta^{\mathfrak{r}_{\mathcal{G}}}} \big(2^{k_1 \epsilon}\|\Delta_{k_1}D^{\beta(v_1,1)} f_1\|_{p_1}\big) \big(2^{k_{\max}(v_2) \beta^{\mathfrak{r}_{\mathcal{G}}}}\|\Delta_{k_{\max}(v_2)}D^{\beta(v_2,l_0)}f_{l_0}\|_{p_{l_0}}\big)\\
\qquad \lesssim & \sum_{k_{\max}(v_2)}2^{-k_{\max}(v_2) \epsilon} \|D^{\beta(v_1,1)} f_1\|_{\dot{B}^\epsilon_{p_1}}\|D^{\beta(v_2,l_0)}f_{l_0}\|_{\dot{B}^{\beta^{\mathfrak{r}_{\mathcal{G}}}}_{p_{l_0}}}.
\end{align*}
The optimization and interpolation can be applied to conclude that $\|II\|_{r}$ is bounded above by
\begin{align*}
\big(\|D^{\beta^{\mathfrak{r}_{\mathcal{G}}} + \beta(v_1,1)}f_1\|_{p_1}\|D^{\beta(v_2,l_0)}f_{l_0}\|_{p_{l_0}} + \|D^{\beta(v_1,1)}f_1\|_{p_1}\|D^{\beta^{\mathfrak{r}_{\mathcal{G}}} + \beta(v_2,l_0)}f_{l_0}\|_{p_{l_0}} \big)\sum_{\restr{\delta}{(\mathcal{V}_1^{v_1})^c \cup (\mathcal{V}_{l_0}^{v_2})^c}}\prod_{\substack{l\in \mathcal{L}(v_1) \\ l \neq 1, l_0}}\|D^{\delta^{-1}(l)}f_{l}\|_{p_l}.
\end{align*}
We thus have arrived at the expression \eqref{induction_cone_statement} claimed in the inductive statement. With this, we end the proof of Proposition \ref{prop:1param:ind:statements}, if $R$ is of the form \eqref{eq:cone:ums}.

\vskip .15in
\item
\textbf{Case 2: The conical region $R$ is of the form \eqref{eq:cone:nonums}.}\\
We notice that (\ref{induction_cone_statement}) is a direct consequence of (\ref{induction_whitney_cube_statement-2max}) for the given tree $\mathcal{G}$. In particular, we apply (\ref{induction_whitney_cube_statement-2max}) on the union of Whitney cubes at a fixed scale to derive 
\begin{align}\label{two_maximal_final}
& \big\|\sum_{k_{\max}}T_{\mathcal{G}}\left((P_{k_{\max}}f_l)_{1 \leq l \leq n}\right) \big\|_r \lesssim  \sum_{k_{\max}}\left\|T_{\mathcal{G}}\left((P_{k_{\max}}f_l)_{1 \leq l \leq n}\right)\right\|_{r} \nonumber\\
\lesssim& \sum_{k_{\max}} 2^{k_{\max}\cdot\beta(\mathfrak{r}_{\mathcal{G}}, v^{l_1, l_2})}\|\Delta_{k_{\max}}D^{\beta(w^{l_1}, l_1)}f_{l_1}\|_{p_{l_1}}\|\Delta_{k_{\max}}D^{\beta(w^{l_2}, l_2)}f_{l_2}\|_{p_{l_2}}\sum_{\restr{\delta}{(\mathcal{V}_{l_1}^{\mathfrak{r}_{\mathcal{G}}})^c \setminus \mathcal{V}^{w^{l_2}}_{l_2}}}\prod_{\substack{ \\ l \neq l_1, l_2}}\|D^{\delta^{-1}(l)}f_{l}\|_{p_l}.
\end{align}
We then distribute the derivatives as before:
\begin{align}\label{two_maximal_optimized}
& \sum_{k_{\max}} 2^{k_{\max}\cdot\beta(\mathfrak{r}_{\mathcal{G}}, v^{l_1, l_2})}\|\Delta_{k_{\max}}D^{\beta(w^{l_1}, l_1)}f_{l_1}\|_{p_{l_1}}\|\Delta_{k_{\max}}D^{\beta(w^{l_2}, l_2)}f_{l_2}\|_{p_{l_2}} \nonumber\\
 \leq & \sum_{k_{\max}} \min\big(2^{k_{\max} \beta(\mathfrak{r}_{\mathcal{G}}, v^{l_1,l_2})}\|D^{\beta(w^{l_1},l_1)}f_{l_1}\|_{\dot{B}^0_{p_{l_1}}}\| \|D^{\beta(w^{l_2},l_2)} f_{l_2}\|_{\dot{B}^0_{p_{l_2}}}, \nonumber \\
 &\ \ \quad \quad \quad \quad  2^{-k_{\max} \beta(\mathfrak{r}_{\mathcal{G}}, v^{l_1,l_2})}\|D^{\beta(w^{l_1},l_1)}f_{l_1}\|_{\dot{B}^{\beta(\mathfrak{r}_{\mathcal{G}}, v^{l_1,l_2})}_{p_{l_1}}} \|D^{\beta(w^{l_2},l_2)} f_{l_2}\|_{\dot{B}^{\beta(\mathfrak{r}_{\mathcal{G}}, v^{l_1,l_2})}_{p_{l_2}}}\big) \nonumber\\
 \lesssim & \|D^{\beta(\mathfrak{r}_{\mathcal{G}},l_1)}f_{l_1}\|_{p_{l_1}} \| D^{\beta(w^{l_2},l_2)} f_{l_2}\|_{p_{l_2}} + \|D^{\beta(w^{l_1},l_1)}f_{l_1}\|_{p_{l_1}} \| D^{\beta(\mathfrak{r}_{\mathcal{G}},l_2)} f_{l_2}\|_{p_{l_2}}.
\end{align}
We plug (\ref{two_maximal_optimized}) into (\ref{two_maximal_final}) and obtain the estimate (\ref{induction_cone_statement}) claimed in the inductive statement.

\end{enumerate}
\end{enumerate}
\end{proof}

\subsection{Bi-parameter flag Leibniz rule}\label{subsection_bi_leibniz}
We will extend our inductive argument to bi-parameter flag Leibniz rules of arbitrary complexity in dimension\footnote{As mentioned in the beginning of Section \ref{generic_induction}, the methods are adaptable to higher dimensions, in a straightforward manner.} one. We follow the same notation as before, except for the addition of subscription to indicate which parameter is involved. For example, for any $v \in \mathcal{V}$, $\beta^v_j$ represents the derivative for the $j$-th parameter (for $j = 1,2$). As in the one-parameter setting, the frequency space for each parameter can be decomposed into conical regions of the form \eqref{eq:cone:ums} or \eqref{eq:cone:nonums}. We correspondingly define the maps $\mathfrak{M}_j$ and $\mathfrak{m}_j$ for $j = 1, 2$ on the collection of conical regions for the $j$-th parameter:
{\fontsize{9}{9}
\begin{align*}
\begin{matrix}
\mathfrak{M}_1(R):= \{1 \leq l \leq n:(\xi_1, \ldots, \xi_n) \in R,\ \ |\xi_l| \sim \max_{1 \leq l' \leq n} |\xi_{l'}|\},  & \mathfrak{m}_1(R) :=  \{1 \leq l \leq n: (\xi_1, \ldots, \xi_n) \in R,\ \ |\xi_l| \ll \max_{1 \leq l' \leq n} |\xi_{l'}|\},\\
\mathfrak{M}_2(R'):= \{1 \leq l \leq n: (\eta_1, \ldots, \eta_n) \in R',\ \ |\eta_l| \sim \max_{1 \leq l' \leq n} |\eta_{l'}|\},    & \mathfrak{m}_2(R') :=  \{1 \leq l \leq n: (\eta_1, \ldots, \eta_n) \in R',\ \ |\eta_l| \ll \max_{1 \leq l' \leq n} |\eta_{l'}|\},
\end{matrix}
\end{align*}}where $R$ denotes a conical region for the first parameter and $R'$ a conical region for the second parameter.

Fix any integer $k$; we define the projections $P^{(j)}_k$ ($j = 1,2$) depending on $l$ as follows:
\begin{equation*}
P^{(1)}_k(l) := 
\begin{cases}
\Delta^{(1)}_k  \ \ \text{if} \ \  l \in \mathfrak{M}_1(R) \\
S^{(1)}_k  \ \ \text{if} \ \  l \in \mathfrak{m}_1(R),
\end{cases} \qquad \text{and} \qquad P^{(2)}_{k}(l) := 
\begin{cases}
\Delta^{(2)}_{k}  \ \ \text{if} \ \  l \in \mathfrak{M}_2(R') \\
S^{(2)}_{k}  \ \ \text{if} \ \  l \in \mathfrak{m}_2(R').
\end{cases}
\end{equation*}
Similarly, we define $P^{(j)}_{k,L}$ ($j = 1,2$) for fixed $k \in \mathbb{Z}$ and $L \in \BBR$ by
\begin{equation*}
P^{(1)}_{k,L}(l) := 
\begin{cases}
\Delta^{(1)}_{k,L}  \ \ \text{if} \ \  l \in \mathfrak{M}_1(R) \\
S^{(1)}_{k,L}  \ \ \text{if} \ \  l \in \mathfrak{m}_1(R),
\end{cases} \qquad \text{and} \qquad P^{(2)}_{k,L}(l) := 
\begin{cases}
\Delta^{(2)}_{k,L}  \ \ \text{if} \ \  l \in \mathfrak{M}_2(R') \\
S^{(2)}_{k,L}  \ \ \text{if} \ \  l \in \mathfrak{m}_2(R').
\end{cases}
\end{equation*}
Let $k_{\max}$ and $m_{\max}$ specify the Whitney cubes at scales $k_{\max}$ and $m_{\max}$ in the cone $R$ for the first parameter and $R'$ for the second parameter. 
We can thus express the multilinear expression localized to a union of Whitney rectangles\footnote{In this case, a Whitney rectangle is simply the product of two Whitney cubes, one in each parameter.} at the scale $k_{\max} \times m_{\max}$ in the conical region $R \times R'$ as 
$$
T_{\mathcal{G}}^{R\times R'}\big((P^{(1)}_{k_{\max}}P^{(2)}_{m_{\max}} f_l)_{1 \leq l \leq n}\big).
$$
Since we will always focus on a certain conical region, the above expression will be abbreviated as  
$$
T_{\mathcal{G}}\big((P^{(1)}_{k_{\max}}P^{(2)}_{m_{\max}} f_l)_{1 \leq l \leq n}\big).
$$
The multilinear expression restricted to this cone can then be written as
$$
\sum_{k_{\max}, m_{\max} \in \mathbb{Z}}T_{\mathcal{G}}^{}\big((P^{(1)}_{k_{\max}}P^{(2)}_{m_{\max}} f_l)_{1 \leq l \leq n}\big).
$$

Similarly to the one-parameter analysis, the cone decomposition can also be applied to a subtree $\mathcal{G}^v$ with $v \in \mathcal{V}$ for both parameters and let $k_{\max}(v)$ and $m_{\max}(v)$ specify the Whitney cubes in these conical regions for $\mathcal{G}^v$. We will follow the abbreviation in the one-parameter setting so that  $k_{\max}$ and $m_{\max}$ refer to $k_{\max}(\mathfrak{r}_{\mathcal{G}})$ and $m_{\max}(\mathfrak{r}_{\mathcal{G}})$ respectively.

Define the sum of partial derivatives from the vertex $v$ to the leaf $f_l$ by 
\begin{align*}
 \beta_j(v, l):= \sum_{w \in \mathcal{V}^{v}_{l}} \beta_j^w, \ \  \text{for}\ \  j = 1,2.
\end{align*}
With some abuse of notation, if $v = f_l$ is a leaf, then $ \beta_j(v, l) = 0$.

Our goal is to prove bi-parameter versions of the inductive statements in Proposition \ref{prop:1param:ind:statements}. The statements \eqref{prop:2param:ind:1}, \eqref{prop:2param:ind:2} and \eqref{prop:2param:ind:3} below describe estimates for the multi-linear expression $T_{\mathcal{G}}(P^{(1)}_{k_{\max}}P^{(2)}_{m_{\max}} f_l)_{1 \leq l \leq n }$ on a union of Whitney rectangles (the product of two Whitney cubes) at a fixed scale $k_{\max} \times m_{\max}$ localized to a conical region (the product of two conical regions, one for each parameter); several possibilities need to be investigated:
\begin{itemize}
\item
the Whitney cubes can be in any conical region for both parameters;
\item
for at least one parameter its Whitney cube is located in a ``diagonal'' conical region;
\item
the Whitney cubes are in diagonal conical regions for both parameters. 
\end{itemize}

One observes that  \eqref{prop:2param:ind:1} - \eqref{prop:2param:ind:3} impose  conditions from weak to strong so that  \eqref{prop:2param:ind:3} implies  \eqref{prop:2param:ind:2} and  \eqref{prop:2param:ind:2} leads to  \eqref{prop:2param:ind:1}. The statements  \eqref{prop:2param:ind:4} and  \eqref{prop:2param:ind:5} concern Leibniz rules when the frequency space for one parameter is localized on a conical region and on a union of Whitney cubes at a fixed scale for the other parameter.  \eqref{prop:2param:ind:4} corresponds to the case when the Whitney cubes for one parameter are in any conical region while  \eqref{prop:2param:ind:5} describes the case when the Whitney cubes lie in a diagonal conical region.  \eqref{prop:2param:ind:6} is the Leibniz rule when the frequency spaces for both parameters are localized on a conical region.

\begin{proposition} \label{prop:biparameter_induction}
Suppose that all the Lebesgue exponents in the inductive statement satisfy the condition described in Theorem \ref{thm:main} and that $T_{\mathcal G}$ is restricted to the cone denoted by $R$ for the first parameter and $R'$ for the second parameter. 
\begin{enumerate}[leftmargin=*]
\item \label{prop:2param:ind:1}
Suppose that $l_0 \in \mathfrak{M}_1(R)$ and $l'_0 \in  \mathfrak{M}_2(R')$.  %Fix $k_{\max}$ and $m_{\max}$, there holds:
\begin{enumerate}
\item
If $l_0 = \l'_0$, then
\begin{align} \label{induction_2parameter_whitney_one}
%& \Bigg\|T_{\mathcal{G}}\bigg((\Delta^{(1)}_{k_{\max}}\Delta^{(2)}_{m_{\max}}f_{l})_{\substack{\\ l \in \mathfrak{M}_1(R)} \cap  \mathfrak{M}_2(R')}, (\Delta^{(1)}_{k_{\max}}S^{(2)}_{m_{\max}}f_{l})_{\substack{\\ l \in \mathfrak{M}_1(R)} \cap  \mathfrak{M}_2(R')}, \nonumber \\
%& \quad \quad (S^{(1)}_{k_{\max}}\Delta^{(2)}_{m_{\max}}f_{l})_{\substack{\\ l \in \mathfrak{M}_1(R)} \cap  \mathfrak{M}_2(R')}, (S^{(1)}_{k_{\max}}S^{(2)}_{m_{\max}} f_l)_{\substack{ \\ l \in \mathfrak{M}_1(R)}\cap  \mathfrak{M}_2(R')}\bigg)\Bigg\|_{\vec{r}} \nonumber\\
& \|T_{\mathcal{G}}(P^{(1)}_{k_{\max}}P^{(2)}_{m_{\max}} f_l)_{1 \leq l \leq n }\|_{\vec{r}} \lesssim   2^{k_{\max}\cdot\beta_1(\mathfrak{r}_{\mathcal{G}}, l_0)}2^{m_{\max}\cdot\beta_2(\mathfrak{r}_{\mathcal{G}}, l_0)}\|\Delta^{(1)}_{k_{\max}}\Delta^{(2)}_{m_{\max}}f_{l_0}\|_{\vec{p}_{l_0}}\sum_{\restr{\delta_1 \otimes \delta_2}{(\mathcal{V}_{l_0}^{\mathfrak{r}_{\mathcal{G}}})^c}}\prod_{\substack{ \\ l \neq l_0}}\|D_{(1)}^{\delta_1^{-1}(l)}D_{(2)}^{\delta_2^{-1}(l)}f_{l}\|_{\vec{p_l}}.
\end{align}
\item
If $l_0 \neq l'_0$, then
\begin{align} \label{induction_2parameter_whitney_two}
%& \Bigg\|T_{\mathcal{G}}\bigg((\Delta^{(1)}_{k_{\max}}\Delta^{(2)}_{m_{\max}}f_{l})_{\substack{\\ l \in \mathfrak{M}_1(R)} \cap  \mathfrak{M}_2(R')}, (\Delta^{(1)}_{k_{\max}}S^{(2)}_{m_{\max}}f_{l})_{\substack{\\ l \in \mathfrak{M}_1(R)} \cap  \mathfrak{M}_2(R')}, \nonumber\\
%& \quad \quad (S^{(1)}_{k_{\max}}\Delta^{(2)}_{m_{\max}}f_{l})_{\substack{\\ l \in \mathfrak{M}_1(R)} \cap  \mathfrak{M}_2(R')}, (S^{(1)}_{k_{\max}}S^{(2)}_{m_{\max}} f_l)_{\substack{ \\ l \in \mathfrak{M}_1(R)}\cap  \mathfrak{M}_2(R')}\bigg)\Bigg\|_{\vec{r}} \nonumber\\
\|T_{\mathcal{G}}(P^{(1)}_{k_{\max}}P^{(2)}_{m_{\max}} f_l)_{1 \leq l \leq n }\|_{\vec{r}} 
&\lesssim  2^{k_{\max}\cdot\beta_1(\mathfrak{r}_{\mathcal{G}}, l_0)}2^{m_{\max}\cdot\beta_2(\mathfrak{r}_{\mathcal{G}}, l'_0)}  \nonumber\\&\cdot \sum_{\substack{\restr{\delta_1}{(\mathcal{V}_{l_0}^{\mathfrak{r}_{\mathcal{G}}})^c} \\ \restr{\delta_2}{(\mathcal{V}_{l'_0}^{\mathfrak{r}_{\mathcal{G}}})^c}}} \|\Delta^{(1)}_{k_{\max}}D_{(2)}^{\delta_2^{-1}(l_0)}f_{l_0}\|_{\vec{p}_{l_0}}\|\Delta^{(2)}_{m_{\max}}D_{(1)}^{\delta_1^{-1}(l'_0)}f_{l'_0}\|_{\vec{p}_{l'_0}} \prod_{\substack{ \\ l \neq l_0, l'_0}}\|D_{(1)}^{\delta_1^{-1}(l)}D_{(2)}^{\delta_2^{-1}(l)}f_{l}\|_{\vec{p_l}}.
\end{align}
\end{enumerate}
\item  \label{prop:2param:ind:2}
Suppose that $l_1, l_2 \in \mathfrak{M}_1(R)$ with $l_1 \neq l_2$ and $l_0' \in  \mathfrak{M}_2(R')$.
\begin{enumerate}
\item \label{prop:2param:ind:2:a}
If $l_1 = l_0'$, then
\begin{align}\label{induction_2parameter_whitney_2:1}
 \|T_{\mathcal{G}}(P^{(1)}_{k_{\max}}P^{(2)}_{m_{\max}} f_l)_{1 \leq l \leq n }\|_{\vec{r}} \lesssim &2^{k_{\max}\beta_1(\mathfrak{r}_{\mathcal{G}}, v^{l_1,l_2})}2^{m_{\max}\beta_2(\mathfrak{r}_{\mathcal{G}}, l_1)}\|\Delta^{(1)}_{k_{\max}}\Delta^{(2)}_{m_{\max}}D_{(1)}^{\beta_1(w^{l_1},l_1)}f_{l_1}\|_{\vec{p}_{l_1}} \nonumber\\
 & \cdot \sum_{\substack{\restr{\delta_1}{(\mathcal{V}_{l_1}^{\mathfrak{r}_{\mathcal{G}}})^c\setminus \mathcal{V}^{w^{l_2}}_{l_2}} \\ \restr{\delta_2}{(\mathcal{V}_{l_1}^{\mathfrak{r}_{\mathcal{G}}})^c}}}\|\Delta^{(1)}_{k_{\max}}D_{(1)}^{\beta_1(w^{l_2},l_2)}D_{(2)}^{\delta_2^{-1}(l_2)}f_{l_2}\|_{\vec{p}_{l_2}}\prod_{\substack{ \\ l \neq l_1,l_2}}\|D_{(1)}^{\delta_1^{-1}(l)}D_{(2)}^{\delta_2^{-1}(l)}f_{l}\|_{\vec{p_l}}.
\end{align}
\item \label{prop:2param:ind:2:b}
If $l_1, l_2 \neq l_0'$, then
\begin{align*}
  \|T_{\mathcal{G}}(P^{(1)}_{k_{\max}} & P^{(2)}_{m_{\max}} f_l)_{1 \leq l \leq n }\|_{\vec{r}} \lesssim  2^{k_{\max}\beta_1(\mathfrak{r}_{\mathcal{G}}, v^{l_1,l_2})}2^{m_{\max}\beta_2(\mathfrak{r}_{\mathcal{G}}, l_0')} \prod_{\substack{ \\ l \neq l_1,l_2}}\|D_{(1)}^{\delta_1^{-1}(l)}D_{(2)}^{\delta_2^{-1}(l)}f_{l}\|_{\vec{p_l}} \\
  & \cdot\sum_{\substack{\restr{\delta_1}{(\mathcal{V}_{l_1}^{\mathfrak{r}_{\mathcal{G}}})^c\setminus \mathcal{V}^{w^{l_2}}_{l_2}} \\ \restr{\delta_2}{(\mathcal{V}_{l_0'}^{\mathfrak{r}_{\mathcal{G}}})^c}}}\|\Delta^{(1)}_{k_{\max}}D_{(1)}^{\beta_1(w^{l_1},l_1)}D_{(2)}^{\delta_2^{-1}(l_1)}f_{l_1}\|_{\vec{p}_{l_1}} \|\Delta^{(1)}_{k_{\max}}D_{(1)}^{\beta_1(w^{l_2},l_2)}D_{(2)}^{\delta_2^{-1}(l_2)}f_{l_2}\|_{\vec{p}_{l_2}}\|\Delta^{(2)}_{m_{\max}}D_{(1)}^{\delta_1^{-1}(l_0')}f_{l_0'}\|_{p_{l_0'}}.
\end{align*}
\end{enumerate}
\item \label{prop:2param:ind:3}
Suppose that $l_1, l_2 \in \mathfrak{M}_1(R)$ with $l_1 \neq l_2$ and $l_1', l_2' \in  \mathfrak{M}_2(R')$ with $l_1'\neq l_2'$.
\begin{enumerate}
\item
If $l_1 = l_1'$ and $l_2 = l_2'$, then
%{\fontsize{8.5}{8.5}
\begin{align*}
 \|T_{\mathcal{G}}(P^{(1)}_{k_{\max}}P^{(2)}_{m_{\max}} f_l)_{1 \leq l \leq n }\|_{\vec{r}} \lesssim & 2^{k_{\max}\beta_1(\mathfrak{r}_{\mathcal{G}}, v^{l_1,l_2})}2^{m_{\max}\beta_2(\mathfrak{r}_{\mathcal{G}}, v^{l_1,l_2})}  \sum_{\substack{\restr{\delta_1 \otimes \delta_2}{(\mathcal{V}_{l_1}^{\mathfrak{r}_{\mathcal{G}}})^c\setminus \mathcal{V}^{w^{l_2}}_{l_2}} \\ }}\prod_{\substack{ \\ l \neq l_1,l_2}}\|D_{(1)}^{\delta_1^{-1}(l)}D_{(2)}^{\delta_2^{-1}(l)}f_{l}\|_{\vec{p_l}} \\
 &\cdot \|\Delta^{(1)}_{k_{\max}}\Delta^{(2)}_{m_{\max}}D_{(1)}^{\beta_1(w^{l_1},l_1)}D_{(2)}^{\beta_2(w^{l_1},l_1)}f_{l_1}\|_{\vec{p}_{l_1}}\|\Delta^{(1)}_{k_{\max}}\Delta^{(2)}_{m_{\max}}D_{(1)}^{\beta_1(w^{l_2},l_2)}D_{(2)}^{\beta_2(w^{l_2},l_2)}f_{l_2}\|_{\vec{p}_{l_2}}.
\end{align*}%}
\item 
If $l_1 = l_1'$ and $l_2 \neq l_2'$, then
\begin{align*}
& \|T_{\mathcal{G}}(P^{(1)}_{k_{\max}}P^{(2)}_{m_{\max}} f_l)_{1 \leq l \leq n }\|_{\vec{r}} \lesssim  2^{k_{\max}\beta_1(\mathfrak{r}_{\mathcal{G}}, v^{l_1,l_2})}2^{m_{\max}\beta_2(\mathfrak{r}_{\mathcal{G}}, v^{l_1,l_2'})}\|\Delta^{(1)}_{k_{\max}}\Delta^{(2)}_{m_{\max}}D_{(1)}^{\beta_1(w^{l_1},l_1)}D_{(2)}^{\beta_2(w^{l_1},l_1)}f_{l_1}\|_{\vec{p}_{l_1}} \\
& \qquad \cdot\sum_{\substack{\restr{\delta_1 }{(\mathcal{V}_{l_1}^{\mathfrak{r}_{\mathcal{G}}})^c\setminus \mathcal{V}^{w^{l_2}}_{l_2}} \\ \restr{\delta_2}{(\mathcal{V}_{l_1}^{\mathfrak{r}_{\mathcal{G}}})^c\setminus \mathcal{V}^{w^{l_2'}}_{l_2'}} }}\|\Delta^{(1)}_{k_{\max}}D_{(1)}^{\beta_1(w^{l_2},l_2)}D_{(2)}^{\delta_2^{-1}(l_2)}f_{l_2}\|_{\vec{p}_{l_2}} \|\Delta^{(2)}_{m_{\max}}D_{(1)}^{\delta_1^{-1}(l_2')}D_{(2)}^{\beta_2(w^{l_2'},l_2')} f_{l_2'}\|_{p_{l_2'}}
\prod_{\substack{ \\ l \neq l_1,l_2, l_2'}}\|D_{(1)}^{\delta_1^{-1}(l)}D_{(2)}^{\delta_2^{-1}(l)}f_{l}\|_{\vec{p_l}}.
\end{align*}
\item
If $l_1 \neq l_1'$ and $l_2 \neq l_2'$, then
\begin{align*}
\|T_{\mathcal{G}}(P^{(1)}_{k_{\max}}P^{(2)}_{m_{\max}} f_l)_{1 \leq l \leq n }\|_{\vec{r}}
\lesssim & 2^{k_{\max}\beta_1(\mathfrak{r}_{\mathcal{G}}, v^{l_1,l_2})}2^{m_{\max}\beta_2(\mathfrak{r}_{\mathcal{G}}, v^{l'_1,l_2'})} \prod_{\substack{ \\ l \neq l_1,l_2, l_1', l_2'}}\|D_{(1)}^{\delta_1^{-1}(l)}D_{(2)}^{\delta_2^{-1}(l)}f_{l}\|_{\vec{p_l}} \\
& \cdot \sum_{\substack{\restr{\delta_1 }{(\mathcal{V}_{l_1}^{\mathfrak{r}_{\mathcal{G}}})^c\setminus \mathcal{V}^{w^{l_2}}_{l_2}} 
\\ \restr{\delta_2}{(\mathcal{V}_{l_1'}^{\mathfrak{r}_{\mathcal{G}}})^c\setminus \mathcal{V}^{w^{l_2'}}_{l_2'}} }}\|\Delta^{(1)}_{k_{\max}}D_{(1)}^{\beta_1(w^{l_1},l_1)}D_{(2)}^{\delta_2^{-1}(l_1)}f_{l_1}\|_{\vec{p}_{l_1}} \|\Delta^{(2)}_{m_{\max}}D_{(1)}^{\delta_1^{-1}(l_1')}D_{(2)}^{\beta_2(w^{l_1'},l_1')} f_{l_1'}\|_{p_{l_1'}}\\
&\cdot \|\Delta^{(1)}_{k_{\max}}D_{(1)}^{\beta_1(w^{l_2},l_2)}D_{(2)}^{\delta_2^{-1}(l_2)}f_{l_2}\|_{\vec{p}_{l_2}} \|\Delta^{(2)}_{m_{\max}}D_{(1)}^{\delta_1^{-1}(l_2')}D_{(2)}^{\beta_2(w^{l_2'},l_2')} f_{l_2'}\|_{p_{l_2'}}.
\end{align*}
\end{enumerate}
\item \label{prop:2param:ind:4}
Suppose that $l_0 \in  \mathfrak{M}_1(R)$. 
Then
\begin{align} \label{induction_2parameter_fix1scale}
& \Big\|\sum_{m_{\max} \in \mathbb{Z}}T_{\mathcal{G}}(P^{(1)}_{k_{\max}}P^{(2)}_{m_{\max}} f_l)_{1 \leq l \leq n }\Big\|_{\vec{r}} \lesssim  2^{k_{\max}\cdot\beta_1(\mathfrak{r}_{\mathcal{G}}, l_0)}\sum_{\substack{\delta_2 \\ \restr{\delta_1}{(\mathcal{V}_{l_0}^{\mathfrak{r}_{\mathcal{G}}})^c}}}\|\Delta^{(1)}_{k_{\max}}D^{\delta_2^{-1}(l_0)}_{(2)}f_{l_0}\|_{\vec{p}_{l_0}}\prod_{\substack{ \\ l \neq l_0}}\|D_{(1)}^{\delta_1^{-1}(l)}D_{(2)}^{\delta_2^{-1}(l)}f_{l}\|_{\vec{p_l}}.
\end{align}
\item \label{prop:2param:ind:5}
Suppose that $l_1, l_2 \in  \mathfrak{M}_1(R)$ with $l_1 \neq l_2$. Then 
\begin{align*}
 \Big\|\sum_{m_{\max}\in \mathbb{Z}}T_{\mathcal{G}}(P^{(1)}_{k_{\max}}P^{(2)}_{m_{\max}} f_l)_{1 \leq l \leq n }\Big\|_{\vec{r}}  \lesssim &  2^{k_{\max}\cdot\beta_1(\mathfrak{r}_{\mathcal{G}}, v^{l_1,l_2})} \prod_{\substack{ \\ l \neq l_1,l_2}}\|D_{(1)}^{\delta_1^{-1}(l)}D_{(2)}^{\delta_2^{-1}(l)}f_{l}\|_{\vec{p_l}}\\ 
 & \cdot \sum_{\substack{\delta_2 \\ \restr{\delta_1}{(\mathcal{V}_{l_1}^{\mathfrak{r}_{\mathcal{G}}})^c \setminus \mathcal{V}^{w^{l_2}}_{l_2}}}}\|\Delta^{(1)}_{k_{\max}}D_{(1)}^{\beta_1(w^{l_1},l_1)}D^{\delta_2^{-1}(l_1)}_{(2)}f_{l_1}\|_{\vec{p}_{l_1}}\|\Delta^{(1)}_{k_{\max}}D_{(1)}^{\beta_1(w^{l_2},l_2)}D^{\delta_2^{-1}(l_2)}_{(2)}f_{l_2}\|_{\vec{p}_{l_2}}.
\end{align*}
\item \label{prop:2param:ind:6}
If we sum over the whole conical regions, we have
\begin{align}\label{induction_2parameter_cone_statement}
& \Big\|\sum_{\substack{k_{\max}\in \mathbb{Z} \\ m_{\max}\in \mathbb{Z}}}T_{\mathcal{G}}(P^{(1)}_{k_{\max}}P^{(2)}_{m_{\max}} f_l)_{1 \leq l \leq n }\Big\|_{\vec{r}} \lesssim  \sum_{\delta}\prod_{l}\|D_{(1)}^{\delta_1^{-1}(l)}D_{(2)}^{\delta_2^{-1}(l)}f_l\|_{\vec{p_l}}.
\end{align}
\end{enumerate}
\end{proposition}
\begin{remark}
\textbf{(i)}
Observe that the cases summarized in  \eqref{prop:2param:ind:2} and  \eqref{prop:2param:ind:3} are not exhaustive but typical: other cases not explicitly stated (such as $l_2 = l_0'$ in  \eqref{prop:2param:ind:2} and $l_1= l_2'$, $l_2 \neq l_1'$ in  \eqref{prop:2param:ind:3} can be estimated analogously.\\
\noindent
\textbf{(ii)}
Due to symmetry, the term
\begin{align*}
& \sum_{k_{\max}\in \mathbb{Z}}T_{\mathcal{G}}(P^{(1)}_{k_{\max}}P^{(2)}_{m_{\max}} f_l)_{1 \leq l \leq n }
\end{align*}
is to be treated in the same way as statements \eqref{prop:2param:ind:4} or  \eqref{prop:2param:ind:5} -- based on the structure of the conical regions.\\
\textbf{(iii)}
We realize that the semi-localized operator
\begin{align}\label{induction_2parameter_semilocal_to_est}
& T_{\mathcal{G}}\big((P^{(1)}_{k_{\max}}f_{l})_{\substack{\\1 \leq  l \leq n}}\big) 
\end{align}
can be written as a finite sum of terms appearing in the laft-hand side of \eqref{induction_2parameter_fix1scale}, so that it satisfies the same estimate:
\begin{align}  \label{induction_2parameter_semilocal}
& \|\eqref{induction_2parameter_semilocal_to_est}\|_{\vec{r}} \lesssim 2^{k_{\max}\cdot\beta_1(\mathfrak{r}_{\mathcal{G}}, l_0)}\sum_{\substack{\delta_2 \\ \restr{\delta_1}{(\mathcal{V}_{l_0}^{\mathfrak{r}_{\mathcal{G}}})^c}}}\|\Delta^{(1)}_{k_{\max}}D^{\delta_2^{-1}(l_0)}_{(2)}f_{l_0}\|_{\vec{p}_{l_0}}\prod_{\substack{ \\ l \neq l_0}}\|D_{(1)}^{\delta_1^{-1}(l)}D_{(2)}^{\delta_2^{-1}(l)}f_{l}\|_{\vec{p_l}}.
\end{align}
By symmetry, a similar estimate holds for $
 T_{\mathcal{G}}\big((P^{(2)}_{m_{\max}}f_{l})_{\substack{\\ 1 \leq l \leq n }}\big). 
$
\\
\textbf{(iv)}
We notice that as before \eqref{induction_2parameter_cone_statement} generates the following global Leibniz rule:
\begin{equation}\label{leibniz_2parameter_global}
\big\|T_{\mathcal{G}}\big((f_l)_{1 \leq l \leq n}\big)\big\|_{\vec{r}} \lesssim \sum_{\delta}\prod_{l}\|D_{(1)}^{\delta_1^{-1}(l)}D_{(2)}^{\delta_2^{-1}(l)}f_l\|_{\vec{p_l}}. 
\end{equation}
and the Leibniz rule corresponding to a cone localization in the first parameter:
\begin{equation}\label{leibniz_2parameter_semi_global}
\Big\|\sum_{k_{\max} \in \mathbb{Z}}T_{\mathcal{G}}\left((P^{(1)}_{k_{\max}}f_l)_{1 \leq l \leq n}\right)\Big\|_{\vec{r}} \lesssim \sum_{\delta}\prod_{l}\|D_{(1)}^{\delta_1^{-1}(l)}D_{(2)}^{\delta_2^{-1}(l)}f_l\|_{\vec{p_l}}. 
\end{equation}
\end{remark}

Due to the range of different conical regions for each parameter, and hence to the various ways the root symbol can split, more auxiliary inductive statements (necessary for proving Theorem \ref{thm:main} in the bi-parameter case) appear. The strategy of the proof is however the same as in the previous Section \ref{sec:one:param:generic:ind}, and we will especially focus on two aspects: the \textit{splitting of the root symbol} (when the corresponding cone is of the type \eqref{eq:cone:ums}) and the \textit{Fourier series decompositions}, which allows to systematically reduce the estimation to subtrees of lower complexity. Because of this, the proof of many of these auxiliary statements will be left to the reader.

\begin{proof}[Proof of Proposition \ref{prop:biparameter_induction}]
%We will develop the argument for the inductive statements (1) and (2) which can be extended to prove (3). We will also provide the proof for the statement (6). We remark that the inductive statements (4) and (5)
%(\ref{induction_2parameter_fix1scale}) 
%can be perceived as a hybrid of (1)-(3) (depending on the number of maximal scales involved)
%(\ref{induction_whitney_cube_statement}) 
%for the first parameter and (6)
%(\ref{induction_cone_statement}) 
%for the second parameter. It is therefore natural to expect that the proof would be a hybrid of the arguments described in (1)-(3) and (6). We will skip the proof for (4) and (5) and encourage the interested reader to reproduce the proof. 
%\vskip .1in
\noindent
\textbf{(1)}
We first illustrate the proof for \eqref{prop:2param:ind:1}, namely (\ref{induction_2parameter_whitney_one}) and (\ref{induction_2parameter_whitney_two}), focusing on the second case (\ref{induction_2parameter_whitney_two}) since (\ref{induction_2parameter_whitney_one}) follows a similar and indeed simpler argument. We observe that the base case for (\ref{induction_2parameter_whitney_two}) can be verified easily by using the Fourier series decomposition and extending the argument in the one-parameter setting. 

We assume that all the inductive statements hold for trees of \textit{all} lower complexities. We recall that the multilinear expression $T_{\mathcal{G}}(P^{(1)}_{k_{\max}}P^{(2)}_{m_{\max}} f_l)_{1 \leq l \leq n }$ yields the localization to the frequency region $R_{k_{\max}} \times R'_{m_{\max}}$, where
\begin{align} 
R_{k_{\max}}:= \{(\xi_1, \ldots, \xi_n): |\xi_l| \sim 2^{k_{\max}} \ \ \text{for} \ \  l \in \mathfrak{M}_1(R) \ \  \text{and} \ \ |\xi_l| \ll 2^{k_{\max}} \ \ \text{for} \ \  l \in  \mathfrak{m}_1(R) \}, \label{loc:sym_diff_xi} \\
R'_{m_{\max}}:= \{(\eta_1, \ldots, \eta_n): |\eta_l| \sim 2^{m_{\max}} \ \ \text{for} \ \  l \in  \mathfrak{M}_2(R') \ \  \text{and} \ \ |\eta_l| \ll 2^{m_{\max}} \ \ \text{for} \ \  l \in  \mathfrak{m}_2(R') \}, \label{loc:sym_diff_eta}
\end{align}
so that $|\sum_{l=1}^n \xi_l| \leq n 2^{k_{\max}}$ and $|\sum_{l=1}^n \eta_l| \leq n 2^{m_{\max}}$. As in the one-parameter setting, we smoothly restrict the symbol 
\begin{equation} \label{symbol:biparameter_xi}
m_{\beta_1^{\mathfrak{r}_{\mathcal{G}}}}(\sum_{l=1}^n \xi_l):= |\sum_{l =1}^n \xi_l|^{\beta_1^{\mathfrak{r}_{\mathcal{G}}}} 
\end{equation}
to the interval $[-n2^{k_{\max}}, n2^{k_{\max}}]$ and denote it by $m^{k_{\max}}_{\beta_1^{\mathfrak{r}_{\mathcal{G}}}}(\sum_{l=1}^n\xi_l)$. Similarly, we denote by $m^{m_{\max}}_{\beta_2^{\mathfrak{r}_{\mathcal{G}}}}(\sum_{l=1}^n \eta_l)$ the symbol 
\begin{equation} \label{symbol:biparameter_eta}
m_{\beta_2^{\mathfrak{r}_{\mathcal{G}}}}(\sum_{l=1}^n \eta_l) :=  |\sum_{l =1}^n \eta_l|^{\beta_2^{\mathfrak{r}_{\mathcal{G}}}}
\end{equation}
localized to the interval $[-n2^{m_{\max}}, n2^{m_{\max}}]$. We undertake the Fourier series decomposition of the localized symbols 
\begin{align}
& m^{k_{\max}}_{\beta_1^{\mathfrak{r}_{\mathcal{G}}}}(\sum_{l=1}^n \xi_l) = 
(2^{k_{\max}})^ {\beta_1^{\mathfrak{r}_{\mathcal{G}}}} \sum_{L \in \mathbb{Z}}C_{L} e^{2\pi i \frac{L}{n2^{k_{\max}}} \sum_{l=1}^n \xi_l}, \label{symbol_parameter1_fourier_decomp}\\
& m^{m_{\max}}_{\beta_2^{\mathfrak{r}_{\mathcal{G}}}} (\sum_{l=1}^n \eta_l)= 
(2^{m_{\max}})^{\beta_2^{\mathfrak{r}_{\mathcal{G}}}} \sum_{L' \in \mathbb{Z}}C_{L'} e^{2\pi i \frac{L'}{n2^{m_{\max}}} \sum_{l=1}^n \eta_l},
\end{align}
where the (renormalized) Fourier coefficients satisfy the decaying conditions
\begin{align*}
|C_{L}| \lesssim & \frac{1}{(1+|L|)^{1+\beta_1^{\mathfrak{r}_{\mathcal{G}}}}}, \qquad  |C_{L'}| \lesssim \frac{1}{(1+|L'|)^{1+\beta_2^{\mathfrak{r}_{\mathcal{G}}}}}.
\end{align*}
By applying the Fourier series representations on the multiplier, we indeed obtain
\begin{align}\label{induction_2parameter_Fourier_decomp}
& T_{\mathcal{G}}(P^{(1)}_{k_{\max}}P^{(2)}_{m_{\max}} f_l)_{1 \leq l \leq n }(x,y) =  \sum_{L,L' \in \mathbb{Z}} C_{L}C_{L'} 2^{k_{\max}\cdot\beta_1^{\mathfrak{r}_{\mathcal{G}}}}2^{m_{\max}\cdot\beta_2^{\mathfrak{r}_{\mathcal{G}}}}\cdot  \prod_{i=1}^{n_1}T_{\mathcal{G}^{v_i}}\left((P^{(1)}_{k_{\max}}P^{(2)}_{m_{\max}}f_l)_{l \in \mathcal{L}(v_i)}\right)(x +  \frac{L}{n2^{k_{\max}}}, y + \frac{L'}{n2^{m_{\max}}}).
%\bigg((\Delta^{(1)}_{k_{\max}, \frac{L_1}{2^{k_{\max}(\mathfrak{r}_{\mathcal{G})}}}}\Delta^{(2)}_{m_{\max}, \frac{L_2}{2^{m_{\max}(\mathfrak{r}_{\mathcal{G})}}}}f_{l})_{\substack{\\ l \in \mathcal{L}(v_i) \cap \mathfrak{M}_1(R) \cap  \mathfrak{M}_2(R')}},\nonumber \\
%& \quad \quad \quad \quad (\Delta^{(1)}_{k_{\max}, \frac{L_1}{2^{k_{\max}(\mathfrak{r}_{\mathcal{G})}}}}S^{(2)}_{m_{\max}, \frac{L_2}{2^{m_{\max}(\mathfrak{r}_{\mathcal{G})}}}}  f_l)_{\substack{ \\ l \in \mathcal{L}(v_i) \cap \mathfrak{M}_1(R) \cap  \mathfrak{M}_2(R')}},\nonumber \\
%& \quad \quad \quad \quad (S^{(1)}_{k_{\max}, \frac{L_1}{2^{k_{\max}(\mathfrak{r}_{\mathcal{G})}}}}\Delta^{(2)}_{m_{\max}, \frac{L_2}{2^{m_{\max}(\mathfrak{r}_{\mathcal{G})}}}}  f_l)_{\substack{ \\ l \in \mathcal{L}(v_i) \cap \mathfrak{M}_1(R) \cap  \mathfrak{M}_2(R')}},\nonumber \\ 
%& \quad \quad \quad \quad (S^{(1)}_{k_{\max}, \frac{L_1}{2^{k_{\max}(\mathfrak{r}_{\mathcal{G})}}}}S^{(2)}_{m_{\max}, \frac{L_2}{2^{m_{\max}(\mathfrak{r}_{\mathcal{G})}}}}  f_l)_{\substack{ \\ l \in \mathcal{L}(v_i) \cap \mathfrak{M}_1(R) \cap  \mathfrak{M}_2(R')}}\bigg),
\end{align}
Therefore \eqref{induction_2parameter_Fourier_decomp} in its $\|\cdot\|_{\vec r}$ norm\footnote{We recall that if $\|\cdot\|_{\vec r}$ is not subadditive (i.e. if $r^1<1$ or $r^2<1$), we need to use instead $\|\cdot\|_{\vec r}^\tau$ with $\tau \leq \min(1, r^1, r^2)$. The analysis is similar to the reasoning in Section \ref{sec:2:param:5:flag:several:max}; in particular, we derive analogous estimate to \eqref{trick:summation:bi}, which generate the appropriate conditions on the Lebesgue exponents \eqref{cond:thm:1}.} can be majorized by
\begin{equation} \label{induction_2parameter_Fourier_decomp_0}
 2^{k_{\max}\cdot\beta_1^{\mathfrak{r}_{\mathcal{G}}}}2^{m_{\max}\cdot\beta_2^{\mathfrak{r}_{\mathcal{G}}}}\prod_{i=1}^{n_1}\|T_{\mathcal{G}^{v_i}}\left((P^{(1)}_{k_{\max}}P^{(2)}_{m_{\max}}f_l)_{l \in \mathcal{L}(v_i)}\right)\|_{\vec{p}_{v_i}}.
\end{equation}
We notice that there are 2 possibilities for the tree structure with respect to $l_0$ and $l'_0$ where $l_0 \in \mathfrak{M}_1(R), l'_0 \in  \mathfrak{M}_2(R')$ with $l_0 \neq l'_0$:
\vskip .05in
\begin{enumerate}[label=(\roman*), leftmargin=*]
\item \label{prop:2param:1case1}
$l_0, l'_0$ belong to the same subtree $\mathcal{V}^{v_{i_0}}$ for some $1 \leq i_0 \leq n_1$. Assume without loss of generality that $l_0, l'_0 \in \mathcal{L}(v_1)$. 
\vskip .05in
\noindent
\item  \label{prop:2param:1case2}
$l_0 \in \mathcal{L}(v_{i_0})$ and $l'_0 \in \mathcal{L}(v_{i'_0})$ for some $i_0 \neq i'_0$. Assume that $i_0 =1 $ and $i'_0 = 2$.
\end{enumerate}
\vskip .05in
In Case \ref{prop:2param:1case1}, one observes that $\mathfrak{M}_1(R) \cap \mathcal{L}(v_1) \neq \emptyset$ and $\mathfrak{M}_2(R') \cap \mathcal{L}(v_1) \neq \emptyset$. This implies that the subtree $\mathcal{G}^{v_1}$ is automatically restricted to the conical regions
\begin{align}
R(v_1) := & \{(\xi_l)_{l \in \mathcal{L}(v_1)}: |\xi_{l}| \gg |\xi_{l'}| \ \  \text{for} \ \ l \in  \mathcal{L}(v_1) \cap \mathfrak{M}_1(R),\ \  l' \in \mathcal{L}(v_1)\setminus \mathfrak{M}_1(R) \}; \label{subcone_automatic:1}\\
R'(v_1) := & \{(\eta_l)_{l \in \mathcal{L}(v_1)}: |\eta_{l}| \gg |\eta_{l'}| \ \  \text{for} \ \ l \in  \mathcal{L}(v_1)\cap \mathfrak{M}_2(R'),\ \  l' \in \mathcal{L}(v_1)\setminus \mathfrak{M}_2(R') \}.\label{subcone_automatic:2}
\end{align}
Furthermore, $k_{\max}= k_{\max}(v_1) $ and $m_{\max} = m_{\max}(v_1)$ specify the Whitney cubes in the cones \eqref{subcone_automatic:1} and \eqref{subcone_automatic:2}. One can then invoke the inductive hypothesis (\ref{induction_2parameter_whitney_two}) on $T_{\mathcal{G}^{v_1}}$ localized on Whitney cubes (at fixed scales) for both parameters: 
\begin{align} \label{indcuction_2parameter_whitney_two_hyp}
& \big\|T_{\mathcal{G}^{v_1}}\big((P^{(1)}_{k_{\max}}P^{(2)}_{m_{\max}}f_l)_{l \in \mathcal{L}(v_1)}\big)\big\|_{\vec{p}_{v_1}} \nonumber \\
  \lesssim  & 2^{k_{\max}\cdot\beta_1(v_1, l_0)}2^{m_{\max}\cdot\beta_2(v_1, l'_0)}\sum_{\substack{\restr{\delta_1}{(\mathcal{V}_{l_0}^{v_1})^c} \\ \restr{\delta_2}{(\mathcal{V}_{l'_0}^{v_1})^c}}}\|\Delta^{(1)}_{k_{\max}}D_{(2)}^{\delta_2^{-1}(l_0)}f_{l_0}\|_{\vec{p}_{l_0}}\|\Delta^{(2)}_{m_{\max}}D_{(1)}^{\delta_1^{-1}(l'_0)}f_{l'_0}\|_{\vec{p}_{l'_0}} \prod_{\substack{ \\ l \neq l_0, l'_0}}\|D_{(1)}^{\delta_1^{-1}(l)}D_{(2)}^{\delta_2^{-1}(l)}f_{l}\|_{\vec{p_l}}.
\end{align}
Meanwhile the inductive hypothesis \eqref{induction_2parameter_cone_statement} and thus (\ref{leibniz_2parameter_global}) can be invoked to estimate $T_{\mathcal{G}^{v_i}}$ for $i \neq 1$:
\begin{align}\label{leibniz_2parameter_subtree}
& \big\|T_{\mathcal{G}^{v_i}}\left((P^{(1)}_{k_{\max}}P^{(2)}_{m_{\max}}f_l)_{l \in \mathcal{L}(v_i)}\right)\big\|_{\vec{p}_{v_i}} \lesssim  \sum_{\restr{\delta}{\mathcal{V}^{v_i}}}\prod_{l \in \mathcal{L}(v_i)} \|D_{(1)}^{\delta_1^{-1}(l)}D_{(2)}^{\delta_2^{-1}(l)}f_l\|_{\vec{p}_l}.
\end{align}
By applying the estimates (\ref{indcuction_2parameter_whitney_two_hyp}) and (\ref{leibniz_2parameter_subtree}) to (\ref{induction_2parameter_Fourier_decomp}), we obtain the desired estimate claimed in the inductive statement (\ref{induction_2parameter_whitney_two}).

In Case \ref{prop:2param:1case2}, we define for $l \in \mathcal{L}(v_1)$
\begin{equation*}
\tilde{f_l} := P^{(2)}_{m_{\max}} f_l, 
\end{equation*}
and for $l \in \mathcal{L}(v_2)$
\begin{equation*}
\tilde{f_l} := P^{(1)}_{k_{\max}} f_l.
\end{equation*}
Then we apply the inductive hypothesis ((\ref{induction_2parameter_fix1scale}) and thus) (\ref{induction_2parameter_semilocal}) to estimate 
\begin{align} \label{induction_2parameter_use_1}
& \big\|T_{\mathcal{G}^{v_1}}\big((P^{(1)}_{k_{\max}}P^{(2)}_{m_{\max}}f_l)_{l \in \mathcal{L}(v_1)}\big)\big\|_{\vec{p}_{v_1}} =  \big\|T_{\mathcal{G}^{v_1}}\big((P^{(1)}_{k_{\max}}\tilde{f}_{l})_{\substack{\\ l \in \mathcal{L}(v_1)}}\big)\big\|_{\vec{p}_{v_1}} \nonumber \\
\lesssim & 2^{k_{\max}\cdot\beta_1(v_1, l_0)}\sum_{\substack{\restr{\delta_2}{\mathcal{V}^{v_1}} \\ \restr{\delta_1}{(\mathcal{V}_{l_0}^{v_1})^c} \\ }}\|\Delta^{(1)}_{k_{\max}}D_{(2)}^{\delta_2^{-1}(l_0)}f_{l_0}\|_{\vec{p}_{l_0}}\prod_{\substack{l \in \mathcal{L}(v_1) \\ l \neq l_0}}\|D_{(1)}^{\delta_1^{-1}(l)}D_{(2)}^{\delta_2^{-1}(l)}f_{l}\|_{\vec{p_l}}, %\nonumber\\
%  \lesssim  & 2^{k_{\max}\cdot\beta_1(v_1, l_0)}\sum_{\substack{\restr{\delta_2}{\mathcal{V}^{v_1}} \\ \restr{\delta_1}{(\mathcal{V}_{l_0}^{v_1})^c} \\ }}\|\Delta^{(1)}_{k_{\max}}D_{(2)}^{\delta_2^{-1}(l_0)}f_{l_0}\|_{\vec{p}_{l_0}}\prod_{\substack{l \in \mathcal{L}(v_1) \\ l \neq l_0}}\|D_{(1)}^{\delta_1^{-1}(l)}D_{(2)}^{\delta_2^{-1}(l)}f_{l}\|_{\vec{p_l}},
\end{align}
and 
\begin{align}\label{induction_2parameter_use_2}
&  \big\|T_{\mathcal{G}^{v_2}}\big((P^{(1)}_{k_{\max}}P^{(2)}_{m_{\max}}f_l)_{l \in \mathcal{L}(v_2)}\big)\big\|_{\vec{p}_{v_2}} =  \big\|T_{\mathcal{G}^{v_2}}\big((P^{(2)}_{m_{\max}}\tilde{f}_{l})_{\substack{\\ l \in \mathcal{L}(v_2)}}\big)\big\|_{\vec{p}_{v_2}} \nonumber \\
\lesssim & 2^{m_{\max}\cdot\beta_2(v_2, l'_0)}\sum_{\substack{\restr{\delta_1}{\mathcal{V}^{v_2}} \\ \restr{\delta_2}{(\mathcal{V}_{l'_0}^{v_2})^c} \\ }}\|\Delta^{(2)}_{m_{\max}}D_{(1)}^{\delta_1^{-1}(l'_0)}f_{l'_0}\|_{\vec{p}_{l'_0}}\prod_{\substack{l \in \mathcal{L}(v_2) \\ l \neq l'_0}}\|D_{(1)}^{\delta_1^{-1}(l)}D_{(2)}^{\delta_2^{-1}(l)}f_{l}\|_{\vec{p_l}}. %\nonumber \\
%  \lesssim  & 2^{m_{\max}\cdot\beta_2(v_2, l'_0)}\sum_{\substack{\restr{\delta_1}{\mathcal{V}^{v_2}} \\ \restr{\delta_2}{(\mathcal{V}_{l'_0}^{v_2})^c} \\ }}\|\Delta^{(2)}_{m_{\max}}D_{(1)}^{\delta_1^{-1}(l'_0)}f_{l'_0}\|_{\vec{p}_{l'_0}}\prod_{\substack{l \in \mathcal{L}(v_2) \\ l \neq l'_0}}\|D_{(1)}^{\delta_1^{-1}(l)}D_{(2)}^{\delta_2^{-1}(l)}f_{l}\|_{\vec{p_l}}.
\end{align}
We also recall the inductive hypothesis (\ref{induction_2parameter_cone_statement}) and hence (\ref{leibniz_2parameter_global}) to the other subtrees corresponding to the multilinear forms for $i \neq 1,2$ and obtain the estimate (\ref{leibniz_2parameter_subtree}). Plugging the estimates (\ref{induction_2parameter_use_1}), (\ref{induction_2parameter_use_2}) and (\ref{leibniz_2parameter_subtree}) into (\ref{induction_2parameter_Fourier_decomp}), we derive the estimate specified on the right hand side of (\ref{induction_2parameter_whitney_two}) as desired.

\vskip .1in
\noindent
\textbf{(2)}
For the inductive statement \eqref{prop:2param:ind:2}, we will provide a proof for Case \eqref{prop:2param:ind:2:a}. Case \eqref{prop:2param:ind:2:b} can be verified using a similar (although not perfectly identical) argument -- the hypotheses taking part in the inductive argument are different for the two cases. The base case concerns the estimate for
\begin{equation*}
D^{\beta_1}_{(1)}D^{\beta_2}_{(2)}\big((P^{(1)}_{k_{\max}}P^{(2)}_{m_{\max}} f_l)_{1 \leq l \leq n}\big).
\end{equation*}
The Fourier series decomposition of the symbol gives
\begin{equation*}
 \sum_{L,L' \in \mathbb{Z}} C_{L}C_{L'} 2^{k_{\max}\beta_1}2^{m_{\max}\beta_2} \prod_{i=1}^{n}P^{(1)}_{k_{\max}}P^{(2)}_{m_{\max}}f_l \big( x +  \frac{L}{n2^{k_{\max}}}, y + \frac{L'}{n2^{m_{\max}}} \big)
\end{equation*}
whose $L^{\vec r}$ norm for $r^1, r^2 \geq 1$ \footnote{The estimate in $L^{\vec r}$ norm when $r^1, r^2 \geq 1$ doesn't hold requests appropriate conditions on the Lebesgue exponents --  see \eqref{cond:thm:1}.} can be majorized by
\begin{equation*}
2^{k_{\max}\beta_1}2^{m_{\max}\beta_2} \prod_{i=1}^{n}\|P^{(1)}_{k_{\max}}P^{(2)}_{m_{\max}}f_l\|_{\vec{p}_l}
\end{equation*}
due to H\"older's inequality. We recall that in the base case $v^{l_1,l_2} $ is the root and $w^{l_1} = l_1$, $w^{l_2} = l_2$ so that 
\begin{align*}
& \beta_1(\mathfrak{r}_{\mathcal{G}}, v^{l_1,l_2})= \beta_1, \ \ \beta_2(\mathfrak{r}_{\mathcal{G}}, l_1) = \beta_2,\\
& \beta_1(w^{l_1},l_1) = 0, \ \ \beta_2(w^{l_2}, l_2) = 0.
\end{align*}
We have thus verified the base case for (\ref{induction_2parameter_whitney_2:1}).

To prove the inductive statement, we apply the Fourier series decomposition on the root symbol as before to tensorize the operator into operators associated to subtrees and obtain (\ref{induction_2parameter_Fourier_decomp}); its $L^{\vec r}$ norm can now be estimated by (\ref{induction_2parameter_Fourier_decomp_0}). There are 2 possible positions for $l_1, l_2 \in \mathfrak{M}_1(R)$ and $l_0' \in  \mathfrak{M}_2(R')$ with $l_1 = l_0'$ and $l_1 \neq l_2$.
\vskip .1in
\noindent
\begin{enumerate}[label=(\roman*), leftmargin=*]
\item \label{prop:2param:2case1}
$l_1, l_2 \in \mathcal{L}(v_{i_0})$. Assume without loss of generality that $i_0 = 1$.
\vskip .05in
\noindent
\item  \label{prop:2param:2case2}
$l_1\in \mathcal{L}(v_{i_0})$ and $l_2 \in \mathcal{L}(v_{i'_0})$ with $i_0 \neq i_0'$. Assume that $i_0 =1$ and $i_0' = 2$.
\end{enumerate}

\vskip .1in

For Case \ref{prop:2param:2case1}, we deduce that the multilinear expression associated to the subtree $\mathcal{G}^{v_1}$ is automatically restricted to the cones $R(v_1)$ and $R'(v_1)$ of the form \eqref{subcone_automatic:1} and \eqref{subcone_automatic:2}. By assumption,  
$l_1, l_2 \in \mathfrak{M}_1(R(v_1))$ and $l_0' = l_1 \in \mathfrak{M}_2(R'(v_1))$. Let
$k_{\max} = k_{\max}(v_1)$ and $m_{\max} = m_{\max}(v_1)$ indicate the Whitney cubes in the cones $R(v_1)$ and $R'(v_1)$. By the inductive hypothesis (\ref{induction_2parameter_whitney_2:1}),
\begin{align} \label{leibniz_rule_2parameter_major_fixedscale}
 \|T_{\mathcal{G}^{v_1}}(P^{(1)}_{k_{\max}}P^{(2)}_{m_{\max}} f_l)_{l \in \mathcal{L}(v_1)}\|_{\vec{p}_{v_1}} \lesssim & 2^{k_{\max}\beta_1(v_1, v^{l_1,l_2})}2^{m_{\max}\beta_2(v_1, l_1)} \prod_{\substack{l \in \mathcal{L}(v_1) \\ l \neq l_1,l_2}}\|D_{(1)}^{\delta_1^{-1}(l)}D_{(2)}^{\delta_2^{-1}(l)}f_{l}\|_{\vec{p_l}} \\
& \cdot \|\Delta^{(1)}_{k_{\max}}\Delta^{(2)}_{m_{\max}}D_{(1)}^{\beta_1(w^{l_1},l_1)}f_{l_1}\|_{\vec{p}_{l_1}} \sum_{\substack{\restr{\delta_1}{(\mathcal{V}_{l_1}^{v_1})^c\setminus \mathcal{V}^{w^{l_2}}_{l_2}} \nonumber\\ \restr{\delta_2}{(\mathcal{V}_{l_1}^{v_1})^c}}}\|\Delta^{(1)}_{k_{\max}}D_{(1)}^{\beta_1(w^{l_2},l_2)}D_{(2)}^{\delta_2^{-1}(l_2)}f_{l_2}\|_{\vec{p}_{l_2}}.
\end{align}

We can invoke the inductive hypothesis (\ref{leibniz_2parameter_global}) -- assumed to hold for trees of lower complexities -- to estimate $\|T_{\mathcal{G}^{v_i}}(P^{(1)}_{k_{\max}}P^{(2)}_{m_{\max}} f_l)_{l \in \mathcal{L}(v_i)}\|_{\vec{p}_{v_i}}$ for $i \neq 1$; we obtain the bound
\begin{equation} \label{leibniz_rule_2parameter_subtree}
 \sum_{\restr{\delta}{\mathcal{V}^{v_i}}}\prod_{l \in \mathcal{L}(v_i)} \|D_{(1)}^{\delta_1^{-1}(l)}D_{(2)}^{\delta_2^{-1}(l)}f_l\|_{\vec{p}_l}.
\end{equation}
By plugging the estimates (\ref{leibniz_rule_2parameter_major_fixedscale})  and (\ref{leibniz_rule_2parameter_subtree}) into (\ref{induction_2parameter_Fourier_decomp_0}), we conclude with (\ref{induction_2parameter_whitney_2:1}).
\vskip .1in

In Case \ref{prop:2param:2case2}, the localization of the original $T_{\mathcal{G}}$ to the cones $R$ and $R'$ imposes a similar restriction to conical regions on $T_{\mathcal{G}^{v_1}}$ and $T_{\mathcal{G}^{v_2}}$. More precisely,  let $R(v_1)$ and  $R'(v_1)$ denote the conical regions for the subtrees $\mathcal{G}^{v_1}$ for the first and second parameters respectively as before. Let $R(v_2)$ represent the conical region for the subtree $\mathcal{G}^{v_2}$ for the first parameter. 

Then we have
$l_1 \in \mathfrak{M}_1(R(v_1))$, $l_2 \in \mathfrak{M}_1(R(v_2))$ and $l_0' \in \mathfrak{M}_2(R'(v_1))$ with $l_1 = l_0'$. Also, $k_{\max} = k_{\max}(v_1)$ and $m_{\max} = m_{\max}(v_1)$ indicate the Whitney cubes in the cones $R(v_1)$ and $R'(v_1)$ for the first and second parameters. The inductive hypothesis (\ref{induction_2parameter_whitney_one}) describes exactly the estimate for the subtree $\mathcal{G}^{v_1}$: %stemming from $v_1$:
\begin{align} \label{induction_2parameter_whitney_2a_major}
 \|T_{\mathcal{G}^{v_1}}(P^{(1)}_{k_{\max}}P^{(2)}_{m_{\max}} f_l)_{l \in \mathcal{L}(v_1)}\|_{\vec{p}_{v_1}} \lesssim  & 2^{k_{\max}\beta_1(v_1, l_1)}2^{m_{\max}\beta_2(v_1, l_1)}\|\Delta^{(1)}_{k_{\max}}\Delta^{(2)}_{m_{\max}}f_{l_1}\|_{\vec{p}_{l_1}}\sum_{\restr{\delta_1 \otimes \delta_2}{(\mathcal{V}_{l_1}^{v_1})^c}}\prod_{\substack{ l \in \mathcal{L}(v_1) \\ l \neq l_1}}\|D_{(1)}^{\delta_1^{-1}(l)}D_{(2)}^{\delta_2^{-1}(l)}f_{l}\|_{\vec{p_l}} \nonumber \\
 \lesssim & 2^{m_{\max}\beta_2(v_1, l_1)}\|\Delta^{(1)}_{k_{\max}}\Delta^{(2)}_{m_{\max}}D_{(1)}^{\beta_1(v_1, l_1)}f_{l_1}\|_{\vec{p}_{l_1}}\sum_{\restr{\delta_1 \otimes \delta_2}{(\mathcal{V}_{l_1}^{v_1})^c}}\prod_{\substack{ l \in \mathcal{L}(v_1) \\ l \neq l_1}}\|D_{(1)}^{\delta_1^{-1}(l)}D_{(2)}^{\delta_2^{-1}(l)}f_{l}\|_{\vec{p_l}} ,
\end{align}
where the last inequality follows from (\ref{eq:obs:besov}). 

Due to the localization to the cone $R(v_2)$ for the first parameter, we also apply the (corollary of the) inductive hypothesis -- (\ref{induction_2parameter_semilocal}) -- to derive the following estimate
%a similar estimate as (\ref{induction_2parameter_use_1}) 
for the subtree $\mathcal{G}^{v_2}$: %stemming from $v_2$, in particular
\begin{align} \label{induction_2parameter_whitney_2a_major_2}
 \|T_{\mathcal{G}^{v_2}}(P^{(1)}_{k_{\max}}P^{(2)}_{m_{\max}} f_l)_{l \in \mathcal{L}(v_2)}\|_{\vec{p}_{v_2}} %\lesssim & 2^{k_{\max}\beta_1(v_2, l_1)}\sum_{\substack{\restr{\delta_2}{\mathcal{V}^{v_2}} \\ \restr{\delta_1}{(\mathcal{V}_{l_2}^{v_2})^c}}}\|\Delta^{(1)}_{k_{\max}}D^{\delta_2^{-1}(l_2)}_{(2)}f_{l_2}\|_{\vec{p}_{l_2}}\prod_{\substack{ l \in \mathcal{L}(v_2) \\ l \neq l_2}}\|D_{(1)}^{\delta_1^{-1}(l)}D_{(2)}^{\delta_2^{-1}(l)}f_{l}\|_{\vec{p_l}} \\
 \lesssim & \sum_{\substack{\restr{\delta_2}{\mathcal{V}^{v_2}} \\ \restr{\delta_1}{(\mathcal{V}_{l_2}^{v_2})^c}}}\|\Delta^{(1)}_{k_{\max}}D_{(1)}^{\beta_1(v_2, l_1)}D^{\delta_2^{-1}(l_2)}_{(2)}f_{l_2}\|_{\vec{p}_{l_2}}\prod_{\substack{ l \in \mathcal{L}(v_2) \\ l \neq l_2}}\|D_{(1)}^{\delta_1^{-1}(l)}D_{(2)}^{\delta_2^{-1}(l)}f_{l}\|_{\vec{p_l}}.
\end{align}
Last but not least, we use the bound (\ref{leibniz_rule_2parameter_subtree}) for $\|T_{\mathcal{G}^{v_i}}(P^{(1)}_{k_{\max}}P^{(2)}_{m_{\max}} f_l)_{l \in \mathcal{L}(v_i)}\|_{\vec{p}_{v_i}}$, $i \neq 1,2$.  With the application of (\ref{induction_2parameter_whitney_2a_major}), (\ref{induction_2parameter_whitney_2a_major_2}) and (\ref{leibniz_rule_2parameter_subtree}) to (\ref{induction_2parameter_Fourier_decomp_0}), we complete the proof of the inductive statement (\ref{induction_2parameter_whitney_2:1}).
\vskip .1in
\noindent
\textbf{(6)}
The base case of the statement (\ref{induction_2parameter_cone_statement}), corresponding to a tree of complexity $1$, is contained in \cite{OhWu}; in Section \ref{sec:2param:5linflag}, trees of complexity $2$ were treated.

We will focus on the case when the multilinear expression is localized on conical regions of type \eqref{eq:cone:ums} for both parameters, as this is the situation which requires the use of commutators -- the tools that allow to depart from the usual methods relying on Coifman-Meyer multipliers. The other cases follow similar arguments with application of possibly different inductive hypotheses. The treatment presented here resembles the proof for the bi-parameter $5$-linear flag Leibniz rule presented in Section \ref{sec:5:flag:bi:param:k_1:m_1}. 

For the symbol in each parameter, we independently carry out the procedure described in the one-parameter setting to derive a similar expression to (\ref{commutator_fourier_series}). One will first \textit{split the root symbol} and introduce appropriate commutators in both parameters. Let $l_0,   l'_0 \in \mathcal{L}(\mathfrak{r}_{\mathcal{G}})$ denote the indices such that 
\begin{equation} \label{local:conical_biparameter}
\mathfrak{M}_1(R) = \{l_0\}, \ \  \mathfrak{M}_2(R')= \{l'_0\},
\end{equation}
and suppose 
$$l_0 \in \mathcal{L}(v_{i_0}),  \ \ l'_0 \in \mathcal{L}(v_{i'_0}).$$
Then the root symbols (\eqref{symbol:biparameter_xi} and \eqref{symbol:biparameter_eta}) localized to the conical regions specified by \eqref{local:conical_biparameter} are decomposed as follows: 
\begin{align}\label{root_symbol_decomp_commutator}
%|\sum_{l =1}^n\xi_l|^{\beta_1^{\mathfrak{r}_{\mathcal{G}}}} \cdot \chi_{R_1} 
m_{\beta_1^{\mathfrak{r}_{\mathcal{G}}}}(\sum_{l=1}^n \xi_l)
= &  \sum_{\tilde{i} \neq i_0}   \underbrace{m^{k_{l_0}}_{C_{\beta_1^{\mathfrak{r_{\mathcal{G}}}}}}\big(\sum_{l \in \mathcal{L}(v_{i_0})}\xi_l, \sum_{l \nin \mathcal{L}(v_{i_0})}\xi_l\big)\cdot \sum_{l \in \mathcal{L}(v_{\tilde{i}})}\xi_l }_{\mathcal{A}_1}+  \underbrace{|\sum_{l \in \mathcal{L}(v_{i_0})} \xi_l|^{\beta^{\mathfrak{r}_{\mathcal{G}}}} }_{\mathcal{A}_2}, \nonumber \\
%|\sum_{l =1}^n\eta_l|^{\beta_2^{\mathfrak{r}_{\mathcal{G}}}} \cdot \chi_{R_2}
m_{\beta_2^{\mathfrak{r}_{\mathcal{G}}}}(\sum_{l=1}^n \eta_l)
= &  \sum_{\tilde{i}' \neq i'_0}   \underbrace{m^{m_{l'_0}}_{C_{\beta_2^{\mathfrak{r_{\mathcal{G}}}}}}\big(\sum_{l \in \mathcal{L}(v_{i'_0})}\eta_l, \sum_{l \nin \mathcal{L}(v_{i'_0})}\eta_l\big)\cdot \sum_{l \in \mathcal{L}(v_{\tilde{i}'})}\eta_l }_{\mathcal{B}_1} + \underbrace{ |\sum_{l \in \mathcal{L}(v_{i'_0})} \eta_l|^{\beta_2^{\mathfrak{r}_{\mathcal{G}}}} }_{\mathcal{B}_2},
\end{align}
where %$R_1, R_2$ are defined in (\ref{def_symbols_local_1}) and 
\begin{align}\label{symbol:commutator_biparameter}
m^{}_{C_{\beta_1^{\mathfrak{r_{\mathcal{G}}}}}}\left(\tilde{\xi}_1, \tilde{\xi}_2\right) := &\frac{\displaystyle |\tilde{\xi}_1 + \tilde{\xi}_2|^{\beta_1^{\mathfrak{r}_{\mathcal{G}}}} - |\tilde{\xi}_1|^{\beta_1^{\mathfrak{r}_{\mathcal{G}}}}}{\displaystyle \tilde{\xi}_2}, \ \ m^{}_{C_{\beta_2^{\mathfrak{r_{\mathcal{G}}}}}}\left(\tilde{\eta}_1, \tilde{\eta}_2\right) := \frac{\displaystyle |\tilde{\eta}_1 + \tilde{\eta}_2|^{\beta_2^{\mathfrak{r}_{\mathcal{G}}}} - |\tilde{\eta}_1|^{\beta_2^{\mathfrak{r}_{\mathcal{G}}}}}{\displaystyle \tilde{\eta}_2}.
\end{align}
Recall that $R_{k_{\max}}$ and $R'_{m_{\max}}$ are defined in \eqref{loc:sym_diff_xi} and \eqref{loc:sym_diff_eta} and due to the assumption \eqref{local:conical_biparameter} on $\mathfrak{M}_1(R)$ and $ \mathfrak{M}_2(R')$, they take the form
\begin{align}
R_{k_{\max}} =&  R_{k_{l_0}} = \{(\xi_1, \ldots, \xi_n): |\xi_{l_0}| \sim 2^{k_{l_0}}, |\xi_l| \ll 2^{k_{l_0}} \ \ \text{for}\ \ l \neq l_0 \},  \\
R'_{m_{\max}} =&  R'_{m_{l'_0}} = \{(\eta_1, \ldots, \eta_n): |\eta_{l'_0}| \sim 2^{m_{l'_0}}, |\eta_l| \ll 2^{m_{l'_0}} \ \ \text{for}\ \ l \neq l'_0 \}. 
\end{align}
We notice that $R_{k_{l_0}}$ and $R'_{m_{l'_0}}$ can be decomposed as Whitney cubes on which we will perform double Fourier series decompositions. In particular,
\begin{align*}
R_{k_{l_0}} =& R_{k_{l_0}}^+ \cup R_{k_{l_0}}^-, \ \  R'_{m_{l'_0}} = (R'_{m_{l'_0}})^+ \cup (R'_{m_{l'_0}})^-, %\\
%R'_{m_{l'_0}} =& (R'_{m_{l'_0}})^+ \cup (R'_{m_{l'_0}})^-,
\end{align*}
%\textbf{(II)}
where each Whitney cube is defined by
\begin{align}
R_{k_{l_0}}^+ := \{(\xi_1, \ldots, \xi_n): \xi_{l_0} \sim 2^{k_{l_0}}, |\xi_l| \ll 2^{k_{l_0}} \ \ \text{for}\ \ l \neq l_0 \}, \ \ &R_{k_{l_0}}^- := \{(\xi_1, \ldots, \xi_n): \xi_{l_0} \sim -2^{k_{l_0}}, |\xi_l| \ll 2^{k_{l_0}} \ \ \text{for}\ \ l \neq l_0 \}, \label{local:symbol_whitney_cube_xi}\\
R_{m_{l'_0}}^+ := \{(\eta_1, \ldots, \eta_n): \eta_{l'_0} \sim 2^{m_{l'_0}}, |\eta_l| \ll 2^{m_{l'_0}} \ \ \text{for}\ \ l \neq l'_0 \}, \ \ &R_{m_{l'_0}}^- := \{(\eta_1, \ldots, \eta_n): \eta_{l'_0} \sim -2^{m_{l'_0}}, |\eta_l| \ll 2^{m_{l'_0}} \ \ \text{for}\ \ l \neq l'_0 \}.\label{local:symbol_whitney_cube_eta}
\end{align}

As previously, we smoothly restrict the symbol $m^{}_{C_{\beta_1^{\mathfrak{r_{\mathcal{G}}}}}}$ \eqref{symbol:commutator_biparameter} to the Whitney cube $R^+_{k_{l_0}}$ and to $R^-_{k_{l_0}}$ \eqref{local:symbol_whitney_cube_xi} and denote the localized symbols by $m^{k_{l_0},+}_{C_{\beta_1^{\mathfrak{r_{\mathcal{G}}}}}}$ and $m^{k_{l_0},-}_{C_{\beta_1^{\mathfrak{r_{\mathcal{G}}}}}}$ respectively. Similarly, the symbol $m^{}_{C_{\beta_2^{\mathfrak{r_{\mathcal{G}}}}}}$ restricted to $(R'_{m_{l'_0}})^{\pm}$% and to $(R'_{m_{l'_0}})^-$
 regions in \eqref{local:symbol_whitney_cube_eta} 
are denoted by
$m^{m_{l'_0},\pm}_{C_{\beta_2^{\mathfrak{r_{\mathcal{G}}}}}}$. %and $m^{m_{l'_0},-}_{C_{\beta_2^{\mathfrak{r_{\mathcal{G}}}}}}$. 

We use Fourier series decomposition to rewrite the symbol $\displaystyle m^{k_{l_0},\pm}_{C_{\beta_1^{\mathfrak{r_{\mathcal{G}}}}}}\big(\sum_{l \in \mathcal{L}(v_{i_0})}\xi_l, \sum_{l \nin \mathcal{L}(v_{i_0})}\xi_l\big)$, which is indeed \eqref{commutator_1parameter_fourier_decomp+} 
with $\beta^{\mathfrak{r}_{\mathcal{G}}}$ replaced by $\beta_1^{\mathfrak{r}_{\mathcal{G}}}$. Similarly, 
\begin{align*}
m^{m_{l'_0},\pm}_{C_{\beta_2^{\mathfrak{r_{\mathcal{G}}}}}}\big(\sum_{l \in \mathcal{L}(v_{i'_0})}\eta_l, \sum_{l \nin \mathcal{L}(v_{i'_0})}\eta_l\big)  = \sum_{L'_1, L'_2} C^{\pm}_{L'_1, L'_2} 2^{m_{l'_0}(\beta_2^{\mathfrak{r}_{\mathcal{G}}}-1)} e^{2 \pi i L'_1 \frac{\sum_{l \in \mathcal{L'}(v_{i'_0})}\eta_l} {2^{m_{l'_0}}}} \, e^{2 \pi i L'_2 \frac{\sum_{l \nin \mathcal{L}(v_{i'_0})} \eta_l}{2^{m_{l'_0}}}}. %,\\
%m^{m_{l'_0},-}_{C_{\beta_2^{\mathfrak{r_{\mathcal{G}}}}}}\big(\sum_{l \in \mathcal{L}(v_{i'_0})}\eta_l, \sum_{l \nin \mathcal{L}(v_{i'_0})}\eta_l\big)  = \sum_{L'_1, L'_2} C^-_{L'_1, L'_2} 2^{m_{l'_0}(\beta_2^{\mathfrak{r}_{\mathcal{G}}}-1)} e^{2 \pi i L'_1 \frac{\sum_{l \in \mathcal{L'}(v_{i'_0})}\eta_l} {2^{m_{l'_0}}}} \, e^{2 \pi i L'_2 \frac{\sum_{l \nin \mathcal{L}(v_{i'_0})} \eta_l}{2^{m_{l'_0}}}}.
\end{align*}
\vskip .1in
The estimate for the multiplier in the biparameter setting concerns the different combinations of the symbols involving the commutators and the symbols for differential operators on subtrees of lower complexity:
$$
\mathcal{A}_1 \cdot \mathcal{B}_1, \ \ \mathcal{A}_1 \cdot \mathcal{B}_2,\ \  \mathcal{A}_2 \cdot \mathcal{B}_1, \ \ \mathcal{A}_2 \cdot \mathcal{B}_2.
$$
where $\mathcal{A}_1$ and $\mathcal{B}_1$ are symbols for commutators while $\mathcal{A}_2$ and $\mathcal{B}_2$ represent symbols for differential operators on subtrees. There are 3 possibilities with respect to the relation between $l_0$ and $l'_0$:
\begin{enumerate}[label=(\roman*), leftmargin=*]
\item \label{prop:2param:6case1}
$l_0 = l'_0$; \\                                            
\item \label{prop:2param:6case2}
$l_0 \neq l'_0$ and $l_0, l'_0 \in \mathcal{L}(v_{i_0})$ for some $i_0$. Assume that $i_0 = 1$; \\
\item \label{prop:2param:6case3}
$l_0 \in \mathcal{L}(v_{i_0})$ and $l'_0 \in \mathcal{L}(v_{i'_0})$ for some $i_0 \neq i'_0$. Assume that $i_0 =1 $ and $i'_0 = 2$. 
\end{enumerate}

Different possibilities generate multipliers that are analogous to the operators discussed in Section \ref{sec:5:flag:bi:param:k_1:m_1}. In the generic induction, estimates for those multipliers are reduced to estimates on subtrees that request various inductive hypotheses. Since the procedure of reduction to subtree estimates is similar in all cases and the computations after the application of the inductive hypotheses are analogous, we will focus on the proof in Case \ref{prop:2param:6case1}.

When $l_0 = l'_0$, assume without loss of generality that $l_0 =1 \in \mathcal{L}(v_1)$ and thus $i_0= 1$. The multipliers involved are listed and estimated as follows. 
%The typical terms that would appear are listed as follows: \\
\vskip .1in
\begin{enumerate} [leftmargin=*]
\item[$\bullet$] \underline{estimating $\mathcal{A}_1 \cdot \mathcal{B}_1$ \eqref{root_symbol_decomp_commutator}:}
%We will first consider the ones generated by $\mathcal{A}_1 \cdot \mathcal{B}_1$.  

Due to the assumption, we will refer to $k_{\max}$ as $k_1$ and $m_{\max}$ as $m_1$. When $\tilde{i} = \tilde{i}' = 2$ in (\ref{root_symbol_decomp_commutator}), the multiplier takes the form:  
\begin{align}\label{commu_commu_fourier}
\sum_{\substack{L_1, L_2 \in \mathbb{Z} \\ L'_1, L'_2 \in \mathbb{Z}}}C^{\pm}_{L_1, L_2}  C^{\pm}_{L'_1, L'_2} &  \sum_{\substack{k_{\max}(v_2) \ll k_1\\ m_{\max}(v_2) \ll m_1}}  2^{k_{\max}(v_2)} 2^{k_1 \cdot (\beta_1^{\mathfrak{r}_{\mathcal{G}}}-1)} 2^{m_{\max}(v_2)} 2^{m_1 \cdot(\beta_2^{\mathfrak{r}_{\mathcal{G}}}-1)}\cdot \nonumber \\
& T_{\mathcal{G}^{v_1}}\Big(\Delta^{(1)}_{k_1,\pm}\Delta^{(2)}_{m_1,\pm}f_1,  (S^{(1)}_{k_1}S^{(2)}_{m_1}f_l)_{\substack{l \in \mathcal{L}(v_1) \\ l \neq 1}})\Big)(x+ \frac{L_1}{2^{k_1}}, y + \frac{L'_1}{2^{m_1}}) \cdot\nonumber \\
&  T_{\mathcal{G}^{v_2}}\Big((P^{(1)}_{k_{\max}(v_2)}P^{(2)}_{m_{\max}(v_2)}f_{l})_{\substack{\\ l \in \mathcal{L}(v_2)}}\Big)(x+ \frac{L_2}{2^{k_1}} , y + \frac{L'_2}{2^{m_1}}) \cdot \nonumber\\
& \prod_{i=3}^{n_1} T_{\mathcal{G}^{v_{i}}}\left((S^{(1)}_{k_1 }S^{(2)}_{m_1}f_l)_{l \in \mathcal{L}(v_i)}\right)(x + \frac{L_2}{2^{k_1}}, y + \frac{L'_2}{2^{m_1}}).
\end{align}
The case when $\tilde{i} \neq \tilde{i}'$ follows a similar argument and will not be discussed in details. 

Due to the decay of the Fourier coefficients, the $L^{\vec r}$ norm of (\ref{commu_commu_fourier}) can be bounded by 
\begin{align}\label{2parameter_max_same_cc}
\sum_{\substack{k_{\max}(v_2) \ll k_1 \\ m_{\max}(v_2) \ll m_1}}  2^{k_{\max}(v_2)} 2^{k_1 \cdot (\beta_1^{\mathfrak{r}_{\mathcal{G}}}-1)} 2^{m_{\max}(v_2)} 2^{m_1 \cdot(\beta_2^{\mathfrak{r}_{\mathcal{G}}}-1)} \big\|T_{\mathcal{G}^{v_1}}\big(\Delta^{(1)}_{k_1,+}\Delta^{(2)}_{m_1,+}f_1,  (S^{(1)}_{k_1}S^{(2)}_{m_1}f_l)_{\substack{l \in \mathcal{L}(v_1) \\ l \neq 1}})\big)\big\|_{\vec{p}_{v_1}} &\nonumber\\
 \cdot \big\|T_{\mathcal{G}^{v_2}}\big((P^{(1)}_{k_{\max}(v_2)}P^{(2)}_{k_{\max}(v_2)}f_{l})_{\substack{\\ l \in \mathcal{L}(v_2)}}\big)\big\|_{\vec{p}_{v_2}} \prod_{i=3}^{n_1} \big\|T_{\mathcal{G}^{v_{i}}}\big((S^{(1)}_{k_1}S^{(2)}_{m_1}f_l)_{l \in \mathcal{L}(v_i)}\big)\big\|_{\vec{p}_{v_i}}.&
\end{align}
We can apply the inductive hypothesis (\ref{induction_2parameter_whitney_one}) on $T_{\mathcal{G}^{v_1}}\big(\Delta^{(1)}_{k_1,\pm}\Delta^{(2)}_{m_1,\pm}f_1,  (S^{(1)}_{k_1}S^{(2)}_{m_1}f_l)_{\substack{l \in \mathcal{L}(v_1) \\ l \neq 1}})\big)$ and obtain
{\fontsize{8}{8}
\begin{align}\label{est:inductive_G(v_1)}
\big\|T_{\mathcal{G}^{v_1}}\big(\Delta^{(1)}_{k_1,\pm}\Delta^{(2)}_{m_1,\pm}f_1,  (S^{(1)}_{k_1}S^{(2)}_{m_1}f_l)_{\substack{l \in \mathcal{L}(v_1) \\ l \neq 1}})\big)\big\|_{\vec{p}_{v_1}} \lesssim &
\|\Delta^{(1)}_{k_{\max}}\Delta^{(2)}_{m_{\max}}D^{\beta_1(v_1, 1)}_{(1)}D^{\beta_2(v_1, 1)}_{(2)}\Delta^{(1)}_{k_1,\pm}\Delta^{(2)}_{m_1,\pm}f_{1}\|_{\vec{p}_{1}}\sum_{\restr{\delta_1 \otimes \delta_2}{(\mathcal{V}_{1}^{v_1})^c}}\prod_{\substack{ \\ l \neq 1}}\|D_{(1)}^{\delta_1^{-1}(l)}D_{(2)}^{\delta_2^{-1}(l)}f_{l}\|_{\vec{p_l}} \nonumber \\
\lesssim & \|\Delta^{(1)}_{k_{\max}}\Delta^{(2)}_{m_{\max}}D^{\beta_1(v_1, 1)}_{(1)}D^{\beta_2(v_1, 1)}_{(2)}f_{1}\|_{\vec{p}_{1}}\sum_{\restr{\delta_1 \otimes \delta_2}{(\mathcal{V}_{1}^{v_1})^c}}\prod_{\substack{ \\ l \neq 1}}\|D_{(1)}^{\delta_1^{-1}(l)}D_{(2)}^{\delta_2^{-1}(l)}f_{l}\|_{\vec{p_l}}.
%2^{k_1\cdot\beta_1(\mathfrak{r}_{\mathcal{G}}, l_0)}2^{m_2 \cdot\beta_2(\mathfrak{r}_{\mathcal{G}}, l_0)}\|\Delta^{(1)}_{k_{\max}}\Delta^{(2)}_{m_{\max}}f_{l_0}\|_{\vec{p}_{l_0}}\sum_{\restr{\delta_1 \otimes \delta_2}{(\mathcal{V}_{l_0}^{\mathfrak{r}_{\mathcal{G}}})^c}}\prod_{\substack{ \\ l \neq l_0}}\|D_{(1)}^{\delta_1^{-1}(l)}D_{(2)}^{\delta_2^{-1}(l)}f_{l}\|_{\vec{p_l}}.
\end{align}}
Meanwhile, we further decompose the multilinear expression associated to the subtree $\mathcal{G}^{v_2}$ and denote by $R(v_2)$ and $R'(v_2)$ the conical regions for the first and second parameters. We apply (\ref{induction_2parameter_whitney_one}) or (\ref{induction_2parameter_whitney_two}) depending on the type of conical regions on
$$
T^{R(v_2) \times R'(v_2)}_{\mathcal{G}^{v_2}}\big((P^{(1)}_{k_{\max}(v_2)}P^{(2)}_{k_{\max}(v_2)}f_{l})_{\substack{\\ l \in \mathcal{L}(v_2)}}\big).
$$
We also invoke (\ref{leibniz_2parameter_global}) on 
$$
T_{\mathcal{G}^{v_{i}}}\left((S^{(1)}_{k_1}S^{(2)}_{m_1}f_l)_{l \in \mathcal{L}(v_i)}\right)
$$
for $i \neq 1, 2$. Suppose that $\tilde{l} \in \mathfrak{M}_1(R(v_2))$ and $\tilde{l}' \in \cap \mathfrak{M}_2(R'(v_2))$ with $\tilde{l} \neq \tilde{l}'$. Then \eqref{induction_2parameter_whitney_two} together with other inductive hypotheses and the estimates \eqref{est:inductive_G(v_1)} and \eqref{eq:direction:LP:same} imply that
\begin{align*}
\|(\ref{2parameter_max_same_cc})\|_{\vec{r}} \lesssim \sum_{\substack{\restr{\delta_1}{\mathcal{V} \setminus \{\mathfrak{r}_{\mathcal{G}}\} \setminus \mathcal{V}^{v_1}_{1} \setminus \mathcal{V}^{v_2}_{\tilde{l}}} \\ \restr{\delta_2}{\mathcal{V} \setminus \{\mathfrak{r}_{\mathcal{G}}\} \setminus \mathcal{V}^{v_1}_{1} \setminus \mathcal{V}^{v_2}_{\tilde{l}'}}}} \prod_{\substack{ \\ l \neq 1, \tilde{l}, \tilde{l}'}}\|D_{(1)}^{\delta_1^{-1}(l)}D_{(2)}^{\delta_2^{-1}(l)}f_{l}\|_{\vec{p_l}} \cdot  \sum_{\substack{k_{\max}(v_2) \ll k_1 \\ m_{\max}(v_2) \ll m_1}}  2^{k_{\max}(v_2)} 2^{k_1 \cdot (\beta_1^{\mathfrak{r}_{\mathcal{G}}}-1)} 2^{m_{\max}(v_2)} 2^{m_1 \cdot(\beta_2^{\mathfrak{r}_{\mathcal{G}}}-1)}\cdot &\\
 \|\Delta^{(1)}_{k_{\max}}\Delta^{(2)}_{m_{\max}}D^{\beta_1(v_1, 1)}_{(1)}D^{\beta_2(v_1, 1)}_{(2)}f_{1}\|_{\vec{p}_{1}}\cdot \|\Delta^{(1)}_{k_{\max}(v_2)}D^{\beta_1(v_2, \tilde{l})}_{(1)}D_{(2)}^{\delta_2^{-1}(\tilde{l})}f_{\tilde{l}}\|_{\vec{p}_{\tilde{l}}}\|\Delta^{(2)}_{m_{\max}(v_2)}D_{(1)}^{\delta_1^{-1}(\tilde{l}')}D^{\beta_2(v_2, \tilde{l}')}_{(2)}f_{\tilde{l}'}\|_{\vec{p}_{\tilde{l}'}}.&
\end{align*}
The similar estimate developed in Section \ref{sec:5:flag:bi:param:k_1:m_1} can be applied to distribute derivatives as follows: for any fixed $\restr{\delta_1}{\mathcal{V} \setminus \{\mathfrak{r}_{\mathcal{G}}\} \setminus \mathcal{V}^{v_1}_{1} \setminus \mathcal{V}^{v_2}_{\tilde{l}}}$ and $\restr{\delta_2}{\mathcal{V} \setminus \{\mathfrak{r}_{\mathcal{G}}\} \setminus \mathcal{V}^{v_1}_{1} \setminus \mathcal{V}^{v_2}_{\tilde{l}'}}$, the inner sum can be bounded by
\begin{align*}
\sum_{\substack{k_1 \\  m_1}}  \min\big(&2^{k_1 \beta_1^{\mathfrak{r}_{\mathcal{G}}}} 2^{m_1\beta_2^{\mathfrak{r}_{\mathcal{G}}}} \|D^{\beta_1(v_1, 1)}_{(1)}D^{\beta_2(v_1, 1)}_{(2)}f_{1}\|_{\dot{B}^0_{p^1_1, \infty}\dot{B}^0_{p^2_1, \infty}} \|D^{\beta_1(v_2, \tilde{l})}_{(1)}D_{(2)}^{\delta_2^{-1}(\tilde{l})}f_{\tilde{l}}\|_{\dot{B}^0_{p^1_{\tilde{l}}, \infty}L^{p^2_{\tilde{l}}}}\|D_{(1)}^{\delta_1^{-1}(\tilde{l}')}D^{\beta_2(v_2, \tilde{l}')}_{(2)}f_{\tilde{l}'}\|_{L^{p^1_{\tilde{l}'}}\dot{B}^0_{p^2_{\tilde{l}'}, \infty}}, \\
& 2^{k_1 \beta_1^{\mathfrak{r}_{\mathcal{G}}}} 2^{-m_1\epsilon_2}\|D^{\beta_1(v_1, 1)}_{(1)}D^{\beta_2(v_1, 1)}_{(2)}f_{1}\|_{\dot{B}^0_{p^1_1, \infty}\dot{B}^{\beta_2^{\mathfrak{r}_{\mathcal{G}}}}_{p^2_1, \infty}} \|D^{\beta_1(v_2, \tilde{l})}_{(1)}D_{(2)}^{\delta_2^{-1}(\tilde{l})}f_{\tilde{l}}\|_{\dot{B}^0_{p^1_{\tilde{l}}, \infty}L^{p^2_{\tilde{l}}}}\|D_{(1)}^{\delta_1^{-1}(\tilde{l}')}D^{\beta_2(v_2, \tilde{l}')}_{(2)}f_{\tilde{l}'}\|_{L^{p^1_{\tilde{l}'}}\dot{B}^{\epsilon_2}_{p^2_{\tilde{l}'}, \infty}},\\
& 2^{-k_1 \epsilon_1} 2^{m_1\beta_2^{\mathfrak{r}_{\mathcal{G}}}} \|D^{\beta_1(v_1, 1)}_{(1)}D^{\beta_2(v_1, 1)}_{(2)}f_{1}\|_{\dot{B}^{\beta_1^{\mathfrak{r}_{\mathcal{G}}}}_{p^1_1, \infty}\dot{B}^0_{p^2_1, \infty}} \|D^{\beta_1(v_2, \tilde{l})}_{(1)}D_{(2)}^{\delta_2^{-1}(\tilde{l})}f_{\tilde{l}}\|_{\dot{B}^{\epsilon_1}_{p^1_{\tilde{l}}, \infty}L^{p^2_{\tilde{l}}}}\|D_{(1)}^{\delta_1^{-1}(\tilde{l}')}D^{\beta_2(v_2, \tilde{l}')}_{(2)}f_{\tilde{l}'}\|_{L^{p^1_{\tilde{l}'}}\dot{B}^0_{p^2_{\tilde{l}'}, \infty}},\\
&  2^{-k_1 \epsilon_1} 2^{-m_1\epsilon_2} \|D^{\beta_1(v_1, 1)}_{(1)}D^{\beta_2(v_1, 1)}_{(2)}f_{1}\|_{\dot{B}^{\beta_1^{\mathfrak{r}_{\mathcal{G}}}}_{p^1_1, \infty}\dot{B}^{\beta_2^{\mathfrak{r}_{\mathcal{G}}}}_{p^2_1, \infty}} \|D^{\beta_1(v_2, \tilde{l})}_{(1)}D_{(2)}^{\delta_2^{-1}(\tilde{l})}f_{\tilde{l}}\|_{\dot{B}^{\epsilon_1}_{p^1_{\tilde{l}}, \infty}L^{p^2_{\tilde{l}}}}\|D_{(1)}^{\delta_1^{-1}(\tilde{l}')}D^{\beta_2(v_2, \tilde{l}')}_{(2)}f_{\tilde{l}'}\|_{L^{p^1_{\tilde{l}'}}\dot{B}^{\epsilon_2}_{p^2_{\tilde{l}'}, \infty}} \big)
\end{align*}
for $0 < \epsilon_j < \min(1, \beta_j^{\mathfrak{r}_{\mathcal{G}}})$, $j = 1,2$. By the optimization and interpolation procedure specified in Section \ref{generic_optim_interp}, we attain the right hand side of the inductive statement (\ref{induction_2parameter_cone_statement}).
 \vskip .1in

\item[$\bullet$] \underline{estimating $\mathcal{A}_1 \cdot \mathcal{B}_2$ \eqref{root_symbol_decomp_commutator}:}

The multipliers generated by $\mathcal{A}_1 \cdot \mathcal{B}_2$ are similar to the ones generated by $\mathcal{A}_2 \cdot \mathcal{B}_1$; by symmetry it will be enough to focus on the former. The symbol $\mathcal{A}_1\cdot \mathcal{B}_2$ with $\tilde{i} = 2$ in (\ref{root_symbol_decomp_commutator}) generates the multiplier
\begin{align*}
\sum_{\substack{L_1, L_2 \in \mathbb{Z} \\ }}C^{\pm}_{L_1, L_2} \sum_{\substack{k_{\max}(v_2) \ll k_1\\ m_1}} & 2^{k_{\max}(v_2)} 2^{k_1 \cdot (\beta_1^{\mathfrak{r}_{\mathcal{G}}}-1)} D_{(2)}^{\beta_2^{\mathfrak{r}_{\mathcal{G}}}}T_{\mathcal{G}^{v_1}}\Big(\Delta^{(1)}_{k_1,\pm, \frac{L_1}{2^{k_1}}}\Delta^{(2)}_{m_1}f_1,  (S^{(1)}_{k_1, \frac{L_1}{2^{k_1}}}S^{(2)}_{m_1}f_l)_{\substack{l \in \mathcal{L}(v_1) \\ l \neq 1}})\Big)\cdot \\
& T_{\mathcal{G}^{v_2}}\Big((P^{(1)}_{k_{\max}(v_2), \frac{L_2}{2^{k_1}} }S^{(2)}_{m_1}f_{l})_{\substack{\\ l \in \mathcal{L}(v_2)}}\Big)\prod_{i=3}^{n_1} T_{\mathcal{G}^{v_{i}}}\Big((S^{(1)}_{k_1, \frac{L_2}{2^{k_1}}}S^{(2)}_{m_1}f_l)_{l \in \mathcal{L}(v_i)}\Big).
\end{align*}
Let $\tilde{\mathcal{G}}^{v_1}$ denote the tree having the same structure as $\mathcal{G}^{v_1}$, except that the differential operator associated to the vertex $v_1$ is replaced by $D_{(1)}^{\beta_1^{v_1}}D_{(2)}^{\beta_2^{v_1} + \beta_2^{\mathfrak{r}_{\mathcal{G}}}}$. As before, the estimate for the above term is the same as the simpler term when $L_i= 0$ for $i = 1, 2$:

\begin{align*}
\sum_{\substack{k_{\max}(v_2) \ll k_1\\ m_1}} 2^{k_{\max}(v_2)} 2^{k_1 \cdot (\beta_1^{\mathfrak{r}_{\mathcal{G}}}-1)} \cdot & T_{\tilde{\mathcal{G}}^{v_1}}\Big(\Delta^{(1)}_{k_1,\pm}\Delta^{(2)}_{m_1}f_1,  (S^{(1)}_{k_1}S^{(2)}_{m_1}f_l)_{\substack{l \in \mathcal{L}(v_1) \\ l \neq 1}})\Big)\cdot T_{\mathcal{G}^{v_2}}\left((P^{(1)}_{k_{\max}(v_2)}S^{(2)}_{m_1}f_{l})_{\substack{\\ l \in \mathcal{L}(v_2)}}\right) \\
\cdot& \prod_{i=3}^{n_1} T_{\mathcal{G}^{v_{i}}}\left((S^{(1)}_{k_1}S^{(2)}_{m_1}f_l)_{l \in \mathcal{L}(v_i)}\right).
\end{align*}
As in the one-parameter setting, we perform finer paraproduct decompositions on the subtrees $\mathcal{G}^{v_i}$, $i \neq 1$, for the second parameter:
\begin{align} \label{commu_diff_before_hilow}
\sum_{\substack{k_{\max}(v_2) \ll k_1\\  m_1 \\m_{\max}(v_2) }} & 2^{k_{\max}(v_2)} 2^{k_1 \cdot (\beta_1^{\mathfrak{r}_{\mathcal{G}}}-1)} T_{\tilde{\mathcal{G}}^{v_1}}\Big(\Delta^{(1)}_{k_1,\pm}\Delta^{(2)}_{m_1}f_1,  (S^{(1)}_{k_1}S^{(2)}_{m_1}f_l)_{\substack{l \in \mathcal{L}(v_1) \\ l \neq 1}})\Big)\cdot \nonumber\\
& T_{\mathcal{G}^{v_2}}\Big((P^{(1)}_{k_{\max}(v_2)}P^{(2)}_{m_{\max}(v_2)}S_{m_1}^{(2)}f_{l})_{\substack{\\ l \in \mathcal{L}(v_2)}}\Big) \prod_{i=3}^{n_1} \sum_{m_{\max}(v_i)}T_{\mathcal{G}^{v_{i}}}\Big((S^{(1)}_{k_1}P^{(2)}_{m_{\max}(v_i)}S_{m_1}^{(2)}f_l)_{l \in \mathcal{L}(v_i)}\Big).
\end{align}
Due to the observation that $\Delta_{m_{\max}(v_i)}^{(2)}S_{m_1}^{(2)} \nequiv 0 $ if only if $m_{\max}(v_i) \ll m_1$, (\ref{commu_diff_before_hilow}) can be simplified as
\begin{align*}
 \sum_{\substack{k_{\max}(v_2) \ll k_1\\  \\m_{\max}(v_2) \ll m_1}}& 2^{k_{\max}(v_2)} 2^{k_1 \cdot (\beta_1^{\mathfrak{r}_{\mathcal{G}}}-1)} T_{\tilde{\mathcal{G}}^{v_1}}\Big(\Delta^{(1)}_{k_1,\pm}\Delta^{(2)}_{m_1}f_1,  (S^{(1)}_{k_1}S^{(2)}_{m_1}f_l)_{\substack{l \in \mathcal{L}(v_1) \\ l \neq 1}})\Big)\cdot \\
& T_{\mathcal{G}^{v_2}}\big((P^{(1)}_{k_{\max}(v_2)}P^{(2)}_{m_{\max}(v_2)}f_{l})_{\substack{\\ l \in \mathcal{L}(v_2)}}\big) \prod_{i=3}^{n_1} \sum_{m_{\max}(v_i): m_{\max}(v_i) \ll m_1}T_{\mathcal{G}^{v_{i}}}\big((S^{(1)}_{k_1}P^{(2)}_{m_{\max}(v_i)}f_l)_{l \in \mathcal{L}(v_i)}\big).
\end{align*}
We then apply the high-low switch technique to reduce the expression above to a sum of terms that can be estimated using the inductive hypotheses. In particular, 
\begin{align*}
& \sum_{\substack{k_{\max}(v_2) \ll k_1\\ m_1 \\ m_{\max}(v_2)}}  2^{k_{\max}(v_2)} 2^{k_1 (\beta_1^{\mathfrak{r}_{\mathcal{G}}}-1)} T_{\tilde{\mathcal{G}}^{v_1}}\Big(\Delta^{(1)}_{k_1,\pm}\Delta^{(2)}_{m_1}f_1,  (S^{(1)}_{k_1}S^{(2)}_{m_1}f_l)_{\substack{l \in \mathcal{L}(v_1) \\ l \neq 1}})\Big) T_{\mathcal{G}^{v_2}}\Big((P^{(1)}_{k_{\max}(v_2)}P^{(2)}_{m_{\max}(v_2)}f_{l})_{\substack{\\ l \in \mathcal{L}(v_2)}}\Big)\\
&\quad \quad \quad \quad \quad \cdot \prod_{i=3}^{n_1} \sum_{m_{\max}(v_i)}T_{\mathcal{G}^{v_{i}}}\left((S^{(1)}_{k_1}P^{(2)}_{m_{\max}(v_i)}f_l)_{l \in \mathcal{L}(v_i)}\right)  \\
& -\sum_{\substack{k_{\max}(v_2) \ll k_1\\ m_{\max(v_2)} \succ m_1 \\ }}  2^{k_{\max}(v_2)} 2^{k_1 (\beta_1^{\mathfrak{r}_{\mathcal{G}}}-1)} T_{\tilde{\mathcal{G}}^{v_1}}\Big(\Delta^{(1)}_{k_1,\pm}\Delta^{(2)}_{m_1}f_1,  (S^{(1)}_{k_1}S^{(2)}_{m_1}f_l)_{\substack{l \in \mathcal{L}(v_1) \\ l \neq 1}})\Big)T_{\mathcal{G}^{v_2}}\Big((P^{(1)}_{k_{\max}(v_2)}P^{(2)}_{m_{\max}(v_2)}f_{l})_{\substack{\\ l \in \mathcal{L}(v_2)}}\Big) \\
&\quad \quad \quad \quad \quad \cdot \prod_{i=3}^{n_1}\sum_{m_{\max}(v_i)} T_{\mathcal{G}^{v_{i}}}\big((S^{(1)}_{k_1}P^{(2)}_{m_{\max}(v_i)}f_l)_{l \in \mathcal{L}(v_i)}\big) \\
& \pm \text{similar terms} : = I - II \pm \text{similar terms}. 
\end{align*}

We will elaborate on the estimates for the first term denoted by $I$ and the second term denoted by $II$.
For $I$, we first recall the inductive hypothesis (\ref{induction_2parameter_fix1scale}) that allows to control, for   $k_1$ and $k_{\max}(v_2)$ fixed,
\[
\quad \sum_{m_1} T_{\tilde{\mathcal{G}}^{v_1}}\big(\Delta^{(1)}_{k_1,\pm}\Delta^{(2)}_{m_1}f_1,  (S^{(1)}_{k_1}S^{(2)}_{m_1}f_l)_{\substack{l \in \mathcal{L}(v_1) \\ l \neq 1}})\big)  \qquad \text{and} \qquad \sum_{m_{\max}(v_2)}T_{\mathcal{G}^{v_2}}\big((P^{(1)}_{k_{\max}(v_2)}P^{(2)}_{m_{\max}(v_2)}f_{l})_{\substack{\\ l \in \mathcal{L}(v_2)}}\big).
\]
For $i \geq 3$ and $l \in \mathcal{L}(v_i)$, define
\begin{equation*}
\tilde{f}_l := S^{(1)}_{k_1} f_l.
\end{equation*}
and apply the inductive hypothesis (\ref{leibniz_2parameter_semi_global}) to 
$$
\sum_{m_{\max}(v_i)}T_{\mathcal{G}^{v_{i}}}\big((P^{(2)}_{m_{\max}(v_i)}\tilde{f}_l)_{l \in \mathcal{L}(v_i)}\big).
$$
Suppose that $l_0 \in \mathfrak{M}_1(R(v_2))$. Combining the estimates from the inductive hypotheses, we derive that

\begin{align*}
\|I\|_{\vec{r}} \lesssim & \sum_{\substack{\delta_2 \\ \restr{\delta_1}{\mathcal{V} \setminus \{\mathfrak{r}_{\mathcal{G}}\} \setminus \mathcal{V}^{v_1}_1\setminus \mathcal{V}^{v_2}_{l_0} }}} \prod_{\substack{ \\ l \neq 1, l_0}}\|D_{(1)}^{\delta_1^{-1}(l)}D_{(2)}^{\delta_2^{-1}(l)}f_{l}\|_{\vec{p_l}}\sum_{k_{\max}(v_2) \ll k_1}  2^{k_{\max}(v_2)} 2^{k_1 (\beta_1^{\mathfrak{r}_{\mathcal{G}}}-1)} \\
& \qquad \cdot  \|\Delta^{(1)}_{k_{1}}D_{(1)}^{\beta_1(v_1,1)}D^{\delta_2^{-1}(1)}_{(2)}f_{1}\|_{\vec{p}_{1}} \|\Delta^{(1)}_{k_{\max}(v_2)}D^{\beta_1(v_2,\tilde{l})}_{(1)}D^{\delta_2^{-1}(\tilde{l})}_{(2)}f_{\tilde{l}}\|_{\vec{p}_{\tilde{l}}}. 
%&\ \ \quad \quad \quad \quad \quad \sum_{k_{\max}(v_2) \ll k_1}  2^{k_{\max}(v_2)} 2^{k_1 (\beta_1^{\mathfrak{r}_{\mathcal{G}}}-1)}  \|\Delta^{(1)}_{k_{1}}D_{(1)}^{\beta_1(v_1,1)}D^{\delta_2^{-1}(1)}_{(2)}f_{1}\|_{\vec{p}_{1}} \|\Delta^{(1)}_{k_{\max}(v_2)}D^{\beta_1(v_2,\tilde{l})}_{(1)}D^{\delta_2^{-1}(\tilde{l})}_{(2)}f_{\tilde{l}}\|_{\vec{p}_{\tilde{l}}}.
\end{align*}
Now we fix $\delta_2$ and $\restr{\delta_1}{\mathcal{V} \setminus \{\mathfrak{r}_{\mathcal{G}}\} \setminus \mathcal{V}^{v_1}_1\setminus \mathcal{V}^{v_2}_{l_0} }$ and distribute the partial derivatives in the first parameter:
\begin{align*}
 \sum_{k_{\max}(v_2) \ll k_1}  & 2^{k_{\max}(v_2)} 2^{k_1 (\beta_1^{\mathfrak{r}_{\mathcal{G}}}-1)}  \|\Delta^{(1)}_{k_{1}}D_{(1)}^{\beta_1(v_1,1)}D^{\delta_2^{-1}(1)}_{(2)}f_{1}\|_{\vec{p}_{1}} \|\Delta^{(1)}_{k_{\max}(v_2)}D^{\beta_1(v_2,\tilde{l})}_{(1)}D^{\delta_2^{-1}(\tilde{l})}_{(2)}f_{\tilde{l}}\|_{\vec{p}_{\tilde{l}}} \\
\lesssim & \sum_{k_1} \min\Big(2^{k_1 \beta_1^{\mathfrak{r}_{\mathcal{G}}}}  \|D^{\beta_1(v_1, 1)}_{(1)}D^{\delta_2^{-1}(1)}_{(2)}f_{1}\|_{\dot{B}^0_{p^1_1, \infty}L^{p^2_1}} \|D^{\beta_1(v_2, \tilde{l})}_{(1)}D_{(2)}^{\delta_2^{-1}(\tilde{l})}f_{\tilde{l}}\|_{\dot{B}^0_{p^1_{\tilde{l}}, \infty}L^{p^2_{\tilde{l}}}}, \\
& \ \  \quad \quad\quad  2^{-k_1 \epsilon_1}  \|D^{\beta_1(v_1, 1)}_{(1)}D^{\delta_2^{-1}(1)}_{(2)}f_{1}\|_{\dot{B}^{\beta_1^{\mathfrak{r}_{\mathcal{G}}}}_{p^1_1, \infty}L^{p^2_1}} \|D^{\beta_1(v_2, \tilde{l})}_{(1)}D_{(2)}^{\delta_2^{-1}(\tilde{l})}f_{\tilde{l}}\|_{\dot{B}^{\epsilon_1}_{p^1_{\tilde{l}}, \infty}L^{p^2_{\tilde{l}}}}\Big);
\end{align*}
using the optimization and interpolation described in Section \ref{generic_optim_interp}, we deduce the inductive statement (\ref{induction_2parameter_cone_statement}). The term $II$ requires the same inductive hypothesis (\ref{leibniz_2parameter_semi_global}) on 
$$
\sum_{m_{\max}(v_i)}T_{\mathcal{G}^{v_{i}}}\big((S^{(1)}_{k_1}P^{(2)}_{m_{\max}(v_i)}f_l)_{l \in \mathcal{L}(v_i)}\big)
$$
for $i \geq 3$, while (\ref{induction_2parameter_whitney_one}) yields estimates on $T_{\tilde{\mathcal{G}}^{v_1}}\big(\Delta^{(1)}_{k_1,+}\Delta^{(2)}_{m_1}f_1,  (S^{(1)}_{k_1}S^{(2)}_{m_1}f_l)_{\substack{l \in \mathcal{L}(v_1) \\ l \neq 1}})\big)$.

Recall that $\tilde{l} \in \mathfrak{M}_1(R(v_2))$ and further assume that $\tilde{l}' \in \mathfrak{M}_2(R'(v_2))$ with %$\tilde{l}, \tilde{l}' \in \mathcal{L}(v_2)$ and 
$\tilde{l} \neq \tilde{l}'$.\footnote{If $l = \tilde{l}'$, then the inductive hypothesis (\ref{induction_2parameter_whitney_one}) will be used instead.} Then (\ref{induction_2parameter_whitney_two}) is applicable to
$$
T_{\mathcal{G}^{v_2}}\big((P^{(1)}_{k_{\max}(v_2)}P^{(2)}_{m_{\max}(v_2)}f_{l})_{\substack{\\ l \in \mathcal{L}(v_2)}}\big).$$
As a consequence,
\begin{align*}
 \|II\|_{\vec{r}} \lesssim &  \sum_{\substack{\restr{\delta_1}{\mathcal{V} \setminus \{\mathfrak{r}_{\mathcal{G}}\} \setminus \mathcal{V}^{v_1}_1 \setminus \mathcal{V}^{v_2}_{\tilde{l}} } \\ \restr{\delta_2}{\mathcal{V}}\setminus \{\mathfrak{r}_{\mathcal{G}} \}  \setminus \mathcal{V}^{v_2}_{\tilde{l}'} }}\prod_{l \neq 1, \tilde{l}, \tilde{l}'} \|D^{\delta_1^{-1}(l)}_{(1)} D^{\delta_2^{-1}(l)}_{(2)}f_l \|_{\vec{p}_l}  \sum_{\substack{k_{\max}(v_2) \ll k_1\\ m_{\max(v_2)} > m_1 \\ }}  2^{k_{\max}(v_2)} 2^{k_1 (\beta_1^{\mathfrak{r}_{\mathcal{G}}}-1)} 2^{m_1\beta_2^{\mathfrak{r}_{\mathcal{G}}}}\\
&\ \ \quad \quad \quad \quad \|\Delta^{(1)}_{k_{1}}\Delta^{(2)}_{m_1}D_{(1)}^{\beta_1(v_1,1)}D^{\beta_2(v_1,1)}_{(2)}f_{1}\|_{\vec{p}_{1}} \|\Delta^{(1)}_{k_{\max}(v_2)}D^{\beta_1(v_2,\tilde{l})}_{(1)}D^{\delta_2^{-1}(\tilde{l})}_{(2)}f_{\tilde{l}}\|_{\vec{p}_{\tilde{l}}} \|\Delta^{(2)}_{m_{\max}(v_2)}D^{\delta_1^{-1}(\tilde{l}')}_{(2)}D^{\beta_2(v_2,\tilde{l}')}_{(2)}f_{\tilde{l}'}\|_{\vec{p}_{\tilde{l}'}}.
\end{align*}
In the inner sum the partial derivatives can be appropriately distributed as before:
{\fontsize{8}{8}
\begin{align*}
& \sum_{\substack{k_{\max}(v_2) \ll k_1\\ m_{\max(v_2)} > m_1 \\ }}  2^{k_{\max}(v_2)} 2^{k_1 (\beta_1^{\mathfrak{r}_{\mathcal{G}}}-1)} 2^{m_1\beta_2^{\mathfrak{r}_{\mathcal{G}}}}\|\Delta^{(1)}_{k_{1}}\Delta^{(2)}_{m_1}D_{(1)}^{\beta_1(v_1,1)}D^{\beta_2(v_1,1)}_{(2)}f_{1}\|_{\vec{p}_{1}} \\
&\quad \quad \quad \quad \quad \quad   \|\Delta^{(1)}_{k_{\max}(v_2)}D^{\beta_1(v_2,\tilde{l})}_{(1)}D^{\delta_2^{-1}(\tilde{l})}_{(2)}f_{\tilde{l}}\|_{\vec{p}_{\tilde{l}}} \|\Delta^{(2)}_{m_{\max}(v_2)}D^{\delta_1^{-1}(\tilde{l}')}_{(2)}D^{\beta_2(v_2,\tilde{l}')}_{(2)}f_{\tilde{l}'}\|_{\vec{p}_{\tilde{l}'}} \\
& \lesssim \sum_{\substack{k_1 \\ m_{\max}(v_2)}} \min(2^{k_1\beta_1^{\mathfrak{r}_{\mathcal{G}}}} 2^{m_{\max}(v_2)\beta_2^{\mathfrak{r}_{\mathcal{G}}}}\|D_{(1)}^{\beta_1(v_1,1)}D^{\beta_2(v_1,1)}_{(2)}f_{1}\|_{\dot{B}^{0}_{p_1^1, \infty} \dot{B}^{0}_{p_1^2, \infty}} \|D^{\beta_1(v_2,\tilde{l})}_{(1)}D^{\delta_2^{-1}(\tilde{l})}_{(2)}f_{\tilde{l}}\|_{\dot{B}^{0}_{p_{\tilde{l}}^1, \infty} L^{p_1^2}} \|D^{\delta_1^{-1}(\tilde{l}')}_{(2)}D^{\beta_2(v_2,\tilde{l}')}_{(2)}f_{\tilde{l}'}\|_{L^{p_{\tilde{l}'}^1} \dot{B}^{0}_{p_{\tilde{l}'}^2, \infty}},\\
& \quad \quad \quad \quad \quad \quad 2^{k_1\beta_1^{\mathfrak{r}_{\mathcal{G}}}} 2^{-m_{\max}(v_2)\epsilon_2}\|D_{(1)}^{\beta_1(v_1,1)}D^{\beta_2(v_1,1)}_{(2)}f_{1}\|_{\dot{B}^{0}_{p_1^1, \infty} \dot{B}^{\epsilon_2}_{p_1^2, \infty}} \|D^{\beta_1(v_2,\tilde{l})}_{(1)}D^{\delta_2^{-1}(\tilde{l})}_{(2)}f_{\tilde{l}}\|_{\dot{B}^{0}_{p_{\tilde{l}}^1, \infty} L^{p_1^2}} \|D^{\delta_1^{-1}(\tilde{l}')}_{(2)}D^{\beta_2(v_2,\tilde{l}')}_{(2)}f_{\tilde{l}'}\|_{L^{p_{\tilde{l}'}^1} \dot{B}^{\beta_2^{\mathfrak{r}_{\mathcal{G}}}}_{p_{\tilde{l}'}^2, \infty}},\\
& \quad \quad \quad \quad \quad \quad 2^{-k_1\epsilon_1} 2^{m_{\max}(v_2)\beta_2^{\mathfrak{r}_{\mathcal{G}}}}\|D_{(1)}^{\beta_1(v_1,1)}D^{\beta_2(v_1,1)}_{(2)}f_{1}\|_{\dot{B}^{\beta_1^{\mathfrak{r}_{\mathcal{G}}}}_{p_1^1, \infty} \dot{B}^{0}_{p_1^2, \infty}} \|D^{\beta_1(v_2,\tilde{l})}_{(1)}D^{\delta_2^{-1}(\tilde{l})}_{(2)}f_{\tilde{l}}\|_{\dot{B}^{\epsilon_1}_{p_{\tilde{l}}^1, \infty} L^{p_1^2}} \|D^{\delta_1^{-1}(\tilde{l}')}_{(2)}D^{\beta_2(v_2,\tilde{l}')}_{(2)}f_{\tilde{l}'}\|_{L^{p_{\tilde{l}'}^1} \dot{B}^{0}_{p_{\tilde{l}'}^2, \infty}},\\
& \quad \quad \quad \quad \quad \quad 2^{-k_1\epsilon_1} 2^{-m_{\max}(v_2)\epsilon_2}\|D_{(1)}^{\beta_1(v_1,1)}D^{\beta_2(v_1,1)}_{(2)}f_{1}\|_{\dot{B}^{\beta_1^{\mathfrak{r}_{\mathcal{G}}}}_{p_1^1, \infty} \dot{B}^{\epsilon_2}_{p_1^2, \infty}} \|D^{\beta_1(v_2,\tilde{l})}_{(1)}D^{\delta_2^{-1}(\tilde{l})}_{(2)}f_{\tilde{l}}\|_{\dot{B}^{\epsilon_1}_{p_{\tilde{l}}^1, \infty} L^{p_1^2}} \|D^{\delta_1^{-1}(\tilde{l}')}_{(2)}D^{\beta_2(v_2,\tilde{l}')}_{(2)}f_{\tilde{l}'}\|_{L^{p_{\tilde{l}'}^1} \dot{B}^{\beta_2^{\mathfrak{r}_{\mathcal{G}}}}_{p_{\tilde{l}'}^2, \infty}}),
\end{align*}}
to which we can apply optimization and interpolation to obtain the desired estimates described in the right hand side of (\ref{induction_2parameter_cone_statement}). 
\vskip .1in
\item[$\bullet$] \underline{estimating $\mathcal{A}_2 \cdot \mathcal{B}_2$ \eqref{root_symbol_decomp_commutator}:}

The terms generated by $\mathcal{A}_2 \cdot \mathcal{B}_2$ take the form
\begin{align*}
& \sum_{k_1, m_1}D_{(1)}^{\beta_1^{\mathfrak{r}_{\mathcal{G}}}}D_{(2)}^{\beta_2^{\mathfrak{r}_{\mathcal{G}}}}T_{\mathcal{G}^{v_1}}\big(\Delta^{(1)}_{k_1}\Delta^{(2)}_{m_1}f_1,  (S^{(1)}_{k_1}S^{(2)}_{m_1}f_l)_{\substack{l \in \mathcal{L}(v_1) \\ l \neq 1}})\big) \prod_{i=2}^{n_1} T_{\mathcal{G}^{v_{i}}}\big((S^{(1)}_{k_1}S^{(2)}_{m_1}f_l)_{l \in \mathcal{L}(v_i)}\big).
\end{align*}
Let $\tilde{\mathcal{G}}^{v_1}$ be the tree having the same configuration as $\mathcal{G}^{v_1}$, with the original differential operator associated to the vertex $v_1$ 
$$
D_{(1)}^{\beta_1^{v_1}}D_{(2)}^{\beta_2^{v_1}}
$$
replaced by 
$$
D_{(1)}^{\beta_1^{v_1} + \beta_1^{\mathfrak{r}_{\mathcal{G}}}}D_{(2)}^{\beta_2^{v_1}+ \beta_2^{\mathfrak{r}_{\mathcal{G}}}}.
$$
The treatment of the second parameter in last section (\underline{estimating $\mathcal{A}_1 \cdot \mathcal{B}_2$}) can be used in both parameters here. In particular, the finer paraproduct decompositions on both parameters yield
\begin{align} \label{indcution_2parameter_c}
& \sum_{\substack{k_1 \\ m_1 \\ }}T_{\tilde{\mathcal{G}}^{v_1}}\big(\Delta^{(1)}_{k_1}\Delta^{(2)}_{m_1}f_1,  (S^{(1)}_{k_1}S^{(2)}_{m_1}f_l)_{\substack{l \in \mathcal{L}(v_1) \\ l \neq 1}})\big) \prod_{i=2}^{n_1} \sum_{\substack{k_{\max}(v_i) \\ m_{\max}(v_i)}}T_{\mathcal{G}^{v_{i}}}\big((P_{k_{\max}(v_i)}^{(1)}P_{m_{\max}(v_i)}^{(2)}S^{(1)}_{k_1}S^{(2)}_{m_1}f_l)_{l \in \mathcal{L}(v_i)}\big) \nonumber\\
= & \sum_{\substack{k_1 \\ m_1}}T_{\tilde{\mathcal{G}}^{v_1}}\big(\Delta^{(1)}_{k_1}\Delta^{(2)}_{m_1}f_1,  (S^{(1)}_{k_1}S^{(2)}_{m_1}f_l)_{\substack{l \in \mathcal{L}(v_1) \\ l \neq 1}})\big) \prod_{i=2}^{n_1} \sum_{\substack{k_{\max}(v_i): k_{\max}(v_i) \ll k_1 \\ m_{\max}(v_i): m_{\max}(v_i)\ll m_1}}T_{\mathcal{G}^{v_{i}}}\big((P_{k_{\max}(v_i)}^{(1)}P_{m_{\max}(v_i)}^{(2)}f_l)_{l \in \mathcal{L}(v_i)}\big),
\end{align}
where the equality follows from the fact that $\Delta_{\tilde{k}}^{(j)} S_k^{(j)}\nequiv 0$ for $j = 1,2$ if and only if $\tilde{k} \ll k$. We then apply the high-low switch technique to rewrite (\ref{indcution_2parameter_c}) as
{\fontsize{8}{8}
\begin{align*}
& \sum_{\substack{k_1 \\ m_1 \\ }}T_{\tilde{\mathcal{G}}^{v_1}}\big(\Delta^{(1)}_{k_1}\Delta^{(2)}_{m_1}f_1,  (S^{(1)}_{k_1}S^{(2)}_{m_1}f_l)_{\substack{l \in \mathcal{L}(v_1) \\ l \neq 1}})\big) \prod_{i=2}^{n_1} \sum_{\substack{k_{\max}(v_i) \\ m_{\max}(v_i)}}T_{\mathcal{G}^{v_{i}}}\left((P_{k_{\max}(v_i)}^{(1)}P_{m_{\max}(v_i)}^{(2)}f_l)_{l \in \mathcal{L}(v_i)}\right) + \\
& \sum_{\substack{k_{\max}(v_2) > k_1 \\ m_{\max}(v_2) > m_1}} T_{\tilde{\mathcal{G}}^{v_1}}\big(\Delta^{(1)}_{k_1}\Delta^{(2)}_{m_1}f_1,  (S^{(1)}_{k_1}S^{(2)}_{m_1}f_l)_{\substack{l \in \mathcal{L}(v_1) \\ l \neq 1}})\big)T_{\mathcal{G}^{v_{2}}}\big((P_{k_{\max}(v_2)}^{(1)}P_{m_{\max}(v_2)}^{(2)}f_l)_{l \in \mathcal{L}(v_i)}\big) \\
& \qquad \cdot \prod_{i=3}^{n_1} \sum_{\substack{k_{\max}(v_i) \\ m_{\max}(v_i)}}T_{\mathcal{G}^{v_{i}}}\big((P_{k_{\max}(v_i)}^{(1)}P_{m_{\max}(v_i)}^{(2)}f_l)_{l \in \mathcal{L}(v_i)}\big)  \pm \text{similar terms} := I + II \pm \text{similar terms}. 
\end{align*}}
By applying the inductive hypothesis (\ref{induction_2parameter_cone_statement}) to both 
$$
\sum_{k_1,m_1}T_{\tilde{\mathcal{G}}^{v_1}}\big(\Delta^{(1)}_{k_1}\Delta^{(2)}_{m_1}f_1,  (S^{(1)}_{k_1}S^{(2)}_{m_1}f_l)_{\substack{l \in \mathcal{L}(v_1) \\ l \neq 1}})\big) \quad \text{and} \quad \sum_{\substack{k_{\max}(v_i) \\ m_{\max}(v_i)}}T_{\mathcal{G}^{v_{i}}}\big((P_{k_{\max}(v_i)}^{(1)}P_{m_{\max}(v_i)}^{(2)}f_l)_{l \in \mathcal{L}(v_i)}\big)
$$
for $i \geq 2$, we derive the estimate on the right hand side of (\ref{induction_2parameter_cone_statement}) for $I$. 

On the other hand, different inductive hypotheses are called for in dealing with different subtrees involved in $II$. More precisely, (\ref{induction_2parameter_cone_statement}) is invoked to estimate
$$
\sum_{\substack{k_{\max}(v_i)\\ m_{\max}(v_i)}}T_{\mathcal{G}^{v_{i}}}\big((P_{k_{\max}(v_i)}^{(1)}P_{m_{\max}(v_i)}^{(2)}f_l)_{l \in \mathcal{L}(v_i)}\big) 
$$
for $i \geq 3$. Meanwhile, (\ref{induction_2parameter_whitney_one}) is used for
$$
T_{\tilde{\mathcal{G}}^{v_1}}\big(\Delta^{(1)}_{k_1}\Delta^{(2)}_{m_1}f_1,  (S^{(1)}_{k_1}S^{(2)}_{m_1}f_l)_{\substack{l \in \mathcal{L}(v_1) \\ l \neq 1}})\big).
$$ 
Also, we denote by $R(v_2)$ and $R'(v_2)$ the conical regions associated to the subtree $\mathcal{G}^{v_2}$ for the first and second parameters. Suppose that $\tilde{l} \in \mathfrak{M}_1(R(v_2))$, $\tilde{l}' \in \mathfrak{M}_2(R'(v_2))$ with 
%$\tilde{l}, \tilde{l}' \in \mathcal{L}(v_2)$ and 
$\tilde{l} \neq \tilde{l}'$, then (\ref{induction_2parameter_whitney_two}) can be applied. Combining all the subtree estimates, we conclude that
\begin{align*}
\|II\|_{\vec{r}}  \lesssim& \sum_{\substack{\restr{\delta_1}{\mathcal{V} \setminus \{\mathfrak{r}_{\mathcal{G}}\} \setminus \mathcal{V}^{v_1}_{1} \setminus \mathcal{V}^{v_2}_{\tilde{l}}} \\ \restr{\delta_2}{\mathcal{V} \setminus \{\mathfrak{r}_{\mathcal{G}}\} \setminus \mathcal{V}^{v_1}_{1} \setminus \mathcal{V}^{v_2}_{\tilde{l}'}}}} \prod_{\substack{ \\ l \neq 1, \tilde{l}, \tilde{l}'}}\|D_{(1)}^{\delta_1^{-1}(l)}D_{(2)}^{\delta_2^{-1}(l)}f_{l}\|_{\vec{p_l}} \cdot  \sum_{\substack{k_{\max}(v_2) > k_1 \\ m_{\max}(v_2) > m_1}}  2^{k_1 \beta_1^{\mathfrak{r}_{\mathcal{G}}}} 2^{m_1 \beta_2^{\mathfrak{r}_{\mathcal{G}}}}\\
& \quad \cdot  \|\Delta^{(1)}_{k_{\max}}\Delta^{(2)}_{m_{\max}}D^{\beta_1(v_1, 1)}_{(1)}D^{\beta_2(v_1, 1)}_{(2)}f_{1}\|_{\vec{p}_{1}}\cdot \|\Delta^{(1)}_{k_{\max}(v_2)}D^{\beta_1(v_2, \tilde{l})}_{(1)}D_{(2)}^{\delta_2^{-1}(\tilde{l})}f_{\tilde{l}}\|_{\vec{p}_{\tilde{l}}}\|\Delta^{(2)}_{m_{\max}(v_2)}D_{(1)}^{\delta_1^{-1}(\tilde{l}')}D^{\beta_2(v_2, \tilde{l}')}_{(2)}f_{\tilde{l}'}\|_{\vec{p}_{\tilde{l}'}}.
\end{align*}
The distribution of partial derivatives on both parameters gives the following estimate of the inner sum with $\restr{\delta_1}{\mathcal{V} \setminus \{\mathfrak{r}_{\mathcal{G}}\} \setminus \mathcal{V}^{v_1}_{1} \setminus \mathcal{V}^{v_2}_{\tilde{l}}}$ and $\restr{\delta_2}{\mathcal{V} \setminus \{\mathfrak{r}_{\mathcal{G}}\} \setminus \mathcal{V}^{v_1}_{1} \setminus \mathcal{V}^{v_2}_{\tilde{l}'}}$ fixed: 
{\fontsize{9}{9}
\begin{align*}
 \sum_{\substack{k_{\max}(v_2) \\ m_{\max}(v_2)}} &\min(2^{k_{\max}(v_2)\beta_1^{\mathfrak{r}_{\mathcal{G}}}} 2^{m_{\max}(v_2)\beta_2^{\mathfrak{r}_{\mathcal{G}}}}\|D_{(1)}^{\beta_1(v_1,1)}D^{\beta_2(v_1,1)}_{(2)}f_{1}\|_{\dot{B}^{0}_{p_1^1, \infty} \dot{B}^{0}_{p_1^2, \infty}} \|D^{\beta_1(v_2,\tilde{l})}_{(1)}D^{\delta_2^{-1}(\tilde{l})}_{(2)}f_{\tilde{l}}\|_{\dot{B}^{0}_{p_{\tilde{l}}^1, \infty} L^{p_1^2}} \|D^{\delta_1^{-1}(\tilde{l}')}_{(2)}D^{\beta_2(v_2,\tilde{l}')}_{(2)}f_{\tilde{l}'}\|_{L^{p_{\tilde{l}'}^1} \dot{B}^{0}_{p_{\tilde{l}'}^2, \infty}},\\
& 2^{k_{\max}(v_2)\beta_1^{\mathfrak{r}_{\mathcal{G}}}} 2^{-m_{\max}(v_2)\epsilon_2}\|D_{(1)}^{\beta_1(v_1,1)}D^{\beta_2(v_1,1)}_{(2)}f_{1}\|_{\dot{B}^{0}_{p_1^1, \infty} \dot{B}^{\epsilon_2}_{p_1^2, \infty}} \|D^{\beta_1(v_2,\tilde{l})}_{(1)}D^{\delta_2^{-1}(\tilde{l})}_{(2)}f_{\tilde{l}}\|_{\dot{B}^{0}_{p_{\tilde{l}}^1, \infty} L^{p_1^2}} \|D^{\delta_1^{-1}(\tilde{l}')}_{(2)}D^{\beta_2(v_2,\tilde{l}')}_{(2)}f_{\tilde{l}'}\|_{L^{p_{\tilde{l}'}^1} \dot{B}^{\beta_2^{\mathfrak{r}_{\mathcal{G}}}}_{p_{\tilde{l}'}^2, \infty}},\\
& 2^{-k_{\max}(v_2)\epsilon_1} 2^{m_{\max}(v_2)\beta_2^{\mathfrak{r}_{\mathcal{G}}}}\|D_{(1)}^{\beta_1(v_1,1)}D^{\beta_2(v_1,1)}_{(2)}f_{1}\|_{\dot{B}^{\epsilon_1}_{p_1^1, \infty} \dot{B}^{0}_{p_1^2, \infty}} \|D^{\beta_1(v_2,\tilde{l})}_{(1)}D^{\delta_2^{-1}(\tilde{l})}_{(2)}f_{\tilde{l}}\|_{\dot{B}^{\beta_1^{\mathfrak{r}_{\mathcal{G}}}}_{p_{\tilde{l}}^1, \infty} L^{p_1^2}} \|D^{\delta_1^{-1}(\tilde{l}')}_{(2)}D^{\beta_2(v_2,\tilde{l}')}_{(2)}f_{\tilde{l}'}\|_{L^{p_{\tilde{l}'}^1} \dot{B}^{0}_{p_{\tilde{l}'}^2, \infty}},\\
& 2^{-k_{\max}(v_2)\epsilon_1} 2^{-m_{\max}(v_2)\epsilon_2}\|D_{(1)}^{\beta_1(v_1,1)}D^{\beta_2(v_1,1)}_{(2)}f_{1}\|_{\dot{B}^{\epsilon_1}_{p_1^1, \infty} \dot{B}^{\epsilon_2}_{p_1^2, \infty}} \|D^{\beta_1(v_2,\tilde{l})}_{(1)}D^{\delta_2^{-1}(\tilde{l})}_{(2)}f_{\tilde{l}}\|_{\dot{B}^{\beta_1^{\mathfrak{r}_{\mathcal{G}}}}_{p_{\tilde{l}}^1, \infty} L^{p_1^2}} \|D^{\delta_1^{-1}(\tilde{l}')}_{(2)}D^{\beta_2(v_2,\tilde{l}')}_{(2)}f_{\tilde{l}'}\|_{L^{p_{\tilde{l}'}^1} \dot{B}^{\beta_2^{\mathfrak{r}_{\mathcal{G}}}}_{p_{\tilde{l}'}^2, \infty}}),
\end{align*}}so that the optimization and interpolation can be carried out to derive the estimate on the right hand side of (\ref{induction_2parameter_cone_statement}). This completes proof of Case \ref{prop:2param:6case1} and provides a generic recipe for treating the remaining cases. 
\end{enumerate}
\end{proof}

We make a final remark on the biparameter flag Leibniz-type estimates associated to asymmetric symbols stated in Theorem \ref{thm:an:example}. The inductive procedure described in the previous section, while still applicable, requires a certain modification; we only elaborate on this: the main difference arises in the reduction of the frequency trees. We focus on the example \eqref{example:asymmetric_diff} with the frequency trees specified below:
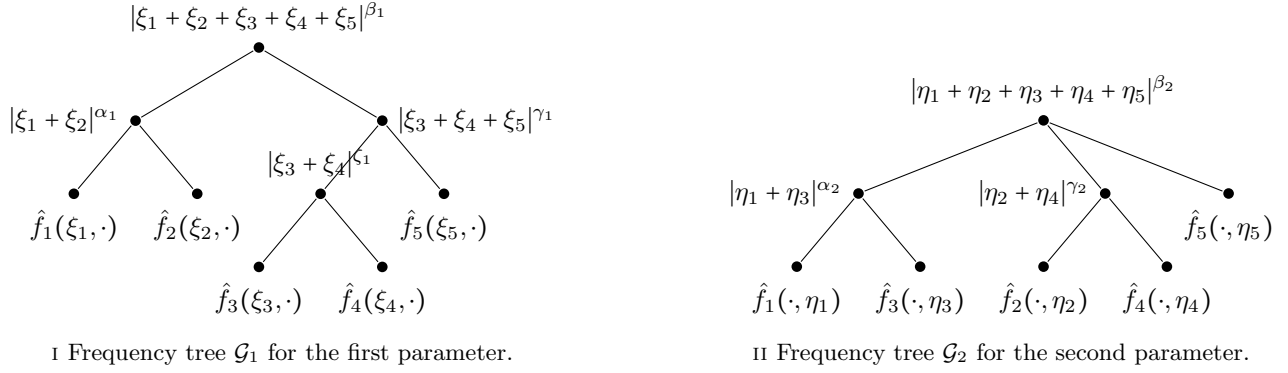
\begin{figure}[!htbp]
\begin{subfigure}[t]{.45\textwidth}
  \centering
 \begin{forest}
my treeSep
[, label={above: $|\xi_1+\xi_2+\xi_3+\xi_4+\xi_5|^{\beta_1}$}
  	[, label = {left: $|\xi_1+\xi_2|^{\alpha_1}$}
		[, label = {below: $\hat f_1(\xi_1, \cdot)$}
		]
		[, label = {below: $\hat f_2(\xi_2, \cdot)$}
		]
	]
	[,  label = {right: $|\xi_3+\xi_4+\xi_5|^{\gamma_1}$}
		[, label = {above: $|\xi_3+\xi_4|^{\zeta_1}$}
			[, label = {below: $\hat f_3(\xi_3, \cdot)$}
			]
			[, label = {below: $\hat f_4(\xi_4, \cdot)$}
			]
		]
		[, label = {below: $\hat f_5(\xi_5, \cdot)$}
		]
	]
]
\end{forest}
  \caption{Frequency tree $\mathcal{G}_1$ for the first parameter.} 
 \label{fig:assym:step1:1} 
    \end{subfigure}
 \hfill 
 \begin{subfigure}[t]{.45\textwidth}
  \centering
 \begin{forest}
my treeSep
[, label={above: $|\eta_1+\eta_2+\eta_3+\eta_4 +\eta_5|^{\beta_2}$}
	[, label = {left: $|\eta_1+\eta_3|^{\alpha_2}$}
		[, label = {below: $\hat f_1(\cdot, \eta_1)$}
		]
		[, label = {below: $\hat f_3(\cdot, \eta_3)$}
		]
	]
  	[, label = {left: $|\eta_2+\eta_4|^{\gamma_2}$}
		[, label = {below: $\hat f_2(\cdot, \eta_2)$}
		]
		[, label = {below: $\hat f_4(\cdot, \eta_4)$}
		]
	]
	[, label = {below: $\hat f_5(\cdot, \eta_5)$}
	]
]
\end{forest}
  \caption{Frequency tree $\mathcal{G}_2$ for the second parameter.} 
\label{fig:assym:step1:2} 
    \end{subfigure}
 \caption{An asymmetric bi-parameter symbol}   
\label{fig:assym:gl}    
\end{figure}  
\vskip .1in
We further assume that the multilinear expression \eqref{example:asymmetric_diff} is localized on the conical region 
$$
\{(\xi_1, \ldots, \xi_5): |\xi_1| \gg |\xi_l| \ \ \text{for} \ \ l \neq 1\} \times \{(\eta_1, \ldots, \eta_5): |\eta_4| \gg |\eta_l| \ \ \text{for} \ \ l \neq 4\}.
$$

Then we can split the root symbol into a commutator and a symbol associated to subtrees of lower complexity for both parameters. However, due to the asymmetricity, all the leaves $f_1, \ldots, f_5$ are intertwined and we cannot decouple any subsets of the leaves (previously associated to subtrees) as we did in the symmetric setting. Instead, we obtain a product of subtrees of lower complexity for both parameters as a reduction.  

In the conical region above, one term (modulo modulation and after simplifications) that appears in our estimation is the commutator tensorized with a symbol associated to a subtree: 
\begin{align} \label{multiplier:asymmtric_union}
 \sum_{\substack{k_1 \gg k_3 \\ m_3 \succ m_4 }}2^{k_1(\beta_1-1)} 2^{k_3} \int  |\xi_1+\xi_2|^{\alpha_1}|\xi_3+\xi_4+\xi_5|^{\gamma_1}|\xi_3+\xi_4|^{\zeta_1} |\eta_1+\eta_3|^{\alpha_2}|\eta_2+\eta_4|^{\gamma_2+\beta_2} &\nonumber \\
 \ii F(\Delta_{k_1,+}^{(1)}S_{m_3}^{(2)} f_1)(\xi_1, \eta_1) \ii F(S_{k_1}^{(1)}S_{m_4}^{(2)} f_2)(\xi_2, \eta_2)  \ii F(\Delta_{k_3}^{(1)}\Delta_{m_3}^{(2)} f_3)(\xi_3, \eta_3) \ii F(S_{k_3}^{(1)}\Delta_{m_4}^{(2)} f_4)(\xi_4, \eta_4) \ii F(S_{k_3}^{(1)} f_5)(\xi_5, \eta_5) & \nonumber \\
 e^{2 \pi i (x, y) \cdot (\xi_1+\xi_2+\xi_3+\xi_4+\xi_5, \eta_1+\eta_2+\eta_3+\eta_4+\eta_5)} d \xi d \eta.&
\end{align}
This hints to the necessity of establishing inductive statements associated to disjoint unions of rooted (sub)trees in both parameters, as opposed to inductive statements just for rooted subtrees. Thus the induction is performed based on the maximal complexity of the rooted trees involved.

In our example, we started with a frequency tree $\mathcal{G}_1$ of complexity $3$ in the first parameter (see Figure \ref{fig:assym:gl}(\subref{fig:assym:step1:1})), and a frequency tree $\mathcal{G}_2$ of complexity $2$ in the second parameter (see Figure \ref{fig:assym:gl}(\subref{fig:assym:step1:2})); so the maximal complexity of $\mathcal{G}_1 \times \mathcal{G}_2$ is $3$. By breaking down the root symbols $m_{\mathfrak{r}_{\mathcal{G}_1}}$ and $m_{\mathfrak{r}_{\mathcal{G}_2}}$, we are led, as suggested by \eqref{multiplier:asymmtric_union}, to considering the frequency forest $ \tilde {\mathcal{G}}_1$ of maximal complexity 2 in the first parameter (see Figure \ref{fig:assym:step2:1}), and the frequency forest $\tilde {\mathcal{G}}_2$ of maximal complexity $1$ in the second parameter (see Figure \ref{fig:assym:step2:2}). Overall, the splitting of the roots' symbols reduces the maximal complexity. More concretely, the inductive hypothesis will be applied to the expression
\begin{align}\label{induc_hypo}
& \int  |\xi_1+\xi_2|^{\alpha_1}|\xi_3+\xi_4+\xi_5|^{\gamma_1}|\xi_3+\xi_4|^{\zeta_1} |\eta_1+\eta_3|^{\alpha_2}|\eta_2+\eta_4|^{\gamma_2+\beta_2} \ii F(\Delta_{k_1}^{(1)}S_{m_3}^{(2)} f_1)(\xi_1, \eta_1) \ii F(S_{k_1}^{(1)}S_{m_4}^{(2)} f_2)(\xi_2, \eta_2)\nonumber \\
&   \ii F(\Delta_{k_3}^{(1)}\Delta_{m_3}^{(2)} f_3)(\xi_3, \eta_3) \ii F(S_{k_3}^{(1)}\Delta_{m_4}^{(2)} f_4)(\xi_4, \eta_4) \ii F(S_{k_3}^{(1)} f_5)(\xi_5, \eta_5)e^{2 \pi i (x, y) \cdot (\xi_1+\xi_2+\xi_3+\xi_4+\xi_5, \eta_1+\eta_2+\eta_3+\eta_4+\eta_5)} d \xi d \eta,
\end{align}
which is a multiplier appearing in \eqref{multiplier:asymmtric_union} with maximal complexity $2$ (attained by the subtree with root symbol $|\xi_3+\xi_4+\xi_5|^{\gamma_1}$).
%we will need to prove the boundedness for:
%\begin{align*}
%& \int  |\xi_1+\xi_2|^{\alpha_1}|\xi_3+\xi_4+\xi_5|^{\gamma_1}|\xi_3+\xi_4|^{\zeta_1} |\eta_1+\eta_3|^{\alpha_2}|\eta_2+\eta_4|^{\gamma_2+\beta_2} \ii F(\Delta_{k_1}^{(1)}S_{m_3}^{(2)} f_1)(\xi_1, \eta_1) \ii F(S_{k_1}^{(1)}S_{m_4}^{(2)} f_2)(\xi_2, \eta_2)\nonumber \\
%&   \ii F(\Delta_{k_3}^{(1)}\Delta_{m_3}^{(2)} f_3)(\xi_3, \eta_3) \ii F(S_{k_3}^{(1)}\Delta_{m_4}^{(2)} f_4)(\xi_4, \eta_4) \ii F(S_{k_3}^{(1)} f_5)(\xi_5, \eta_5)e^{2 \pi i (x, y) \cdot (\xi_1+\xi_2+\xi_3+\xi_4+\xi_5, \eta_1+\eta_2+\eta_3+\eta_4+\eta_5)} d \xi d \eta,
%\end{align*}
\begin{figure}[!htbp]
\label{fig:asymetric}
\begin{subfigure}[t]{.45\textwidth}
  \centering
 \begin{forest}
my treeSep
  	[, label = {above: $|\xi_1+\xi_2|^{\alpha_1}$}
		[, label = {below: $\hat f_1(\xi_1, \cdot)$}
		]
		[, label = {below: $\hat f_2(\xi_2, \cdot)$}
		]
	]
\end{forest}
\end{subfigure}
\begin{subfigure}[t]{.45\textwidth}
 \centering
\begin{forest}
my treeSep
	[,  label = {above: $|\xi_3+\xi_4+\xi_5|^{\gamma_1}$}
		[, label = {left: $|\xi_3+\xi_4|^{\zeta_1}$}
			[, label = {below: $\hat f_3(\xi_3, \cdot)$}
			]
			[, label = {below: $\hat f_4(\xi_4, \cdot)$}
			]
		]
		[, label = {right: $\hat f_5(\xi_5, \cdot)$}
		]
	]
\end{forest}
 \end{subfigure}
  \caption{Frequency forest $\tilde {\mathcal{G}}_1$ for the first parameter.} 
 \label{fig:assym:step2:1}
 \end{figure}
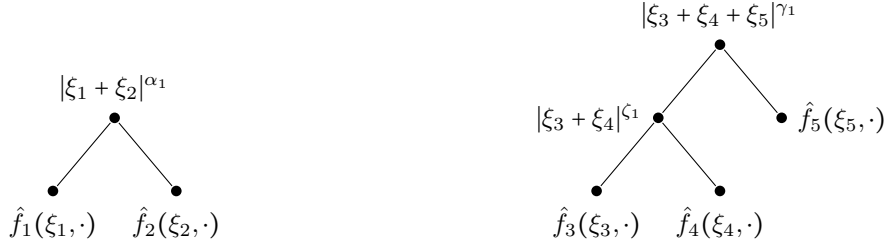

 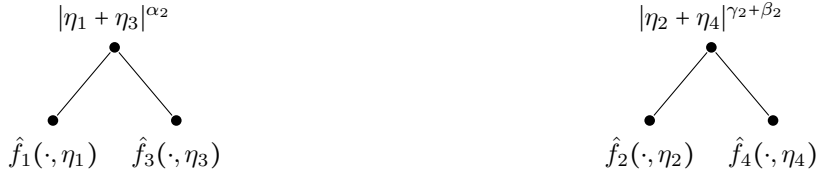
\begin{figure}
 \begin{subfigure}[t]{.45\textwidth}
  \centering
 \begin{forest}
my treeSep
	[, label = {above: $|\eta_1+\eta_3|^{\alpha_2}$}
		[, label = {below: $\hat f_1(\cdot, \eta_1)$}
		]
		[, label = {below: $\hat f_3(\cdot, \eta_3)$}
		]
	]
\end{forest}
\end{subfigure}
\begin{subfigure}[t]{.45\textwidth}
\centering
\begin{forest}
my treeSep
  	[, label = {above: $|\eta_2+\eta_4|^{\gamma_2+\beta_2}$}
		[, label = {below: $\hat f_2(\cdot, \eta_2)$}
		]
		[, label = {below: $\hat f_4(\cdot, \eta_4)$}
		]
	]
\end{forest}
    \end{subfigure}
\caption{Frequency forest $\tilde {\mathcal{G}}_2$ for the second parameter.} 
\label{fig:assym:step2:2}
\end{figure}  

The base case of such an inductive procedure involves an analysis of symbols associated to frequency forests of complexity less than or equal to 1 in each parameter. This can be verified directly by implementing the usual procedure: paraproduct decompositions, splitting of root symbols, Fourier series decompositions, optimization and Besov norm interpolation.

This ends the discussion on the example \eqref{thm:an:example}, which is generic enough to illustrate the main ingredients in the proof for multi-parameter flag Leibniz-type estimates associated to asymmetric symbols of arbitrary complexity.

\subsection{Optimization and interpolation in the $N$-parameters setting}\label{generic_optim_interp}
This section is devoted to the optimization and interpolation procedure which allows to redistribute the derivatives and produce the expected geometric and arithmetic means in the generic $N$-parameters setting. As a consequence of this procedure, we derive the desired estimates in the statements \eqref{prop:2param:ind:1}-- \eqref{prop:2param:ind:6} of Proposition \ref{prop:biparameter_induction}.

We first introduce some notation. Let $\h$ and $\el$ denote maps
\begin{align*}
& \h: i \in \{\text{Index set of the parameters}\} \mapsto l \in \{\text{Index set for the functions}\} \\
& \el: i \in \{\text{Index set of the parameters}\} \mapsto l \in \{\text{Index set for the functions}\},
\end{align*}
where the map $\h$ indicates which functions would be hit by the full order of derivatives and the map $\el$ illustrates the functions hit by the lower order of derivatives. Equations \eqref{eq:besov:ill:1}-\eqref{eq:besov:ill:3} are representative of this action, and in \eqref{eq:besov:ill:2} in particular we can see that in the region $\{ |\xi_1| \gg |\xi_2|, \ldots, |\xi_n| \}$ the function $f_1$ receives $D^{\beta^{\mathfrak{r}_{\mathcal{G}}}}$ derivatives and $f_{l_0}$ receives $D^\epsilon$ derivatives.

We denote by $\Sigma_N$ the set of length-$N$ signatures:
\begin{equation}
\label{eq:def:sign:N}
\Sigma_N:= \{ \sigma=(\varsigma_1, \ldots, \varsigma_N): \varsigma_1, \ldots, \varsigma_N \in \{ +, -  \}   \}
\end{equation}

Our index set of parameters is $\{1, \ldots, N \}$, and the index set of functions $\{ 1, \ldots, n \}$. The above convention implies that, on the specific frequency conical region to which we restrict our operator, in the first parameter the scales $k_{\h(1)}\geq k_{\el(1)}$ (and thus the functions $f_{\h(1)}$ and $f_{\el(1)}$) are involved in the estimation of the flag, in the second parameter the scales $k_{\h(2)} \geq k_{\el(2)}$ (and the functions $f_{\h(2)}$ and $f_{\el(2)}$), and so on. 

In order to simplify the notation, let us assume that in estimating an $n$-linear, $N$-parameter flag operator $T_{\mathcal G}$ associated to a rooted tree with root information $(\beta_1, \ldots, \beta_N)$ and restricted\footnote{Otherwise $T_{\mathcal G}$ will be bounded above by a some of similar terms, each to be estimated through the present analysis.} to frequency conical regions on each subtree, the functions $f_1, f_2, \ldots, f_m$ are involved. Given the discussion in the previous section, we are reduced to\footnote{Again, for simplicity, we omit the exponent $\ds\tau \leq  \min_{1 \leq l \leq N} p^l$; at this stage of the proof it plays no role.}
\begin{align*}
\|T_{\mathcal{G}}&(f_1, \ldots, f_n)\|_{L^{\vec p}} \lesssim \prod_{l=m+1}^n \|F_l\|_{L^{\vec {p_l}}} \\
\sum_{\ell_1, \ldots, \ell_N} \min( &2^{\beta_1 \ell_1} \cdot \ldots  \cdot 2^{\beta_{N-1} \ell_{N-1}} \cdot 2^{\beta_N \ell_N} \cdot \prod_{l=1}^m \|F_l\|_{\vec W_l; (+, \ldots, +, +)}, 2^{\beta_1 \ell_1} \cdot \ldots  \cdot 2^{\beta_{N-1} \ell_{N-1}} \cdot 2^{-\epsilon \ell_N} \cdot \prod_{l=1}^m \|F_l\|_{\vec W_l; (+, \ldots, +, -)} \\
&\vdots \\
&  2^{-\epsilon_1 \ell_1} \cdot \ldots  \cdot 2^{-\epsilon_{N-1} \ell_{N-1}} \cdot 2^{\beta_N \ell_N} \cdot \prod_{l=1}^m \|F_l\|_{\vec W_l; (-, \ldots, -, +)}, 2^{-\epsilon_1 \ell_1} \cdot \ldots  \cdot 2^{-\epsilon_{N-1} \ell_{N-1}} \cdot 2^{-\epsilon \ell_N} \cdot \prod_{l=1}^m \|F_l\|_{\vec W_l; (-, \ldots, -, -)} ),
\end{align*}
where $\|F_l\|_{\vec W_l; \sigma}$ -- for $\sigma \in \Sigma_N$ -- is the mixed norm
\begin{equation}
\label{eq:def:mixed:W:sigma}
\|F_l\|_{\vec W_l; \sigma} := \|F_l\|_{W_l^{1, \varsigma_1}W_l^{2, \varsigma_2} \ldots W_l^{{N-1}, \varsigma_{N-1}} W_l^{N, \varsigma_N}},
\end{equation}
and for any $1 \leq l \leq n$
\begin{equation}
\label{eq:def:deriv:subtrees}
F_l:= D^{\alpha_l^1}_{(1)} \ldots D^{\alpha_l^N}_{(N)} f_l
\end{equation}
records the derivatives picked up by the function $f_l$ on the subtrees of $\mathcal{G}$ (which are of lower complexity).

Above, we define for any $1 \leq l \leq n$ and any $1 \leq i \leq N$,
\begin{equation}
W_l^{i, \varsigma_i}:=\begin{cases} \dot B^0_{p_l^i, \infty}, \text{      if    } l= \h(i) \text{    or    } l= \el(i) \text{   and     } \varsigma_i=+\\
                                                      \dot B^{\beta_i}_{p_l^i, \infty},  \text{      if    } l= \h(i)  \text{   and     } \varsigma_i=-\\
 \dot B^{\epsilon_i}_{p_l^i, \infty},  \text{      if    } l= \el(i)  \text{   and     } \varsigma_i=- \\
 L^{p_l^i}  \text{      if    } l\notin \{ \h(i), \el(i)  \}.
\end{cases}
\end{equation}

Then it is not difficult to see\footnote{Indeed, such a result can very easily be proved via an induction argument on $N$, the number of parameters.} that 
\begin{align*}
\sum_{\ell_1, \ldots, \ell_N} \min\Big( &2^{\beta_1 \ell_1} \cdot \ldots  \cdot 2^{\beta_{N-1} \ell_{N-1}} \cdot 2^{\beta_N \ell_N} \cdot \prod_{l=1}^m \|F_l\|_{\vec W_l; (+, \ldots, +, +)}, 2^{\beta_1 \ell_1} \cdot \ldots  \cdot 2^{\beta_{N-1} \ell_{N-1}} \cdot 2^{-\epsilon \ell_N} \cdot \prod_{l=1}^m \|F_l\|_{\vec W_l; (+, \ldots, +, -)} \\
&\vdots \\
&  2^{-\epsilon_1 \ell_1} \cdot \ldots  \cdot 2^{-\epsilon_{N-1} \ell_{N-1}} \cdot 2^{\beta_N \ell_N} \cdot \prod_{l=1}^m \|F_l\|_{\vec W_l; (-, \ldots, -, +)}, 2^{-\epsilon_1 \ell_1} \cdot \ldots  \cdot 2^{-\epsilon_{N-1} \ell_{N-1}} \cdot 2^{-\epsilon \ell_N} \cdot \prod_{l=1}^m \|F_l\|_{\vec W_l; (-, \ldots, -, -)} \Big) \\
\quad \lesssim & \prod_{\sigma= (\varsigma_1, \ldots, \varsigma_N) \in \Sigma_N}  \big(  \prod_{l=1}^{m}  \|F_l\|_{\vec W_l; \sigma} \big)^{\nu_{\varsigma_1} \cdot \ldots \cdot  \nu_{\varsigma_N}},
\end{align*}
where for any $1 \leq i \leq N$, $\nu_{\varsigma_i}$ is defined by 
\begin{equation}
\nu_{\varsigma_i}:=\begin{cases} \frac{\epsilon_i}{\beta_i+\epsilon_1},  \text{      if    } \varsigma_i=+\\
\frac{\beta_i}{\beta_i+\epsilon_1},  \text{      if    } \varsigma_i=-.
\end{cases}
\end{equation}

If we denote $\nu_{\sigma}:=\nu_{\varsigma_1} \cdot \ldots \cdot  \nu_{\varsigma_N}$, we have 
\begin{align*}
\sum_{\sigma \in \Sigma_N} \nu_{\sigma}= \prod_{i=1}^N \Big( \frac{\epsilon_i}{\beta_i+\epsilon_1} + \frac{\beta_i}{\beta_i+\epsilon_1} \Big)=1.
\end{align*}

This observation and the previous notation allows us to deduce 
\begin{align}
\label{eq:est:geom:mean:mess}
\|T_{\mathcal G}(f_1, \ldots, f_n)\|_{L^{\vec p}} \lesssim  \prod_{\sigma= (\varsigma_1, \ldots, \varsigma_N) \in \Sigma_N}  \big(  \prod_{l=1}^{n}  \|F_l\|_{\vec W_l; \sigma} \big)^{\nu_{\varsigma_1} \cdot \ldots \cdot  \nu_{\varsigma_N}}.
\end{align}

Now our task is to replace $ \ds \prod_{l=1}^{n} \|F_l\|_{\vec W_l; \sigma}$ by more suitable expressions that allow to keep track of the distribution of derivatives. This will be done in $N$ steps, which corresponds to the number of parameters.

Before proceeding, we recall the interpolation result
\begin{equation}
\label{eq:optim:interp:reminder}
\|F_l\|_{\dot B^{\epsilon_1}_{p_l^1, \infty} (\tilde W_l; \tilde \sigma)} \leq \|F_l\|_{\dot B^{0}_{p_l^1, \infty} (\tilde W_l; \tilde \sigma)}^{\frac{\beta_1-\epsilon_1}{\beta_1}} \cdot \|F_l\|_{\dot B^{\beta_1}_{p_l^1, \infty} (\tilde W_l; \tilde \sigma)}^{\frac{\epsilon_1}{\beta_1}}, 
\end{equation}
where $\tilde \sigma \in \Sigma_{N-1}$ is a length-$(N-1)$ signature vector and $\tilde W_l$ is the iteration of $N-1$ vector spaces.

We want to replace $\dot B^{\epsilon_1}_{p_l^1, \infty}$ appearing in \eqref{eq:est:geom:mean:mess} either by $\dot B^{0}_{p_l^1, \infty}$ or by $\dot B^{\beta_1}_{p_l^1, \infty}$. We notice that $\dot B^{\epsilon_1}_{p_l^1, \infty}$ appears in the norm $$\|F_l\|_{\vec W_l; \sigma}$$ of precisely those $l$ satisfying $l =\el(1)$, and $\varsigma_1=-$.

This observation, together with the fact that the right-hand side of \eqref{eq:est:geom:mean:mess} equals\footnote{We also use the fact that for $ l \neq \h(1), \el(1)$, $W_l^1= L^{p_l^1}.$}
\begin{align*}
\sum_{\tilde \sigma \in \Sigma_{N-1}}   & \big(  \|F_{\h(1)}\|_{ W_{\h(1)}^{1,+}  (\tilde W_l; \tilde \sigma)} \cdot \|F_{\el(1)}\|_{ W_{\el(1)}^{1,+}  (\tilde W_l; \tilde \sigma)} \prod_{\substack{l=1 \\ l \neq \h(1), \el(1)}}^{n}  \|F_l\|_{  L^{p_l^1}   (\tilde W_l;\tilde \sigma)} \big)^{\nu_{+} \cdot \nu_{\tilde \sigma}} \\
& \cdot  \big( \|F_{\h(1)}\|_{ W_{\h(1)}^{1,-} (\tilde W_l; \tilde \sigma)} \cdot \|F_{\el(1)}\|_{ W_{\el(1)}^{1,-}(\tilde W_l; \tilde \sigma)}   \prod_{\substack{l=1 \\ l \neq \h(1), \el(1)}}^{n}  \|F_l\|_{ L^{p_l^1}   (\tilde W_l;\tilde \sigma)} \big)^{\nu_{-} \cdot  \nu_{\tilde \sigma}} \\
&= \sum_{\tilde \sigma \in \Sigma_{N-1}}   \big(  \|F_{\h(1)}\|_{ \dot B^0_{p_{\h(1)}^1, \infty}  (\tilde W_l; \tilde \sigma)} \cdot \|F_{\el(1)}\|_{\dot B^0_{p_{\el(1)}^1, \infty}   (\tilde W_l; \tilde \sigma)} \prod_{\substack{l=1 \\ l \neq \h(1), \el(1)}}^{n}  \|F_l\|_{ L^{p_l^1} (\tilde W_l;\tilde \sigma)} \big)^{\nu_{+} \cdot \nu_{\tilde \sigma}} \\
& \cdot  \big( \|F_{\h(1)}\|_{ \dot B^{\beta_1}_{p_{\h(1)}^1, \infty}  (\tilde W_l; \tilde \sigma)} \cdot \|F_{\el(1)}\|_{ \dot B^{\epsilon_1}_{p_{\h(1)}^1, \infty} (\tilde W_l; \tilde \sigma)}   \prod_{\substack{l=1 \\ l \neq \h(1), \el(1)}}^{n}  \|F_l\|_{ L^{p_l^1}   (\tilde W_l;\tilde \sigma)} \big)^{\nu_{-} \cdot  \nu_{\tilde \sigma}}, 
\end{align*}
naturally suggests the use of interpolation. If $\tilde \sigma \in \Sigma_{N-1}$ is fixed, then the expression appearing in the last display equals 
\begin{align*}
&\big( \prod_{\substack{l=1 \\ l \neq \h(1), \el(1)}}^{n}  \|F_l\|_{ L^{p_l^1}   (\tilde W_l;\tilde \sigma)} \big)^{ \nu_{\tilde \sigma}} \\
& \cdot \big( ( \|F_{\h(1)}\|_{ \dot B^0_{p_{\h(1)}^1, \infty}  (\tilde W_l; \tilde \sigma)} \cdot \|F_{\el(1)}\|_{\dot B^0_{p_{\el(1)}^1, \infty}(\tilde W_l; \tilde \sigma)} )^{\frac{\epsilon_1}{\beta_1+\epsilon_1}} 
\cdot ( \|F_{\h(1)}\|_{ \dot B^{\beta_1}_{p_{\h(1)}^1, \infty}  (\tilde W_l; \tilde \sigma)} \cdot \|F_{\el(1)}\|_{ \dot B^{\epsilon_1}_{p_{\h(1)}^1, \infty} (\tilde W_l; \tilde \sigma)}  )^{\frac{\beta_1}{\beta_1+\epsilon_1}} \big)^{\nu_{\tilde \sigma}}.
\end{align*}

The interpolation in \eqref{eq:interpolation:Besov:epsilon} and direct computations similar to \eqref{ea:conc:comm} allow to bound this expression by
\begin{align*}
&\big( \prod_{\substack{l=1 \\ l \neq \h(1), \el(1)}}^{n}  \|F_l\|_{ L^{p_l^1}   (\tilde W_l;\tilde \sigma)} \big)^{ \nu_{\tilde \sigma}} \\
& \cdot \big( ( \|F_{\h(1)}\|_{ \dot B^{0}_{p_{\h(1)}^1, \infty}  (\tilde W_l; \tilde \sigma)} \cdot \|F_{\el(1)}\|_{\dot B^{\beta_1}_{p_{\el(1)}^1, \infty}(\tilde W_l; \tilde \sigma)} )^{\frac{\epsilon_1}{\beta_1+\epsilon_1}} 
\cdot ( \|F_{\h(1)}\|_{ \dot B^{\beta_1}_{p_{\h(1)}^1, \infty}  (\tilde W_l; \tilde \sigma)} \cdot \|F_{\el(1)}\|_{ \dot B^{0}_{p_{\h(1)}^1, \infty} (\tilde W_l; \tilde \sigma)}  )^{\frac{\beta_1}{\beta_1+\epsilon_1}} \big)^{\nu_{\tilde \sigma}} \\
&= \big(  \|F_{\h(1)}\|_{ \dot B^{0}_{p_{\h(1)}^1, \infty}  (\tilde W_l; \tilde \sigma)} \cdot \|F_{\el(1)}\|_{\dot B^{\beta_1}_{p_{\el(1)}^1, \infty}(\tilde W_l; \tilde \sigma)}  \prod_{\substack{l=1 \\ l \neq \h(1), \el(1)}}^{n}  \|F_l\|_{ L^{p_l^1}   (\tilde W_l;\tilde \sigma)}   \big)^{ \nu_{(+, \tilde \sigma)}} \\
&\cdot \big(    \|F_{\h(1)}\|_{ \dot B^{\beta_1}_{p_{\h(1)}^1, \infty}  (\tilde W_l; \tilde \sigma)} \cdot \|F_{\el(1)}\|_{ \dot B^{0}_{p_{\h(1)}^1, \infty} (\tilde W_l; \tilde \sigma)}  )^{\frac{\beta_1}{\beta_1+\epsilon_1}}   \prod_{\substack{l=1 \\ l \neq \h(1), \el(1)}}^{n}  \|F_l\|_{ L^{p_l^1}   (\tilde W_l;\tilde \sigma)}  \big)^{ \nu_{(-, \tilde \sigma)}}. 
\end{align*}

In conclusion, \eqref{eq:est:geom:mean:mess} is replaced by 
\begin{align}
\label{eq:est:geom:mean:mess:1}
\|T_{\mathcal{G}}(f_1, \ldots, f_n)\|_{L^{\vec p}} \lesssim  \prod_{\sigma= (\varsigma_1, \ldots, \varsigma_N) \in \Sigma_N}  \big(  \prod_{l=1}^{n}  \|F_l\|_{\vec W_l^1; \sigma} \big)^{\nu_{\varsigma_1} \cdot \ldots \cdot  \nu_{\varsigma_N}},
\end{align}
where $\ds \vec W_l^1; \sigma: = X_l^{1, \varsigma_1} W_l^{2, \varsigma_2} \ldots W_l^{{N-1}, \varsigma_{N-1}} W_l^{N, \varsigma_N}$ and
\begin{equation}
\label{eq:def:X1}
X_l^{1, \varsigma_1}:=\begin{cases} \dot B^0_{p_l^1, \infty}, \text{      if    } l= \h(1)  \text{   and     } \varsigma_1=+\\
                                                        \dot B^{\beta_1}_{p_l^1, \infty}, \text{      if    } l= \el(1)  \text{   and     } \varsigma_1=+\\
                                                      \dot B^{\beta_1}_{p_l^1, \infty},  \text{      if    } l= \h(1)  \text{   and     } \varsigma_1=-\\
 \dot B^{0}_{p_l^1, \infty},  \text{      if    } l= \el(1)  \text{   and     } \varsigma_1=- \\
 L^{p_l^1}  \text{      if    } l\notin \{ \h(1), \el(1)  \}.
\end{cases}
\end{equation}

We iterate the procedure to obtain 
\begin{align}
\label{eq:est:geom:mean:mess:2}
\|T_{\mathcal{G}}(f_1, \ldots, f_n)\|_{L^{\vec p}} \lesssim  \prod_{\sigma= (\varsigma_1, \ldots, \varsigma_N) \in \Sigma_N}  \big(  \prod_{l=1}^{n}  \|F_l\|_{\vec W_l^2; \sigma} \big)^{\nu_{\varsigma_1} \cdot \ldots \cdot  \nu_{\varsigma_N}},
\end{align}
where $\ds \vec W_l^2; \sigma: = X_l^{1, \varsigma_1} X_l^{2, \varsigma_2} W_l^{3, \varsigma_3} \ldots W_l^{{N-1}, \varsigma_{N-1}} W_l^{N, \varsigma_N}$ and $ X_l^{2, \varsigma_2}$ is defined in a similar way to \eqref{eq:est:geom:mean:mess:2}.

After $N$ steps we conclude that 
\begin{align}
\label{eq:est:geom:mean:mess:N}
\|T_{\mathcal{G}}(f_1, \ldots, f_n)\|_{L^{\vec p}} \lesssim  \prod_{\sigma= (\varsigma_1, \ldots, \varsigma_N) \in \Sigma_N}  \big(  \prod_{l=1}^{n}  \|F_l\|_{\vec W_l^N; \sigma} \big)^{\nu_{\sigma}},
\end{align}
where $\ds \vec W_l^N; \sigma: = X_l^{1, \varsigma_1}  \ldots  X_l^{N, \varsigma_N}$ and
\begin{equation}
\label{eq:def:Xi}
X_l^{i, \varsigma_i}:=\begin{cases} \dot B^0_{p_l^i, \infty}, \text{      if    } l= \h(i)  \text{   and     } \varsigma_i=+\\
                                                        \dot B^{\beta_i}_{p_l^i, \infty}, \text{      if    } l= \el(i)  \text{   and     } \varsigma_i=+\\
                                                      \dot B^{\beta_i}_{p_l^i, \infty},  \text{      if    } l= \h(i)  \text{   and     } \varsigma_i=-\\
 \dot B^{0}_{p_l^i, \infty},  \text{      if    } l= \el(i)  \text{   and     } \varsigma_i=- \\
 L^{p_l^i}  \text{      if    } l\notin \{ \h(i), \el(i)  \}.
\end{cases}
\end{equation}

If we carefully read inequality \eqref{eq:est:geom:mean:mess:N}, we have obtained that the $N$-parameter $n$-linear flag $T_{\mathcal{G}}(f_1, \ldots, f_n)$ is bounded above in the mixed norm $\| \cdot \|_{L^{\vec r}}$ by the geometric mean of $2^N$ terms (this is because $\ds \sum_{\sigma \in \Sigma_N} \nu_{\sigma}=1$), and each term is of the form 
\[
 \prod_{l=1}^{n}  \|F_l\|_{X_l^{1, \varsigma_1}  \ldots  X_l^{N, \varsigma_N}}.
\]
For every parameter $1 \leq i \leq N$, there are exactly two indices $\h(i)$ and $\el(i)$ in the index set $\{1, \ldots, n \}$ for the functions for which  
\begin{equation}
\label{eq:distrib:deriv}
( X_{\h(i)}^{i, \varsigma_i} , X_{\el(i)}^{i, \varsigma_i})=(  \dot B^0_{p_{\h(i)}^i, \infty},  \dot B^{\beta_i}_{p_{\el(i)}^i, \infty} ) \quad \text{or} \quad ( X_{\h(i)}^{i, \varsigma_i} , X_{\el(i)}^{i, \varsigma_i})=(  \dot B^{\beta_1}_{p_{\h(i)}^i, \infty},  \dot B^{0}_{p_{\el(i)}^i, \infty} ),
\end{equation}
while for the other indices $l \neq \h(i),\el(i)$ we simply have $X_l^{i, \varsigma_i}= L^{p_l^i}$. The identity \eqref{eq:distrib:deriv} means precisely that the derivatives are being distributed accordingly: for every $1 \leq i \leq N$, each $D^{\beta_i}_{(i)}$ is shared between $F_{\h(i)}$ and $F_{\el(i)}$. Finally, if we undo the definition \eqref{eq:def:deriv:subtrees}, we deduce that $T_{\mathcal{G}}(f_1, \ldots, f_n)$, restricted to conical regions according to the subtree structures, is indeed controlled by the geometric mean of terms of the form $\prod_{l=1}^n\|D_{(1)}^{\delta_1^{-1}(l)} \ldots D_{(N)}^{\delta_N^{-1}(l)}f_{l}\|_{\vec{p_l}}$, where $\delta_1\otimes \ldots \otimes \delta_N$ satisfy conditions \ref{deriv:distrib:i} and \ref{deriv:distrib:ii} from Section \ref{introduction}.

\section{Multilinear operators of positive order} \label{sec:generic:multipliers}

After having worked out the multi-parameter flag Leibniz rules, we discuss in this section Leibniz-type estimates for flag structures associated to Mikhlin symbols of positive order. Our intention here is to provide more examples for which the method introduced by Bourgain and Li \cite{BourgainLi-kato} offers an alternative to well-established techniques.

We start with the observation that on the region\footnote{On the other hand, in the region $\ds \tilde R_{1,2}=\{(\xi_1, \ldots, \xi_n) \in \BBR^{dn}: |\xi_1|\sim |\xi_2| \gg |\xi_3|, \ldots, |\xi_n|\}$, $\ds |\xi_1+\ldots + \xi_n|^\beta$ is less regular than a Mikhlin symbol of order $\beta$, unless $\beta \in 2 \BBN$. In general, for a Mikhlin symbol $m$ of order $\beta$, $m(\xi_1+\ldots +\xi_n)$ is singular along the subspace $\{(\xi_1, \ldots, \xi_n): \xi_1 + \ldots + \xi_n = 0\}$, and it can be seen as a natural extension of $|\xi_1+\ldots + \xi_n|^\beta$. Nonetheless, it is not difficult to verify that our method developed in this section also applies to such symbols.} $\ds R_1=\{(\xi_1, \ldots, \xi_n) \in \BBR^{dn}: |\xi_1| \gg |\xi_2|, \ldots, |\xi_n|\}$, the symbol $\ds |\xi_1+\ldots + \xi_n|^\beta$ -- which is naturally associated to Leibniz rules for $D^\beta(f_1 \ldots f_n)$ -- is a Mikhlin symbol of order $\beta$, satisfying \eqref{eq:positive:order:class} for arbitrarily many multi-indices $\gamma_1, \ldots, \gamma_n$. Because of this, in what follows we consider Mikhlin symbols of positive order, arbitrarily smooth away from the origin. Under this assumption, we provide a sketch of the proofs of Theorems \ref{thm:multipliers:positive:order}, \ref{thm:Nparam:tensor:positive:multipl}, \ref{thm:multiparameter_nontensor}. Afterwards we discuss \emph{smoothing properties} for multilinear operators associated to symbols of negative order, such as multilinear fractional integral operators. 

We would like to remark that Theorem \ref{thm:multipliers:positive:order} follows from the boundedness of one-parameter flag paraproducts \cite{flag_paraproducts} and Theorem \ref{thm:multiparameter_nontensor} is implied by the boundedness of multi-parameter paraproducts \cite{bi-parameter_paraproducts}. Moreover, the smoothing property described in Theorem \ref{thm:smoothing:N:param} in the one-parameter mixed-norm setting has been resolved by Hart-Torres-Wu \cite{HartTorresWu-smoothing-bil}, and its bi-parameter, non-mixed-norm variant was established in Yang-Liu-Wu \cite{YangLiuWu-smoothing}; both cases focus on symbols of limited regularity whereas we handle smooth symbols satisfying pointwise decay conditions.  

\subsection{Leibniz-type estimates for flag operators associated to Mikhlin symbols of positive order}

The natural multi-parameter adaptation of Theorem \ref{thm:main} in this context is provided exactly by Theorem \ref{thm:Nparam:tensor:positive:multipl}, in which the Mikhlin symbols associated to each vertex are assumed to be products of Mikhlin symbols in each parameter, i.e. they satisfy \eqref{symbol_N_parameter}. Because the symbols tensorize, the multi-parameter extension will follow closely the procedures described in Section \ref{sec:2param:5linflag}, Section \ref{subsection_bi_leibniz} and Section \ref{generic_optim_interp}. For that reason, we will focus on the one-parameter case -- Theorem \ref{thm:multipliers:positive:order}.

\begin{remark}
Since the multipliers considered are smooth away from the origin (they satisfy \eqref{eq:positive:order:class}), they are smooth on every Whitney rectangle and as a consequence the corresponding Fourier coefficients have arbitrary decay. This comes in contrast with the Leibniz rules presented in Theorem \ref{thm:main}, where conditions \eqref{cond:thm:1} and \eqref{cond:thm:2} are necessary. 
Notice also that in this situation we do not obtain the endpoint estimates corresponding to $p_l=1$ or $p_l=\infty$, although some of the $L^\infty$ endpoints can be proved through the Coifman-Meyer approach.
\end{remark}

\begin{proof}[Sketch of Proof of Theorem \ref{thm:multipliers:positive:order}]
Here again the inductive procedure will follow closely the steps described in Section \ref{sec:one:param:generic:ind}, and it will be enough to discuss steps \eqref{dec:root:symbol} and \eqref{dec:Fourier:series} of our strategy presented in Section \ref{strategy}: the splitting of the root symbol (and the simultaneous appearance of commutators), and the Fourier series decomposition for the new symbols. Once these steps performed, the operator associated to the rooted tree $\mathcal{G}$ naturally tensorized into operators associated to rooted subtrees of lower complexity; in many situations, the Mikhlin symbol $m_{\beta_{\tilde v}}(\xi_{i_1}, \ldots, \xi_{i_s})$ -- associated to the root of a subtree -- will be replaced by the product $m_{\beta_{\tilde v}}(\xi_{i_1}, \ldots, \xi_{i_s}) \cdot m_{\beta_{r_{\mathcal{G}}}}(0, \ldots, \xi_{i_1}, 0, \ldots, \xi_{i_s}, 0 , \ldots)$, which is again a Mikhlin symbol of order $\beta_{r_{\mathcal{G}}}+\beta_{\tilde v}>0$ in the variables $\xi_{i_1}, \ldots, \xi_{i_s}$.

For illustrative purposes, we will focus on the one-parameter Leibniz-type estimate of complexity 1 and explain how to achieve the steps \eqref{dec:root:symbol} and \eqref{dec:Fourier:series} mentioned above. In particular, we consider
\[
T_{m_\beta}(f_1, \ldots, f_n)(x):= \int_{\BBR^{nd}} m_\beta(\xi_1, \ldots, \xi_n) \hat f_1(\xi_1) \cdot \ldots \cdot \hat f_n(\xi_n) e^{2 \pi i x \cdot (\xi_1+\ldots+\xi_n)} d \xi_1 \ldots d \xi_n,
\]
when restricted to two typical regions:
\begin{equation} \label{cone:type}
\begin{split}
R_1:= & \{ (\xi_1, \ldots, \xi_n)  : |\xi_1| \gg |\xi_2|,|\xi_3| \ldots, |\xi_n|\} \\
\tilde R_{1,2}:= & \{ (\xi_1, \ldots, \xi_n)  : |\xi_1| \sim |\xi_2|\gg |\xi_3| \ldots, |\xi_n|\}.
\end{split}
\end{equation}

The first region corresponds to the situation when one of $|\xi_{j_0}|$ is much larger than the remaining variables, and so $|(\xi_1, \ldots, \xi_n)|\sim |\xi_{j_0}|$. For simplicity and without loss of generality, we assume that $j_0=1$. The second situation corresponds to the ``diagonal case'', when there exist $j_1 \neq j_2$ with $|\xi_{j_1}|\sim |\xi_{j_2}|$ larger than the norm of the remaining variables; then $|(\xi_1, \ldots, \xi_n)|\sim |\xi_{j_1}|\sim |\xi_{j_2}|$ and we assume that $j_1=1, j_2=2$.

\begin{enumerate}[leftmargin=*]
\item 
The multiplier localized on the off-diagonal conical region $\{(\xi_1, \ldots \xi_n): |\xi_1| \gg |\xi_2|, \ldots, |\xi_n| \}$ can be expressed as
%We would like to estimate
\begin{align}
\label{eq:positive:order:start}
\sum_{k_1 \gg k_2, \ldots, k_n} \int_{\BBR^{nd}} m_\beta(\xi_1, \ldots, \xi_n) \widehat{\Delta_{k_1} f_1}(\xi_1)\cdot \ldots \cdot \widehat{\Delta_{k_n} f_n}(\xi_n) e^{2 \pi i x \cdot (\xi_1+\ldots+\xi_n)} d \xi_1 \ldots d \xi_n.
\end{align}
For this, we approximate $m_\beta(\xi_1, \xi_2, \ldots, \xi_n)$ by $m_\beta(\xi_1, 0,  \ldots, 0)$ and in consequence we need to study the ``commutator"  $m_\beta(\xi_1, \xi_2, \ldots, \xi_n)-m_\beta(\xi_1, 0,  \ldots, 0)$. We notice the following\footnote{Throughout the section, for a function $m: \BBR^{dn} \to \BBC$ and any $1 \leq j \leq n$, $\partial_{\xi_j}m$ or $\nabla_j m$ denotes the vector \[  \partial_{\xi_j}m(\xi_1, \ldots, \xi_n)= (\partial_{\xi_j^1} m (\xi_1, \ldots, \xi_n), \ldots, \partial_{\xi_j^d} m (\xi_1, \ldots, \xi_n) ).\]}:
\begin{align*}
m_\beta(\xi_1, \xi_2, \ldots, \xi_n)-m_\beta(\xi_1, 0,  \ldots, 0)= \int_0^1 \frac{d}{d \,t } \big(  m_\beta(\xi_1, t \xi_2, \ldots, t \xi_n)  \big) dt = \int_0^1 \sum_{l=2}^n \nabla_{\xi_l} m_\beta(\xi_1, t \xi_2, \ldots, t \xi_n) \cdot \xi_l dt.
\end{align*}

Now we fix $2 \leq l \leq n$ and assume without loss of generality that $l= 2$. The functions $f_1$ and $f_2$ will play a prominent role, and for this reason we can sum over $k_3, \ldots, k_n \ll k_1$ in \eqref{eq:positive:order:start}. We would like to implement the Fourier series decomposition of
%the symbol
\begin{equation} \label{Miksymbol:annuli}
\left(\int_0^1 \nabla_2 m_\beta(\xi_1, t \xi_2, \ldots, t \xi_n) \cdot \xi_2 dt \right) \tilde \psi_{k_1}(\xi_1) \tilde\psi_{k_2}(\xi_2) \tilde\varphi_{k_1}(\xi_3) \cdot \ldots \cdot\tilde\varphi_{k_1}(\xi_n),
\end{equation}
which however requires a further restriction of the symbol. For this, we cover the annuli $\{|\xi_1| \sim 2^{k_1} \}$ and $\{|\xi_2| \sim 2^{k_2} \}$ with Whitney cubes associated to directional cones as described in Remark \ref{remark:higher:dim}:
\begin{equation*}
\{|\xi_1| \sim 2^{k_1}\} \subseteq \bigcup_{c_1 \in \mathfrak{C}}\{|\xi_1| \sim 2^{k_1}\}\cap c_1, \quad \{|\xi_2| \sim 2^{k_2}\} \subseteq \bigcup_{c_2 \in \mathfrak{C}}\{|\xi_2| \sim 2^{k_2}\}\cap c_2.
\end{equation*}
Hence the symbol \eqref{Miksymbol:annuli} can be rewritten as
\[
\sum_{c_1, c_2 \in \mathfrak{C}}\int_0^1 \nabla_2 m_\beta(\xi_1, t \xi_2, \ldots, t \xi_n) \cdot \xi_2 \, \tilde \psi_{k_1,c_1}(\xi_1) \tilde \psi_{k_1}(\xi_1) \tilde \psi_{k_2,c_2}(\xi_2) \tilde \psi_{k_2}(\xi_2) \cdot \tilde \varphi_{k_1}(\xi_3) \ldots \tilde \varphi_{k_1}(\xi_n)  dt.
\]
For each summand with $c_1, c_2 \in \mathfrak{C}$ fixed, we apply the Fourier series decomposition to obtain
\begin{align*}
\sum_{L_1, \ldots, L_n \in \BBZ^d} C_{L_1, \ldots, L_n}^{k_1, k_2,c_1,c_2} e^{2 \pi i L_1 \cdot {{\xi_1} \over {2^{k_1}}}} e^{2 \pi i L_2 \cdot {{\xi_2} \over {2^{k_2}}}} e^{2 \pi i L_3 \cdot {{\xi_3} \over {2^{k_1}}}} \ldots e^{2 \pi i L_n \cdot {{\xi_n} \over {2^{k_1}}}}
\end{align*}
on the Whitney rectangle $Q_{k_1, c_1} \times Q_{k_2, c_2} \times \{ |\xi_3|\lesssim 2^{k_1}\} \times \ldots \times \{ |\xi_n|\lesssim 2^{k_1}\}$.

The Fourier coefficients $C_{L_1, \ldots, L_n}^{k_1, k_2,c_1,c_2}$ can be written as
\begin{align*}
2^{-k_1d(n-1)}2^{-k_2d}\int_0^1 \int_{\BBR^{nd}}& \nabla_2 m_\beta(\xi_1, t \xi_2, t\xi_3,\ldots, t \xi_n) \cdot \xi_2 \, \tilde \psi_{k_1,c_1}(\xi_1) \tilde \psi_{k_2,c_2}(\xi_2) \cdot  \tilde \psi_{k_1}(\xi_1) \cdot  \tilde \psi_{k_2}(\xi_2)
\\ & \cdot \tilde \varphi_{k_1}(\xi_3) \ldots\tilde  \varphi_{k_1}(\xi_n) \cdot
 e^{2 \pi i L_1 \cdot \frac{\xi_1}{2^{k_1}}}e^{2 \pi i L_2 \cdot \frac{\xi_2}{2^{k_2}}} e^{2 \pi i L_3 \cdot \frac{\xi_3}{2^{k_1}}} \ldots e^{2 \pi i L_n \cdot \frac{\xi_n}{2^{k_1}}} d\xi_1 \ldots d\xi_n dt.
\end{align*}
By change of variables $\xi'_{\tilde{l}} := 2^{-k_1}\xi_{\tilde{l}}$ for $\tilde{l}\neq 2$ and $\xi_2' := 2^{-k_2}\xi_2$, it becomes
\begin{align} \label{fourier_coeff:Mikhlin}
 2^{k_2} \int_0^1 \int_{\BBR^{nd}} & \nabla_2 m_\beta(2^{k_1}\xi'_1, t 2^{k_2}\xi'_2, t2^{k_1}\xi'_3, \ldots, t 2^{k_1}\xi'_n) \cdot \xi'_2 \, \tilde \psi_{0,c_1}(\xi'_1) \tilde \psi_{0,c_2}(\xi'_2) \tilde \psi(\xi'_1) \tilde \psi(\xi'_2)  \nonumber
 \\& \cdot \tilde \varphi_{0}(\xi'_3) \ldots \tilde \varphi_{0}(\xi'_n)
  e^{2 \pi i L_1 \cdot \xi'_1}e^{2 \pi i L_2 \cdot \xi'_2} e^{2 \pi i L_3 \cdot \xi'_3} \ldots e^{2 \pi i L_n \cdot \xi'_n} d\xi'_1 \ldots d\xi'_n dt.
\end{align}
Using integration by parts, we can bound \eqref{fourier_coeff:Mikhlin} by
\begin{align} \label{fourier_coeff:Mikhlin:ibp}
2^{k_2} (1+|L_1| + \ldots |L_n|)^{-|\tilde{M}|} & \int_0^1 \int_{\BBR^{nd}}  \big|\partial^{\tilde{M}}_{\xi'} \left(\nabla_2 m_\beta(2^{k_1}\xi'_1, t 2^{k_2}\xi'_2, t2^{k_1}\xi'_3, \ldots, t 2^{k_1}\xi'_n) \cdot \xi'_2 \right) \big| \nonumber \\ 
& \big|\tilde \psi_{0,c_1}(\xi'_1) \tilde \psi_{0,c_2}(\xi'_2) \tilde \psi(\xi'_1) \tilde \psi(\xi'_2) \cdot \varphi_{0}(\xi'_3) \ldots \varphi_{0}(\xi'_n)\big| d\xi'_1 \ldots d\xi'_n dt,
\end{align}
for any multi-indices $\tilde{M}$, where $\xi' := (\xi'_1, \ldots, \xi'_n)$. 

Since $m_\beta$ is of order $\beta>0$ and its derivatives decay away from $0$, we can majorize, for $(\xi'_1, \ldots \xi'_n)$ with $|\xi'_1| \sim 1, |\xi'_2| \sim 1, |\xi'_3| \lesssim 1 \ldots, |\xi'_n| \lesssim 1$,
\begin{equation}\label{decay:Mikhlin:Fourier}
 \big|\partial^{\tilde{M}}_{\xi'} \left( \nabla_2 m_\beta(2^{k_1}\xi'_1, t 2^{k_2}\xi'_2, t2^{k_1}\xi'_3, \ldots, t 2^{k_1}\xi'_n) \cdot \xi'_2 \right) \big| \lesssim 2^{k_1(\beta-1)}.
\end{equation}
By applying \eqref{decay:Mikhlin:Fourier} to \eqref{fourier_coeff:Mikhlin:ibp}, we conclude that for any $M>0$,
\begin{equation}
\big| C_{L_1, \ldots, L_n}^{k_1, k_2,c_1,c_2} \big| \lesssim_M  \frac{2^{k_1 (\beta-1)} 2^{k_2} }{ \left(1+|L_1|+ \ldots + |L_n| \right)^M}.
\end{equation}

The initial estimate \eqref{eq:positive:order:start} becomes the sum of terms of the form: 
\begin{align*}
I_A:= \sum_{c_1, c_2 \in \mathfrak{C}}\ \  \sum_{k_1 \gg k_2} C_{L_1, \ldots, L_n}^{k_1, k_2,c_1,c_2} \Delta_{k_1, c_1, \frac{L_1}{2^{k_1}}}f_1(x) \Delta_{k_2, c_2, \frac{L_2}{2^{k_1}}}f_2(x) \cdot S_{k_1, \frac{L_3}{2^{k_1}}}f_3(x) \cdot \ldots \cdot S_{k_1, \frac{L_n}{2^{k_1}}}f_n(x)
\end{align*}
and 
\begin{align} \label{simpler_symbol_multiplier}
I_B:= \sum_{k_1 \gg k_2, \ldots, k_n} (T_{m_\beta (\cdot, 0, \ldots, 0)} \Delta_{k_1 }f_1)(x) \cdot  (\Delta_{k_2 }f_2)(x) \cdot \ldots \cdot  (\Delta_{k_n }f_n)(x).
\end{align}

Above, the operator $ \Delta_{k_1, c_1}$ is defined in frequency by
\begin{equation}
\label{eq:def:proj:W}
\widehat{ \Delta_{k_1, c_1} f_1}(\xi_1):= \tilde \psi_{k_1,c_1}(\xi_1) \tilde \psi_{k_1}(\xi_1) \widehat{\Delta_{k_1} f_1}(\xi_1), 
\end{equation}
so that for any $1 \leq p_1 \leq \infty$, we still\footnote{$\Delta_{k_1, c_1, a}$ represents the $a$-modulation of $\Delta_{k_1, c_1} $.} have 
\begin{equation}
\label{eq:proj:W:same}
\|\Delta_{k_1, c_1, a} f_1\|_{p_1}= \| \Delta_{k_1, c_1} f_1 \|_{p_1} \lesssim \|\Delta_{k_1} f_1\|_{p_1}.
\end{equation}

For $I_A$, we get the usual estimates analogous to the ones in Section \ref{sec:one:param:generic:ind}:
\begin{align*}
\|I_A\|_r^\tau \lesssim %\sum_{j=2}^n 
\sum_{k_1 \gg k_2} 2^{k_1(\beta-1) \tau} 2^{k_2 \tau} \|\Delta_{k_1 }f_1\|_{p_1}^\tau \|\Delta_{k_2}f_2\|_{p_2}^\tau \cdot \|f_3\|_{p_3}^\tau \cdot \ldots \|f_n\|_{p_n}^\tau.
\end{align*}

For $I_B$, we switch the order of summation to rewrite it as
\[
 (T_{m_\beta (\cdot, 0, \ldots, 0)}f_1)(x) \cdot f_2(x) \cdot \ldots \cdot f_n(x) - \sum_{k_1\prec k_2} (T_{m_\beta (\cdot, 0, \ldots, 0)} \Delta_{k_1 }f_1)(x) (\Delta_{k_2 }f_2)(x) f_3(x) \cdot \ldots \cdot f_n(x) + \text{many similar terms}. 
\]
The first term is the reason why we cannot obtain endpoint estimates $p_1=1$ or $p_1=\infty$,\footnote{In other regions where for example $|\xi_{j_0}| \gg |\xi_1|, \ldots, |\xi_n|$, we will miss the endpoints $p_{j_0}=1$ or $p_{j_0}=\infty$; so overall we simply have the conditions $1<p_1, \ldots, p_n < \infty$.} but its boundedness reduces to H\"older's inequality and to the fact that
\[
m_{\beta}(\xi_1, 0, \ldots, 0) \cdot |\xi_1|^{- \beta}
\]
is a Mikhlin symbol of order $0$; hence 
\[
 \| T_{m_\beta (\cdot, 0, \ldots, 0)}f_1 \|_{p_1} \lesssim \| D^\beta f_1\|_{p_1}
\] 
for any $1<p_1< \infty$. The remaining terms reduce to familiar estimates of the form
\[
\sum_{k_1< k_j + n+10} 2^{k_1 \beta \tau} \|\Delta_{k_1} f_1\|_{p_1}^\tau \|\Delta_{k_j} f_j\|_{p_j}^\tau,  
\]
which can easily be bounded -- see the treatment of $I_A$ in Section \ref{Bourgain-Li_hilow}.
\vskip .2cm
\item In the diagonal cone $\{(\xi_1, \ldots, \xi_n): |\xi_1| \sim |\xi_l| \geq |\xi_{\tilde{l}}| \ \ \text{for} \ \ \tilde{l} \neq 1, l \}$, we assume without loss of generality that $l= 2$ and we would like to estimate 
\begin{align}
\label{eq:positive:order:start:diag}
\sum_{k_1 \sim  k_2} \int_{\BBR^{nd}} m_\beta(\xi_1, \ldots, \xi_n) \widehat{\Delta_{k_1} f_1}(\xi_1) \widehat{\Delta_{k_2} f_2}(\xi_2) \widehat{\Delta_{ \leq k_1} f_3}(\xi_3)\cdot \ldots \cdot \widehat{\Delta_{ \leq k_1} f_n}(\xi_n) e^{2 \pi i x \cdot (\xi_1+\ldots+\xi_n)} d \xi_1 \ldots d \xi_n.
\end{align}

We apply the Whitney decomposition 
 \begin{align} \label{diagonal:whitney:highdim}
 \displaystyle \bigcup_{k_1 \in \mathbb{Z}} \bigcup_{c_1, c_2 \in \mathfrak{C}}(\{ |\xi_1| \sim 2^{k_1} \} \cap c_1) \times (\{ |\xi_2| \sim 2^{k_1} \} \cap c_2) \times \{ |\xi_3| \leq 2^{k_1}  \} \times \ldots \times \{ |\xi_n| \leq 2^{k_1} \}
 \end{align}
and then perform a Fourier series decomposition of $m_\beta(\xi_1, \ldots, \xi_n)$ restricted to each Whitney cube with fixed $k_1 \in \mathbb{Z}$ and $c_1, c_2 \in \mathfrak{C}$. We thus obtain
\begin{align*}
m_\beta(\xi_1, \ldots, \xi_n) \tilde \psi_{k_1,c_1}(\xi_1) \tilde\psi_{k_1,c_2}(\xi_2) \tilde \varphi_{k_1}(\xi_3) \cdot \ldots \cdot \tilde \varphi_{k_1}(\xi_n)= \sum_{L_1, \ldots, L_n \in \BBZ^d} C_{L_1, \ldots, L_n}^{k_1,c_1, c_2} e^{2 \pi i L_1 \cdot {{\xi_1} \over {2^{k_1}}}} e^{2 \pi i L_2 \cdot {{\xi_2} \over {2^{k_1}}}}  e^{2 \pi i L_3 \cdot {{\xi_3} \over {2^{k_1}}}} \ldots e^{2 \pi i L_n \cdot {{\xi_n} \over {2^{k_1}}}}.
\end{align*}
Since this Whitney cube is away from the origin, we have again for any $M>0$,
\[
\big| C_{L_1, \ldots, L_n}^{k_1} \big| \lesssim_M \frac{2^{k_1 \beta}}{\left( 1+|L_1|+ \ldots+ |L_n| \right)^M}. 
\]
Hence estimating \eqref{eq:positive:order:start:diag} in $\|\cdot \|_{r}^\tau$ reduces to summing
\[
\sum_{k_1} 2^{k_1 \beta \tau} \|\Delta_{k_1} f_1\|_{p_1}^\tau  \|\Delta_{k_1} f_2\|_{p_2}^\tau   \|f_3\|_{p_3}^\tau  \cdot \ldots \cdot  \|f_n\|_{p_n}^\tau,
\]
which is a routine computation by now.
\end{enumerate}
\end{proof}

\begin{remark}
For depth-1 trees, we obtain ``Kato-Ponce''-type estimates:
\begin{equation}
\label{eq:KatoPonce:symbols:positive:order}
\|T_{m_\beta} (f, g)-  (T_{m_\beta(\cdot, 0)} f) \cdot g - f \cdot (T_{m_\beta(0, \cdot)} g) \|_r \lesssim \|D^\beta f\|_{p_1} \|g\|_{p_2} + \|f\|_{p_1}  \|D^\beta g\|_{p_2} 
\end{equation}
for any $1 \leq p_1, p_2 \leq \infty$, $1/r=1/{p_1}+1/{p_2}$. In this case we can allow for $L^1$ and $L^\infty$ endpoints. Of course, the right hand side of \eqref{eq:KatoPonce:symbols:positive:order} can be replaced by a geometric mean of appropriate Besov norms.
\end{remark}

The next interesting situation corresponds to mutilinear multi-parameter operators associated to generic Marcinkiewicz symbols of positive orders; that is, we consider trees of complexity $1$ for multi-parameter non-tensorized symbols satisfying \eqref{eq:positive:order:class:multi:param}. In order to avoid over-burdening the notation, we simply assume that only two parameters and two functions are involved: $N=2$, $n=2$.

\begin{proof}[Sketch of Proof of Theorem \ref{thm:multiparameter_nontensor}]
We start with $m_{\beta_1,\beta_2}(\xi,\eta)$ a symbol in $\mathbb{R}^{2d_1} \times \mathbb{R}^{2d_2}$ that is smooth away from the planes $\{(\xi_1, \xi_2) = 0 \}$ and $\{(\eta_1, \eta_2) = 0 \}$, and satisfies the Marcinkiewicz condition 
\begin{equation*}
|\partial^{\zeta_1}_{\xi}\partial^{\zeta_2}_{\eta}m_{\beta_1,\beta_2}(\xi, \eta)| \lesssim |\xi|^{\beta_1 - |\zeta_1|} |\eta|^{\beta_2 - |\zeta_2|}
\end{equation*}
for sufficiently many multi-indices $\zeta_1, \zeta_2$, where $\xi = (\xi_1, \xi_2) = \left((\xi_1^i)_{i=1}^{d_1}, (\xi^i_2)_{i=1}^{d_1}\right)$ and $\eta = (\eta_1, \eta_2) = \left((\eta_1^i)_{i=1}^{d_2}, (\eta^i_2)_{i=1}^{d_2}\right)$. 

The associated bilinear operator is given by
\begin{equation} \label{Marcink_condition}
T_{m_{\beta_1,\beta_2}}(f_1, f_2)(x) := \int_{\mathbb{R}^{2d_1} \times \BBR^{2d_2}} m_{\beta_1,\beta_2}(\xi, \eta)\f{f_1}(\xi_1,\eta_1)\f{f_2}(\xi_2,\eta_2)e^{2 \pi i x\cdot(\xi_1+\xi_2)}e^{2 \pi i y\cdot(\eta_1+\eta_2)} d\xi d\eta,
\end{equation}
and our aim is to prove that $T_{m_{\beta_1,\beta_2}}$ satisfies the same estimate described in \eqref{Leib_bi_paraproduct:mixed} for the Lebesgue exponents $1 < p^j_1, p^j_2 < \infty, \frac{1}{p_1^j} + \frac{1}{p_2^j} = \frac{1}{r^j},  0 < r^j <\infty, $ for all $1 \leq j \leq 2$.

As in the one-parameter setting, we first decompose the frequency spaces for both parameters into cones as in \eqref{cone:type}. Depending on the type of cones, the arguments will be different and we will develop a case-by-case study as before.
\vskip .1in
\noindent
\textbf{(1)}
In the case when the cones for both parameters are of type (1) as in \eqref{cone:type}, the root symbol split will be two-folded, involving commutators in each parameter.

Assume without loss of generality that the symbol is smoothly restricted to the region 
\begin{equation*}
R:= \{(\xi_1, \xi_2, \eta_1, \eta_2): |\xi_1| \gg |\xi_2|, |\eta_1| \gg |\eta_2| \}.
%\displaystyle \bigcup_{\substack{k_1 \gg k_2 \\ m_1 \gg m_2}}R_{k_1,k_2,m_1,m_2}.
\end{equation*}
Let $\tilde{\chi}_R$ denote the smooth restriction to the region $R$.
We split the symbol on this region as
\begin{align*}
& m_{\beta_1,\beta_2}(\xi_1, \xi_2, \eta_1,\eta_2) \\
=&  \left(m_{\beta_1,\beta_2}(\xi_1, \xi_2, \eta_1, \eta_2) - m_{\beta_1,\beta_2}(\xi_1, 0, \eta_1, \eta_2)\right) + \left(m_{\beta_1,\beta_2}(\xi_1, 0, \eta_1, \eta_2) - m_{\beta_1,\beta_2}(\xi_1, 0, \eta_1, 0)\right) + m_{\beta_1,\beta_2}(\xi_1, 0, \eta_1, 0) \\
= & \int_{0}^1\partial_{\xi_2} m_{\beta_1,\beta_2}(\xi_1,t\xi_2, \eta_1,\eta_2) \cdot \xi_2 dt + \int_0^1\partial_{\eta_2} m_{\beta_1,\beta_2}(\xi_1, 0, \eta_1,s\eta_2) \cdot  \eta_2 ds + m_{\beta_1,\beta_2}(\xi_1, 0, \eta_1, 0) \\
= &  \int_{0}^1\left( \partial_{\xi_2} m_{\beta_1,\beta_2}(\xi_1,t\xi_2, \eta_1,\eta_2) - \partial_{\xi_2} m_{\beta_1,\beta_2}(\xi_1, t\xi_2, \eta_1,0)\right) \cdot \xi_2 dt + \int_0^1 \partial_{\xi_2} m_{\beta_1,\beta_2}(\xi_1, t\xi_2, \eta_1,0) \cdot \xi_2  dt +\\
&   \int_0^1\partial_{\eta_2} m_{\beta_1,\beta_2}(\xi_1, 0, \eta_1,s\eta_2) \cdot \eta_2 ds + m_{\beta_1,\beta_2}(\xi_1, 0, \eta_1, 0)\\
= & \int_0^1 \int_{0}^1\sum_{i=1}^{d_1}\sum_{i'=1}^{d_2} \xi^i_2 \eta^{i'}_2\partial_{\xi^i_2}\partial_{\eta^{i'}_2} m_{\beta_1,\beta_2}(\xi_1,t\xi_2, \eta_1,s\eta_2) dtds + \int_0^1 \partial_{\xi_2} m_{\beta_1,\beta_2}(\xi_1, t\xi_2, \eta_1,0) \cdot  \xi_2 dt + \\
& \int_0^1\partial_{\eta_2} m_{\beta_1,\beta_2}(\xi_1, 0, \eta_1,s\eta_2) \cdot  \eta_2 ds + m_{\beta_1,\beta_2}(\xi_1, 0, \eta_1, 0)\\
:= & \mathcal{I} + \mathcal{II} + \mathcal{III} + \mathcal{IV}.
\end{align*}
In the case when the symbol tensorizes, $\mathcal{I}$ corresponds to the tensor product of two commutator symbols, $\mathcal{II}$ and $\mathcal{III}$ are generalizations of the mix of a commutator symbol and a symbol of lower complexity whereas $\mathcal{IV}$ is a biparameter variant of $I_B$ defined in (\ref{simpler_symbol_multiplier}). 

It suffices to prove the boundedness of the multipliers $T_{\mathcal{I}\tilde{\chi}_R}$, $T_{\mathcal{II}\tilde{\chi}_R}$, $T_{\mathcal{III}\tilde{\chi}_R}$ and $T_{\mathcal{IV}\tilde{\chi}_R}$.
\vskip .05in
\noindent
\textbf{Estimate for $T_{\mathcal{I}\tilde{\chi}_R}.$}
We recall the Whitney decomposition used in the one parameter setting and decompose frequency spaces for both parameters as
\begin{align*}
\{(\xi_1, \xi_2) \in \BBR^{2d_1}: |\xi_1| \gg |\xi_2| \} = & \bigcup_{c_1, c_2 \in \mathfrak{C}_1}\bigcup_{k_1 \gg k_2 }(\{|\xi_1| \sim 2^{k_1} \} \cap c_1) \times (\{|\xi_2| \sim 2^{k_2} \} \cap c_2), \\
\{(\eta_1, \eta_2) \in \BBR^{2d_2}: |\eta_1| \gg |\eta_2| \} = & \bigcup_{c_1', c_2' \in \mathfrak{C}_2}\bigcup_{m_1 \gg m_2 }(\{|\eta_1| \sim 2^{m_1} \} \cap c_1') \times (\{|\eta_2| \sim 2^{m_2} \} \cap c_2'),
\end{align*}
where $\mathfrak{C}_1$ and $\mathfrak{C}_2$ represent the collections of directional cones in the frequency spaces $\BBR^{d_1}$ and $\BBR^{d_2}$ respectively. We then smoothly restrict the symbol $$
m_{C_{\beta_1}, C_{\beta_2}}(\xi_1,\xi_2, \eta_1,\eta_2):=\int_0^1 \int_{0}^1\sum_{i=1}^{d_1}\sum_{i'=1}^{d_2} \xi^i_2 \eta^{i'}_2\partial_{\xi^i_2}\partial_{\eta^{i'}_2} m_{\beta_1,\beta_2}(\xi_1,t\xi_2, \eta_1,s\eta_2) dtds$$
to each Whitney rectangle with fixed $k_1  \gg k_2, m_1 \gg m_2 $, $c_1, c_2 \in \mathfrak{C}_1$ and $c_1', c_2' \in \mathfrak{C}_2$, and denote it by $m^{k_1, k_2, m_1,m_2,c_1, c_2, c_1',c_2'}_{C_{\beta_1},C_{\beta_2}}$. We perform the quadruple Fourier series decomposition to obtain
\begin{align} \label{fourier_series_biparameter_marcink}
m_{C_{\beta_1}, C_{\beta_2}}^{k_1,k_2,m_1,m_2, c_1, c_2, c_1',c_2'}(\xi_1,\xi_2,\eta_1,\eta_2) =
\sum_{\substack{L_1,L_2 \in \mathbb{Z}^{d_1} \\ L_1',L_2' \in \mathbb{Z}^{d_2}}}C^{k_1,k_2,m_1,m_2,c_1, c_2, c_1',c_2'}_{L_1,L_2,L_1',L_2'}e^{2 \pi i L_1\cdot \frac{\xi_1}{2^{k_1}}}e^{2 \pi i  L_2\cdot\frac{\xi_2}{2^{k_2}}}e^{2 \pi i L_1'\cdot\frac{\eta_1}{2^{m_1}}}e^{2 \pi i L_2'\cdot\frac{\eta_2}{2^{m_2}}}, 
\end{align}
where the Fourier coefficients decay rapidly due to the Marcinkiewicz condition (\ref{Marcink_condition}) on $m_{\beta_1,\beta_2}$:
\begin{align*}
|C^{k_1,k_2,m_1,m_2,c_1,c_2, c_1',c_2'}_{L_1,L_2,L_1',L_2'}| \lesssim 2^{k_1(\beta_1-1)}2^{k_2}2^{m_1(\beta_2-1)}2^{m_2}(1+|L_1|+|L_2|)^{-N}(1+|L_1'|+|L_2'|)^{-N'}
\end{align*}
for sufficiently large $N$ and $N'$.

By applying the Fourier series representation (\ref{fourier_series_biparameter_marcink}), we rewrite the multiplier as
\begin{align*}
&T_{\mathcal{I}\tilde{\chi}_R} (f_1,f_2)(x,y) =  \sum_{\substack{c_1,c_2 \in \mathfrak{C}_1 \\ c_1',c_2' \in \mathfrak{C}_2}}\sum_{\substack{L_1, L_2 \in \mathbb{Z}^{d_1} \\ L_1', L_2' \in \mathbb{Z}^{d_2} \\ }}\sum_{\substack{k_1 \gg k_2 \\ m_1 \gg m_2}}C^{k_1,k_2,m_1,m_2,c_1, c_2, c_1',c_2'}_{L_1,L_2, L_1',L_2'} \\
& \int_{\mathbb{R}^{2d_1} \times \BBR^{2d_2}} \ii F(\Delta^{(1)}_{k_1,c_1}\Delta_{m_1,c_1'}^{(2)}{f_1}) (\xi_1,\eta_1) \ii F(\Delta_{k_2,c_2}^{(1)}\Delta^{(2)}_{m_2,c_2'}f_2)(\xi_2,\eta_2) e^{2 \pi i \xi_1\cdot(x+\frac{L_1}{2^{k_1}})} e^{2 \pi i  \xi_2\cdot(x+\frac{L_2}{2^{k_2}} )}e^{2 \pi i \eta_1\cdot(y+\frac{L_1'}{2^{m_1}})} e^{2 \pi i \eta_2\cdot(y+\frac{L_2'}{2^{m_2}})} d\xi d\eta.
\end{align*}
We can invoke now the analysis developed in Sections \ref{sec:2param:5linflag} and \ref{generic_induction} to conclude the discussion.
\vskip .05in 
\noindent
\textbf{Estimate for $T_{\mathcal{II}\tilde{\chi}_R}$ and $T_{\mathcal{III}\tilde{\chi}_R}$.}
We shall notice that the treatment of the multipliers corresponding to $\mathcal{II}$ and $\mathcal{III}$ are symmetric and for that reason we will focus on $T_{\mathcal{II}\tilde{\chi}_R}$. We rewrite the symbol as
\begin{align*}
\mathcal{II} =m_{C_{\beta_1},H}(\xi_1,\xi_2, \eta_1)|\eta_1|^{\beta_2},
\end{align*}
where
\begin{equation*}
m_{C_{\beta_1},H}(\xi_1,\xi_2, \eta_1) := \int_0^1 \frac{\partial_{\xi_2} m(\xi_1, t\xi_2, \eta_1, 0)\cdot \xi_2}{|\eta_1|^{\beta_2}} dt.
\end{equation*}
As the notation suggests, $m_{C_{\beta_1},H}(\xi_1,\xi_2, \eta_1)$ is a symbol generating a commutator in the first parameter and a Mikhlin symbol of order $0$ (the simplest example being the symbol corresponding to the Hilbert transform) in the second parameter localized on the region $\{(\xi_1, \xi_2, \eta_1): |\xi_1| \gg |\xi_2| \}$. In order to use the Fourier series decomposition on this symbol, we need to decompose  
\begin{align*}
\{(\xi_1, \xi_2, \eta_1): |\xi_1| \gg |\xi_2| \} = \bigcup_{c_1, c_2 \in \mathfrak{C}_1}  \bigcup_{c_1' \in \mathfrak{C}_2} \bigcup_{\substack{k_1 \gg k_2 \\ m_1 \in \mathbb{Z}}}| (\{|\xi_1| \sim 2^{k_1} \} \cap c_1) \times (\{|\xi_2| \sim 2^{k_2} \} \cap c_2) \times (\{|\eta_1| \sim 2^{m_1}  \} \cap c_1').
\end{align*}
We smoothly restrict $m_{C_{\beta_1},H}(\xi_1,\xi_2, \eta_1)$ to a region with fixed $k_1 \gg k_2, m_1 \in \mathbb{Z}$, $c_1, c_2 \in \mathfrak{C}_1$ and $c_1' \in \mathfrak{C_2}$, and let $m^{k_1,k_2,m_1,c_1,c_2,c_1'}_{C_{\beta_1},H}(\xi_1,\xi_2, \eta_1)$ denote the localized symbol. We then perform the triple Fourier series decomposition for
\begin{align} \label{fourier_series_commu_H}
m^{k_1,k_2, m_1,c_1,c_2,c_1'}_{C_{\beta_1},H}(\xi_1,\xi_2, \eta_1) = \sum_{\substack{L_1,L_2 \in \mathbb{Z}^{d_1} \\  L' \in \mathbb{Z}^{d_2}}} C^{k_1,k_2,m_1,c_1,c_2,c_1'}_{L_1,L_2, L'}  e^{2 \pi i L_1\cdot \frac{\xi_1}{2^{k_1}}} e^{2 \pi i L_2\cdot \frac{\xi_2}{2^{k_2}}} e^{2 \pi i L' \cdot\frac{\eta_1}{2^{m_1}}}.
\end{align}
Because of the regularity of $m_{C_{\beta_1},H}(\xi_1,\xi_2, \eta_1)$, the Fourier coefficients satisfy the decay condition
\begin{equation} \label{fourier_coef_commu_H}
|C^{k_1,k_2,m_1,c_1,c_2,c_1'}_{L_1,L_2, L'}|\lesssim 2^{k_1(\beta_1-1)}2^{k_2} (1+|L_1|+|L_2|)^{-N}(1+|L'|)^{-N'}
\end{equation}
for sufficiently large $N$ and $N'$.
Thanks to (\ref{fourier_series_commu_H}), we can rewrite the multiplier as
\begin{align*}
T_{\mathcal{II}\tilde{\chi}_R} &= \sum_{\substack{c_1,c_2 \in \mathfrak{C}_1 \\ c'_1 \in \mathfrak{C}_2}} \sum_{\substack{L_1,L_2 \in \mathbb{Z}^{d_1} \\  L' \in \mathbb{Z}^{d_2}}} \sum_{\substack{k_1 \gg k_2 \\ m_1 \gg m_2}}C^{k_1,k_2,m_1,c_1,c_2,c_1'}_{L_1,L_2,L'} \int |\eta_1|^{\beta_2} \ii F(\Delta_{k_1,c_1}^{(1)}\Delta_{m_1,c_1'}^{(2)} f_1) (\xi_1,\eta_1) \ii F(\Delta_{k_2,c_2}^{(1)}\Delta_{m_2}^{(2)}f_2)(\xi_2, \eta_2)e^{2 \pi i \xi_1\cdot(x+\frac{L_1}{2^{k_1}})} \\
&\quad \quad \quad \quad \quad \quad \quad \quad\quad \quad \quad \quad \quad \quad \quad \quad \quad e^{2 \pi i  \xi_2\cdot(x+\frac{L_2}{2^{k_2}})}e^{2 \pi i \eta_1\cdot(y+\frac{L'}{2^{m_1}})} e^{2 \pi i \eta_2\cdot y} d\xi d\eta\\
= & \sum_{\substack{c_1,c_2 \in \mathfrak{C}_1 \\ c'_1 \in \mathfrak{C}_2}} \sum_{\substack{L_1,L_2 \in \mathbb{Z}^{d_1} \\  L' \in \mathbb{Z}^{d_2}}} \sum_{\substack{k_1 \gg k_2 \\ m_1, m_2}}C_{L_1,L_2,L'}^{k_1,k_2,m_1,c_1,c_2,c_1'}  \int  \ii F(\Delta_{k_1,c_1}^{(1)}\Delta_{m_1,c_1'}^{(2)}D^{\beta_2}_{(2)} f_1) (\xi_1,\eta_1) \ii F(\Delta_{k_2,c_2}^{(1)}\Delta_{m_2}^{(2)}f_2)(\xi_2, \eta_2)e^{2 \pi i \xi_1\cdot(x+\frac{L_1}{2^{k_1}})} \\
&\quad \quad \quad \quad \quad \quad \quad \quad\quad \quad \quad \quad \quad \quad \quad \quad \quad e^{2 \pi i  \xi_2\cdot(x+\frac{L_2}{2^{k_2}} )}e^{2 \pi i \eta_1\cdot(y+\frac{L'}{2^{m_1}})} e^{2 \pi i \eta_2 \cdot y} d\xi d\eta\\
& - \sum_{\substack{c_1,c_2 \in \mathfrak{C}_1 \\ c'_1 \in \mathfrak{C}_2}} \sum_{\substack{L_1,L_2 \in \mathbb{Z}^{d_1} \\  L' \in \mathbb{Z}^{d_2}}} \sum_{\substack{k_1 \gg k_2 \\ m_1 \prec m_2}} C_{L_1,L_2,L'}^{k_1,k_2,m_1,c_1,c_2,c_1'} \int |\eta_1|^{\beta_2} \ii F(\Delta_{k_1,c_1}^{(1)}\Delta_{m_1,c_1'}^{(2)} f_1) (\xi_1,\eta_1) \ii F(\Delta_{k_2,c_2}^{(1)}\Delta_{m_2}^{(2)}f_2)(\xi_2, \eta_2)e^{2 \pi i \xi_1\cdot(x+\frac{L_1}{2^{k_1}})} \\
&\quad \quad \quad \quad \quad \quad \quad \quad\quad \quad \quad \quad \quad \quad \quad \quad \quad e^{2 \pi i  \xi_2\cdot(x+\frac{L_2}{2^{k_2}} )}e^{2 \pi i \eta_1\cdot(y+\frac{L'}{2^{m_1}})} e^{2 \pi i \eta_2\cdot y} d\xi d\eta := \mathcal{A} - \mathcal{B}.
\end{align*}
It is not difficult to verify that $\mathcal{B}$ can be treated using the argument presented in Section \ref{sec:2param:5linflag}. In contrast, the term $\mathcal{A}$ is trickier to estimate due to the fact that 
the non-tensorized symbol $m_{C_{\beta_1},H}(\xi_1,\xi_2, \eta_1)$ exhibits different types of behaviours in the $(\xi_1, \xi_2)$ and $\eta_1$ variables.

Nonetheless, we can simplify $\mathcal{A}$ as follows 
\begin{align*}
&\sum_{\substack{c_1,c_2 \in \mathfrak{C}_1 \\c_1' \in \mathfrak{C}_2}}\sum_{\substack{L_1,L_2 \in \mathbb{Z}^{d_1} \\  L' \in \mathbb{Z}^{d_2}}} \sum_{\substack{k_1 \gg k_2 \\ }}\sum_{m_1} C_{L_1,L_2,L'}^{k_1,k_2,m_1,c_1,c_2,c_1'} \Delta_{k_1,c_1}^{(1)}\Delta_{m_1,c_1'}^{(2)}D^{\beta_2}_{(2)} f_1(x + \frac{L_1}{2^{k_1}}, y + \frac{L'}{2^{m_1}})\sum_{m_2} \Delta_{k_2,c_2}^{(1)}\Delta_{m_2}^{(2)}f_2(x + \frac{L_2}{2^{k_2}}, y ) \\
= &\sum_{\substack{c_1,c_2 \in \mathfrak{C}_1 \\c_1' \in \mathfrak{C}_2}} \sum_{\substack{L_1,L_2 \in \mathbb{Z}^{d_1} \\  L' \in \mathbb{Z}^{d_2}}} \sum_{\substack{k_1 \gg k_2 \\ }}\sum_{m_1}  C_{L_1,L_2,L'}^{k_1,k_2,m_1,c_1,c_2,c_1'} \Delta_{k_1,c_1}^{(1)}\Delta_{m_1,c_1'}^{(2)}D^{\beta_2}_{(2)} f_1(x + \frac{L_1}{2^{k_1}}, y + \frac{L'}{2^{m_1}})\Delta_{k_2,c_2}^{(1)}f_2(x + \frac{L_2}{2^{k_2}}, y ).
\end{align*}
Its $\|\cdot \|^\tau_{L^{r^1}_x(L^{r^2}_y)}$ norm with $\tau \leq \min(1, r^1, r^2)$ can be majorized by
\begin{align} \label{commu_H_almost_final}
\sum_{\substack{c_1,c_2 \in \mathfrak{C}_1 \\c_1' \in \mathfrak{C}_2}}\sum_{\substack{L_1,L_2 \in \mathbb{Z}^{d_1} \\  L' \in \mathbb{Z}^{d_2}}} \sum_{\substack{k_1 \gg k_2 \\ }}  \big\|\sum_{m_1} C_{L_1,L_2,L'}^{k_1,k_2,m_1,c_1,c_2,c_1'}\Delta_{k_1,c_1}^{(1)}\Delta_{m_1,c_1'}^{(2)}D^{\beta_2}_{(2)} f_1(\cdot, \cdot + \frac{L'}{2^{m_1}})\big\|_{L^{p_1^1}_x(L^{p_1^2}_y)}^\tau \big\|\Delta_{k_2,c_2}^{(1)}f_2 \big\|_{L^{p_2^1}_x(L^{p_2^2}_y)}^\tau.
\end{align}
Let $F:= \Delta_{k_1,c_1}^{(1)}D^{\beta_2}_{(2)}f_1$. We claim that, for any $1<p_1^1, p_1^2<\infty$, 
\begin{equation} \label{claim_using_sq}
%\big\|\sum_{m_1} \Delta_{k_1}^{(1)}\Delta_{m_1}^{(2)}D^{\beta_2}_{(2)} f_1(\cdot, \cdot + \frac{L'}{2^{m_1}})\big\|_{L^{p_1^1}_x(L^{p_1^2}_y)} \lesssim 
\big\|\sum_{m_1} C_{L_1,L_2,L'}^{k_1,k_2,m_1,c_1,c_2,c_1'}\Delta_{m_1,c_1'}^{(2)}F(\cdot, \cdot + \frac{L'}{2^{m_1}})\big\|_{L^{p_1^1}_x(L^{p_1^2}_y)} \lesssim (1+|L_1|+|L_2|)^{-N}(1+|L'|)^{-N'+1}2^{k_1(\beta_1-1)}2^{k_2}\|F\|_{L^{p_1^1}_x(L^{p_1^2}_y)}.
\end{equation}
Applying (\ref{claim_using_sq}) to (\ref{commu_H_almost_final}), we are left with a familiar expression that can be easily dealt with.

To prove the claim, we first linearize the left hand side of (\ref{claim_using_sq}) by choosing an appropriate function $h \in L^{{p_1^1}'}_x(L^{{p_1^2}'}_y)$ with $\|h\|_{L^{{p_1^1}'}_x(L^{{p_1^2}'}_y)} = 1$ such that
\begin{align} \label{ineq:translated_sq}
& \big\|\sum_{m_1}C_{L_1,L_2,L'}^{k_1,k_2,m_1,c_1,c_2,c_1'} \Delta_{m_1,c_1'}^{(2)}F(\cdot, \cdot + \frac{L'}{2^{m_1}})\big\|_{L^{p_1^1}_x(L^{p_1^2}_y)}= \big\|\sum_{m_1}C_{L_1,L_2,L'}^{k_1,k_2,m_1,c_1,c_2,c_1'}\Delta_{m_1,c_1'}^{(2)}F \ast_2 \tilde\psi_{m_1}(\cdot, \cdot + \frac{L'}{2^{m_1}})\big\|_{L^{p_1^1}_x(L^{p_1^2}_y)} \nonumber\\
= &\big| \int \sum_{m_1} C_{L_1,L_2,L'}^{k_1,k_2,m_1,c_1,c_2,c_1'}\Delta_{m_1,c_1'}^{(2)}F(x, y + \frac{L'}{2^{m_1}}) \tilde \Delta_{m_1}^{(2)}h(x,y) dx dy\big| \nonumber\\
 \leq &  \sup_{m_1}|C_{L_1,L_2,L'}^{k_1,k_2,m_1,c_1,c_2,c_1'}|\int \big(\sum_{m_1} | \Delta_{m_1,c_1'}^{(2)}F(x, y + \frac{L'}{2^{m_1}})|^2\big)^{\frac{1}{2}}\big( \sum_{m_1}|\tilde \Delta_{m_1}^{(2)}h(x,y)|^2\big)^{\frac{1}{2}} dxdy \nonumber \\
 \lesssim &\frac{2^{k_1(\beta_1-1)}2^{k_2}}{(1+|L_1|+|L_2|)^{N}(1+|L'|)^{N'}} \int \Big(\int \big(\sum_{m_1} |\Delta_{m_1,c_1'}^{(2)}F(x, y + \frac{L'}{2^{m_1}})|^2\big)^{\frac{p_1^2}{2}} dy\Big)^{\frac{1}{p_1^2}} \Big(\int \big( \sum_{m_1}|\tilde \Delta_{m_1}^{(2)}h(x,y)|^2 \big)^{\frac{{p_1^2}'}{2}} dy \Big)^{\frac{1}{{p_1^2}'}}dx,
 \end{align}
where the third inequality follows from Cauchy-Schwartz and the last one holds due to H\"older. We recall the boundedness\footnote{The bounds provided here are far from being optimal, but they are sufficient for our purpose.} of the shifted square function 
 \begin{equation} \label{sqaure:translated}
\Big(\int \big(\sum_{m_1} |\Delta_{m_1,c_1'}^{(2)}F(x, y + \frac{L'}{2^{m_1}})|^2\big)^{\frac{p_1^2}{2}} dy\Big)^{\frac{1}{p_1^2}} \lesssim   O\big(1+ |L'|^{100}\big)\big(\int |F_{c_1'}(x,y)|^{p_1^2} dy\big)^{\frac{1}{p_1^2}},
 \end{equation}
 where $F_{c_1'} := \ii F^{-1}(\hat{F}(\xi_1,\eta_1) \sum_{m_1 \in \mathbb{Z}}\psi_{m_1,c_1}(\eta_1))$ and $\psi_{m_1,c_1}$ is defined in Remark \ref{remark:higher:dim}. We also have that 
 \begin{equation}
 \big(\int |F_{c_1'}(x,y)|^{p_1^2} dy\big)^{\frac{1}{p_1^2}} \lesssim \big(\int |F_{}(x,y)|^{p_1^2} dy\big)^{\frac{1}{p_1^2}},
 \end{equation}
 for $1<p_1^2 <\infty$. This ends the proof of the claim.

\vskip .05in
\noindent
\textbf{Estimate for $T_{\mathcal{IV}\tilde{\chi}_R}.$}
\begin{align*}
m_{\beta_1,\beta_2}(\xi_1,0,\eta_1,0)\tilde{\chi}_R= \sum_{\substack{k_1\gg k_2 \\ m_1 \gg m_2}} m_{\beta_1,\beta_2}(\xi_1,0,\eta_1,0) \tilde{\chi}_{R_{k_1,k_2,m_1,m_2}},
\end{align*}
where $\tilde{\chi}_{R_{k_1,k_2,m_1,m_2}}$ denote the smooth restriction to the region $R_{k_1,k_2,m_1,m_2}$ achieved by the Littlewood-Paley decomposition.
Now we apply the high-low switch technique as previously to deduce 
\begin{align*}
& \sum_{\substack{k_1\gg k_2 \\ m_1 \gg m_2}} m_{\beta_1,\beta_2}(\xi_1,0,\eta_1,0) \tilde{\chi}_{R_{k_1,k_2,m_1,m_2}} =   \sum_{\substack{k_1,k_2 \\ m_1,m_2}}m_{\beta_1,\beta_2}(\xi_1,0,\eta_1,0) \tilde{\chi}_{R_{k_1,k_2,m_1,m_2}} - \sum_{\substack{k_1 \prec  k_2 \\ m_1,m_2}}m_{\beta_1,\beta_2}(\xi_1,0,\eta_1,0) \tilde{\chi}_{R_{k_1,k_2,m_1,m_2}} \\
& - \sum_{\substack{k_1, k_2 \\ m_1 \prec m_2}}m_{\beta_1,\beta_2}(\xi_1,0,\eta_1,0) \tilde{\chi}_{R_{k_1,k_2,m_1,m_2}} + \sum_{\substack{k_1\prec k_2 \\ m_1 \prec m_2}}m_{\beta_1,\beta_2}(\xi_1,0,\eta_1,0) \tilde{\chi}_{R_{k_1,k_2,m_1,m_2}} := \mathcal{IV}_a - \mathcal{IV}_b - \mathcal{IV}_c + \mathcal{IV}_d.
\end{align*}
It is straightforward to verify that $\mathcal{IV}_a$ generates the symbol for the (linear) bi-parameter Marcinkiewicz multiplier of order $(\beta_1, \beta_2)$, whose boundedness is well-known. $\mathcal{IV}_d$ can be estimated using the routine procedures developed in Section \ref{sec:2param:5linflag}; $\mathcal{IV}_b$ and $\mathcal{IV}_c$ are symmetric and follow the similar analysis as $\mathcal{II}$. 
\vskip .05in
\noindent
\textbf{(2)} When the cone for the first parameter is of type (1) as in \eqref{cone:type} and the cone for the second parameter is of type (2),\footnote{The other case is similar by symmetry.} the argument is a hybrid of the reasoning in cases (1) and (2) developed in the one-parameter setting. In particular, the symbol is smoothly restricted to the region 
$$\tilde{R}:= \displaystyle \{(\xi_1, \xi_2, \eta_1, \eta_2): |\xi_1| \gg |\xi_2|, |\eta_1| \sim |\eta_2|\},$$ where we split it as follows:
\begin{align*}
m_{\beta_1,\beta_2}(\xi_1,\xi_2,\eta_1,\eta_2) =& m_{\beta_1,\beta_2}(\xi_1,\xi_2,\eta_1,\eta_2) - m_{\beta_1,\beta_2}(\xi_1, 0, \eta_1, \eta_2)  + m_{\beta_1,\beta_2}(\xi_1, 0, \eta_1, \eta_2) \\
= & \int_0^1 \partial_{\xi_2} m_{\beta_1,\beta_2}(\xi_1, t\xi_2, \eta_1,\eta_2)\cdot \xi_2\, dt + m_{\beta_1,\beta_2}(\xi_1, 0, \eta_1, \eta_2).
\end{align*}
The study of the multiplier associated to the symbol 
$$m_{C_{\beta_1}, \beta_2}(\xi_1, \xi_2, \eta_1, \eta_2) := \int_0^1 \partial_{\xi_2} m_{\beta_1,\beta_2}(\xi_1, t\xi_2, \eta_1,\eta_2) \cdot \xi_2 dt$$
requests a Fourier series decomposition. With a by now routine decomposition of the frequency space $\tilde{R}$, we have
\begin{align} \label{fourier_series:whitney:highdim}
\tilde{R} = \bigcup_{\substack{c_1, c_2 \in \mathfrak{C}_1 \\ c_1', c_2' \in \mathfrak{C_2}}}
\bigcup_{\substack{k_1 \gg k_2 \\ m_1 \in \mathbb{Z}}}(\{|\xi_1| \sim 2^{k_1} \} \cap c_1) \times (\{|\xi_2| \sim 2^{k_2} \} \cap c_2) \times (\{|\eta_1| \sim 2^{m_1} \} \cap c_1') \times (\{|\eta_2| \sim 2^{m_1} \} \cap c_2'),
\end{align}
where we recall that $\mathfrak{C_1}$ and $\mathfrak{C_2}$ are collections of directional cones in the frequency spaces $\BBR^{d_1}$ and $\BBR^{d_2}$, respectively. 

We smoothly restrict the symbol to each piece of \eqref{fourier_series:whitney:highdim} with fixed $k_1 \gg k_2$, $m_1 \in \mathbb{Z}$, $c_1, c_2 \in \mathfrak{C}_1$ and $c'_1, c'_2 \in \mathfrak{C}_2$ and denote it by $m_{C_{\beta_1}, \beta_2}^{k_1,k_2,m_1, c_1, c_2, c_1',c_2'}$. By applying the Fourier series decomposition of the symbol to the corresponding multiplier, we obtain an expression whose analysis follows the standard procedure of this paper. 

The symbol $m_{\beta_1,\beta_2}(\xi_1, 0, \eta_1, \eta_2)$ can be rewritten as 
\begin{equation} \label{m:H,pos}
\frac{m_{\beta_1,\beta_2}(\xi_1, 0, \eta_1, \eta_2)}{|\xi_1|^{\beta_1}} |\xi_1|^{\beta_1} =: m_{H, \beta_2}(\xi_1,\eta_1,\eta_2) |\xi_1|^{\beta_1}.
\end{equation}

Last but not least, we notice that the symbol $m_{H, \beta_2}(\xi_1,\eta_1,\eta_2)$, as implied by the notation, is of order $0$ for the first parameter and order $\beta_2$ for the second parameter; because of that, the multiplier corresponding to \eqref{m:H,pos} can be treated similarly to $\mathcal{II}$. 
\end{proof}

\subsection{Some remarks on smoothing properties of multipliers}\label{remark:smoothing}

Finally, we present some results in the vein of \cite{HartTorresWu-smoothing-bil} and \cite{YangLiuWu-smoothing}. Although our methods do not apply to symbols satisfying only Sobolev conditions,\footnote{In the one-parameter bilinear setting, instead of the Sobolev conditions in \cite{HartTorresWu-smoothing-bil}, we require that $m$ is continuously differentiable away from the origin and that it satisfies
\begin{align*}
\sup_{c_1, c_2 \in \mathfrak{C}}\sup_{k} & \|m_{k,c_1,c_2}^\nu\|_{W^{(\tilde{r}, \tilde{r}),2}} < \infty, \ \ 
\sup_{c_1, c_2 \in \mathfrak{C}}\sup_{k_1 \ll k_2} \sup_{t \in [0,1]} \|m^{\nu, t}_{k_1,k_2,c_1,c_2}\|_{W^{(\tilde{r},\tilde{r}),2}} < \infty, \ \ \text{and} \\ 
& \sup_{c_1, c_2 \in \mathfrak{C}}\sup_{k_1 \ll k_2} \sup_{t \in [0,1]} \|m^{\nu, t, \nabla_i}_{k_1,k_2,c_1,c_2}\|_{W^{(\tilde{r},\tilde{r}),2}} < \infty \ \ \text{for} \ \ i=1,2,
\end{align*}
where $\mathfrak{C}$ is a collection of directional cones in the frequency space $\BBR^d$ and
\begin{align*}
m_{k,c_1,c_2}^\nu := & 2^{k\nu} m(2^{k} \xi_1, 2^{k} \xi_2) \psi_{0,c_1}(\xi_1) \psi_{0,c_2}(\xi_2), \\
%m^{\nu, \nabla_1}_{k_1,k_2}(\xi_1, \xi_2)  :=&2^{-k_2(\nu-1)}\left(\int_0^1 \nabla_1 m(t2^{k_1}\xi_1, 2^{k_2}\xi_2) \cdot \xi_1 dt\right) \psi_{0}(\xi_1) \psi_0(\xi_2),\\
m_{k_1,k_2,c_1,c_2}^{\nu,t}(\xi_1, \xi_2) :=& 2^{k_1 \nu} m(2^{k_1} \xi_1, t 2^{k_2} \xi_2) \psi_{0,c_1}(\xi_1) \psi_{0,c_2}(\xi_2), \\
m^{\nu, t, \nabla_2}_{k_1,k_2,c_1,c_2}(\xi_1,\xi_2) :=&  2^{k_1(\nu+1)}\left( \nabla_2 m(2^{k_1}\xi_1, t 2^{k_2}\xi_2)\cdot \xi_2 \right) \psi_{0,c_1}(\xi_1) \psi_{0,c_2}(\xi_2),
\end{align*}
and $m^{\nu, t, \nabla_1}_{k_1,k_2,c_1,c_2}$ is defined similarly. 
The discrete Sobolev norm $W^{(\tilde{r},\tilde{r}),2}$ for a function $f$ is defined by 
\begin{equation*}
\|f\|_{W^{(\tilde{r},\tilde{r}),2}} := \left(\sum_{n_1,n_2 \in \mathbb{Z}^d} (1+|n_1|^2)^{\tilde{r}} (1+|n_2|^2)^{\tilde{r}}|\hat{f}(n_1,n_2)|^2  \right)^{\frac{1}{2}}.
%\|f\|_{W^{(\tilde{r},\tilde{r}),2}} := \left(\int_{\BBR^{2d}} (1+|x|^2)^{\tilde{r}} (1+|y|^2)^{\tilde{r}}|\hat{f}(x,y)|^2 dxdy \right)^{\frac{1}{2}}.
\end{equation*}
%$$2^{-k} 2^{l(\nu+1)} (|2^{k-l} \xi_1+\xi_2|^s m(2^k \xi_1, 2^l \xi_2)- |\xi_2|^s m(0, \xi_2))$$
%$$
%2^{-l (\nu-1)} \nabla_1 m(t 2^k \xi_1,2^l \xi_2),
%$$
Using the argument in the proof of Theorem \ref{thm:smoothing:N:param}, \eqref{eq:smoothing:1} can be verified for symbols $m$ satisfying the conditions above for $0 \leq \nu < 2d$, $\max(\frac{d}{\tilde{r}}, \frac{d}{d+s}) < r < \infty$, $1 < p_1, p_2 < \infty$ and $\frac{1}{r} = \frac{1}{p_1} + \frac{1}{p_2}$. A similar Sobolev condition can be formulated in the biparameter setting. 
%We do require a certain regularity from our symbols... in particular we could ask them to be $C^1$ away from the origin... or of bounded variation in each variable on any dyadic interval (as in the Marcinkiwecz condition); since we do not recover results as general as in ....., we will simply assume the $C^1$ condition; or see what we want to assume 
} we do recover the results from \cite{HartTorresWu-smoothing-bil} and from \cite{YangLiuWu-smoothing} for fractional integral operators also in the $N$-parameter, mixed-norm setting. Once more, the purpose of this section is only illustrative.

\begin{theorem}
\label{thm:smoothing:N:param}
Let $0 \leq \nu_1, \ldots, \nu_N < nd$ and let $m \in \mathcal{M}_{- \nu_1, \ldots, -\nu_N}(\BBR^{d_1} \times \ldots \times \BBR^{d_N})$ be a Mikhlin symbol of order $(- \nu_1, \ldots, -\nu_N)$. Then for any $s_1, \ldots, s_N \geq 0$ so that $s_j \neq \nu_j$ for any $1 \leq j \leq N$ and for any functions $f_1, \ldots, f_n \in \mathcal{S}(\BBR^{d_1} \times \ldots \times \BBR^{d_N})$,
we have
\begin{align}\label{eq:smoothing:multi:param}
\|D^{s_1}_{(1)}\ldots D^{s_N}_{(N)} T_m(f_1, \ldots, f_n)\|_{L^{\vec r}} \lesssim   \sum_{\substack{\sigma_1^1, \ldots, \sigma^N_n \in \{ 0, 1 \} \\ \sigma_1^j+ \ldots +\sigma_n^j=1}}\prod_{l=1}^n\|D^{\sigma_l^1 (s_1- \nu_1)}_{(1)} \ldots D^{\sigma_l^N (s_N-\nu_N)}_{(N)} f_{l}\|_{\vec p_l},
\end{align}
provided that $1< p_1^1, \ldots, p_1^N, p_2^1, \ldots, p_2^n, \ldots, p_n^1, \ldots, p_n^N < \infty$, $\frac{1}{n}<r^1, \ldots, r^n < \infty$ are satisfying component-wise the H\"older condition 
\[
\frac{1}{\vec r}= \frac{1}{ \vec p_1}+\ldots+\frac{1}{\vec p_n}
\] 
and 
\[
\frac{d_N}{d_N+s_N}< r^N, \quad \max(\frac{d_N}{d_N+s_N}, \frac{d_{N-1}}{d_{N-1}+s_{N-1}})< r^{N-1}, \ldots, \max \big( \frac{d_N}{d_N+s_N}, \ldots, \frac{d_1}{d_1+s_1}  \big)< r^1.
\]
\end{theorem}

We call attention to the fact that \eqref{eq:smoothing:multi:param} remains true whenever $s_1-\nu_1 \neq 0, \ldots, s_N- \nu_N \neq 0$, so even when they are negative numbers. 

A typical example of multipliers satisfying the boundedness property in Theorem \ref{thm:smoothing:N:param} is the $N$-parameter fractional integral operator $I_\nu(f_1, \ldots, f_n)$ given by
\begin{align*}
 \int_{\BBR^{nd_1} \times \ldots \times \BBR^{nd_N}} (|\xi^1_1|^2 + \ldots +|\xi^1_n|^2)^{-\frac{\nu_1}{2}} \ldots (|\xi_1^N|^2 + \ldots + |\xi_n^N|^2)^{-\frac{\nu_N}{2}}\hat{f_1}(\xi_1) \ldots \f{f_n}(\xi_n)e^{2\pi i x\cdot(\xi_1 + \ldots \xi_n)} d\xi^1 \ldots d\xi^N,
\end{align*}
where $\xi^j := (\xi^j_1, \ldots, \xi^j_n) \in \BBR^n$ for $1 \leq j \leq N$. In this particular case we can of course allow for $s_j=\nu_j$ and \eqref{eq:smoothing:multi:param} holds for $D^{s_1}_{(1)}\ldots D^{s_N}_{(N)} I_\nu(f_1, \ldots, f_n)$. The smoothing property for fractional integral operators, naturally implied by Theorem \ref{thm:smoothing:N:param}, seem to be new in the  multi-parameter, mixed-norm setting.

Again, for simplicity we only present the case $n=2$ and $N=2$ of the proof. First, we start with the one-parameter case, which relies on the results in Theorem \ref{thm:multipliers:positive:order} corresponding to a rooted tree of depth $1$; next we discuss the bi-parameter case, which makes use of Theorem \ref{thm:multiparameter_nontensor}.

\begin{proof}[Sketch of Proof of Theorem \ref{thm:smoothing:N:param} when $N=1$]

Here, like in \cite{HartTorresWu-smoothing-bil}, we would like to obtain smoothing properties for operators of order $\nu$, where $0 \leq \nu< 2d$. This simply means that we take $m(\xi, \eta) \in \mathcal{M}_{-\nu}(\BBR^{2d})$ and would like to prove for any $s \neq \nu$
\begin{align}
\label{eq:smoothing:1}
\|D^s T_m(f, g)\|_r \lesssim \|D^{s-\nu}f\|_{p_1} \|g\|_{p_2} + \|f\|_{p_1} \|D^{s-\nu} g\|_{p_2}.
\end{align}

As usual, $D^s T_m(f, g)$ breaks down as a sum of
\begin{align*}
&\sum_{k \ll \ell} \int_{\BBR^{2d}} |\xi_1+\xi_2|^s m_\nu(\xi_1, \xi_2) \widehat{\Delta_k f}(\xi_1) \widehat{\Delta_\ell g}(\xi_2) e^{2 \pi i x \cdot (\xi_1+\xi_2)} d \xi_1 d \xi_2 \\
& +\sum_{k \gg \ell} \int_{\BBR^{2d}} |\xi_1+\xi_2|^s m_\nu(\xi_1, \xi_2) \widehat{\Delta_k f}(\xi_1) \widehat{\Delta_\ell g}(\xi_2) e^{2 \pi i x \cdot (\xi_1+\xi_2)} d \xi_1 d \xi_2  \\
& +\sum_{ |k - \ell| \leq 3} \int_{\BBR^{2d}} |\xi_1+\xi_2|^s m_\nu(\xi_1, \xi_2) \widehat{\Delta_k f}(\xi_1) \widehat{\Delta_\ell g}(\xi_2) e^{2 \pi i x \cdot (\xi_1+\xi_2)} d \xi_1 d \xi_2 :=I+II+III
\end{align*}
and it will be enough to estimate the terms $I$ and $III$.

\begin{enumerate}[leftmargin=*]
\item For $I$, we simply notice that on the region 
\[
R:=\bigcup_{k \ll \ell} \{ (\xi_1, \xi_2): \quad  |\xi_1| \sim 2^k, \quad |\xi_2| \sim 2^\ell  \},
\]
the symbol $ |\xi_1+\xi_2|^s m_\nu(\xi_1, \xi_2)$ becomes a Mikhlin symbol of order $s-\nu$. So as long as $s>\nu$, the result follows from Theorem \ref{thm:multipliers:positive:order}. When $s< \nu$, the proof of Theorem \ref{thm:multipliers:positive:order} still holds: in that case, we would use $\| \cdot \|_{\dot B^{s-\nu}_{p_l, \infty}}$ and $\| \cdot \|_{\dot B^{-\epsilon}_{p_l, \infty}}$ Besov norms, with $s-\nu < - \epsilon< 0$.

\item For the diagonal term, we want to estimate 
\[
III:= \sum_{ k } \int_{\BBR^{2d}} |\xi_1+\xi_2|^s m_\nu(\xi_1, \xi_2) \widehat{\Delta_k f}(\xi_1) \widehat{\Delta_k g}(\xi_2) e^{2 \pi i x \cdot (\xi_1+\xi_2)} d \xi_1 d \xi_2.
\]

For this, we would like to use a Fourier series decomposition for $|\xi_1+\xi_2|^s m_\nu(\xi_1, \xi_2)$ in the region $\{(\xi_1, \xi_2): |\xi_1| \sim |\xi_2| \}$, which can be decomposed as \eqref{diagonal:whitney:highdim}:
\begin{align*}
\{(\xi_1, \xi_2) \in \BBR^{2d}: |\xi_1| \sim |\xi_2| \} = \bigcup_{c_1, c_2 \in \mathfrak{C}}\bigcup_{k \in \mathbb{Z}}( \{|\xi_1| \sim 2^{k} \} \cap c_1) \times ( \{|\xi_2| \sim 2^{k} \} \cap c_2),
\end{align*}
where $\mathfrak{C}$ is the collection of all directional cones in the frequency space $\BBR^{d}$.

We proceed with a Fourier series decomposition of the symbol restricted to the Whitney cube corresponding to a fixed scale $k \in \mathbb{Z}$ and fixed directional cones $c_1, c_2, \in \mathfrak{C}$, namely
\begin{equation} \label{symbol:smoothing:2para}
|\xi_1+\xi_2|^s m_\nu(\xi_1, \xi_2) \psi_{k.c}(\xi_1) \psi_{k,c_2}(\xi_2).
\end{equation}

In this region, while $m_\nu(\xi_1, \xi_2)$ is smooth, $|\xi_1+\xi_2|^s$ is less regular since we cannot exclude that $|\xi_1+\xi_2|=0$. So we proceed with Fourier series decompositions on both symbols separately. We first notice that \eqref{symbol:smoothing:2para} can be rewritten as
\begin{equation*}
\left(|\xi_1+\xi_2|^s \tilde \varphi_{k}(\xi_1+\xi_2)\right) \left(m_\nu(\xi_1, \xi_2) \psi_{k.c}(\xi_1) \psi_{k,c_2}(\xi_2)\right).
\end{equation*}

Then the Fourier series of the symbol in the first parenthesis yields 
\begin{equation} \label{eq:one_parameter_smoothing_diff}
|\xi_1+\xi_2|^s \tilde \varphi_k(\xi_1+\xi_2)= \sum_{L \in \BBZ^d} 2^{k s} C_L e^{2 \pi i {L \over {2^k}} \cdot (\xi_1+\xi_2)}, 
\end{equation}
where the renormalized Fourier coefficients $C_L$ satisfy
\begin{equation} \label{fourier_coef:one_parameter_smoothing_diff}
|C_L| \lesssim \frac{1}{\left( 1+|L|  \right)^{d+s}}.
\end{equation}

Next, we use a Fourier decomposition on $m_\nu(\xi_1, \xi_2) \psi_{k,c_1}(\xi_1) \psi_{k,c_2}(\xi_2)$:
\begin{align*}
m_\nu(\xi_1, \xi_2) \psi_{k,c_1}(\xi_1) \psi_{k,c_2}(\xi_2)= \sum_{L_1, L_2 \in \BBZ^d} C_{L_1, L_2}^{k,c_1,c_2} e^{2 \pi i {L_1 \over {2^k}} \cdot \xi_1} e^{2 \pi i {L_2 \over {2^k}} \cdot \xi_2},
\end{align*}
with 
\[
\big|  C_{L_1, L_2}^{k,c_1,c_2}  \big| \lesssim_M \frac{2^{-k \nu}}{\left( 1+ |L_1|+|L_2|  \right)^M}.
\]
for $M >0$ sufficiently large.
Overall we get 
\begin{align*}
|\xi_1+\xi_2|^s m_\nu(\xi_1, \xi_2) \psi_{k,c_1}(\xi_1) \psi_{k,c_2}(\xi_2) = \sum_{L \in \BBZ^d} \sum_{L_1, L_2 \in \BBZ^d} 2^{k s} C_L C_{L_1, L_2}^{k,c_1,c_2} e^{2 \pi i { {L+ L_1} \over {2^k}} \cdot \xi_1} e^{2 \pi i {{L+L_2} \over {2^k}} \cdot \xi_2}
\end{align*}
and in consequence, whenever $0<\tau \leq \min(1,r)$,
\begin{align*}
\|III\|_r^\tau \lesssim \sum_L |C_L|^\tau \sum_k  2^{k (s - \nu) \tau}  \|\Delta_k f\|_{p_1}^\tau \|\Delta_k   g\|_{p_2}^\tau.
\end{align*}
This imposes the restriction $r> \frac{d}{d+s}$, and implies the desired \eqref{eq:smoothing:multi:param} for either $s>\nu$ or $s<\nu$.
\end{enumerate}
\end{proof}

\begin{proof}[Sketch of Proof of Theorem \ref{thm:smoothing:N:param} when $N=2$.]
Now we prove the bi-parameter version of the smoothing property discussed above; that is, we will show that 
\begin{align}
\label{eq:smooth:bi:param}
\|D^{s_1}_{(1)}D^{s_2}_{(2)} T_m(f_1, f_2)\|_{L^{r^1}_x(L^{r^2}_y)} \lesssim & \|D_{(1)}^{s_1-\nu_1}D_{(2)}^{s_2-\nu_2}f_1\|_{L^{p_1^1}_x(L^{p_1^2}_y)}\|f_2\|_{L^{p_2^1}_x(L^{p_2^2}_y)} +  \|f_1\|_{L^{p_1^1}_x(L^{p_1^2}_y)}\|D_{(1)}^{s_1-\nu_1}D_{(2)}^{s_2-\nu_2}f_2\|_{L^{p_2^1}_x(L^{p_2^2}_y)}  \nonumber \\
&+  \|D_{(1)}^{s_1-\nu_1}f_1\|_{L^{p_1^1}_x(L^{p_1^2}_y)}\|D_{(2)}^{s_2-\nu_2}f_2\|_{L^{p_2^1}_x(L^{p_2^2}_y)} +  \|D_{(2)}^{s_2-\nu_2}f_1\|_{L^{p_1^1}_x(L^{p_1^2}_y)}\|D_{(1)}^{s_1-\nu_1}f_2\|_{L^{p_2^1}_x(L^{p_2^2}_y)},
\end{align}
for $1 < p_1^j, p_2^j < \infty$, $\frac{d}{d+s_j} < r^j$ and $\frac{1}{p_1^j} + \frac{1}{p_2^j} = \frac{1}{r^j}$ with $j = 1,2$.

As indicated by the argument in the one-parameter setting, various considerations are necessary for different regions of the frequency space.  
\vskip .05in
\noindent
\textbf{(1)}
In the ``off-diagonal'' region
$$
\{(\xi_1,\xi_2, \eta_1, \eta_2): |\xi_1| \gg |\xi_2|, |\eta_1| \gg |\eta_2| \}
$$
for both parameters (or other similar regions, obtained by permuting the roles of $\xi_1, \xi_2$ or $\eta_1, \eta_2$), the symbols $|\xi_1+\xi_2|^{s_1}$ and $|\eta_1+\eta_2|^{s_2}$ are Mikhlin symbols of order $s_1$ and $s_2$ respectively. Hence
$$
m_{s_1-\nu_1, s_2-\nu_2}(\xi_1,\xi_2, \eta_1,\eta_2) := |\xi_1+\xi_2|^{s_1}|\eta_1+\eta_2|^{s_2}m(\xi_1, \xi_2, \eta_1, \eta_2)
$$ is a bi-parameter Marcinkiewicz symbol satisfying 
\begin{equation*}
|\partial^{\zeta_1}_{\xi}\partial^{\zeta_2}_{\eta}m_{\beta_1,\beta_2}(\xi, \eta)| \lesssim |\xi|^{s_1-\nu_1 - |\zeta_1|} |\eta|^{s_2-\nu_2 - |\zeta_2|}.
\end{equation*}
We can now invoke Theorem \ref{thm:multiparameter_nontensor} to conclude the discussion for this case.
\vskip .05in
\noindent
\textbf{(2)}
In the diagonal region for both parameters
$$
\{(\xi_1,\xi_2, \eta_1, \eta_2): |\xi_1| \sim |\xi_2|, |\eta_1| \sim |\eta_2| \},
$$
we no longer need to split the symbol and consider commutator terms. Instead, we perform a quadruple Fourier series decomposition of $m(\xi_1,\xi_2,\eta_1,\eta_2)$ restricted to a Whitney rectangle and a Fourier series decompositions of the symbols $|\xi_1+\xi_2|^{s_1}$ and $|\eta_1+\eta_2|^{s_2}$ respectively, restricted accordingly. The subsequent analysis mimics the argument developed in Section \ref{sec:2param:5linflag}.
\vskip .05in
\noindent
\textbf{(3)}
In the region that is ``off-diagonal" for the first parameter and  ``diagonal" for the second parameter (and in the other similar regions, alike)
$$
\{(\xi_1,\xi_2, \eta_1, \eta_2): |\xi_1| \gg |\xi_2|, |\eta_1| \sim |\eta_2| \},
$$
we will use a hybrid of the arguments for \textbf{(1)} and \textbf{(2)}. In particular,
 \begin{align*}
&|\xi_1 +\xi_2|^{s_1}m(\xi_1,\xi_2, \eta_1, \eta_2) \\
= &\underbrace{|\xi_1 + \xi_2|^{s_1}m(\xi_1,\xi_2, \eta_1, \eta_2) - |\xi_1|^{s_1}m(\xi_1,0,\eta_1, \eta_2)}_{=:m_{C_{s_1-\nu_1}, -\nu_2}(\xi_1, \xi_2, \eta_1, \eta_2)} + \underbrace{|\xi_1|^{s_1}m(\xi_1, 0, \eta_1, \eta_2)}_{=:m_{s_1-\nu_1, -\nu_2}(\xi_1, \eta_1, \eta_2)}.
 \end{align*}
 
The quadruple Fourier series decomposition can be applied to the symbol $m_{C_{s_1-\nu_1}, -\nu_2}$ smoothly restricted to the region 
\begin{equation}\label{loc:biparameter_smoothing}
(\{|\xi_1| \sim 2^{k_1}\} \cap c_1) \times (\{|\xi_2| \sim 2^{k_2}\} \cap c_2) \times (\{|\eta_1| \sim 2^{m_1}\} \cap c_1') \times (\{|\eta_1| \sim 2^{m_1}\} \cap c_2'),
%\{(\xi_1, \xi_2, \eta_1, \eta_2): |\xi_1| \sim 2^{k_1}, |\xi_2| \sim 2^{k_2}, |\eta_1| \sim |\eta_2| \sim 2^{m_1} \}
\end{equation}
where $k_1 \gg k_2$ and $c_1, c_2 \in \mathfrak{C}_1$ are directional cones in the frequency space for the first parameter and $c_1', c_2' \in \mathfrak{C}_2$ are directional cones in the frequency space for the second parameter.
 
Meanwhile, we can use the Fourier series decomposition to express the symbol $|\eta_1+\eta_2|^{s_2}$ localized to the cube $[-2^{m_1+1}, 2^{m_1+1}]^{d_2}$ which yields an expression similar to \eqref{eq:one_parameter_smoothing_diff}. The corresponding Fourier coefficients satisfy a decaying condition similar to \eqref{fourier_coef:one_parameter_smoothing_diff}, which imposes the constraint $\ds r^2 > \frac{d_2}{d_2+s_2}$ on the Lebesgue exponent $r^2$. Next, we use the Fourier series representation for the symbol $m_{C_{s_1-\nu_1}, -\nu_2}|\eta_1+\eta_2|^{s_2}$ in the multiplier; the remaining argument is routine and will be omitted here.
 
Last but not least, the symbol $m_{s_1-\nu_1, -\nu_2}(\xi_1, \eta_1, \eta_2)$ can be rewritten as
\begin{equation*}
m_{s_1-\nu_1, -\nu_2}(\xi_1, \eta_1, \eta_2) = \frac{m_{s_1-\nu_1, -\nu_2}(\xi_1, \eta_1, \eta_2)}{|\xi_1|^{s_1-\nu_1}}|\xi_1|^{s_1-\nu_1} =: m_{H, -\nu_2}(\xi_1, \eta_1, \eta_2) |\xi_1|^{s_1-\nu_1}.
\end{equation*}

We smoothly restrict the symbol $m_{H, -\nu_2}(\xi_1, \eta_1, \eta_2)$ to the region 
$$
(\{|\xi_1| \sim 2^{k_1}\} \cap c_1) \times (\{|\eta_1| \sim 2^{m_1}\} \cap c_1') \times  (\{|\eta_2| \sim 2^{m_1}\} \cap c_2'),
%\{(\xi_1, \eta_1, \eta_2): |\xi_1| \sim 2^{k_1}, |\eta_1| \sim |\eta_2| \sim 2^{m_1}\}
$$
where $k_1, m_1 \in \mathbb{Z}$ are fixed, $c_1\in \mathfrak{C}_1$ is a  directional cone in the frequency space for the first parameter and $c_1', c_2' \in \mathfrak{C}_2$ are directional cones in the frequency space for the second parameter. We denote the localized symbol by $m^{k_1, m_1}_{H, -\nu_2}$, on which we perform the triple Fourier series decomposition: 
\begin{equation*}
m^{k_1, m_1,c_1,c_1',c_2'}_{H, -\nu_2}(\xi_1, \eta_1, \eta_2) = \sum_{\substack{L \in \mathbb{Z}^{d_1} \\ L_1', L_2' \in \mathbb{Z}^{d_2}}}C^{k_1, m_1, c_1, c_1',c_2'}_{L, L_1',L_2'} e^{2 \pi i  L \cdot \frac{\xi_1}{2^{k_1}}}e^{2 \pi i  L_1' \cdot \frac{\eta_1}{2^{m_1}}}e^{2 \pi i  L_2' \cdot \frac{\eta_2}{2^{m_1}}},
\end{equation*} 
where the Fourier coefficients satisfy the decay condition 
\begin{equation*}
|C^{k_1, m_1,c_1,c_1',c_2'}_{L, L_1',L_2'}| \lesssim 2^{-m_1\nu_2}(1+|L|)^{-N}(1+|L_1'|+|L_2'|)^{-N'}
\end{equation*}
for sufficiently large $N, N'$.
When combined with the Fourier series representation for $|\eta_1+\eta_2|^{s_2}$ on $\{ |\eta_1+\eta_2| \leq 2^{m_1+1}  \}$, we obtain for fixed $c_1 \in \mathfrak{C_1}$ and $c_1',c_2' \in \mathfrak{C}_2$,
\begin{align*}
&\sum_{\substack{L \in \mathbb{Z}^{d_1} \\ L_1', L_2', L' \in \mathbb{Z}^{d_2}}} \sum_{\substack{k_1 \gg k_2 \\ m_1}}C^{k_1, m_1,c_1,c_1',c_2'}_{L, L_1', L_2'} C_{L'} 2^{m_1 s_1}\\
& \int |\xi_1|^{s_1-\nu_1}\ii F(\Delta_{k_1,c_1}^{(1)}\Delta_{m_1,c_1'}^{(2)}f_1)(\xi_1, \eta_1) \ii F(\Delta_{k_2}^{(1)}\Delta_{m_1,c_2'}^{(2)}f_2) (\xi_2, \eta_2)e^{2 \pi i \xi_1\cdot (x + \frac{L}{2^{k_1}})}e^{2 \pi i \xi_2\cdot x} e^{2 \pi i \eta_1 \cdot (y+  \frac{L_1'+L'}{2^{m_1}})}e^{2 \pi i  \eta_2 \cdot (y+\frac{L_2'+L'}{2^{m_1}})} d\xi d\eta .
\end{align*}

We then use the high-low switch technique to split the sum in two parts -- one with the sum over all $k_1, k_2, m_1$ and the second sum over $k_2 \succ k_1$ and all $m_1$. The second sum can be estimated using the usual analysis (optimization, Besov norms, etc), while the first can be simplified as 
\begin{align*}
\sum_{\substack{L \in \mathbb{Z}^{d_1} \\ L_1', L_2', L' \in \mathbb{Z}^{d_2}}} \sum_{k_1, m_1\in \mathbb{Z}} C^{k_1, m_1,c_1,c_1',c_2'}_{L, L_1', L_2'}C_{L'} 2^{m_1 s_1}\Delta_{k_1,c_1}^{(1)}\Delta_{m_1,c_1'}^{(2)}D^{s_1-\nu_1}_{(1)}f_1(x+\frac{L}{2^{k_1}}, y +  \frac{L_1'+L'}{2^{m_1}})\Delta_{m_1,c_2'}^{(2)}f_2(x, y + \frac{L_2'+L'}{2^{m_1}}). 
\end{align*}
By subadditivity and H\"older's inequality, its $\|\cdot \|_{L^{r^1}_x(L^{r^2}_y)}^\tau$ norm with $\tau \leq \min(1, r^1, r^2)$ can be estimated by 
\begin{align*}
\sum_{\substack{L \in \mathbb{Z}^{d_1} \\ L_1', L_2', L' \in \mathbb{Z}^{d_2}}} \sum_{m_1}|C_{L'}|^\tau 2^{m_1 s_1\tau} \big\|\sum_{k_1}C^{k_1, m_1,c_1,c_1',c_2'}_{L, L_1', L_2'}\Delta_{k_1,c_1}^{(1)}\Delta_{m_1,c_1'}^{(2)}D^{s_1-\nu_1}_{(1)}f_1(\cdot+\frac{L}{2^{k_1}}, \cdot )\big\|_{L^{p_1^1}_x(L^{p_1^2}_y)}^\tau \big\|\Delta_{m_1,c_2'}^{(2)}f_2(\cdot, \cdot)\big\|_{L^{p_2^1}_x(L^{p_2^2}_y)}^\tau.
\end{align*}
The reasoning used for proving the inequality \eqref{claim_using_sq} also implies that
\begin{equation*}
\big\|\sum_{k_1}C^{k_1, m_1,c_1,c_1',c_2'}_{L, L_1', L_2'}\Delta_{k_1,c_1}^{(1)}\Delta_{m_1,c_1'}^{(2)}D^{s_1-\nu_1}_{(1)}f_1(\cdot+\frac{L}{2^{k_1}}, \cdot)\big\|_{L^{p_1^1}_x(L^{p_1^2}_y)} \lesssim (1+|L|)^{-N+1}(1+|L_1'|+|L_2'|)^{-N'}2^{-m_1\nu_2}\|\Delta_{m_1}^{(2)}D^{s_1-\nu_1}_{(1)}f_1\|_{L^{p_1^1}_x(L^{p_1^2}_y)}.
\end{equation*}
From here on the summation in $m_1$ is standard.

The reasoning presented above yields the desired estimates in the case when $s_1> \nu_1, s_2>\nu_2$. As in the one-parameter case, a similar and even simpler\footnote{The high-low switch procedure is no longer necessary if $s_2<\nu_2$, when one needs to sum \[
\sum_{m_1} 2^{m_1(s_2-\nu_2)\tau} \|\Delta_{m_1}^{(2)}D^{s_1-\nu_1}_{(1)}f_1\|_{L^{p_1^1}_x(L^{p_1^2}_y)}^\tau \|\Delta_{m_1}^{(2)}f_2\|_{L^{p_2^1}_x(L^{p_2^2}_y)}^\tau.
\]} argument implies \eqref{eq:smooth:bi:param} for any $s_1 \neq \nu_1, s_2 \neq \nu_2$.
\end{proof}

\bibliographystyle{plain}
\bibliography{MultiParamFlagLeibniz.bbl}
%\bibliography{harmonic.bib}
\end{document}